\documentclass[a4paper,11pt]{article}
\pdfoutput=1
\usepackage{import}
\usepackage{lipsum}
\usepackage{mainpreamble}

\title{Presentations for the ghost algebra and the label algebra}
\author{Madeline Nurcombe}
\date{}

\begin{document}
\maketitle
\begin{abstract}
The ghost algebra is a two-boundary extension of the Temperley-Lieb algebra, constructed recently via a diagrammatic presentation. The existing two-boundary Temperley-Lieb algebra has a basis of two-boundary string diagrams, where the number of strings connected to each boundary must be even. The ghost algebra is similar, but allows this number to be odd, using bookkeeping dots called \textit{ghosts} to assign a consistent parity to each string endpoint on each boundary. Equivalently, one can discard the ghosts and label each string endpoint with its parity; the resulting algebra is readily generalised to allow any number of possible labels, instead of just odd or even. We call the generalisation the \textit{label algebra}, and establish a non-diagrammatic presentation for it. A similar presentation for the ghost algebra follows from this.

\end{abstract}

\tableofcontents

\pagebreak

\section{Introduction}\label{s:intro}
The ghost algebra \cite{Ghost,Thesis} is a two-boundary extension of the Temperley-Lieb (TL) algebra \cite{TL,JonesTL}, introduced as an alternative to the existing two-boundary TL algebra \cite{MNdGB,dGN}. Similar to the two-boundary TL algebra, the ghost algebra is an associative unital algebra with a diagrammatic presentation. Both algebras have a basis of rectangular string diagrams, where non-crossing strings may be connected to all four sides of the rectangle, including the two opposing sides identified as \textit{boundaries}. The ghost algebra primarily differs from the two-boundary TL algebra by including basis diagrams with an odd number of strings connected to each boundary; in the two-boundary TL algebra, this number must be even. In order to preserve both associativity, and the desired subalgebras isomorphic to the one-boundary TL algebra \cite{MartinSaleur}, basis diagrams of the ghost algebra can include bookkeeping dots called \textit{ghosts} on their boundaries, such that the number of strings plus ghosts on each boundary is even. For example, diagrams such as
\begin{align}
\begin{tikzpicture}[baseline={([yshift=-1mm]current bounding box.center)},scale=0.5]
{
\draw[dotted] (0,0.5)--(3,0.5);
\draw[dotted] (0,-4.5)--(3,-4.5);
\draw [very thick](0,-4.5) -- (0,0.5);
\draw [very thick](3,-4.5)--(3,0.5);
\draw (0,0) arc (-90:0:0.6 and 0.5);
\draw (0,-1) to[out=0,in=180] (3,-3);
\draw (0,-2) arc(90:-90:0.5);
\draw (0,-4)--(3,-4);
\draw (3,0) arc(90:270:0.5);
\draw (3,-2) to[out=180,in=270] (1.8,0.5);
}
\end{tikzpicture}\; ,
&&
\begin{tikzpicture}[baseline={([yshift=-1mm]current bounding box.center)},scale=0.5]
{
\draw[dotted] (0,0.5)--(3,0.5);
\draw[dotted] (0,-4.5)--(3,-4.5);
\draw [very thick](0,-4.5) -- (0,0.5);
\draw [very thick](3,-4.5)--(3,0.5);
\draw (0,0) arc (90:-90:0.5);
\draw (0,-2) to[out=0,in=-90] (1.2,0.5);
\draw (0,-3) to[out=0,in=90] (1.2,-4.5);
\draw (0,-4) arc(90:0:0.6 and 0.5);
\draw (3,0) arc(270:180:0.6 and 0.5);
\draw (3,-1) to[out=180,in=90] (1.8,-4.5);
\draw (3,-2) arc(90:270:0.5);
\draw (3,-4) arc(90:180:0.6 and 0.5);
}
\end{tikzpicture} \; ,
&&
\begin{tikzpicture}[baseline={([yshift=-1mm]current bounding box.center)},scale=0.5]
{
\draw[dotted] (0,0.5)--(3,0.5);
\draw[dotted] (0,-4.5)--(3,-4.5);
\draw [very thick](0,-4.5) -- (0,0.5);
\draw [very thick](3,-4.5)--(3,0.5);
\draw (0,0) arc (90:-90:0.5);
\draw (0,-2) to[out=0,in=180] (3,0);
\draw (0,-3) arc(90:-90:0.5);
\draw (3,-1) ..controls (1.5,-1) and (1.5,-4).. (3,-4);
\draw (3,-2) arc(90:270:0.5);
}
\end{tikzpicture} \; ,
&&
\begin{tikzpicture}[baseline={([yshift=-1mm]current bounding box.center)},scale=0.5]
{
\draw[dotted] (0,0.5)--(3,0.5);
\draw[dotted] (0,-4.5)--(3,-4.5);
\draw [very thick](0,-4.5) -- (0,0.5);
\draw [very thick](3,-4.5)--(3,0.5);
\draw (0,0) arc(-90:0:0.6 and 0.5);
\draw (0,-1) to[out=0,in=-90] (1.2,0.5);
\draw (0,-2) arc(90:-90:0.5);
\draw (0,-4) to[out=0,in=-90] (1.8,0.5);
\draw (3,0) arc(270:180:0.6 and 0.5);
\draw (3,-1) arc(90:270:0.5);
\draw (3,-3) arc(90:270:0.5);
}
\end{tikzpicture}
\end{align}
are basis elements of both the two-boundary TL algebra and the ghost algebra---note that they each have an even number of strings connected to each dotted boundary. Diagrams such as
\begin{align}
\begin{tikzpicture}[baseline={([yshift=-1mm]current bounding box.center)},scale=0.5]
{
\draw[dotted] (0,0.5)--(3,0.5);
\draw[dotted] (0,-4.5)--(3,-4.5);
\draw [very thick](0,-4.5) -- (0,0.5);
\draw [very thick](3,-4.5)--(3,0.5);
\draw (0,0) arc(-90:0:0.6 and 0.5);
\draw (0,-1) to[out=0,in=180] (3,-4);
\draw (0,-2) arc(90:-90:0.5);
\draw (0,-4) arc(90:0:0.6 and 0.5);
\draw (3,0) arc(270:180:0.6 and 0.5);
\draw (3,-1) arc (90:270:0.5);
\draw (3,-3) to[out=180,in=270] (1.8,0.5);
}
\end{tikzpicture}
\; ,
&&
\begin{tikzpicture}[baseline={([yshift=-1mm]current bounding box.center)},scale=0.5]
{
\draw[dotted] (0,0.5)--(3,0.5);
\draw[dotted] (0,-4.5)--(3,-4.5);
\draw [very thick](0,-4.5) -- (0,0.5);
\draw [very thick](3,-4.5)--(3,0.5);
\draw (0,0) arc(-90:0:0.6 and 0.5);
\draw (0,-1) to[out=0,in=180] (3,0);
\draw (0,-2) to[out=0,in=180] (3,-1);
\draw (0,-3) to[out=0,in=180] (3,-2);
\draw (0,-4) to[out=0,in=180] (3,-3);
\draw (3,-4) arc(90:180:0.6 and 0.5);
}
\end{tikzpicture}
\; ,
&&
\begin{tikzpicture}[baseline={([yshift=-1mm]current bounding box.center)},scale=0.5]
{
\draw[dotted] (0,0.5)--(3,0.5);
\draw[dotted] (0,-4.5)--(3,-4.5);
\draw [very thick](0,-4.5) -- (0,0.5);
\draw [very thick](3,-4.5)--(3,0.5);
\draw (3,0) arc(270:180:0.5);
\draw (3,-1) ..controls (2.6,-1) and (2.2,-0.6).. (2.2,0.5);
\draw (3,-2) ..controls (2.4,-2) and (1.9,-0.8).. (1.9,0.5);
\draw (3,-3) ..controls (2.2,-3) and (1.6,-0.9).. (1.6,0.5);
\draw (3,-4) ..controls (2,-4) and (1.3,-1).. (1.3,0.5);
\draw (0,0) ..controls (1,0) and (1.7,-3).. (1.7,-4.5);
\draw (0,-1) ..controls (0.8,-1) and (1.4,-3.1).. (1.4,-4.5);
\draw (0,-2) ..controls (0.6,-2) and (1.1,-3.2).. (1.1,-4.5);
\draw (0,-3) ..controls (0.4,-3) and (0.8,-3.3).. (0.8,-4.5);
\draw (0,-4) arc(90:0:0.5);
}
\end{tikzpicture} \; ,
&&
\begin{tikzpicture}[baseline={([yshift=-1mm]current bounding box.center)},scale=0.5]
{
\draw[dotted] (0,0.5)--(3,0.5);
\draw[dotted] (0,-4.5)--(3,-4.5);
\draw [very thick](0,-4.5) -- (0,0.5);
\draw [very thick](3,-4.5)--(3,0.5);
\draw (0,0) to[out=0,in=180] (3,-1);
\draw (0,-1) arc(90:-90:0.5);
\draw (0,-3) to[out=0,in=180] (3,-2);
\draw (0,-4) arc(90:0:0.6 and 0.5);
\draw (3,0) arc(270:180:0.6 and 0.5);
\draw (3,-3) to[out=180,in=90] (1.8,-4.5);
\draw (3,-4) arc(90:180:0.6 and 0.5);
}
\end{tikzpicture}
\end{align}
are not basis diagrams of the two-boundary TL algebra, as they have an odd number of strings connected to each boundary. They are also not basis diagrams of the ghost algebra, as the number of strings plus ghosts is odd along each boundary. However, similar diagrams such as 
\begin{align}
\begin{tikzpicture}[baseline={([yshift=-1mm]current bounding box.center)},scale=0.5]
{
\draw[dotted] (0,0.5)--(3,0.5);
\draw[dotted] (0,-4.5)--(3,-4.5);
\draw [very thick](0,-4.5) -- (0,0.5);
\draw [very thick](3,-4.5)--(3,0.5);
\draw (0,0) arc(-90:0:0.6 and 0.5);
\draw (0,-1) to[out=0,in=180] (3,-4);
\draw (0,-2) arc(90:-90:0.5);
\draw (0,-4) arc(90:0:0.6 and 0.5);
\draw (3,0) arc(270:180:0.6 and 0.5);
\draw (3,-1) arc (90:270:0.5);
\draw (3,-3) to[out=180,in=270] (1.8,0.5);
\filldraw (1.2,0.5) circle (0.08);
\filldraw (1.8,-4.5) circle (0.08);
}
\end{tikzpicture}\; ,
&&
\begin{tikzpicture}[baseline={([yshift=-1mm]current bounding box.center)},scale=0.5]
{
\draw[dotted] (0,0.5)--(3,0.5);
\draw[dotted] (0,-4.5)--(3,-4.5);
\draw [very thick](0,-4.5) -- (0,0.5);
\draw [very thick](3,-4.5)--(3,0.5);
\draw (0,0) arc(-90:0:0.8 and 0.5);
\draw (0,-1) to[out=0,in=180] (3,0);
\draw (0,-2) to[out=0,in=180] (3,-1);
\draw (0,-3) to[out=0,in=180] (3,-2);
\draw (0,-4) to[out=0,in=180] (3,-3);
\draw (3,-4) arc(90:180:0.6 and 0.5);
\filldraw (0.4,0.5) circle (0.08);
\filldraw (1.2,-4.5) circle (0.08);
}
\end{tikzpicture}
\; ,
&&
\begin{tikzpicture}[baseline={([yshift=-1mm]current bounding box.center)},scale=0.5]
{
\draw[dotted] (0,0.5)--(3,0.5);
\draw[dotted] (0,-4.5)--(3,-4.5);
\draw [very thick](0,-4.5) -- (0,0.5);
\draw [very thick](3,-4.5)--(3,0.5);
\draw (3,0) arc(270:180:0.5);
\draw (3,-1) ..controls (2.6,-1) and (2.2,-0.6).. (2.2,0.5);
\draw (3,-2) ..controls (2.4,-2) and (1.9,-0.8).. (1.9,0.5);
\draw (3,-3) ..controls (2.2,-3) and (1.6,-0.9).. (1.6,0.5);
\draw (3,-4) ..controls (2,-4) and (1.3,-1).. (1.3,0.5);
\draw (0,0) ..controls (1,0) and (1.7,-3).. (1.7,-4.5);
\draw (0,-1) ..controls (0.8,-1) and (1.4,-3.1).. (1.4,-4.5);
\draw (0,-2) ..controls (0.6,-2) and (1.1,-3.2).. (1.1,-4.5);
\draw (0,-3) ..controls (0.4,-3) and (0.8,-3.3).. (0.8,-4.5);
\draw (0,-4) arc(90:0:0.5);
\filldraw (0.65,0.5) circle (0.08);
\filldraw (2.35,-4.5) circle (0.08);
}
\end{tikzpicture} \; ,
&&
\begin{tikzpicture}[baseline={([yshift=-1mm]current bounding box.center)},scale=0.5]
{
\draw[dotted] (0,0.5)--(3,0.5);
\draw[dotted] (0,-4.5)--(3,-4.5);
\draw [very thick](0,-4.5) -- (0,0.5);
\draw [very thick](3,-4.5)--(3,0.5);
\draw (0,0) to[out=0,in=180] (3,-1);
\draw (0,-1) arc(90:-90:0.5);
\draw (0,-3) to[out=0,in=180] (3,-2);
\draw (0,-4) arc(90:0:0.6 and 0.5);
\draw (3,0) arc(270:180:0.6 and 0.5);
\draw (3,-3) to[out=180,in=90] (1.4,-4.5);
\draw (3,-4) arc(90:180:0.8 and 0.5);
\filldraw (1.25,0.5) circle (0.08);
\filldraw (1,-4.5) circle (0.08);
\filldraw (1.8,-4.5) circle (0.08);
\filldraw (2.6,-4.5) circle (0.08);
}
\end{tikzpicture}
\end{align}
are indeed basis diagrams of the ghost algebra, because the number of string endpoints plus ghosts on each boundary is even. 

In both algebras, multiplication is by concatenation of basis diagrams, extended bilinearly. Concatenation may produce strings connected to the boundaries at both ends, called \textit{boundary arcs}, or loops; each of these is removed and replaced by a factor of one of the defining parameters of the algebra, except for top-to-bottom boundary arcs in the two-boundary TL algebra, which are not removed.

One motivation for constructing the ghost algebra was that, by relaxing the requirement of an even number of strings connected to each boundary, it would encode more general boundary conditions in the corresponding statistical mechanical lattice model, and thus potentially describe more general physical behaviour, than the two-boundary TL algebra. As seen in \cite{Ghost}, the ghost algebra does admit more general solutions to the boundary Yang-Baxter equation than the two-boundary TL algebra, suggesting that this may be the case, though the resulting physics has not yet been studied.

To better understand the ghost algebra, and its representation theory, we seek a non-diagrammatic presentation. As we will see in Section \ref{ss:ghostdef}, where we define the ghost algebra, the inclusion of ghosts in ghost algebra basis diagrams serves only to consistently assign a parity to each string endpoint on each boundary. We can thus obtain an isomorphic algebra by removing the ghosts, and labelling each string endpoint on each boundary as even or odd. It is natural to generalise the resulting algebra by allowing labels from any fixed set $X$; we call this the \textit{label algebra} $\lab_n(X)$. 

Our goal for this paper is thus to find a non-diagrammatic presentation for the label algebra, from which a similar presentation for the ghost algebra may be deduced. Our strategy is to define an algebra via a list of non-diagrammatic generators and relations---the \textit{algebraic label algebra} $\A_n(X)$---and show that this is isomorphic to the \textit{diagrammatic label algebra} $\lab_n(X)$. We establish an analogous set of generators for the diagrammatic label algebra, and define a map $\phi$ that sends generators of $\A_n(X)$ to the corresponding generators of $\lab_n(X)$. Noting that this map preserves the relations of $\A_n(X)$, it follows that $\phi$ is a surjective homomorphism. We then show that $\phi$ is injective, and thus the algebraic and diagrammatic label algebras are isomorphic.

A more detailed layout of the paper is as follows.

We begin by presenting the diagrammatic definitions of the ghost algebra and the label algebra, in Sections \ref{ss:ghostdef} and \ref{ss:label}, respectively. We also show that the label algebra has a specialisation with two labels that is isomorphic to the ghost algebra, as indicated above.

Next, in Section \ref{s:diagramgenerators}, we establish a set of diagrammatic generators for $\lab_n(X)$. We write each basis diagram as a product of the diagrammatic generators that does not produce any boundary arcs or loops. We separate our basis diagrams into \textit{even} and \textit{odd} diagrams, according to the parity of the number of strings connected to each boundary. In particular, we write each odd diagram as a product $\undertilde{W}\undertilde{T}$, where $\undertilde{W}$ is an odd diagram of a particular form, and $\undertilde{T}$ is an even diagram.

In Section \ref{s:algrels}, we introduce the algebraic label algebra $\A_n(X)$ by a set of generators and relations, distinguishing this from the diagrammatic label algebra $\lab_n(X)$; we wish to show that these algebras are isomorphic. We thus introduce a homomorphism $\phi:\A_n(X) \to \lab_n(X)$, defined by mapping algebraic generators to diagrammatic generators, and deduce from the results in Section \ref{s:diagramgenerators} that it is surjective. The goal for the remainder of the paper is then to show that $\phi$ is injective.

In Section \ref{s:ParityReduced}, we introduce the notion of parity of words in $\A_n(X)$, and two generalisations of the idea of a reduced word. The concept of an \textit{even-reduced} word is particularly useful as it is closely tied to the notion of \textit{reduced monomials} of the symplectic blob algebra \cite{TowersI}, a finite-dimensional quotient of the two-boundary TL algebra, discussed later in Section \ref{s:sympblob}. The term \textit{label-reduced} is more basic, and applies to words containing a minimal number of generators that map to diagrammatic generators with strings connected to the boundaries; this is used even more regularly throughout the rest of the paper.

In Section \ref{s:WTform}, we show that each word in $\A_n(X)$ is equal to a scalar multiple of a word in $WT$ \textit{form}, analogous to the diagram products $\undertilde{W}\undertilde{T}$ from Section \ref{s:diagramgenerators}. That is, each word is equal to a scalar multiple of a word $WT$, where $W$ is either the identity, or an odd word of a particular form, and $T$ is an even-reduced word. The idea is to show that if $\phi$ maps two words in $WT$ form to the same diagram in $\lab_n(X)$, then those words are equal in $\A_n(X)$. However, we first need to show that each word in $WT$ form is mapped to a basis diagram of $\lab_n(X)$.

In Section \ref{s:sympblob}, we establish the connection between the even generators of $\A_n(X)$ and the generators of the symplectic blob algebra. We then exploit this to show that even-reduced words in $\A_n(X)$ map to even diagrams in $\lab_n(X)$, and that $\phi$ is injective when restricted to even words, using the presentation for the symplectic blob algebra established in \cite{SympBlobPres}. This is everything we need for even words and diagrams for injectivity; it remains to handle the odd words and diagrams. Section \ref{s:leftdescentsets} deals primarily with those odd words, providing some further information on odd words in $WT$ form.

Finally, in Section \ref{s:labeliso}, we show that each word in $WT$ form maps to a basis diagram of $\lab_n(X)$. We present a series of technical lemmas, and use them to establish that concatenating the diagrams $\phi(W)$ and $\phi(T)$ does not produce any boundary arcs. We then show that if two words in $WT$ form map to the same diagram in $\lab_n(X)$, then they are equal in $\A_n(X)$. It follows that $\phi$ is an isomorphism.

\subsection*{Conventions}
All algebras are over $\C$, and all parameters are taken to be indeterminates. Each diagram algebra in this paper has a subscript indicating the number of nodes on each side of its basis diagrams, and a superscript indicating the number of boundaries, if any. We use $\N$ to mean the positive integers; that is, $0 \notin \N$.

\subsection*{Acknowledgements}
The author was supported by the Australian Government Research Training Program Stipend and Tuition Fee Offset, as a PhD student at the University of Queensland, and then by the Australian Research Council under the Discovery Project DP240101572.
The author thanks J\o{}rgen Rasmussen, Jon Links and Yvan Saint-Aubin for helpful discussions and comments.

\section{Algebra definitions} \label{s:definitions}
In this section, we define the ghost algebra and the label algebra diagrammatically, and show that the ghost algebra is a specialisation of the label algebra.

\subsection{Ghost algebra} \label{ss:ghostdef}
We first define the basis diagrams of the ghost algebra $\Gh_n$. A $\Gh_n$-\textit{diagram} consists of a rectangle with $n$ nodes on each of its left and right sides, and dotted \textit{boundary} lines at the top and bottom. Non-crossing strings are drawn within the rectangle such that each string connects a node to another node, or to a boundary. Each node has exactly one string endpoint attached to it. The string endpoints attached to each boundary partition the boundaries into a number of disjoint regions called \textit{domains}, and each domain may contain a finite number of filled black circles called \textit{ghosts}. We require that, on each boundary, the number of string endpoints plus ghosts must be even; once we define multiplication, this ensures that the algebra is associative.

For example,
\begin{align}
\begin{tikzpicture}[baseline={([yshift=-1mm]current bounding box.center)},scale=0.5]
{
\draw[dotted] (0,0.5)--(3,0.5);
\draw[dotted] (0,-5.5)--(3,-5.5);
\draw [very thick](0,-5.5) -- (0,0.5);
\draw [very thick](3,-5.5)--(3,0.5);
\draw (0,0) arc(90:-90:0.5);
\draw (0,-2) to[out=0,in=-90] (1.2,0.5);
\draw (0,-3) to[out=0,in=180] (3,-2);
\draw (0,-4) arc(90:-90:0.5);
\draw (3,0) arc(90:270:0.5);
\draw (3,-3) to[out=180,in=90] (2,-5.5);
\draw (3,-4) arc(90:270:0.5);
\filldraw (0.6,0.5) circle (0.08);
\filldraw (1,-5.5) circle (0.08);
}
\end{tikzpicture}\; ,
&&
\begin{tikzpicture}[baseline={([yshift=-1mm]current bounding box.center)},scale=0.5]
{
\draw[dotted] (0,0.5)--(3,0.5);
\draw[dotted] (0,-5.5)--(3,-5.5);
\draw [very thick](0,-5.5) -- (0,0.5);
\draw [very thick](3,-5.5)--(3,0.5);
\draw (0,0) ..controls (1.5,0) and (1.8,-3.5).. (1.8,-5.5);
\draw (0,-1) to[out=0,in=0] (0,-4);
\draw (0,-2) arc(90:-90:0.5);
\draw (0,-5) arc(90:0:0.6 and 0.5);
\draw (3,0) ..controls (1.7,0) and (1.7,-5).. (3,-5);
\draw (3,-1) arc(90:270:0.5);
\draw (3,-3) arc(90:270:0.5);
\filldraw (1.2,-5.5) circle (0.08);
\filldraw (2.4,-5.5) circle (0.08);
}
\end{tikzpicture}\; ,
&&
\begin{tikzpicture}[baseline={([yshift=-1mm]current bounding box.center)},scale=0.5]
{
\draw[dotted] (0,0.5)--(3,0.5);
\draw[dotted] (0,-5.5)--(3,-5.5);
\draw [very thick](0,-5.5) -- (0,0.5);
\draw [very thick](3,-5.5)--(3,0.5);
\draw (0,0) arc(90:-90:0.5);
\draw (0,-2) to[out=0,in=-90] (1,0.5);
\draw (0,-3) arc(-90:0:0.25) arc(180:0:0.25) --(0.75,-3.5) arc(-180:0:0.25) --(1.25,-2) arc(180:0:0.25) --(1.75,-2.5) arc(-180:0:0.25) --(2.25,-1.75) arc(180:0:0.25) arc(180:270:0.25);
\draw (0,-4) arc(90:-90:0.5);
\draw (3,0) arc(90:270:0.5);
\draw (3,-3) to[out=180,in=90] (2,-5.5);
\draw (3,-4) arc(90:270:0.5);
\filldraw (0.5,0.5) circle (0.08);
\filldraw (1.4,0.5) circle (0.08);
\filldraw (1.8,0.5) circle (0.08);
\filldraw (2.2,0.5) circle (0.08);
\filldraw (2.6,0.5) circle (0.08);
\filldraw (1,-5.5) circle (0.08);
}
\end{tikzpicture}\; ,
&&
\begin{tikzpicture}[baseline={([yshift=-1mm]current bounding box.center)},scale=0.5]
{
\draw[dotted] (0,0.5)--(3,0.5);
\draw[dotted] (0,-5.5)--(3,-5.5);
\draw [very thick](0,-5.5) -- (0,0.5);
\draw [very thick](3,-5.5)--(3,0.5);
\draw (0,0) to[out=0,in=180] (3,-2);
\draw (0,-1) to[out=0,in=180] (3,-3);
\draw (0,-2) to[out=0,in=180] (3,-4);
\draw (0,-3) to[out=0,in=180] (3,-5);
\draw (0,-4) arc(90:-90:0.5);
\draw (3,0) arc(270:180:0.5);
\draw (3,-1) to[out=180,in=270] (2,0.5);
}
\end{tikzpicture}\; ,
&&
\begin{tikzpicture}[baseline={([yshift=-1mm]current bounding box.center)},scale=0.5]
{
\draw[dotted] (0,0.5)--(3,0.5);
\draw[dotted] (0,-5.5)--(3,-5.5);
\draw [very thick](0,-5.5) -- (0,0.5);
\draw [very thick](3,-5.5)--(3,0.5);
\draw (0,0) arc(90:-90:0.5);
\draw (0,-2) to[out=0,in=-90] (1.2,0.5);
\draw (0,-3) to[out=0,in=180] (3,-2);
\draw (0,-4) arc(90:-90:0.5);
\draw (3,0) arc(90:270:0.5);
\draw (3,-3) to[out=180,in=90] (2,-5.5);
\draw (3,-4) arc(90:270:0.5);
\filldraw (0.6,0.5) circle (0.08);
\filldraw (2.5,-5.5) circle (0.08);
}
\end{tikzpicture}
\end{align}
are all $\Gh_6$-diagrams.

We say two $\Gh_n$-diagrams are equal if their strings connect the same nodes to each other or to the same boundary, and the numbers of ghosts in corresponding domains are equivalent modulo $2$. This means that the first and third diagrams above are equal, while the first and fifth are not, due to the placement of the ghost on the bottom boundary. Hence each $\Gh_n$-diagram is equal to a $\Gh_n$-diagram with at most one ghost in each domain. Noting that strings with both ends on the boundaries are disallowed, it follows that there are finitely many $\Gh_n$-diagrams for each $n \in \N$, and thus the ghost algebra $\Gh_n$ is finite-dimensional. Indeed, the dimension of $\Gh_n$ is calculated in \cite[App. A]{Ghost}. In contrast, basis diagrams of the two-boundary TL algebra may contain arbitrarily many strings connecting the top boundary to the bottom boundary, so the two-boundary TL algebra is infinite-dimensional, unlike the TL and one-boundary TL algebras.

The \textit{ghost algebra} $\Gh_n$ is the complex vector space with the set of all $\Gh_n$-diagrams as its basis, and multiplication defined on pairs of $\Gh_n$-diagrams as follows, extended bilinearly. To multiply two $\Gh_n$-diagrams, we first concatenate them, and look for any loops or \textit{boundary arcs}---strings with both ends connected to either boundary---formed. Each loop is removed and replaced by a factor of the \textit{loop parameter} $\beta$. Along each boundary, we number the string endpoints and ghosts from left to right, starting from $1$. Each boundary arc is removed, with a ghost left at each of its endpoints, and replaced by a factor of the appropriate \textit{boundary parameter} from Table \ref{tab:params}, according to the parity of its endpoints. Removing the vertical line segment in the middle of the concatenated diagrams yields a scalar multiple of a $\Gh_n$-diagram; this may be neatened by continuously deforming the strings and, in each domain, removing pairs of ghosts to leave at most one ghost, if desired.

\begin{table}[H]
\centering
\caption{Table of boundary arcs and their associated parameters in the ghost algebra.}\label{tab:params}
\vspace{2mm}
\begin{tabular}{|c|c||c|c|}
\hline
Parameter & Boundary arc & Parameter & Boundary arc\\
\hline
$\alpha_1$ & $\begin{tikzpicture}[baseline={([yshift=-1mm]current bounding box.center)},xscale=0.5,yscale=0.5]
{
\draw[dotted] (0,0.5)--(2,0.5);
\draw (0.2,0.5) arc (-180:0:0.8);
\node at (0.2,0.5) [anchor=south] {\scriptsize odd};
\node at (1.8,0.5) [anchor=south] {\scriptsize even};
\node at (1,-0.5) {};
}
\end{tikzpicture}$
&$\delta_1$ &$\begin{tikzpicture}[baseline={([yshift=-1mm]current bounding box.center)},xscale=0.5,yscale=0.5]
{
\draw[dotted] (0,0)--(2,0);
\draw (0.2,0) arc (180:0:0.8);
\node at (0.2,-0.9) [anchor=south] {\scriptsize odd};
\node at (1.8,-0.9) [anchor=south] {\scriptsize even};
\node at (1,1) {};
}
\end{tikzpicture}$\\
\hline
$\alpha_2$ &$\begin{tikzpicture}[baseline={([yshift=-1mm]current bounding box.center)},xscale=0.5,yscale=0.5]
{
\draw[dotted] (0,0.5)--(2,0.5);
\draw (0.2,0.5) arc (-180:0:0.8);
\node at (0.2,0.5) [anchor=south] {\scriptsize even};
\node at (1.8,0.5) [anchor=south] {\scriptsize odd};
\node at (1,-0.5) {};
}
\end{tikzpicture}$ 
&$\delta_2$& $\begin{tikzpicture}[baseline={([yshift=-1mm]current bounding box.center)},xscale=0.5,yscale=0.5]
{
\draw[dotted] (0,0)--(2,0);
\draw (0.2,0) arc (180:0:0.8);
\node at (0.2,-0.9) [anchor=south] {\scriptsize even};
\node at (1.8,-0.9) [anchor=south] {\scriptsize odd};
\node at (1,1) {};
}
\end{tikzpicture}$\\
\hline
$\alpha_3$ & $\begin{tikzpicture}[baseline={([yshift=4mm]current bounding box.south)},xscale=0.5,yscale=0.5]
{
\draw[dotted] (0,0.5)--(2,0.5);
\draw (0.2,0.5) arc (-180:0:0.8);
\node at (0.2,0.5) [anchor=south] {\scriptsize odd};
\node at (1.8,0.5) [anchor=south] {\scriptsize odd};
\node at (1,-0.5) {};
}
\end{tikzpicture}\, , \  \begin{tikzpicture}[baseline={([yshift=4mm]current bounding box.south)},xscale=0.5,yscale=0.5]
{
\draw[dotted] (0,0.5)--(2,0.5);
\draw (0.2,0.5) arc (-180:0:0.8);
\node at (0.2,0.5) [anchor=south] {\scriptsize even};
\node at (1.8,0.5) [anchor=south] {\scriptsize even};
\node at (1,-0.5) {};
}
\end{tikzpicture}$
&$\delta_3$ &$\begin{tikzpicture}[baseline={([yshift=-1mm]current bounding box.center)},xscale=0.5,yscale=0.5]
{
\draw[dotted] (0,0)--(2,0);
\draw (0.2,0) arc (180:0:0.8);
\node at (0.2,-0.9) [anchor=south] {\scriptsize odd};
\node at (1.8,-0.9) [anchor=south] {\scriptsize odd};
\node at (1,1) {};
}
\end{tikzpicture}\, , \ \begin{tikzpicture}[baseline={([yshift=-1mm]current bounding box.center)},xscale=0.5,yscale=0.5]
{
\draw[dotted] (0,0)--(2,0);
\draw (0.2,0) arc (180:0:0.8);
\node at (0.2,-0.9) [anchor=south] {\scriptsize even};
\node at (1.8,-0.9) [anchor=south] {\scriptsize even};
\node at (1,1) {};
}
\end{tikzpicture}$\\
\hline \hline
$\gamma_{12}$ & $\begin{tikzpicture}[baseline={([yshift=5mm]current bounding box.south)},xscale=0.5,yscale=0.5]
{
\draw[dotted] (0,0)--(2,0);
\draw[dotted] (0,1)--(2,1);
\draw (1,1)--(1,0);
\node at (1,1) [anchor=south] {\scriptsize odd};
\node at (1,-0.9) [anchor=south] {\scriptsize even};
}
\end{tikzpicture}\, , \ \  \begin{tikzpicture}[baseline={([yshift=5mm]current bounding box.south)},xscale=0.5,yscale=0.5]
{
\draw[dotted] (0,0)--(2,0);
\draw[dotted] (0,1)--(2,1);
\draw (1,1)--(1,0);
\node at (1,1) [anchor=south] {\scriptsize even};
\node at (1,-0.9) [anchor=south] {\scriptsize odd};
}
\end{tikzpicture}$
&$\gamma_3$ & $\begin{tikzpicture}[baseline={([yshift=5mm]current bounding box.south)},xscale=0.5,yscale=0.5]
{
\draw[dotted] (0,0)--(2,0);
\draw[dotted] (0,1)--(2,1);
\draw (1,1)--(1,0);
\node at (1,1) [anchor=south] {\scriptsize odd};
\node at (1,-0.9) [anchor=south] {\scriptsize odd};
}
\end{tikzpicture}\,  , \ \  \begin{tikzpicture}[baseline={([yshift=5mm]current bounding box.south)},xscale=0.5,yscale=0.5]
{
\draw[dotted] (0,0)--(2,0);
\draw[dotted] (0,1)--(2,1);
\draw (1,1)--(1,0);
\node at (1,1) [anchor=south] {\scriptsize even};
\node at (1,-0.9) [anchor=south] {\scriptsize even};
}
\end{tikzpicture}$
\\
\hline
\end{tabular}
\end{table}

For example, with $n=8$, we have
\begin{align}
\begin{tikzpicture}[baseline={([yshift=-1mm]current bounding box.center)},scale=0.55]
{
\draw[dotted] (0,0.5)--(6,0.5);
\draw[dotted] (0,-7.5)--(6,-7.5);
\draw [very thick] (0,-7.5)--(0,0.5);
\draw [very thick] (3,-7.5)--(3,0.5);
\draw [very thick] (6,-7.5)--(6,0.5);
\draw (0,-2) ..controls (1.5,-2) and (1.5,-7).. (0,-7);
\draw (0,0) arc (90:-90:0.5);
\draw (0,-3) arc (90:-90:0.5);
\draw (0,-5) arc (90:-90:0.5);
\node at (0.6,0.5) [anchor=south] {\scriptsize 1};
\node at (1.2,0.5) [anchor=south] {\scriptsize 2};
\filldraw (0.6,0.5) circle (0.08);
\draw (3,0) arc (90:270:0.5);
\draw (3,-2) to[out=180,in=270] (1.8,0.5);
\draw (3,-3) arc (90:270:0.5);
\draw (3,-5) arc (90:270:0.5);
\draw (3,-7) ..controls (1.5,-7) and (1.2,-2).. (1.2,0.5);
\node at (2.4,0.5) [anchor=south] {\scriptsize 4};
\node at (1.8,0.5) [anchor=south] {\scriptsize 3};
\filldraw (2.4,0.5) circle (0.08);
\draw (3,0) arc(-90:0:0.6 and 0.5);
\draw (3,-1) to[out=0,in=-90] (4.2,0.5);
\draw (3,-2) ..controls (4.6,-2) and (4.8,-5).. (4.8,-7.5);
\draw (3,-3) to[out=0,in=0] (3,-6);
\draw (3,-4) arc(90:-90:0.5);
\draw (3,-7) arc (90:0:0.6 and 0.5);
\node at (3.6,0.5) [anchor=south] {\scriptsize 5};
\node at (4.2,0.5) [anchor=south] {\scriptsize 6};
\node at (4.8,0.5) [anchor=south] {\scriptsize 7};
\node at (5.4,0.5) [anchor=south] {\scriptsize 8};
\node at (3.6,-7.5) [anchor=north] {\scriptsize 1};
\node at (4.2,-7.5) [anchor=north] {\scriptsize 2};
\node at (4.8,-7.5) [anchor=north] {\scriptsize 3};
\node at (5.4,-7.5) [anchor=north] {\scriptsize 4};
\filldraw (4.2,-7.5) circle (0.08);
\filldraw (5.4,-7.5) circle (0.08);
\draw (6,0) arc (270:180:0.6 and 0.5);
\draw (6,-3) arc (90:270:0.5);
\draw (6,-1) to[out=180,in=270] (4.8,0.5);
\draw (6,-2) to[out=180,in=180] (6,-5);
\draw (6,-6) arc(90:270:0.5);
}
\end{tikzpicture}
\ &=\,
\beta\alpha_1\gamma_{12}\gamma_3\;
\begin{tikzpicture}[baseline={([yshift=-1mm]current bounding box.center)},scale=0.55]
{
\draw[dotted] (0,0.5)--(6,0.5);
\draw[dotted] (0,-7.5)--(6,-7.5);
\draw [very thick] (0,-7.5)--(0,0.5);
\draw [very thick] (3,-7.5)--(3,0.5);
\draw [very thick] (6,-7.5)--(6,0.5);
\draw (0,-2) ..controls (1.5,-2) and (1.5,-7).. (0,-7);
\draw (0,0) arc (90:-90:0.5);
\draw (0,-3) arc (90:-90:0.5);
\draw (0,-5) arc (90:-90:0.5);
\node at (0.6,0.5) [anchor=south] {\scriptsize 1};
\node at (1.2,0.5) [anchor=south] {\scriptsize 2};
\filldraw (0.6,0.5) circle (0.08);
\filldraw (1.2,0.5) circle (0.08);
\node at (2.4,0.5) [anchor=south] {\scriptsize 4};
\node at (1.8,0.5) [anchor=south] {\scriptsize 3};
\filldraw (1.8,0.5) circle (0.08);
\filldraw (2.4,0.5) circle (0.08);
\node at (3.6,0.5) [anchor=south] {\scriptsize 5};
\node at (4.2,0.5) [anchor=south] {\scriptsize 6};
\node at (4.8,0.5) [anchor=south] {\scriptsize 7};
\node at (5.4,0.5) [anchor=south] {\scriptsize 8};
\filldraw (3.6,0.5) circle (0.08);
\filldraw (4.2,0.5) circle (0.08);
\node at (3.6,-7.5) [anchor=north] {\scriptsize 1};
\node at (4.2,-7.5) [anchor=north] {\scriptsize 2};
\node at (4.8,-7.5) [anchor=north] {\scriptsize 3};
\node at (5.4,-7.5) [anchor=north] {\scriptsize 4};
\filldraw (3.6,-7.5) circle (0.08);
\filldraw (4.2,-7.5) circle (0.08);
\filldraw (4.8,-7.5) circle (0.08);
\filldraw (5.4,-7.5) circle (0.08);
\draw (6,0) arc (270:180:0.6 and 0.5);
\draw (6,-3) arc (90:270:0.5);
\draw (6,-1) to[out=180,in=270] (4.8,0.5);
\draw (6,-2) to[out=180,in=180] (6,-5);
\draw (6,-6) arc(90:270:0.5);
}
\end{tikzpicture}
\ =\,
\beta \alpha_1 \gamma_{12}\gamma_3\;
\begin{tikzpicture}[baseline={([yshift=-1mm]current bounding box.center)},scale=0.55]
{
\draw[dotted] (0,0.5)--(3,0.5);
\draw[dotted] (0,-7.5)--(3,-7.5);
\draw [very thick] (0,-7.5)--(0,0.5);
\draw [very thick] (3,-7.5)--(3,0.5);
\draw (0,-2) ..controls (1.5,-2) and (1.5,-7).. (0,-7);
\draw (0,0) arc (90:-90:0.5);
\draw (0,-3) arc (90:-90:0.5);
\draw (0,-5) arc (90:-90:0.5);
\draw (3,0) arc (270:180:0.6 and 0.5);
\draw (3,-3) arc (90:270:0.5);
\draw (3,-1) to[out=180,in=270] (1.8,0.5);
\draw (3,-2) to[out=180,in=180] (3,-5);
\draw (3,-6) arc(90:270:0.5);
\node at (1,0.5) [anchor=south] {\phantom{\scriptsize 1}};
\node at (0.6,-7.5) [anchor=north] {\phantom{\scriptsize 1}};
}
\end{tikzpicture}\; ,
\end{align}
and
\begin{align}
\begin{tikzpicture}[baseline={([yshift=-1mm]current bounding box.center)},scale=0.55]
{
\draw[dotted] (0,0.5)--(6,0.5);
\draw[dotted] (0,-7.5)--(6,-7.5);
\draw [very thick] (0,-7.5)--(0,0.5);
\draw [very thick] (3,-7.5)--(3,0.5);
\draw [very thick] (6,-7.5)--(6,0.5);
\draw (0,0) arc (90:-90:0.5);
\draw (0,-2) to[out=0,in=-90] (1.2,0.5);
\draw (0,-3) to[out=0,in=180] (3,-2);
\draw (0,-4) arc (90:-90:0.5);
\draw (0,-6) arc (90:-90:0.5);
\filldraw (0.6,0.5) circle (0.08);
\node at (0.6,0.5) [anchor=south] {\scriptsize 1};
\node at (1.2,0.5) [anchor=south] {\scriptsize 2};
\node at (0.6,-7.5) [anchor=north] {\scriptsize 1};
\draw (3,0) arc (270:180:0.6 and 0.5);
\draw (3,-1) arc (270:180:1.2 and 1.5);
\draw (3,-3) arc (90:270:0.5);
\draw (3,-5) arc (90:180:1.8 and 2.5);
\draw (3,-6) arc (90:180:1.2 and 1.5);
\draw (3,-7) arc (90:180:0.6 and 0.5);
\filldraw (0.6,-7.5) circle (0.08);
\node at (1.8,0.5) [anchor=south] {\scriptsize 3};
\node at (2.4,0.5) [anchor=south] {\scriptsize 4};
\node at (1.2,-7.5) [anchor=north] {\scriptsize 2};
\node at (1.8,-7.5) [anchor=north] {\scriptsize 3};
\node at (2.4,-7.5) [anchor=north] {\scriptsize 4};
\draw (3,0) arc (-90:0:0.6 and 0.5);
\draw (3,-1) arc (-90:0:1.8 and 1.5);
\draw (3,-2) to[out=0,in=180] (6,0);
\draw (3,-3) to[out=0,in=180] (6,-1);
\draw (3,-4) to[out=0,in=180] (6,-6);
\draw (3,-5) arc (90:0:1.725 and 2.5);
\draw (3,-6) ..controls (3.8,-6) and (3.825,-7).. (3.825,-7.5);
\draw (3,-7) arc (90:0:0.375 and 0.5);
\filldraw (4.2,0.5) circle (0.08);
\filldraw (5.4,0.5) circle (0.08);
\filldraw (4.275,-7.5) circle (0.08);
\node at (3.6,0.5) [anchor=south] {\scriptsize 5};
\node at (4.2,0.5) [anchor=south] {\scriptsize 6};
\node at (4.8,0.5) [anchor=south] {\scriptsize 7};
\node at (5.4,0.5) [anchor=south] {\scriptsize 8};
\node at (3.375,-7.5) [anchor=north] {\scriptsize 5};
\node at (3.825,-7.5) [anchor=north] {\scriptsize 6};
\node at (4.275,-7.5) [anchor=north] {\scriptsize 7};
\node at (4.725,-7.5) [anchor=north] {\scriptsize 8};
\draw (6,-2) arc (90:270:0.5);
\draw (6,-4) arc (90:270:0.5);
\draw (6,-7) arc (90:180:0.825 and 0.5);
\filldraw (5.625,-7.5) circle (0.08);
\node at (5.175,-7.5) [anchor=north] {\scriptsize 9};
\node at (5.625,-7.5) [anchor=north] {\scriptsize 10};
}
\end{tikzpicture}
\ = \,\alpha_2 \alpha_3 \delta_1\delta_2\delta_3\;
\begin{tikzpicture}[baseline={([yshift=-1mm]current bounding box.center)},scale=0.55]
{
\draw[dotted] (0,0.5)--(6,0.5);
\draw[dotted] (0,-7.5)--(6,-7.5);
\draw [very thick] (0,-7.5)--(0,0.5);
\draw [very thick] (3,-7.5)--(3,0.5);
\draw [very thick] (6,-7.5)--(6,0.5);
\draw (0,0) arc (90:-90:0.5);
\draw (0,-2) to[out=0,in=-90] (1.2,0.5);
\draw (0,-3) to[out=0,in=180] (3,-2);
\draw (0,-4) arc (90:-90:0.5);
\draw (0,-6) arc (90:-90:0.5);
\filldraw (0.6,0.5) circle (0.08);
\node at (0.6,0.5) [anchor=south] {\scriptsize 1};
\node at (1.2,0.5) [anchor=south] {\scriptsize 2};
\node at (0.6,-7.5) [anchor=north] {\scriptsize 1};
\draw (3,-3) arc (90:270:0.5);
\filldraw (0.6,-7.5) circle (0.08);
\node at (1.8,0.5) [anchor=south] {\scriptsize 3};
\node at (2.4,0.5) [anchor=south] {\scriptsize 4};
\node at (1.2,-7.5) [anchor=north] {\scriptsize 2};
\node at (1.8,-7.5) [anchor=north] {\scriptsize 3};
\node at (2.4,-7.5) [anchor=north] {\scriptsize 4};
\filldraw (1.8,0.5) circle (0.08);
\filldraw (2.4,0.5) circle (0.08);
\filldraw (1.2,-7.5) circle (0.08);
\filldraw (1.8,-7.5) circle (0.08);
\filldraw (2.4,-7.5) circle (0.08);
\draw (3,-2) to[out=0,in=180] (6,0);
\draw (3,-3) to[out=0,in=180] (6,-1);
\draw (3,-4) to[out=0,in=180] (6,-6);
\filldraw (4.2,0.5) circle (0.08);
\filldraw (5.4,0.5) circle (0.08);
\filldraw (4.725,-7.5) circle (0.08);
\node at (3.6,0.5) [anchor=south] {\scriptsize 5};
\node at (4.2,0.5) [anchor=south] {\scriptsize 6};
\node at (4.8,0.5) [anchor=south] {\scriptsize 7};
\node at (5.4,0.5) [anchor=south] {\scriptsize 8};
\node at (3.375,-7.5) [anchor=north] {\scriptsize 5};
\node at (3.825,-7.5) [anchor=north] {\scriptsize 6};
\node at (4.275,-7.5) [anchor=north] {\scriptsize 7};
\node at (4.725,-7.5) [anchor=north] {\scriptsize 8};
\filldraw (3.375,-7.5) circle (0.08);
\filldraw (3.825,-7.5) circle (0.08);
\filldraw (4.275,-7.5) circle (0.08);
\filldraw (3.6,0.5) circle (0.08);
\filldraw (4.8,0.5) circle (0.08);
\draw (6,-2) arc (90:270:0.5);
\draw (6,-4) arc (90:270:0.5);
\draw (6,-7) arc (90:180:0.825 and 0.5);
\filldraw (5.625,-7.5) circle (0.08);
\node at (5.175,-7.5) [anchor=north] {\scriptsize 9};
\node at (5.625,-7.5) [anchor=north] {\scriptsize 10};
}
\end{tikzpicture}
\ = \,\alpha_2 \alpha_3 \delta_1\delta_2\delta_3\;
\begin{tikzpicture}[baseline={([yshift=-1mm]current bounding box.center)},scale=0.55]
{
\draw[dotted] (0,0.5)--(3,0.5);
\draw[dotted] (0,-7.5)--(3,-7.5);
\draw [very thick] (0,-7.5)--(0,0.5);
\draw [very thick] (3,-7.5)--(3,0.5);
\draw (0,0) arc (90:-90:0.5);
\draw (0,-2) to[out=0,in=-90] (1.2,0.5);
\draw (0,-3) to[out=0,in=180] (3,0);
\draw (0,-4) arc (90:-90:0.5);
\draw (0,-6) arc (90:-90:0.5);
\filldraw (0.6,0.5) circle (0.08);
\draw (3,-1) ..controls (1.5,-1) and (1.5,-6).. (3,-6);
\draw (3,-2) arc (90:270:0.5);
\draw (3,-4) arc (90:270:0.5);
\draw (3,-7) arc (90:180:0.825 and 0.5);
\filldraw (2.625,-7.5) circle (0.08);
\node at (1,0.5) [anchor=south] {\phantom{\scriptsize 1}};
\node at (0.6,-7.5) [anchor=north] {\phantom{\scriptsize 1}};
}
\end{tikzpicture}\; .
\end{align}

As shown in \cite[App. B]{Ghost}, the ghost algebra is associative. Broadly, this holds because the process of concatenating diagrams is associative, and the parity assigned to each string endpoint on each boundary---and thus the parameter assigned to each boundary arc---is unaffected by the order of multiplication. Indeed, since parity is determined by numbering ghosts and string endpoints left-to-right, and the number of ghosts plus string endpoints on each boundary is even in each diagram, concatenation can change the number associated with each string endpoint, but not its parity.

The ghost algebra is also unital, with identity
\begin{align}
    \did = \; \begin{tikzpicture}[baseline={([yshift=-1mm]current bounding box.center)},scale=0.5]
{
\draw[dotted] (0,0.5)--(3,0.5);
\draw[dotted] (0,-5.5)--(3,-5.5);
\draw [very thick] (0,-5.5)--(0,0.5);
\draw [very thick] (3,-5.5)--(3,0.5);
\draw (0,-1)--(3,-1);
\draw (0,0) -- (3,0);
\draw (0,-5)--(3,-5);
\node at (1.5,-2.8) [anchor=center] {$\vdots$};
}
\end{tikzpicture}\; . \label{eq:iddiag}
\end{align}

As in \cite{Tipunin} (see also \cite{Thesis}), the one-boundary TL algebra $\BTL_n(\beta;\alpha_1,\alpha_2)$ may be defined in terms of similar diagrams with no ghosts, and no strings connected to the bottom boundary. Multiplication is defined analogously, but without introducing any ghosts. When such diagrams are multiplied in the ghost algebra, each boundary arc has both ends connected to the top boundary, and thus leaves a pair of ghosts in a single domain when all boundary arcs are removed. Hence all of the resulting ghosts may be removed, leaving a scalar multiple of a diagram with no ghosts, and no strings connected to the bottom boundary. Hence the result is the same as if the diagrams were multiplied in $\BTL_n(\beta;\alpha_1,\alpha_2)$. This means that such diagrams span a subalgebra of $\Gh_n$ that is isomorphic to $\BTL_n(\beta;\alpha_1,\alpha_2)$. A similar subalgebra isomorphic to $\BTL_n(\beta;\delta_1,\delta_2)$ is spanned by the $\Gh_n$-diagrams with no ghosts, and no strings connected to the top boundary.

\subsection{Label algebra} \label{ss:label}
We now define the label algebra, and show that the ghost algebra is isomorphic to a specialisation of the label algebra.

Let $X$ be a nonempty set. The \textit{label algebra} $\lab_n(X)$ with \textit{label set} $X$ is a diagram algebra, and has a basis consisting of $\lab_n(X)$-diagrams. A $\lab_n(X)$-\textit{diagram} is constructed similarly to a $\Gh_n$-diagram, except that it has no ghosts, we no longer require the number of ghosts and strings on each boundary to be even, and each string endpoint on each boundary must be labelled with some label from $X$. We say two $\lab_n(X)$-diagrams are equal if their strings connect the same nodes to each other or to the same boundary, and the labels of corresponding string endpoints on the boundaries are the same. 

For example, if $X=\{\mathrm{a},\mathrm{b},\mathrm{c}\}$, some $\lab_6(X)$-diagrams are
\begin{align}
\begin{tikzpicture}[baseline={([yshift=-1mm]current bounding box.center)},xscale=0.5,yscale=0.5]
{
\draw[dotted] (0,0.5)--(3,0.5);
\draw[dotted] (0,-5.5)--(3,-5.5);
\draw [very thick](0,-5.5) -- (0,0.5);
\draw [very thick](3,-5.5)--(3,0.5);
\draw (0,0) arc (-90:0:0.5);
\draw (0,-1) arc (90:-90:0.5);
\draw (0,-3) to[out=0,in=-90] (1.3,0.5);
\draw (0,-4) to[out=0,in=180] (3,-1);
\draw (0,-5) to[out=0,in=180] (3,-2);
\draw (3,0) arc (270:180:0.5);
\draw (3,-3) to[out=180,in=90] (1.7,-5.5);
\draw (3,-4) arc (90:270:0.5);
\node at (0.5,0.4) [anchor=south] {\footnotesize $\mathrm{a}$};
\node at (1.3,0.4) [anchor=south] {\footnotesize $\mathrm{c}$};
\node at (2.5,0.4) [anchor=south] {\footnotesize $\mathrm{b}$};
\node at (1.7,-5.4) [anchor=north] {\footnotesize $\mathrm{a}\vphantom{\mathrm{b}}$};
}
\end{tikzpicture}\; ,
&&
\begin{tikzpicture}[baseline={([yshift=-1mm]current bounding box.center)},xscale=0.5,yscale=0.5]
{
\draw[dotted] (0,0.5)--(3,0.5);
\draw[dotted] (0,-5.5)--(3,-5.5);
\draw [very thick](0,-5.5) -- (0,0.5);
\draw [very thick](3,-5.5)--(3,0.5);
\draw (0,0) to[out=0,in=180] (3,-4);
\draw (0,-1) to[out=0,in=0] (0,-4);
\draw (0,-2) arc (90:-90:0.5);
\draw (0,-5) -- (3,-5);
\draw (3,0) arc (90:270:0.5);
\draw (3,-2) arc (90:270:0.5);
\node at (2.5,0.4) [anchor=south] {\footnotesize $\vphantom{\mathrm{b}}$};
\node at (1.7,-5.4) [anchor=north] {\footnotesize $\vphantom{\mathrm{b}}$};
}
\end{tikzpicture}\; ,
&&
\begin{tikzpicture}[baseline={([yshift=-1mm]current bounding box.center)},xscale=0.5,yscale=0.5]
{
\draw[dotted] (0,0.5)--(3,0.5);
\draw[dotted] (0,-5.5)--(3,-5.5);
\draw [very thick](0,-5.5) -- (0,0.5);
\draw [very thick](3,-5.5)--(3,0.5);
\draw (0,0) arc (-90:0:0.5);
\draw (0,-1) arc (90:-90:0.5);
\draw (0,-3) to[out=0,in=-90] (1.3,0.5);
\draw (0,-4) to[out=0,in=180] (3,-1);
\draw (0,-5) to[out=0,in=180] (3,-2);
\draw (3,0) arc (270:180:0.5);
\draw (3,-3) to[out=180,in=90] (1.7,-5.5);
\draw (3,-4) arc (90:270:0.5);
\node at (0.5,0.4) [anchor=south] {\footnotesize $\mathrm{b}$};
\node at (1.3,0.4) [anchor=south] {\footnotesize $\mathrm{c}$};
\node at (2.5,0.4) [anchor=south] {\footnotesize $\mathrm{a}$};
\node at (1.7,-5.4) [anchor=north] {\footnotesize $\mathrm{c}\vphantom{\mathrm{b}}$};
}
\end{tikzpicture}\; ,
&&
\begin{tikzpicture}[baseline={([yshift=-1mm]current bounding box.center)},xscale=0.5,yscale=0.5]
{
\draw[dotted] (0,0.5)--(3,0.5);
\draw[dotted] (0,-5.5)--(3,-5.5);
\draw [very thick](0,-5.5) -- (0,0.5);
\draw [very thick](3,-5.5)--(3,0.5);
\draw (0,0) arc (-90:0:0.8 and 0.5);
\draw (0,-1) to[out=0,in=180] (3,0);
\draw (0,-2) to[out=0,in=180] (3,-1);
\draw (0,-3) to[out=0,in=180] (3,-2);
\draw (0,-4) to[out=0,in=180] (3,-3);
\draw (0,-5) to[out=0,in=180] (3,-4);
\draw (3,-5) arc (90:180:0.8 and 0.5);
\node at (0.8,0.4) [anchor=south] {\footnotesize $\mathrm{c}\vphantom{\mathrm{b}}$};
\node at (2.2,-5.4) [anchor=north] {\footnotesize $\mathrm{b}$};
}
\end{tikzpicture}\; 
. \label{eq:eglabeldiags}
\end{align}
The first and third diagrams are not equal because their labels differ, even though their strings connect the same nodes and boundaries. 

Multiplication in $\lab_n(X)$ is defined on pairs of $\lab_n(X)$-diagrams by concatenation, replacing each loop by a factor of $\beta$, and each boundary arc by a factor of the appropriate parameter listed in Table \ref{tab:labelparams}.

\begin{table}[H]
\centering
\caption{Table of boundary arcs and their associated parameter values in the label algebra.}\label{tab:labelparams}
\vspace{2mm}
\begin{tabular}{|c|c|}
\hline
Parameter & Boundary arc\\
\hline
$\aup{a}{b}$ & $\begin{tikzpicture}[baseline={([yshift=-1mm]current bounding box.center)},xscale=0.5,yscale=0.5]
{
\draw[dotted] (0,0.5)--(2,0.5);
\draw (0.2,0.5) arc (-180:0:0.8);
\node at (0.2,0.5) [anchor=south] {\scriptsize $\mathrm{a}$};
\node at (1.8,0.5) [anchor=south] {\scriptsize $\mathrm{b}$};
\node at (1,-0.5) {};
}
\end{tikzpicture}$\\
\hline
$\glab{a}{b}$ & $\begin{tikzpicture}[baseline={([yshift=5mm]current bounding box.south)},xscale=0.5,yscale=0.5]
{
\draw[dotted] (0,0)--(2,0);
\draw[dotted] (0,1)--(2,1);
\draw (1,1)--(1,0);
\node at (1,1) [anchor=south] {\scriptsize $\mathrm{a}$};
\node at (1,-0.9) [anchor=south] {\scriptsize $\mathrm{b}$};
}
\end{tikzpicture}$
\\
\hline
$\ddo{a}{b}$ &$\begin{tikzpicture}[baseline={([yshift=-1mm]current bounding box.center)},xscale=0.5,yscale=0.5]
{
\draw[dotted] (0,0)--(2,0);
\draw (0.2,0) arc (180:0:0.8);
\node at (0.2,-0.9) [anchor=south] {\scriptsize $\mathrm{a}$};
\node at (1.8,-0.9) [anchor=south] {\scriptsize $\mathrm{b}$};
\node at (1,1) {};
}
\end{tikzpicture}$
\\
\hline
\end{tabular}
\end{table}

For example, with $X=\{0,1 \}$ and $n=8$, we have
\begin{align}
\begin{tikzpicture}[baseline={([yshift=-1mm]current bounding box.center)},scale=0.55]
{
\draw[dotted] (0,0.5)--(6,0.5);
\draw[dotted] (0,-7.5)--(6,-7.5);
\draw [very thick] (0,-7.5)--(0,0.5);
\draw [very thick] (3,-7.5)--(3,0.5);
\draw [very thick] (6,-7.5)--(6,0.5);
\draw (0,-2) ..controls (1.5,-2) and (1.5,-7).. (0,-7);
\draw (0,0) arc (90:-90:0.5);
\draw (0,-3) arc (90:-90:0.5);
\draw (0,-5) arc (90:-90:0.5);
\node at (1.2,0.5) [anchor=south] {\scriptsize 0};
\draw (3,0) arc (90:270:0.5);
\draw (3,-2) to[out=180,in=270] (1.8,0.5);
\draw (3,-3) arc (90:270:0.5);
\draw (3,-5) arc (90:270:0.5);
\draw (3,-7) ..controls (1.5,-7) and (1.2,-2).. (1.2,0.5);
\node at (1.8,0.5) [anchor=south] {\scriptsize 1};
\draw (3,0) arc(-90:0:0.6 and 0.5);
\draw (3,-1) to[out=0,in=-90] (4.2,0.5);
\draw (3,-2) ..controls (4.6,-2) and (4.8,-5).. (4.8,-7.5);
\draw (3,-3) to[out=0,in=0] (3,-6);
\draw (3,-4) arc(90:-90:0.5);
\draw (3,-7) arc (90:0:0.6 and 0.5);
\node at (3.6,0.5) [anchor=south] {\scriptsize 1};
\node at (4.2,0.5) [anchor=south] {\scriptsize 0};
\node at (4.8,0.5) [anchor=south] {\scriptsize 1};
\node at (5.4,0.5) [anchor=south] {\scriptsize 0};
\node at (3.6,-7.5) [anchor=north] {\scriptsize 1};
\node at (4.8,-7.5) [anchor=north] {\scriptsize 1};
\draw (6,0) arc (270:180:0.6 and 0.5);
\draw (6,-3) arc (90:270:0.5);
\draw (6,-1) to[out=180,in=270] (4.8,0.5);
\draw (6,-2) to[out=180,in=180] (6,-5);
\draw (6,-6) arc(90:270:0.5);
}
\end{tikzpicture}
\ &=\,
\beta \aup{1}{0} \glab{0}{1}\glab{1}{1}\;
\begin{tikzpicture}[baseline={([yshift=-1mm]current bounding box.center)},scale=0.55]
{
\draw[dotted] (0,0.5)--(3,0.5);
\draw[dotted] (0,-7.5)--(3,-7.5);
\draw [very thick] (0,-7.5)--(0,0.5);
\draw [very thick] (3,-7.5)--(3,0.5);
\draw (0,-2) ..controls (1.5,-2) and (1.5,-7).. (0,-7);
\draw (0,0) arc (90:-90:0.5);
\draw (0,-3) arc (90:-90:0.5);
\draw (0,-5) arc (90:-90:0.5);
\draw (3,0) arc (270:180:0.6 and 0.5);
\draw (3,-3) arc (90:270:0.5);
\draw (3,-1) to[out=180,in=270] (1.8,0.5);
\draw (3,-2) to[out=180,in=180] (3,-5);
\draw (3,-6) arc(90:270:0.5);
\node at (1.8,0.5) [anchor=south] {\scriptsize 1};
\node at (2.4,0.5) [anchor=south] {\scriptsize 0};
\node at (1.8,-7.5) [anchor=north] {\scriptsize \vphantom{1}};
}
\end{tikzpicture}\; ,
\end{align}
and
\begin{align}
\begin{tikzpicture}[baseline={([yshift=-1mm]current bounding box.center)},scale=0.55]
{
\draw[dotted] (0,0.5)--(6,0.5);
\draw[dotted] (0,-7.5)--(6,-7.5);
\draw [very thick] (0,-7.5)--(0,0.5);
\draw [very thick] (3,-7.5)--(3,0.5);
\draw [very thick] (6,-7.5)--(6,0.5);
\draw (0,0) arc (90:-90:0.5);
\draw (0,-2) to[out=0,in=-90] (1.2,0.5);
\draw (0,-3) to[out=0,in=180] (3,-2);
\draw (0,-4) arc (90:-90:0.5);
\draw (0,-6) arc (90:-90:0.5);
\node at (1.2,0.5) [anchor=south] {\scriptsize 0};
\draw (3,0) arc (270:180:0.6 and 0.5);
\draw (3,-1) arc (270:180:1.2 and 1.5);
\draw (3,-3) arc (90:270:0.5);
\draw (3,-5) arc (90:180:1.8 and 2.5);
\draw (3,-6) arc (90:180:1.2 and 1.5);
\draw (3,-7) arc (90:180:0.6 and 0.5);
\node at (1.8,0.5) [anchor=south] {\scriptsize 1};
\node at (2.4,0.5) [anchor=south] {\scriptsize 0};
\node at (1.2,-7.5) [anchor=north] {\scriptsize 0};
\node at (1.8,-7.5) [anchor=north] {\scriptsize 1};
\node at (2.4,-7.5) [anchor=north] {\scriptsize 0};
\draw (3,0) arc (-90:0:0.6 and 0.5);
\draw (3,-1) arc (-90:0:1.2 and 1.5);
\draw (3,-2) to[out=0,in=180] (6,0);
\draw (3,-3) to[out=0,in=180] (6,-1);
\draw (3,-4) to[out=0,in=180] (6,-6);
\draw (3,-5) arc (90:0:1.8 and 2.5);
\draw (3,-6) arc(90:0:1.2 and 1.5);
\draw (3,-7) arc (90:0:0.6 and 0.5);
\node at (3.6,0.5) [anchor=south] {\scriptsize 1};
\node at (4.2,0.5) [anchor=south] {\scriptsize 1};
\node at (3.6,-7.5) [anchor=north] {\scriptsize 1};
\node at (4.2,-7.5) [anchor=north] {\scriptsize 0};
\node at (4.8,-7.5) [anchor=north] {\scriptsize 0};
\draw (6,-2) arc (90:270:0.5);
\draw (6,-4) arc (90:270:0.5);
\draw (6,-7) arc (90:180:0.6 and 0.5);
\node at (5.4,-7.5) [anchor=north] {\scriptsize 1};
}
\end{tikzpicture}
\ = \,\aup{1}{1} \aup{0}{1} \ddo{0}{0} \ddo{1}{0} \ddo{0}{1}\;
\begin{tikzpicture}[baseline={([yshift=-1mm]current bounding box.center)},scale=0.55]
{
\draw[dotted] (0,0.5)--(3,0.5);
\draw[dotted] (0,-7.5)--(3,-7.5);
\draw [very thick] (0,-7.5)--(0,0.5);
\draw [very thick] (3,-7.5)--(3,0.5);
\draw (0,0) arc (90:-90:0.5);
\draw (0,-2) to[out=0,in=-90] (1.2,0.5);
\draw (0,-3) to[out=0,in=180] (3,0);
\draw (0,-4) arc (90:-90:0.5);
\draw (0,-6) arc (90:-90:0.5);
\draw (3,-1) ..controls (1.5,-1) and (1.5,-6).. (3,-6);
\draw (3,-2) arc (90:270:0.5);
\draw (3,-4) arc (90:270:0.5);
\draw (3,-7) arc (90:180:0.6 and 0.5);
\node at (1.2,0.5) [anchor=south] {\scriptsize 0\vphantom{1}};
\node at (2.4,-7.5) [anchor=north] {\scriptsize 1\vphantom{0}};
}
\end{tikzpicture}\; .
\end{align}

The label algebra $\lab_n(X)$ is unital, with the same identity diagram as the ghost algebra, \eqref{eq:iddiag}.

\begin{theorem}\label{prop:labelassoc}
The label algebra $\lab_n(X)$ is associative.
\end{theorem}
\begin{proof}
The proof that $\lab_n(X)$ is associative is analogous to the proof that $\Gh_n(X)$ is associative from \cite[App. B]{Ghost}. Briefly, concatenating diagrams is associative, and each boundary arc is assigned the same parameter regardless of bracketing, because the parameters depend only upon the labels associated with each string endpoint, which are unchanged by multiplication. The proof in \cite{Ghost} may be modified to fit the label algebra by omitting the ghosts from the diagrams, adding a label to each string endpoint on each boundary, and redefining the function $\chi$ to use the label algebra parameters instead of the ghost algebra parameters. 
\end{proof}

We note that, if the label set $X$ is finite, then there are finitely many inequivalent $\lab_n(X)$-diagrams, and thus $\lab_n(X)$ is finite-dimensional. Its dimension may be computed similarly to that of the ghost algebra, as in \cite[App. A]{Ghost}, but instead of choosing whether each string endpoint on each boundary is odd or even, we choose its label from $X$. That is,
\begin{align}
\dim \lab_n(X)  &= \sum_{d=0}^n \left(\sum_{j=0}^\floor{\frac{n-d}{2}} |X|^{n-2j-d}(n-2j-d+1)\left(\binom{n}{j}-\binom{n}{j-1}\right) \right)^2. \label{eq:labeldim}
\end{align}
We also observe that if there exists a bijection between $X$ and $Y$, then $\lab_n(X) \cong \lab_n(Y)$, assuming the parameters of each algebra are identified appropriately.

We now give the specialisation of the label algebra that is isomorphic to the ghost algebra.

\begin{proposition}\label{prop:labelghostiso}
The ghost algebra $\Gh_n$ is isomorphic to the label algebra $\lab_n(X)$ with $X = \{ 0,1 \}$, and
\begin{align*}
    \aup{1}{0} &= \alpha_1, & \aup{0}{1} &= \alpha_2, & \aup{1}{1} = \aup{0}{0} &= \alpha_3, \\
    \ddo{1}{0} &= \delta_1, & \ddo{0}{1} &= \delta_2, & \ddo{1}{1} =\ddo{0}{0}&= \delta_3, \\
    \glab{1}{0} = \glab{0}{1} &= \gamma_{12} &&& \glab{1}{1} = \glab{0}{0} &= \gamma_3.
\end{align*}
\end{proposition}

\begin{proof}
The isomorphism is found by taking each $\Gh_n$-diagram, numbering the strings and ghosts at each boundary from left to right, removing the ghosts, then replacing each odd number with the label $1$ and each even number with $0$. For example, the images of the diagrams used in the ghost algebra multiplication examples are precisely the diagrams used in the label algebra multiplication examples. This isomorphism is well-defined because the parity of each string endpoint on the boundary is unchanged by multiplication in the ghost algebra. Indeed, since ghosts are left at the endpoints of removed strings, removing strings does not affect the numbering of any remaining strings, and simplifying a diagram by removing pairs of ghosts from each domain can only decrease the numbering of a string by a multiple of two, thereby preserving its parity. 
\end{proof}

We note that some of the label algebra parameters in the above proposition are set to be equal to each other, such as $\aup{1}{1}$ and $\aup{0}{0}$, as their corresponding boundary arcs in the ghost algebra are assigned the same parameter, even though their parities differ. Assigning such boundary arcs the same parameter means that the ghost algebra is \textit{cellular}, in the sense of \cite{GL}, with respect to the anti-involution given by reflecting its basis diagrams about a vertical line, consistent with the zero- and one-boundary TL algebras. The cellularity of the ghost algebra is discussed in \cite{Thesis}, along with some of the resulting representation theory.

In \cite{Ghost}, we also defined the \textit{dilute ghost algebra}, based on the \textit{dilute TL algebra} \cite{BloteNienhuis,Grimm}, whose basis diagrams may include nodes with no strings attached. A dilute label algebra can be defined analogously---see \cite{Thesis} for more detail. One could also define an alternative label algebra with distinct label sets for the top and bottom boundaries; we have avoided this for simplicity, but the non-diagrammatic presentation would be analogous to the one found in this paper.

\section{Diagrammatic generators}\label{s:diagramgenerators}
In this section, we establish a set of generators for the diagrammatic label algebra $\lab_n(X)$.

To distinguish diagrammatic and non-diagrammatic generators, we will write diagrammatic generators with tildes underneath, and non-diagrammatic generators without tildes. The isomorphism between the algebraic and diagrammatic label algebras discussed throughout the rest of the paper will be obtained by identifying each non-diagrammatic generator with its diagrammatic counterpart by adding a tilde.

For any $i=1, \dots , n-1$ and $\mathrm{a},\mathrm{b} \in X$, let
\begin{align}
\did = \,\begin{tikzpicture}[baseline={([yshift=-1mm]current bounding box.center)},scale=0.5]
{
\draw[dotted] (0,0.5)--(3,0.5);
\draw[dotted] (0,-5.5)--(3,-5.5);
\draw [very thick] (0,-5.5)--(0,0.5);
\draw [very thick] (3,-5.5)--(3,0.5);
\draw (0,0) -- (3,0);
\draw (0,-5)--(3,-5);
\node at (1.5,-2.3) [anchor=center] {$\vdots$};
\node at (2.4,0.4) [anchor=south] {\footnotesize $\vphantom{\mathrm{b}}$};
\node at (0.6,-5.4) [anchor=north] {\footnotesize $\vphantom{\mathrm{b}}$};
}
\end{tikzpicture}\; &,
&\dfup{a}{b} &=\, \begin{tikzpicture}[baseline={([yshift=-1mm]current bounding box.center)},scale=0.5]
{
\draw[dotted] (0,0.5)--(3,0.5);
\draw[dotted] (0,-5.5)--(3,-5.5);
\draw [very thick] (0,-5.5)--(0,0.5);
\draw [very thick] (3,-5.5)--(3,0.5);
\draw (0,0) arc (-90:0:0.6 and 0.5);
\draw (3,0) arc (270:180:0.6 and 0.5);
\draw (0,-1)--(3,-1);
\draw (0,-5)--(3,-5);
\node at (1.5,-2.8) [anchor=center] {$\vdots$};
\node at (0.6,0.4) [anchor=south] {\footnotesize $\mathrm{a}$};
\node at (2.4,0.4) [anchor=south] {\footnotesize $\mathrm{b}$};
\node at (2.4,0.4) [anchor=south] {\footnotesize $\vphantom{\mathrm{b}}$};
\node at (0.6,-5.4) [anchor=north] {\footnotesize $\vphantom{\mathrm{b}}$};
}
\end{tikzpicture}\; ,
&\dfdo{a}{b} &=\, \begin{tikzpicture}[baseline={([yshift=-1mm]current bounding box.center)},scale=0.5]
{
\draw[dotted] (0,0.5)--(3,0.5);
\draw[dotted] (0,-5.5)--(3,-5.5);
\draw [very thick] (0,-5.5)--(0,0.5);
\draw [very thick] (3,-5.5)--(3,0.5);
\draw (0,0) -- (3,0);
\draw (0,-5) arc (90:0:0.6 and 0.5);
\draw (3,-5) arc (90:180:0.6 and 0.5);
\draw (0,-4)--(3,-4);
\node at (1.5,-1.8) [anchor=center] {$\vdots$};
\node at (0.6,-5.4) [anchor=north] {\footnotesize $\mathrm{a}\vphantom{\mathrm{b}}$};
\node at (2.4,-5.4) [anchor=north] {\footnotesize $\mathrm{b}$};
\node at (0.6,0.4) [anchor=south] {\footnotesize $\vphantom{\mathrm{b}}$};
}
\end{tikzpicture}\; ,
\\[2mm]
\de{i} = \begin{tikzpicture}[baseline={([yshift=-1mm]current bounding box.center)},scale=0.5]
{
\draw[dotted] (0,0.5)--(3,0.5);
\draw[dotted] (0,-5.5)--(3,-5.5);
\draw [very thick](0,-5.5) -- (0,0.5);
\draw (0,0) -- (3,0);
\draw (1.5,-0.3) node[anchor=center]{$\vdots$};
\draw (0,-1) -- (3,-1);
\draw (0,-1) node[font=\scriptsize,anchor=east]{$i-1$};
\draw (0,-2) arc (90:-90:0.5);
\draw (3,-2) arc (90:270:0.5);
\draw (0,-2) node[font=\scriptsize,anchor=east]{$i$};
\draw (0,-3) node[font=\scriptsize,anchor=east]{$i+1$};
\draw (0,-4) -- (3,-4);
\draw (0,-4) node[font=\scriptsize,anchor=east]{$i+2$};
\draw (1.5,-4.3) node[anchor=center]{$\vdots$};
\draw (0,-5) -- (3,-5);
\draw (0,-5) node[font=\scriptsize,anchor=east]{$n$};
\draw [very thick](3,-5.5) -- (3,0.5);
\draw (0,0) node[font=\scriptsize,anchor=east]{$1$};
\node at (2.4,0.4) [anchor=south] {\footnotesize $\vphantom{\mathrm{b}}$};
\node at (0.6,-5.4) [anchor=north] {\footnotesize $\vphantom{\mathrm{b}}$};
}
\end{tikzpicture}\; &,
&\dwup{a}{b} &= \,\begin{tikzpicture}[baseline={([yshift=-1mm]current bounding box.center)},scale=0.5]
{
\draw[dotted] (0,0.5)--(3,0.5);
\draw[dotted] (0,-5.5)--(3,-5.5);
\draw [very thick] (0,-5.5)--(0,0.5);
\draw [very thick] (3,-5.5)--(3,0.5);
\draw (0,0) arc (-90:0:0.6 and 0.5);
\draw (0,-1) to[out=0,in=180] (3,0);
\draw (0,-5) to[out=0,in=180] (3,-4);
\draw (3,-5) arc (90:180:0.6 and 0.5);
\node at (1.5,-2.3) [anchor=center] {$\vdots$};
\node at (0.6,0.4) [anchor=south] {\footnotesize $\mathrm{a}$};
\node at (2.4,-5.4) [anchor=north] {\footnotesize $\mathrm{b}$};
\node at (2.4,0.4) [anchor=south] {\footnotesize $\vphantom{\mathrm{b}}$};
\node at (0.6,-5.4) [anchor=north] {\footnotesize $\vphantom{\mathrm{b}}$};
}
\end{tikzpicture}\; ,
&\dwdo{a}{b} &= \,\begin{tikzpicture}[baseline={([yshift=-1mm]current bounding box.center)},scale=0.5]
{
\draw[dotted] (0,0.5)--(3,0.5);
\draw[dotted] (0,-5.5)--(3,-5.5);
\draw [very thick] (0,-5.5)--(0,0.5);
\draw [very thick] (3,-5.5)--(3,0.5);
\draw (3,0) arc (270:180:0.6 and 0.5);
\draw (0,0) to[out=0,in=180] (3,-1);
\draw (0,-4) to[out=0,in=180] (3,-5);
\draw (0,-5) arc (90:0:0.6 and 0.5);
\node at (1.5,-2.3) [anchor=center] {$\vdots$};
\node at (2.4,0.4) [anchor=south] {\footnotesize $\mathrm{a}$};
\node at (0.6,-5.4) [anchor=north] {\footnotesize $\mathrm{b}$};
\node at (2.4,0.4) [anchor=south] {\footnotesize $\vphantom{\mathrm{b}}$};
\node at (0.6,-5.4) [anchor=north] {\footnotesize $\vphantom{\mathrm{b}}$};
}
\end{tikzpicture}\; .
\end{align}

We now introduce some diagrammatic terminology. A \textit{boundary link} is a string connecting a node to a boundary. A \textit{link} is a string connecting two nodes on the same side of a diagram. With these definitions, we stress that a boundary link is not a link. A \textit{simple link} is a link connecting adjacent nodes. The nodes on each side of the diagram are numbered from top to bottom, starting at $1$. We say a diagram has a simple link at $k$ on the left (right) if it has a simple link connecting node $k$ to node $k+1$ on its left (right) side. A \textit{throughline} is a string that connects nodes on opposite sides of a diagram. For example, the identity diagram is the unique $\lab_n(X)$-diagram with $n$ throughlines.

We say a diagram is \textit{even (odd)} if it has an even (odd) number of top boundary links and an even (odd) number of bottom boundary links. Note that the total number of boundary links in an $\lab_n(X)$-diagram is even, because there are $2n$ nodes and every node has exactly one string endpoint attached to it. It follows that every diagram is either even or odd.

\begin{proposition}\label{prop:evengen}
Any even $\lab_n(X)$-diagram $\undertilde{T}$ can be written as a product of the diagrammatic generators $\de{i}$, $i=1, \dots, n-1$; $\dfup{a}{b}$, $\mathrm{a},\mathrm{b} \in X$; and $\dfdo{a}{b}$, $\mathrm{a},\mathrm{b} \in X$; without producing any loops or boundary arcs.
\end{proposition}

\begin{proof}
If $\undertilde{T}$ has no boundary links, then it is a basis diagram of the TL algebra, and can be written as a product of $\de{i}$'s using the method from \cite[\S 2]{RSA}. Otherwise, $\undertilde{T}$ has at least two top boundary links, or at least two bottom boundary links.

If $\undertilde{T}$ has at least two top boundary links, consider the leftmost two, and suppose their labels are $\mathrm{a},\mathrm{b} \in X$. We now consider four cases, according to which side of the diagram each of those boundary links is connected to, and write $\undertilde{T}$ as a product of three simpler diagrams in each case.

If both of the leftmost two top boundary links are connected to the left side of the diagram, then we can write $\undertilde{T}$ as a product of three diagrams, as
\begin{align}
\undertilde{T} &=\;  \begin{tikzpicture}[baseline={([yshift=-3mm]current bounding box.center)},xscale=0.5,yscale=0.5]
{
\filldraw[blue!30!cyan!40] (0,0) arc (90:-90:0.65 and 0.75) --cycle;
\filldraw[orange!40] (0,-2.5) arc (90:-90:0.65 and 0.75) --cycle;
\draw (0,-2) to[out=0,in=-90] (1,0.5);
\draw (0,-4.5) ..controls (1.4,-4.5) and (1.5,-2).. (1.5,0.5);
\filldraw[red!10!magenta!30] (0,-5) ..controls (1.8,-5) and (1.9,-2).. (1.9,0.5) -- (2.5,0.5) arc (180:270:0.5) -- (3,-6) arc (90:180:0.5) -- (0.5,-6.5) arc (0:90:0.5) -- cycle;
\draw[dotted] (0,0.5)--(3,0.5);
\draw[dotted] (0,-6.5)--(3,-6.5);
\draw [very thick](0,-6.5) -- (0,0.5);
\draw [very thick](3,-6.5)--(3,0.5);
\node at (1,0.4) [anchor=south] {\footnotesize $\mathrm{a}$};
\node at (1.5,0.4) [anchor=south] {\footnotesize $\mathrm{b}$};
}
\end{tikzpicture}
\; =\;  
\begin{tikzpicture}[baseline={([yshift=-3mm]current bounding box.center)},xscale=0.5,yscale=0.5]
{
\filldraw[blue!30!cyan!40] (0,0) arc (90:-90:0.65 and 0.75) --cycle;
\filldraw[orange!40] (0,-2.5) arc (90:-90:0.65 and 0.75) --cycle;
\draw (0,-2) to[out=0,in=180] (3,-4);
\draw (0,-4.5) -- (6,-4.5);
\draw (0,-5) -- (6,-5);
\node at (1.5,-5.3) [anchor=center] {$\vdots$};
\draw (0,-6)--(6,-6);
\draw (3,0) arc (90:270:0.3 and 0.25);
\draw (3,-1) arc (90:270:0.3 and 0.25);
\node at (2.7,-2.05) [anchor=center] {$\vdots$};
\draw (3,-3) arc (90:270:0.3 and 0.25);
\begin{scope}[shift={(3,0)}]
\draw (0,0) arc(-90:0:0.5);
\draw (3,0) arc(270:180:0.5);
\draw (0,-0.5) arc (90:-90:0.3 and 0.25);
\draw (3,-0.5) arc (90:270:0.3 and 0.25);
\draw (0,-1.5) arc (90:-90:0.3 and 0.25);
\draw (3,-1.5) arc (90:270:0.3 and 0.25);
\node at (2.7,-2.55) [anchor=center] {$\vdots$};
\node at (0.3,-2.55) [anchor=center] {$\vdots$};
\draw (0,-3.5) arc (90:-90:0.3 and 0.25);
\draw (3,-3.5) arc (90:270:0.3 and 0.25);
\node at (1.5,-5.3) [anchor=center] {$\vdots$};
\end{scope}
\begin{scope}[shift={(6,0)}]
\draw (0,0) arc(90:-90:0.3 and 0.25);
\draw (0,-1) arc(90:-90:0.3 and 0.25);
\node at (0.3,-2.05) [anchor=center] {$\vdots$};
\draw (0,-3) arc(90:-90:0.3 and 0.25);
\draw (0,-4) arc(90:-90:0.3 and 0.25);
\filldraw[red!10!magenta!30] (0,-5) ..controls (1.8,-5) and (1.9,-2).. (1.9,0.5) -- (2.5,0.5) arc (180:270:0.5) -- (3,-6) arc (90:180:0.5) -- (0.5,-6.5) arc (0:90:0.5) -- cycle;
\end{scope}
\draw[dotted] (0,0.5)--(9,0.5);
\draw[dotted] (0,-6.5)--(9,-6.5);
\draw [very thick](0,-6.5) -- (0,0.5);
\draw [very thick](3,-6.5)--(3,0.5);
\draw [very thick](6,-6.5)--(6,0.5);
\draw [very thick](9,-6.5)--(9,0.5);
\node at (3.5,0.4) [anchor=south] {\footnotesize $\mathrm{a}$};
\node at (5.5,0.4) [anchor=south] {\footnotesize $\mathrm{b}$};
}
\end{tikzpicture}\; .
\end{align}
Here, each shaded region in each diagram is a placeholder for some number of strings (possibly none) that lie entirely within that region, such that the full diagram with those strings drawn in is an $\lab_n(X)$-diagram. Corresponding shaded regions on each side of the equality contain the same arrangement of strings; this correspondence is indicated by their matching colours, but should also be clear from their shapes and positions.

Note that the two smaller shaded regions on the left (blue, orange) of $\undertilde{T}$ do not contain any boundary links; the two strings explicitly drawn are the leftmost two top boundary links. This means these regions each cover an even number of nodes. Therefore, on the right side of the leftmost diagram in the product of three diagrams, there is an even number of nodes above each end of the topmost throughline, as is required for this construction to exist, given the zigzagging strings in the middle.

If the leftmost two top boundary links in $\undertilde{T}$ are connected to the left and right sides, then we have two cases, according to which of these boundary links is connected to a higher node. If the node for the leftmost boundary link is at the same height or above that of the second-leftmost, then we have
\begin{align}
\undertilde{T} &=\;  \begin{tikzpicture}[baseline={([yshift=-3mm]current bounding box.center)},xscale=0.5,yscale=0.5]
{
\filldraw[blue!30!cyan!40] (0,0) arc (90:-90:0.65 and 0.75) --cycle;
\filldraw[orange!40] (3,-3.5) ..controls (2.1,-3.5) and (1.9,-2).. (1.9,0.5) -- (2.5,0.5) arc (180:270:0.5) --cycle;
\draw (0,-2) to[out=0,in=-90] (1,0.5);
\draw (3,-4) ..controls (1.7,-4) and (1.5,-2).. (1.5,0.5);
\filldraw[red!10!magenta!30] (0,-2.5) to[out=0,in=180] (3,-4.5) -- (3,-6) arc (90:180:0.5) -- (0.5,-6.5) arc (0:90:0.5) -- cycle;
\draw[dotted] (0,0.5)--(3,0.5);
\draw[dotted] (0,-6.5)--(3,-6.5);
\draw [very thick](0,-6.5) -- (0,0.5);
\draw [very thick](3,-6.5)--(3,0.5);
\node at (1,0.4) [anchor=south] {\footnotesize $\mathrm{a}$};
\node at (1.5,0.4) [anchor=south] {\footnotesize $\mathrm{b}$};
}
\end{tikzpicture}
\; =\;  
\begin{tikzpicture}[baseline={([yshift=-3mm]current bounding box.center)},xscale=0.5,yscale=0.5]
{
\filldraw[blue!30!cyan!40] (0,0) arc (90:-90:0.65 and 0.75) --cycle;
\draw (0,-2) ..controls (1.7,-2) and (1.7,-4).. (3,-4);
\draw (3,-4.5) -- (9,-4.5);
\draw (3,-5)--(9,-5);
\draw (3,-6)--(9,-6);
\draw (6,-4)--(9,-4);
\filldraw[red!10!magenta!30] (0,-2.5) to[out=0,in=180] (3,-4.5) -- (3,-6) arc (90:180:0.5) -- (0.5,-6.5) arc (0:90:0.5) -- cycle;
\draw (3,0) arc (90:270:0.3 and 0.25);
\draw (3,-1) arc (90:270:0.3 and 0.25);
\node at (2.7,-2.05) [anchor=center] {$\vdots$};
\draw (3,-3) arc (90:270:0.3 and 0.25);
\begin{scope}[shift={(3,0)}]
\draw (0,0) arc(-90:0:0.5);
\draw (3,0) arc(270:180:0.5);
\draw (0,-0.5) arc (90:-90:0.3 and 0.25);
\draw (3,-0.5) arc (90:270:0.3 and 0.25);
\draw (0,-1.5) arc (90:-90:0.3 and 0.25);
\draw (3,-1.5) arc (90:270:0.3 and 0.25);
\node at (2.7,-2.55) [anchor=center] {$\vdots$};
\node at (0.3,-2.55) [anchor=center] {$\vdots$};
\draw (0,-3.5) arc (90:-90:0.3 and 0.25);
\draw (3,-3.5) arc (90:270:0.3 and 0.25);
\node at (1.5,-5.3) [anchor=center] {$\vdots$};
\end{scope}
\begin{scope}[shift={(6,0)}]
\filldraw[orange!40] (3,-3.5) ..controls (2.1,-3.5) and (1.9,-2).. (1.9,0.5) -- (2.5,0.5) arc (180:270:0.5) --cycle;
\draw (0,0) arc(90:-90:0.3 and 0.25);
\draw (0,-1) arc(90:-90:0.3 and 0.25);
\node at (0.3,-2.05) [anchor=center] {$\vdots$};
\draw (0,-3) arc(90:-90:0.3 and 0.25);
\end{scope}
\draw[dotted] (0,0.5)--(9,0.5);
\draw[dotted] (0,-6.5)--(9,-6.5);
\draw [very thick](0,-6.5) -- (0,0.5);
\draw [very thick](3,-6.5)--(3,0.5);
\draw [very thick](6,-6.5)--(6,0.5);
\draw [very thick](9,-6.5)--(9,0.5);
\node at (3.5,0.4) [anchor=south] {\footnotesize $\mathrm{a}$};
\node at (5.5,0.4) [anchor=south] {\footnotesize $\mathrm{b}$};
}
\end{tikzpicture} \; .
\end{align}
Note that the shaded region connected to the top boundary in $\undertilde{T}$ (orange) must contain an even number (possibly zero) of top boundary links, because there are two outside this region, and the total must be even. This means that this region covers an even number of nodes, so indeed the number of nodes above the right side of the lower shaded region (pink) must be odd, as required.

If the node for the leftmost boundary link is below that of the second-leftmost, then we have
\begin{align}
\undertilde{T} &=\;  \begin{tikzpicture}[baseline={([yshift=-3mm]current bounding box.center)},xscale=0.5,yscale=0.5]
{
\filldraw[blue!30!cyan!40] (0,0) arc (90:-90:0.8 and 1.75) --cycle;
\filldraw[orange!40] (3,-1.5) ..controls (2.2,-1.5) and (2.1,-0.5).. (2.1,0.5) -- (2.5,0.5) arc (180:270:0.5) --cycle;
\draw (0,-4) ..controls (1.1,-4) and (1.2,-2).. (1.2,0.5);
\draw (3,-2) ..controls (1.7,-2) and (1.7,-0.5).. (1.7,0.5);
\filldraw[red!10!magenta!30] (3,-2.5) to[out=180,in=0] (0,-4.5) -- (0,-6) arc (90:0:0.5) -- (2.5,-6.5) arc (180:90:0.5) -- cycle;
\draw[dotted] (0,0.5)--(3,0.5);
\draw[dotted] (0,-6.5)--(3,-6.5);
\draw [very thick](0,-6.5) -- (0,0.5);
\draw [very thick](3,-6.5)--(3,0.5);
\node at (1.2,0.4) [anchor=south] {\footnotesize $\mathrm{a}$};
\node at (1.7,0.4) [anchor=south] {\footnotesize $\mathrm{b}$};
}
\end{tikzpicture}
\; =\;  
\begin{tikzpicture}[baseline={([yshift=-3mm]current bounding box.center)},xscale=0.5,yscale=0.5]
{
\filldraw[blue!30!cyan!40] (0,0) arc (90:-90:0.8 and 1.75) --cycle;
\draw (0,-4) -- (3,-4);
\draw (0,-4.5) -- (6,-4.5);
\draw (0,-5)--(6,-5);
\draw (0,-6)--(6,-6);
\draw (3,0) arc (90:270:0.3 and 0.25);
\draw (3,-1) arc (90:270:0.3 and 0.25);
\node at (2.7,-2.05) [anchor=center] {$\vdots$};
\draw (3,-3) arc (90:270:0.3 and 0.25);
\begin{scope}[shift={(3,0)}]
\draw (0,0) arc(-90:0:0.5);
\draw (3,0) arc(270:180:0.5);
\draw (0,-0.5) arc (90:-90:0.3 and 0.25);
\draw (3,-0.5) arc (90:270:0.3 and 0.25);
\draw (0,-1.5) arc (90:-90:0.3 and 0.25);
\draw (3,-1.5) arc (90:270:0.3 and 0.25);
\node at (2.7,-2.55) [anchor=center] {$\vdots$};
\node at (0.3,-2.55) [anchor=center] {$\vdots$};
\draw (0,-3.5) arc (90:-90:0.3 and 0.25);
\draw (3,-3.5) arc (90:270:0.3 and 0.25);
\node at (1.5,-5.3) [anchor=center] {$\vdots$};
\end{scope}
\begin{scope}[shift={(6,0)}]
\filldraw[red!10!magenta!30] (3,-2.5) to[out=180,in=0] (0,-4.5) -- (0,-6) arc (90:0:0.5) -- (2.5,-6.5) arc (180:90:0.5) -- cycle;
\filldraw[orange!40] (3,-1.5) ..controls (2.2,-1.5) and (2.1,-0.5).. (2.1,0.5) -- (2.5,0.5) arc (180:270:0.5) --cycle;
\draw (0,0) arc(90:-90:0.3 and 0.25);
\draw (0,-1) arc(90:-90:0.3 and 0.25);
\node at (0.3,-2.05) [anchor=center] {$\vdots$};
\draw (0,-3) arc(90:-90:0.3 and 0.25);
\draw (0,-4) ..controls (1.3,-4) and (1.3,-2).. (3,-2);
\end{scope}
\draw[dotted] (0,0.5)--(9,0.5);
\draw[dotted] (0,-6.5)--(9,-6.5);
\draw [very thick](0,-6.5) -- (0,0.5);
\draw [very thick](3,-6.5)--(3,0.5);
\draw [very thick](6,-6.5)--(6,0.5);
\draw [very thick](9,-6.5)--(9,0.5);
\node at (3.5,0.4) [anchor=south] {\footnotesize $\mathrm{a}$};
\node at (5.5,0.4) [anchor=south] {\footnotesize $\mathrm{b}$};
}
\end{tikzpicture} \; ,
\end{align}
where the top left shaded region (blue) has no boundary links, so covers an even number of nodes, as required.

If both of the leftmost two top boundary links are connected to the right side of the diagram, then
\begin{align}
\undertilde{T} &=\; 
\begin{tikzpicture}[baseline={([yshift=-3mm]current bounding box.center)},xscale=0.5,yscale=0.5]
{
\filldraw[blue!30!cyan!40] (3,-1.5) ..controls (2.2,-1.5) and (2.1,-0.5).. (2.1,0.5) -- (2.5,0.5) arc (180:270:0.5) --cycle;
\filldraw[orange!40] (3,-2.5) arc (90:270:0.65 and 0.75) --cycle;
\draw (3,-2) to[out=180,in=270] (1.7,0.5);
\draw (3,-4.5) ..controls (1.6,-4.5) and (1.2,-2).. (1.2,0.5);
\filldraw[red!10!magenta!30] (0,0) ..controls (1.2,0) and (0.7,-5).. (3,-5) -- (3,-6) arc (90:180:0.5) -- (0.5,-6.5) arc (0:90:0.5) -- cycle;
\draw[dotted] (0,0.5)--(3,0.5);
\draw[dotted] (0,-6.5)--(3,-6.5);
\draw [very thick](0,-6.5) -- (0,0.5);
\draw [very thick](3,-6.5)--(3,0.5);
\node at (1.2,0.4) [anchor=south] {\footnotesize $\mathrm{a}$};
\node at (1.7,0.4) [anchor=south] {\footnotesize $\mathrm{b}$};
}
\end{tikzpicture}
\; = \;
\begin{tikzpicture}[baseline={([yshift=-3mm]current bounding box.center)},xscale=0.5,yscale=0.5]
{
\filldraw[red!10!magenta!30] (0,0) ..controls (1.2,0) and (0.7,-5).. (3,-5) -- (3,-6) arc (90:180:0.5) -- (0.5,-6.5) arc (0:90:0.5) -- cycle;
\draw (3,0) arc (90:270:0.3 and 0.25);
\draw (3,-1) arc (90:270:0.3 and 0.25);
\node at (2.7,-2.05) [anchor=center] {$\vdots$};
\draw (3,-3) arc (90:270:0.3 and 0.25);
\draw (3,-4) arc (90:270:0.3 and 0.25);
\draw (3,-4.5)--(9,-4.5);
\draw (3,-5)--(9,-5);
\draw (3,-6)--(9,-6);
\begin{scope}[shift={(3,0)}]
\draw (0,0) arc(-90:0:0.5);
\draw (3,0) arc(270:180:0.5);
\draw (0,-0.5) arc (90:-90:0.3 and 0.25);
\draw (3,-0.5) arc (90:270:0.3 and 0.25);
\draw (0,-1.5) arc (90:-90:0.3 and 0.25);
\draw (3,-1.5) arc (90:270:0.3 and 0.25);
\node at (2.7,-2.55) [anchor=center] {$\vdots$};
\node at (0.3,-2.55) [anchor=center] {$\vdots$};
\draw (0,-3.5) arc (90:-90:0.3 and 0.25);
\draw (3,-3.5) arc (90:270:0.3 and 0.25);
\node at (1.5,-5.3) [anchor=center] {$\vdots$};
\end{scope}
\begin{scope}[shift={(6,0)}]
\filldraw[blue!30!cyan!40] (3,-1.5) ..controls (2.2,-1.5) and (2.1,-0.5).. (2.1,0.5) -- (2.5,0.5) arc (180:270:0.5) --cycle;
\filldraw[orange!40] (3,-2.5) arc (90:270:0.65 and 0.75) --cycle;
\draw (3,-2) to[out=180,in=0] (0,-4);
\draw (0,0) arc(90:-90:0.3 and 0.25);
\draw (0,-1) arc(90:-90:0.3 and 0.25);
\node at (0.3,-2.05) [anchor=center] {$\vdots$};
\draw (0,-3) arc(90:-90:0.3 and 0.25);
\end{scope}
\draw[dotted] (0,0.5)--(9,0.5);
\draw[dotted] (0,-6.5)--(9,-6.5);
\draw [very thick](0,-6.5) -- (0,0.5);
\draw [very thick](3,-6.5)--(3,0.5);
\draw [very thick](6,-6.5)--(6,0.5);
\draw [very thick](9,-6.5)--(9,0.5);
\node at (3.5,0.4) [anchor=south] {\footnotesize $\mathrm{a}$};
\node at (5.5,0.4) [anchor=south] {\footnotesize $\mathrm{b}$};
}
\end{tikzpicture}\; .
\end{align}
Observe that the top right shaded region (blue) must contain an even number of boundary links, and thus each of the top two shaded regions (blue, orange) must cover an even number of nodes, as required.

In each of these four cases, we have written $\undertilde{T}$ as a product of three diagrams. The left and right diagrams are even diagrams with strictly fewer top boundary links than $\undertilde{T}$, and one of them has no bottom boundary links, while the other has the same number as $\undertilde{T}$. The middle diagram is a product $\dfup{a}{b}\de{2}\de{4} \dots \de{2k}$ for some $k$. 

We can then apply this process to the left or right diagram, if they have any top boundary links. Doing this repeatedly, we arrive at a product consisting of the middle diagrams of the form $\dfup{a}{b}\de{2}\de{4} \dots \de{2k}$, diagrams with no boundary links, and at most one even diagram with bottom boundary links, but no top boundary links. Reflecting this process about a horizontal line, it can similarly be used to express a diagram with only bottom boundary links, as a product of diagrams of the form $\dfdo{a}{b} \de{n-2}\de{n-4} \dots \de{n-2k}$, and diagrams with no boundary links. Each diagram without boundary links is a basis diagram of the TL algebra, and can be written as a product of $\de{i}$ generators using the method from \cite[\S 2]{RSA}, which does not produce any loops. What remains is a product of generators of the form $\de{i}$, $\fup{a}{b}$ and $\fdo{a}{b}$, that is equal to $\undertilde{T}$, and does not contain any loops or boundary arcs, as desired.
\end{proof}

\begin{proposition}\label{prop:oddgen}
Any odd $\lab_n(X)$-diagram $\undertilde{D}$ can be written as a product $\dw{a}{b}(j)\undertilde{T}$, such that $\undertilde{T}$ is an even diagram, $\undertilde{W}^a_b(j)$ is one of the diagrams below, and the concatenation of $\dw{a}{b}(j)$ and $\undertilde{T}$ does not produce any loops or boundary arcs. Moreover, given $\undertilde{D}$ and $\dw{a}{b}(j)$, there is at most one diagram $\undertilde{T}$ such that concatenating $\dw{a}{b}(j)$ and $\undertilde{T}$ gives the diagram $\undertilde{D}$ without forming any loops or boundary arcs.
\begin{align}
\dw{a}{b}(0) &\coloneqq \dwup{a}{b} =\;  \begin{tikzpicture}[baseline={([yshift=-1mm]current bounding box.center)},scale=0.5]
{
\draw[dotted] (0,0.5)--(3,0.5);
\draw[dotted] (0,-5.5)--(3,-5.5);
\draw [very thick] (0,-5.5)--(0,0.5);
\draw [very thick] (3,-5.5)--(3,0.5);
\draw (0,0) arc (-90:0:0.6 and 0.5);
\draw (0,-0.5) to[out=0,in=180] (3,0);
\draw (0,-5) to[out=0,in=180] (3,-4.5);
\draw (3,-5) arc (90:180:0.6 and 0.5);
\node at (1.5,-2.3) [anchor=center] {$\vdots$};
\node at (0.6,0.4) [anchor=south] {\footnotesize $\mathrm{a}$};
\node at (2.4,-5.4) [anchor=north] {\footnotesize $\mathrm{b}$};
}
\end{tikzpicture}\; ;
&\dw{a}{b}(n) &\coloneqq \dwdo{a}{b} = \; 
\begin{tikzpicture}[baseline={([yshift=-1mm]current bounding box.center)},scale=0.5]
{
\draw[dotted] (0,0.5)--(3,0.5);
\draw[dotted] (0,-5.5)--(3,-5.5);
\draw [very thick] (0,-5.5)--(0,0.5);
\draw [very thick] (3,-5.5)--(3,0.5);
\draw (3,0) arc (270:180:0.6 and 0.5);
\draw (0,0) to[out=0,in=180] (3,-0.5);
\draw (0,-4.5) to[out=0,in=180] (3,-5);
\draw (0,-5) arc (90:0:0.6 and 0.5);
\node at (1.5,-2.3) [anchor=center] {$\vdots$};
\node at (2.4,0.4) [anchor=south] {\footnotesize $\mathrm{a}$};
\node at (0.6,-5.4) [anchor=north] {\footnotesize $\mathrm{b}$};
}
\end{tikzpicture}\; ,
\end{align}
\begin{align}
\dw{a}{b}(i) &\coloneqq \; \de{i}\de{i-1}\dots \de{2}\de{1}\dwup{a}{b}\;  =\;  
\begin{tikzpicture}[baseline={([yshift=-1mm]current bounding box.center)},scale=0.5]
{
\draw[dotted] (0,0.5)--(3,0.5);
\draw[dotted] (0,-5.5)--(3,-5.5);
\draw [very thick] (0,-5.5)--(0,0.5);
\draw [very thick] (3,-5.5)--(3,0.5);
\draw (3,0) arc (270:180:0.6 and 0.5);
\draw (0,0) to[out=0,in=180] (3,-0.5);
\draw (0,-1.5) to[out=0,in=180] (3,-2);
\draw (0,-2) arc(90:-90:0.3 and 0.25);
\draw (0,-3) to[out=0,in=180] (3,-2.5);
\draw (0,-5) to[out=0,in=180] (3,-4.5);
\draw (3,-5) arc (90:180:0.6 and 0.5);
\node at (1.5,-0.8) [anchor=center] {$\vdots$};
\node at (1.5,-3.55) [anchor=center] {$\vdots$};
\node at (0,-2) [anchor=east] {\scriptsize $i$};
\node at (0,-2.5) [anchor=east] {\scriptsize $i\! + \! 1$};
\node at (2.4,0.4) [anchor=south] {\footnotesize $\mathrm{a}$};
\node at (2.4,-5.4) [anchor=north] {\footnotesize $\mathrm{b}$};
}
\end{tikzpicture}\; , \qquad 1 \leq i \leq n-1.
\end{align}
\end{proposition}

\begin{proof}
Since $\undertilde{D}$ is odd, it has an odd number of top boundary links, and an odd number of bottom boundary links, and thus at least one of each. Hence let $\mathrm{a},\mathrm{b} \in X$ be the labels of the leftmost top and bottom boundary links in $\undertilde{D}$, respectively.

If $\undertilde{D}$ has a top boundary link at $1$ on the left, take $j=0$, so $\dw{a}{b}(j) = \dw{a}{b}(0) = \dwup{a}{b}$. If the leftmost bottom boundary link comes from the left, we have
\begin{align}
\undertilde{D} &=\;
\begin{tikzpicture}[baseline={([yshift=-1mm]current bounding box.center)},xscale=0.5,yscale=0.5]
{
\draw (0,0) arc (-90:0:0.5);
\draw (0,-3.5) to[out=0,in=90] (1.3,-6.5);
\filldraw[blue!30!cyan!40] (0,-4) arc (90:-90:0.8 and 1) --cycle;
\filldraw[red!10!magenta!30] (0,-0.5) to[out=0,in=-90] (0.9,0.5) -- (2.5,0.5) arc (180:270:0.5) -- (3,-6) arc (90:180:0.5) -- (1.7,-6.5) to[out=90,in=0] (0,-3) -- cycle;
\node at (0.5,0.4) [anchor=south] {\footnotesize $\mathrm{a}$};
\node at (1.3,-6.4) [anchor=north] {\footnotesize $\mathrm{b}$};
\draw[dotted] (0,0.5)--(3,0.5);
\draw[dotted] (0,-6.5)--(3,-6.5);
\draw [very thick](0,-6.5) -- (0,0.5);
\draw [very thick](3,-6.5)--(3,0.5);
}
\end{tikzpicture}
\; =\;
\begin{tikzpicture}[baseline={([yshift=-1mm]current bounding box.center)},xscale=0.5,yscale=0.5]
{
\draw (0,0) arc (-90:0:0.5);
\draw (0,-0.5) to[out=0,in=180] (3,0);
\draw (0,-3) to[out=0,in=180] (3,-2.5);
\draw (0,-3.5) to[out=0,in=180] (3,-3);
\draw (0,-4) to[out=0,in=180] (3,-3.5);
\draw (0,-6) to[out=0,in=180] (3,-5.5);
\draw (3,-6) arc (90:180:0.5);
\node at (1.5,-1.3) [anchor=center] {$\vdots$};
\node at (1.5,-4.55) [anchor=center] {$\vdots$};
\node at (0.5,0.4) [anchor=south] {\footnotesize $\mathrm{a}$};
\node at (2.5,-6.4) [anchor=north] {\footnotesize $\mathrm{b}$};
\begin{scope}[shift={(3,0)}]
\draw (0,-3) arc (90:-90:1.2 and 1.5);
\filldraw[blue!30!cyan!40] (0,-3.5) arc (90:-90:0.8 and 1) --cycle;
\filldraw[red!10!magenta!30] (0,0) arc (-90:0:0.5) -- (2.5,0.5) arc (180:270:0.5) -- (3,-6) arc (90:180:0.5) -- (1.7,-6.5) to[out=90,in=0] (0,-2.5) -- cycle;
\end{scope}
\draw[dotted] (0,0.5)--(6,0.5);
\draw[dotted] (0,-6.5)--(6,-6.5);
\draw [very thick](0,-6.5) -- (0,0.5);
\draw [very thick](3,-6.5)--(3,0.5);
\draw [very thick](6,-6.5)--(6,0.5);
}
\end{tikzpicture}\; . \label{eq:pres21}
\end{align}
Note that the upper region (pink) from $\undertilde{D}$ has been deformed in the product, with its left side positioned one node higher up. The numbers of nodes on each side of the region are the same in each case, however. Similar deformations appear throughout this proof.

If the leftmost boundary link comes from the right, we have
\begin{align}
\undertilde{D} &=\; 
\begin{tikzpicture}[baseline={([yshift=-1mm]current bounding box.center)},xscale=0.5,yscale=0.5]
{
\draw (0,0) arc (-90:0:0.5);
\draw (3,-3.5) to[out=180,in=90] (1.4,-6.5);
\filldraw[blue!30!cyan!40] (3,-4) to[out=180,in=90] (1.8,-6.5) -- (2.5,-6.5) arc (180:90:0.5) --cycle;
\filldraw[red!10!magenta!30] (0,-0.5) to[out=0,in=-90] (0.9,0.5) -- (2.5,0.5) arc (180:270:0.5) -- (3,-3) ..controls (1,-3) and (1.5,-6).. (0,-6) -- cycle;
\node at (0.5,0.4) [anchor=south] {\footnotesize $\mathrm{a}$};
\node at (1.4,-6.4) [anchor=north] {\footnotesize $\mathrm{b}$};
\draw[dotted] (0,0.5)--(3,0.5);
\draw[dotted] (0,-6.5)--(3,-6.5);
\draw [very thick](0,-6.5) -- (0,0.5);
\draw [very thick](3,-6.5)--(3,0.5);
}
\end{tikzpicture}
\; =\;
\begin{tikzpicture}[baseline={([yshift=-1mm]current bounding box.center)},xscale=0.5,yscale=0.5]
{
\draw (0,0) arc (-90:0:0.5);
\draw (0,-0.5) to[out=0,in=180] (3,0);
\draw (0,-6) to[out=0,in=180] (3,-5.5);
\draw (3,-6) arc (90:180:0.5);
\node at (1.5,-2.8) [anchor=center] {$\vdots$};
\node at (0.5,0.4) [anchor=south] {\footnotesize $\mathrm{a}$};
\node at (2.5,-6.4) [anchor=north] {\footnotesize $\mathrm{b}$};
\begin{scope}[shift={(3,0)}]
\draw (3,-3.5) to[out=180,in=0] (0,-6);
\filldraw[blue!30!cyan!40] (3,-4) to[out=180,in=90] (1.8,-6.5) -- (2.5,-6.5) arc (180:90:0.5) --cycle;
\filldraw[red!10!magenta!30] (0,0) arc (-90:0:0.5) -- (2.5,0.5) arc (180:270:0.5) -- (3,-3) ..controls (1.2,-3) and (1.3,-5.5).. (0,-5.5) -- cycle;
\end{scope}
\draw[dotted] (0,0.5)--(6,0.5);
\draw[dotted] (0,-6.5)--(6,-6.5);
\draw [very thick](0,-6.5) -- (0,0.5);
\draw [very thick](3,-6.5)--(3,0.5);
\draw [very thick](6,-6.5)--(6,0.5);
}
\end{tikzpicture}\; . \label{eq:pres22}
\end{align}

If $\undertilde{D}$ does not have a top boundary link at $1$ on the left, then we cannot use $\dw{a}{b}(0)$. Recall that a simple link at $k$ is a link connecting nodes $k$ and $k+1$ on the same side of a diagram. If $\undertilde{D}$ has a simple link on the left, let $i$ be the position of the topmost simple link, and take $j=i$. To find $\undertilde{T}$, we need to consider cases based on which side of the diagram the leftmost boundary links are connected to.

If the leftmost top and bottom boundary links are both connected to the left side of the diagram, then
\vspace{-3mm}
\begin{align}
\undertilde{D} &= \;
\begin{tikzpicture}[baseline={([yshift=-1mm]current bounding box.center)},xscale=0.5,yscale=0.5]
{
\filldraw[blue!30!cyan!40] (0,0) arc (90:-90:1 and 1.25) -- (0,-2) arc (-90:90:0.65 and 0.75) --cycle;
\draw (0,-1) arc (90:-90:0.3 and 0.25);
\draw (0,-3) to[out=0,in=-90] (1.5,0.5);
\filldraw[red!10!magenta!30] (0,-3.5) to[out=0,in=-90] (2,0.5) -- (2.5,0.5) arc (180:270:0.5) -- (3,-6) arc (90:180:0.5) -- (1.6,-6.5) to[out=90,in=0] (0,-4) -- cycle;
\draw (0,-4.5) to[out=0,in=90] (1.2,-6.5);
\filldraw[orange!40] (0,-5) arc (90:-90:0.5)--cycle;
\node at (1.5,0.4) [anchor=south] {\footnotesize $\mathrm{a}$};
\node at (1.2,-6.4) [anchor=north] {\footnotesize $\mathrm{b}$};
\draw[dotted] (0,0.5)--(3,0.5);
\draw[dotted] (0,-6.5)--(3,-6.5);
\draw [very thick](0,-6.5) -- (0,0.5);
\draw [very thick](3,-6.5)--(3,0.5);
}
\end{tikzpicture}
\; =\; 
\begin{tikzpicture}[baseline={([yshift=-1mm]current bounding box.center)},xscale=0.5,yscale=0.5]
{
\draw (0,0) to[out=0,in=180] (3,-0.5);
\draw (0,-0.5) to[out=0,in=180] (3,-1);
\draw (0,-1) arc (90:-90:0.3 and 0.25);
\draw (0,-2) to[out=0,in=180] (3,-1.5);
\draw (0,-2.5) to[out=0,in=180] (3,-2);
\draw (0,-3) to[out=0,in=180] (3,-2.5);
\draw (0,-3.5) to[out=0,in=180] (3,-3);
\draw (0,-4) to[out=0,in=180] (3,-3.5);
\draw (0,-4.5) to[out=0,in=180] (3,-4);
\draw (0,-5) to[out=0,in=180] (3,-4.5);
\draw (0,-6) to[out=0,in=180] (3,-5.5);
\draw (3,0) arc (270:180:0.5);
\draw (3,-6) arc (90:180:0.5);
\node at (1.5,-0.5) {.};
\node at (1.5,-2) {.};
\node at (1.5,-3.5) {.};
\node at (1.5,-5.45) {.};
\node at (1.5,-5.05) {.};
\node at (1.5,-5.25) {.};
\node at (2.5,0.4) [anchor=south] {\footnotesize $\mathrm{a}$};
\node at (2.5,-6.4) [anchor=north] {\footnotesize $\mathrm{b}$};
\begin{scope}[shift={(3,0)}]
\filldraw[blue!30!cyan!40] (0,-0.5) arc (90:-90:0.65 and 0.75) -- (0,-1.5) arc (-90:90:0.25) --cycle;
\draw (0,0) arc (90:-90:1 and 1.25);
\filldraw[red!10!magenta!30] (0,-3) to[out=0,in=-90] (2,0.5) -- (2.5,0.5) arc (180:270:0.5) -- (3,-6) arc (90:180:0.5) -- (1.6,-6.5) to[out=90,in=0] (0,-3.5) -- cycle;
\draw (0,-4) arc (90:-90:0.85 and 1);
\filldraw[orange!40] (0,-4.5) arc (90:-90:0.5)--cycle;
\end{scope}
\draw[dotted] (0,0.5)--(6,0.5);
\draw[dotted] (0,-6.5)--(6,-6.5);
\draw [very thick](0,-6.5) -- (0,0.5);
\draw [very thick](3,-6.5)--(3,0.5);
\draw [very thick](6,-6.5)--(6,0.5);
}
\end{tikzpicture}\; , \label{eq:pres31a}
\end{align}
where a single dot is used in place of a vertical ellipsis where an ellipsis does not fit.

If the leftmost top boundary link in $\undertilde{D}$ comes from the left, and the leftmost bottom boundary link comes from the right, then we have
\begin{align}
\undertilde{D} &=\; 
\begin{tikzpicture}[baseline={([yshift=-1mm]current bounding box.center)},xscale=0.5,yscale=0.5]
{
\filldraw[blue!30!cyan!40] (0,0) arc (90:-90:1.25 and 1.75) -- (0,-2.5) arc (-90:90:0.75) --cycle;
\draw (0,-1.5) arc (90:-90:0.3 and 0.25);
\draw (0,-4) to[out=0,in=-90] (1.5,0.5);
\filldraw[red!10!magenta!30] (0,-4.5) to[out=0,in=-90] (2,0.5) -- (2.5,0.5) arc (180:270:0.5) -- (3,-3.5) to[out=180,in=0] (0,-6) -- cycle;
\draw (3,-4) to[out=180,in=90] (1.5,-6.5);
\filldraw[orange!40] (3,-4.5) to[out=180,in=90] (1.9,-6.5) -- (2.5,-6.5) arc (180:90:0.5)--cycle;
\node at (1.5,0.4) [anchor=south] {\footnotesize $\mathrm{a}$};
\node at (1.5,-6.4) [anchor=north] {\footnotesize $\mathrm{b}$};
\draw[dotted] (0,0.5)--(3,0.5);
\draw[dotted] (0,-6.5)--(3,-6.5);
\draw [very thick](0,-6.5) -- (0,0.5);
\draw [very thick](3,-6.5)--(3,0.5);
}
\end{tikzpicture}
\; =\; 
\begin{tikzpicture}[baseline={([yshift=-1mm]current bounding box.center)},xscale=0.5,yscale=0.5]
{
\draw (0,-1.5) arc (90:-90:0.3 and 0.25);
\draw (0,0) to[out=0,in=180] (3,-0.5);
\draw (0,-1) to[out=0,in=180] (3,-1.5);
\draw (0,-2.5) to[out=0,in=180] (3,-2);
\draw (0,-3.5) to[out=0,in=180] (3,-3);
\draw (0,-4) to[out=0,in=180] (3,-3.5);
\draw (0,-4.5) to[out=0,in=180] (3,-4);
\draw (0,-6) to[out=0,in=180] (3,-5.5);
\draw (3,0) arc (270:180:0.5);
\draw (3,-6) arc (90:180:0.5);
\node at (1.5,-0.75) {.};
\node at (1.5,-0.55) {.};
\node at (1.5,-0.95) {.};
\node at (1.5,-2.75) {.};
\node at (1.5,-2.55) {.};
\node at (1.5,-2.95) {.};
\node at (1.5,-4.8) [anchor=center] {$\vdots$};
\node at (2.5,0.4) [anchor=south] {\footnotesize $\mathrm{a}$};
\node at (2.5,-6.4) [anchor=north] {\footnotesize $\mathrm{b}$};
\begin{scope}[shift={(3,0)}]
\filldraw[blue!30!cyan!40] (0,-0.5) arc (90:-90:1 and 1.25) -- (0,-2) arc (-90:90:0.3 and 0.25) --cycle;
\draw (0,0) arc (90:-90:1.4 and 1.75);
\filldraw[red!10!magenta!30] (0,-4) to[out=0,in=-90] (2,0.5) -- (2.5,0.5) arc (180:270:0.5) -- (3,-3.5) to[out=180,in=0] (0,-5.5) -- cycle;
\draw (0,-6) to[out=0,in=180] (3,-4);
\filldraw[orange!40] (3,-4.5) to[out=180,in=90] (1.9,-6.5) -- (2.5,-6.5) arc (180:90:0.5)--cycle;
\end{scope}
\draw[dotted] (0,0.5)--(6,0.5);
\draw[dotted] (0,-6.5)--(6,-6.5);
\draw [very thick](0,-6.5) -- (0,0.5);
\draw [very thick](3,-6.5)--(3,0.5);
\draw [very thick](6,-6.5)--(6,0.5);
}
\end{tikzpicture}\; . \label{eq:pres31b}
\end{align}

If the leftmost top boundary link comes from the right, and the leftmost bottom boundary link comes from the left, we must also consider whether the bottom boundary link comes from above or below the simple link at $i$ on the left. If below, we have
\begin{align}
\undertilde{D}=\; \begin{tikzpicture}[baseline={([yshift=-1mm]current bounding box.center)},xscale=0.5,yscale=0.5]
{
\filldraw[blue!30!cyan!40] (3,-2.5) ..controls (2,-2.5) and (1.8,-0.5).. (1.8,0.5) -- (2.5,0.5) arc (180:270:0.5) --cycle;
\draw (0,-1.5) arc (90:-90:0.3 and 0.25);
\draw (0,-4) to[out=0,in=90] (1.2,-6.5);
\filldraw[red!10!magenta!30] (0,0) ..controls (1.6,0) and (1.1,-3.5).. (3,-3.5) -- (3,-6) arc (90:180:0.5) -- (1.6,-6.5) to[out=90,in=0] (0,-3.5) -- (0,-2.5) arc (-90:90:0.75) -- cycle;
\draw (3,-3) ..controls (2,-3) and (1.4,-1).. (1.4,0.5);
\filldraw[orange!40] (0,-4.5) arc (90:-90:0.65 and 0.75)--cycle;
\draw[dotted] (0,0.5)--(3,0.5);
\draw[dotted] (0,-6.5)--(3,-6.5);
\draw [very thick](0,-6.5) -- (0,0.5);
\draw [very thick](3,-6.5)--(3,0.5);
\node at (1.4,0.4) [anchor=south] {\footnotesize $\mathrm{a}$};
\node at (1.2,-6.4) [anchor=north] {\footnotesize $\mathrm{b}$};
}
\end{tikzpicture}
\; &=\; 
\begin{tikzpicture}[baseline={([yshift=-1mm]current bounding box.center)},xscale=0.5,yscale=0.5]
{
\draw (0,-1.5) arc (90:-90:0.3 and 0.25);
\draw (0,0) to[out=0,in=180] (3,-0.5);
\draw (0,-1) to[out=0,in=180] (3,-1.5);
\draw (0,-2.5) to[out=0,in=180] (3,-2);
\draw (0,-3.5) to[out=0,in=180] (3,-3);
\draw (0,-4) to[out=0,in=180] (3,-3.5);
\draw (0,-4.5) to[out=0,in=180] (3,-4);
\draw (0,-6) to[out=0,in=180] (3,-5.5);
\draw (3,0) arc (270:180:0.5);
\draw (3,-6) arc (90:180:0.5);
\node at (2.5,0.4) [anchor=south] {\footnotesize $\mathrm{a}$};
\node at (2.5,-6.4) [anchor=north] {\footnotesize $\mathrm{b}$};
\node at (1.5,-0.75) {.};
\node at (1.5,-0.55) {.};
\node at (1.5,-0.95) {.};
\node at (1.5,-2.75) {.};
\node at (1.5,-2.55) {.};
\node at (1.5,-2.95) {.};
\node at (1.5,-4.8) [anchor=center] {$\vdots$};
\begin{scope}[shift={(3,0)}]
\filldraw[blue!30!cyan!40] (3,-2.5) ..controls (2,-2.5) and (1.8,-0.5).. (1.8,0.5) -- (2.5,0.5) arc (180:270:0.5) --cycle;
\draw (0,0) ..controls (1.9,0) and (1.4,-3).. (3,-3);
\filldraw[red!10!magenta!30] (0,-0.5) ..controls (1.6,-0.5) and (1.1,-3.5).. (3,-3.5) -- (3,-6) arc (90:180:0.5) -- (1.6,-6.5) to[out=90,in=0] (0,-3) -- (0,-2) arc (-90:90:0.3 and 0.25) -- cycle;
\draw (0,-3.5) arc (90:-90:1.07 and 1.25);
\filldraw[orange!40] (0,-4) arc (90:-90:0.65 and 0.75)--cycle;
\end{scope}
\draw[dotted] (0,0.5)--(6,0.5);
\draw[dotted] (0,-6.5)--(6,-6.5);
\draw [very thick](0,-6.5) -- (0,0.5);
\draw [very thick](3,-6.5)--(3,0.5);
\draw [very thick](6,-6.5)--(6,0.5);
}
\end{tikzpicture}\; . \label{eq:pres32a}
\end{align}
and if above, we have
\vspace{-3mm}
\begin{align}
\undertilde{D} =\; \begin{tikzpicture}[baseline={([yshift=-1mm]current bounding box.center)},xscale=0.5,yscale=0.5]
{
\draw (0,-2) ..controls (2.1,-2.1) and (1.5,-5.5).. (1.6,-6.5);
\filldraw[orange!40] (3,-2) ..controls (1.9,-2) and (1.8,0).. (1.8,0.5) -- (2.5,0.5) arc (180:270:0.5) --cycle;
\draw (0,-4) arc (90:-90:0.3 and 0.25);
\filldraw[red!10!magenta!30] (0,0) ..controls (1.4,0) and (1.3,-3).. (3,-3) -- (3,-6) arc (90:180:0.5) -- (2,-6.5) ..controls (2.2,-1.5) and (0.5,-1.5)..(0,-1.5) -- cycle;
\draw (3,-2.5) ..controls (1.9,-2.5) and (1.4,-1).. (1.4,0.5);
\filldraw[blue!30!cyan!40] (0,-2.5) arc (90:-90:1.25 and 1.75) -- (0,-5) arc (-90:90:0.75) --cycle;
\node at (1.4,0.4) [anchor=south] {\footnotesize $\mathrm{a}$};
\node at (1.6,-6.4) [anchor=north] {\footnotesize $\mathrm{b}$};
\draw[dotted] (0,0.5)--(3,0.5);
\draw[dotted] (0,-6.5)--(3,-6.5);
\draw [very thick](0,-6.5) -- (0,0.5);
\draw [very thick](3,-6.5)--(3,0.5);
}
\end{tikzpicture}
\; &=\; 
\begin{tikzpicture}[baseline={([yshift=-1mm]current bounding box.center)},xscale=0.5,yscale=0.5]
{
\draw (0,-4) arc (90:-90:0.3 and 0.25);
\draw (0,0) to[out=0,in=180] (3,-0.5);
\draw (0,-1.5) to[out=0,in=180] (3,-2);
\draw (0,-2) to[out=0,in=180] (3,-2.5);
\draw (0,-2.5) to[out=0,in=180] (3,-3);
\draw (0,-3.5) to[out=0,in=180] (3,-4);
\draw (0,-5) to[out=0,in=180] (3,-4.5);
\draw (0,-6) to[out=0,in=180] (3,-5.5);
\draw (3,0) arc (270:180:0.5);
\draw (3,-6) arc (90:180:0.5);
\node at (2.5,0.4) [anchor=south] {\footnotesize $\mathrm{a}$};
\node at (2.5,-6.4) [anchor=north] {\footnotesize $\mathrm{b}$};
\node at (1.5,-5.25) {.};
\node at (1.5,-5.05) {.};
\node at (1.5,-5.45) {.};
\node at (1.5,-3.25) {.};
\node at (1.5,-3.05) {.};
\node at (1.5,-3.45) {.};
\node at (1.5,-0.8) [anchor=center] {$\vdots$};
\begin{scope}[shift={(3,0)}]
\filldraw[orange!40] (3,-2) ..controls (1.9,-2) and (1.8,0).. (1.8,0.5) -- (2.5,0.5) arc (180:270:0.5) --cycle;
\draw (0,0) ..controls (1.9,0) and (1.4,-2.5).. (3,-2.5);
\filldraw[red!10!magenta!30] (0,-0.5) ..controls (1.6,-0.5) and (1.1,-3).. (3,-3) -- (3,-6) arc (90:180:0.5) -- (2,-6.5) ..controls (2,-2.5) and (0.5,-2).. (0,-2) -- cycle;
\draw (0,-2.5) arc (90:-90:1.2 and 1.75);
\filldraw[blue!30!cyan!40] (0,-3) arc (90:-90:0.8 and 1.25) -- (0,-4.5) arc (-90:90:0.3 and 0.25) --cycle;
\end{scope}
\draw[dotted] (0,0.5)--(6,0.5);
\draw[dotted] (0,-6.5)--(6,-6.5);
\draw [very thick](0,-6.5) -- (0,0.5);
\draw [very thick](3,-6.5)--(3,0.5);
\draw [very thick](6,-6.5)--(6,0.5);
}
\end{tikzpicture}\; . \label{eq:pres32b}
\end{align}

If the leftmost top and bottom boundary links both come from the right, then we have
\begin{align}
\begin{tikzpicture}[baseline={([yshift=-1mm]current bounding box.center)},xscale=0.5,yscale=0.5]
{
\filldraw[orange!40] (3,-2) ..controls (1.9,-2) and (1.8,0).. (1.8,0.5) -- (2.5,0.5) arc (180:270:0.5) --cycle;
\draw (0,-2) arc (90:-90:0.3 and 0.25);
\filldraw[red!10!magenta!30] (0,0) ..controls (1.4,0) and (1.1,-3).. (3,-3) -- (3,-4) ..controls (1.3,-4) and (1.4,-6).. (0,-6) -- (0,-3) arc (-90:90:0.65 and 0.75)-- cycle;
\draw (3,-2.5) ..controls (1.9,-2.5) and (1.4,-1).. (1.4,0.5);
\draw (3,-4.5) ..controls (2.1,-4.5) and (1.4,-5.4).. (1.4,-6.5);
\filldraw[blue!30!cyan!40] (3,-5) to[out=180,in=90] (1.8,-6.5) -- (2.5,-6.5) arc (180:90:0.5) --cycle;
\node at (1.4,0.4) [anchor=south] {\footnotesize $\mathrm{a}$};
\node at (1.4,-6.4) [anchor=north] {\footnotesize $\mathrm{b}$};
\draw[dotted] (0,0.5)--(3,0.5);
\draw[dotted] (0,-6.5)--(3,-6.5);
\draw [very thick](0,-6.5) -- (0,0.5);
\draw [very thick](3,-6.5)--(3,0.5);
}
\end{tikzpicture}
\; &=\; 
\begin{tikzpicture}[baseline={([yshift=-1mm]current bounding box.center)},xscale=0.5,yscale=0.5]
{
\draw (0,0) to[out=0,in=180] (3,-0.5);
\draw (0,-1.5) to[out=0,in=180] (3,-2);
\draw (0,-2) arc (90:-90:0.3 and 0.25);
\draw (0,-3) to[out=0,in=180] (3,-2.5);
\draw (0,-6) to[out=0,in=180] (3,-5.5);
\draw (3,0) arc (270:180:0.5);
\draw (3,-6) arc (90:180:0.5);
\node at (2.5,0.4) [anchor=south] {\footnotesize $\mathrm{a}$};
\node at (2.5,-6.4) [anchor=north] {\footnotesize $\mathrm{b}$};
\node at (1.5,-0.8) [anchor=center] {$\vdots$};
\node at (1.5,-4.05) [anchor=center] {$\vdots$};
\begin{scope}[shift={(3,0)}]
\filldraw[orange!40] (3,-2) ..controls (1.9,-2) and (1.8,0).. (1.8,0.5) -- (2.5,0.5) arc (180:270:0.5) --cycle;
\draw (0,0) ..controls (1.9,0) and (1.4,-2.5).. (3,-2.5);
\filldraw[red!10!magenta!30] (0,-0.5) ..controls (1.4,-0.5) and (1.3,-3).. (3,-3) -- (3,-4) ..controls (1.3,-4) and (1.4,-5.5).. (0,-5.5) -- (0,-2.5) arc (-90:90:0.3 and 0.25)-- cycle;
\draw (3,-4.5) to[out=180,in=0] (0,-6);
\filldraw[blue!30!cyan!40] (3,-5) to[out=180,in=90] (1.8,-6.5) -- (2.5,-6.5) arc (180:90:0.5) --cycle;
\end{scope}
\draw[dotted] (0,0.5)--(6,0.5);
\draw[dotted] (0,-6.5)--(6,-6.5);
\draw [very thick](0,-6.5) -- (0,0.5);
\draw [very thick](3,-6.5)--(3,0.5);
\draw [very thick](6,-6.5)--(6,0.5);
}
\end{tikzpicture}\; . \label{eq:pres32c}
\end{align}

Now, if the left side of $\undertilde{D}$ has neither a top boundary link at $1$ nor any simple links, then every string connected to the left side of $\undertilde{D}$ must be a throughline or a bottom boundary link. If all of the strings connected to the left were throughlines, then $\undertilde{D}=\did$ would be even. Hence $\undertilde{D}$ has at least one bottom boundary link on the left. In particular, since planarity implies that bottom boundary links come from nodes below any throughlines, there is a bottom boundary link at $n$ on the left. In this case, take $j=n$. Then we have
\vspace{-3mm}
\begin{align}
\undertilde{D}=\; \begin{tikzpicture}[baseline={([yshift=-1mm]current bounding box.center)},xscale=0.5,yscale=0.5]
{
\filldraw[blue!30!cyan!40] (3,-2) ..controls (1.9,-2) and (1.8,0).. (1.8,0.5) -- (2.5,0.5) arc (180:270:0.5) --cycle;
\filldraw[red!10!magenta!30] (0,0) ..controls (1.4,0) and (1.1,-3).. (3,-3) -- (3,-6) arc (90:180:0.5) -- (0.9,-6.5) to[out=90,in=0] (0,-5.5) -- cycle;
\draw (3,-2.5) ..controls (1.9,-2.5) and (1.4,-1).. (1.4,0.5);
\draw (0,-6) arc (90:0:0.5);
\node at (1.4,0.4) [anchor=south] {\footnotesize $\mathrm{a}$};
\node at (0.5,-6.4) [anchor=north] {\footnotesize $\mathrm{b}$};
\draw[dotted] (0,0.5)--(3,0.5);
\draw[dotted] (0,-6.5)--(3,-6.5);
\draw [very thick](0,-6.5) -- (0,0.5);
\draw [very thick](3,-6.5)--(3,0.5);
}
\end{tikzpicture}
\; &=\; 
\begin{tikzpicture}[baseline={([yshift=-1mm]current bounding box.center)},xscale=0.5,yscale=0.5]
{
\draw (0,0) to[out=0,in=180] (3,-0.5);
\draw (0,-6) arc (90:0:0.5);
\draw (0,-5.5) to[out=0,in=180] (3,-6);
\draw (3,0) arc (270:180:0.5);
\node at (1.5,-2.8) [anchor=center] {$\vdots$};
\node at (2.5,0.4) [anchor=south] {\footnotesize $\mathrm{a}$};
\node at (0.5,-6.4) [anchor=north] {\footnotesize $\vphantom{\mathrm{d}}\mathrm{b}$};
\begin{scope}[shift={(3,0)}]
\filldraw[blue!30!cyan!40] (3,-2) ..controls (1.9,-2) and (1.8,0).. (1.8,0.5) -- (2.5,0.5) arc (180:270:0.5) --cycle;
\draw (0,0) ..controls (1.9,0) and (1.4,-2.5).. (3,-2.5);
\filldraw[red!10!magenta!30] (0,-0.5) ..controls (1.4,-0.5) and (1.3,-3).. (3,-3) -- (3,-6) arc (90:180:0.5) -- (0.5,-6.5) arc(0:90:0.5) -- cycle;
\end{scope}
\draw[dotted] (0,0.5)--(6,0.5);
\draw[dotted] (0,-6.5)--(6,-6.5);
\draw [very thick](0,-6.5) -- (0,0.5);
\draw [very thick](3,-6.5)--(3,0.5);
\draw [very thick](6,-6.5)--(6,0.5);
}
\end{tikzpicture}\; . \label{eq:pres41}
\end{align}

In all cases we have written $\undertilde{D}$ as a product $\dw{a}{b}(j)\undertilde{T}$, where $\undertilde{T}$ is an even diagram, and concatenating $\dw{a}{b}(j)$ and $\undertilde{T}$ does not produce any loops or boundary arcs. 

For the uniqueness of $\undertilde{T}$, fix $\dw{a}{b}(j)$ and $\undertilde{D}$, and suppose there exists $\undertilde{T}$ such that $\dw{a}{b}(j)\undertilde{T}=\undertilde{D}$, and concatenating $\dw{a}{b}(j)$ and $\undertilde{T}$ does not produce any loops or boundary arcs. We will argue that each string in $\undertilde{T}$ is uniquely determined, given $\dw{a}{b}(j)$ and $\undertilde{D}$. The key observation here is that $\dw{a}{b}(j)$ does not have any links on its right side. As the concatenation of $\dw{a}{b}(j)$ and $\undertilde{T}$ does not produce any loops or boundary arcs, each string in $\dw{a}{b}(j)$ and $\undertilde{T}$ must remain in $\undertilde{D}$. Now, each string on the right side of $\dw{a}{b}(j)$ is either a throughline or a boundary link, so one of its endpoints in $\undertilde{D}$ is already fixed. If the other endpoint of the corresponding string in $\undertilde{D}$ is connected to the right side of $\undertilde{D}$, then $\undertilde{T}$ must have a throughline connecting the relevant nodes. Indeed, the only alternative would be if the string from the right of $\dw{a}{b}(j)$ was connected to the node on the right of $\undertilde{D}$ via an alternating sequence of links on the left of $\undertilde{T}$ and the right of $\dw{a}{b}(j)$, followed by a different throughline in $\undertilde{T}$, but there are no links on the right side of $\dw{a}{b}(j)$, so this cannot occur. Similarly, given a string on the right of $\dw{a}{b}(j)$, if the other end of the corresponding string in $\undertilde{D}$ is connected to a node on the left of $\undertilde{D}$, then $\undertilde{T}$ must have a link connecting the given string on the right of $\dw{a}{b}(j)$ to the appropriate throughline of $\dw{a}{b}(j)$. If the other end of the corresponding string in $\undertilde{D}$ is connected to a boundary, then $\undertilde{T}$ must have the appropriate boundary link at this node, unless it is the leftmost top or bottom boundary link, in which case $\undertilde{T}$ has a link connecting this node to node $1$ or $n$ on the left of $\undertilde{T}$, respectively. Thus all strings connected to nodes on the left of $\undertilde{T}$ are uniquely determined by $\dw{a}{b}(j)$ and $\undertilde{D}$. The remaining strings are connected only to nodes on the right of $\undertilde{T}$, and perhaps the boundaries; these must match the links and boundary links on the right side of $\undertilde{D}$. Therefore, given $\dw{a}{b}(j)$ and $\undertilde{D}$, there is at most one diagram $\undertilde{T}$ such that both $\dw{a}{b}(j)\undertilde{T} = \undertilde{D}$, and concatenating $\dw{a}{b}(j)$ and $\undertilde{T}$ does not produce any loops or boundary arcs.
\end{proof}

\begin{theorem}\label{thm:diaggen}
The diagrammatic label algebra $\lab_n(X)$ is generated by $\did$, $e_i$, $\dfup{a}{b}$, $\dfdo{a}{b}$, $\dwup{a}{b}$ and $\dwdo{a}{b}$, where $i = 1, \dots, n-1$ and $\mathrm{a},\mathrm{b} \in X$.
\end{theorem}

\begin{proof}
This follows from Propositions \ref{prop:evengen} and \ref{prop:oddgen}, since every diagram is either even or odd.
\end{proof}

\section{Algebraic presentation and homomorphism}\label{s:algrels}
In this section, we introduce generators and relations for an algebra $\A_n(X)$, and define a map $\phi$ from $\A_n(X)$ to the diagrammatic label algebra $\lab_n(X)$. In later sections, we will show that $\phi$ is an isomorphism.

Let $X$ be a set of labels and $n \in \N$, and let $\beta$, $\aup{a}{b}$, $\ddo{a}{b}$, $\glab{a}{b}$ be indeterminates, for each $\mathrm{a},\mathrm{b} \in X$. We define the \textit{algebraic label algebra} $\A_n(X)$ to be generated by the identity $\id$; $e_i$, $1 \leq i \leq n-1$; $\fup{a}{b}$, $\mathrm{a},\mathrm{b} \in X$; $\fdo{a}{b}$, $\mathrm{a},\mathrm{b} \in X$; $\wup{a}{b}$, $\mathrm{a},\mathrm{b} \in X$; and $\wdo{a}{b}$, $\mathrm{a},\mathrm{b} \in X$, subject to the following relations for all $\mathrm{a},\mathrm{b},\mathrm{c},\mathrm{d} \in X$.

\vspace{2.5mm}
\noindent \textbf{Generators that commute:}
\begin{align}
e_ie_j &= e_je_i, &&\abs{i-j} \geq 2, \tag{L1} \label{eq:eiej}\\
\fup{a}{b} e_j &= e_j \fup{a}{b}, && 2 \leq j \leq n-1, \tag{L2} \label{eq:fupej}\\
\fdo{a}{b} e_j &= e_j \fdo{a}{b}, && 1 \leq j \leq n-2, \tag{L3} \label{eq:fdoej}\\
\fup{a}{b} \fdo{c}{d} &= \fdo{c}{d} \fup{a}{b}, && n \geq 2. \tag{L4} \label{eq:fupfdo}
\end{align}

\noindent\textbf{Reduction of TL generators:}
\begin{align}
    e_ie_je_i &= e_i, &&\abs{i-j}=1. \tag{L5} \label{eq:TLreduce}
\end{align}

\noindent\textbf{Generators that almost commute:}
\begin{align}
e_j \wup{a}{b} &= \wup{a}{b} e_{j-1}, &&2 \leq j \leq n-1, \tag{L6} \label{eq:ejwup}\\
e_j \wdo{a}{b} &= \wdo{a}{b} e_{j+1}, && 1 \leq j \leq n-2, \tag{L7} \label{eq:ejwdo}\\
\fdo{a}{b} \wup{c}{d} &= \wup{c}{a} e_{n-1} \fdo{b}{d}, && n \geq 2, \tag{L8} \label{eq:fdowup}\\
\fup{a}{b} \wdo{c}{d} &= \wdo{a}{d} e_1 \fup{b}{c}, && n \geq 2, \tag{L9} \label{eq:fupwdo}\\
\wup{a}{b} \fup{c}{d} &= \fup{a}{c} e_1 \wup{d}{b}, && n \geq 2, \tag{L10} \label{eq:wupfup}\\
\wdo{a}{b} \fdo{c}{d} &= \fdo{b}{c} e_{n-1} \wdo{a}{d}, && n \geq 2.\tag{L11} \label{eq:wdofdo}
\end{align}

\noindent\textbf{Swapping $\wup{a}{b}$ and $\wdo{a}{b}$:}
\begin{align}
e_1 \wup{a}{b} &= e_1e_2 \dots e_{n-1} \wdo{a}{b}, && n \geq 2, \tag{L12} \label{eq:e1wup}\\
\wup{a}{b} e_{n-1} &= \wdo{a}{b} e_1e_2 \dots e_{n-1}, && n \geq 2, \tag{L13} \label{eq:wupen1}\\
e_{n-1} \wdo{a}{b} &= e_{n-1} e_{n-2} \dots e_1 \wup{a}{b}, && n \geq 2, \tag{L14} \label{eq:en1wdo}\\
\wdo{a}{b} e_1 &= \wup{a}{b}e_{n-1}e_{n-2} \dots e_1, && n \geq 2.\tag{L15} \label{eq:wdoe1}
\end{align}

\noindent\textbf{Reduction of odd generators:}
\begin{align}
\wup{a}{b} \wup{c}{d} &= \fup{a}{c}e_1e_2\dots e_{n-1} \fdo{b}{d}, \tag{L16} \label{eq:wupwup}\\
\wdo{a}{b}\wdo{c}{d} &= \fdo{b}{d} e_{n-1}e_{n-2} \dots e_1 \fup{a}{c}, \tag{L17} \label{eq:wdowdo}\\
\wdo{a}{b}e_1 \wup{c}{d} &= \fup{a}{c}\fdo{b}{d}, && n \geq 2.\tag{L18} \label{eq:wdoe1wup}
\end{align}

\noindent\textbf{Parameter $\beta$:}
\begin{align}
e_j^2 &= \beta e_j, && 1 \leq j \leq n-1.\tag{L19} \label{eq:beta}
\end{align}

\noindent\textbf{Parameter $\aup{a}{b}$:}
\begin{align}
\fup{c}{a} \fup{b}{d} &= \aup{a}{b} \fup{c}{d}, \tag{L20} \label{eq:fupfup}\\
e_1\fup{a}{b} e_1 &= \aup{a}{b} e_1, &&n \geq 2, \tag{L21} \label{eq:e1fupe1}\\
\fup{c}{a}\wup{b}{d} &= \aup{a}{b}\wup{c}{d}, \tag{L22}\label{eq:fupwup}\\
\wdo{a}{d}\fup{b}{c} &= \aup{a}{b} \wdo{c}{d}, \tag{L23}\label{eq:wdofup}\\
\wdo{a}{c} \wup{b}{d} &= \aup{a}{b} \fdo{c}{d}. \tag{L24} \label{eq:wdowup}
\end{align}

\noindent\textbf{Parameter $\ddo{a}{b}$:}
\begin{align}
\fdo{c}{a} \fdo{b}{d} &= \ddo{a}{b}\fdo{c}{d}, \tag{L25} \label{eq:fdofdo}\\
e_{n-1} \fdo{a}{b}e_{n-1} &= \ddo{a}{b} e_{n-1}, &&n \geq 2, \tag{L26} \label{eq:en1fdoen1}\\
\fdo{c}{a} \wdo{d}{b} &= \ddo{a}{b} \wdo{d}{c}, \tag{L27} \label{eq:fdowdo}\\
\wup{c}{a} \fdo{b}{d} &= \ddo{a}{b} \wup{c}{d}, \tag{L28} \label{eq:wupfdo}\\
\wup{c}{a} \wdo{d}{b} &= \ddo{a}{b} \fup{c}{d}.\tag{L29} \label{eq:wupwdo}
\end{align}

\noindent\textbf{Parameter $\glab{a}{b}$:}

\noindent\textit{For $n$ odd:}

\noindent Let
\begin{align*}
O &\coloneqq \prod_{k=1}^{\frac{n-1}{2}} e_{2k-1}, 
&E &\coloneqq \prod_{k=1}^{\frac{n-1}{2}} e_{2k}.
\end{align*}
Then we have
\begin{align}
\wup{c}{b} \,O \, \wup{a}{d}\, O &= \glab{a}{b}\,\wup{c}{d} \,O, \tag{L30}\label{eq:wupOwupO}\\[1mm] 
\fup{c}{a} \,E \,\wdo{d}{b} \,E &= \glab{a}{b} \,\fup{c}{d} \,E, \tag{L31}\label{eq:fupEwdoE}\\[1mm] 
\wdo{a}{d}\,E \,\wdo{c}{b} \,E &= \glab{a}{b} \,\wdo{c}{d}\, E ,\tag{L32}\label{eq:wdoEwdoE}\\[1mm] 
\fdo{c}{b} \,O \,\wup{a}{d} \,O &= \glab{a}{b} \,\fdo{c}{d} \,O. \tag{L33}\label{eq:fdoOwupO} 
\end{align}

\noindent\textit{For $n$ even:}

\noindent Let
\begin{align*}
    \Theta \coloneqq \prod_{k=1}^{n/2} e_{2k-1}.
\end{align*}
Then we have
\begin{align}
\Theta \wup{a}{b} \Theta = \glab{a}{b} \Theta. \tag{L34} \label{eq:ThetawupTheta}
\end{align}

\noindent\textit{For $n=1$ only:}
\begin{align}
\fup{c}{a} \fdo{b}{d} &= \glab{a}{b} \wup{c}{d}, \tag{L35} \label{eq:fupfdo1}\\
\fdo{d}{b} \fup{a}{c} &= \glab{a}{b} \wdo{c}{d}, \tag{L36} \label{eq:fdofup1}\\
\wup{c}{b} \fup{a}{d} &= \glab{a}{b} \fup{c}{d}, \tag{L37} \label{eq:wupfup1}\\
\wdo{a}{c} \fdo{b}{d} &= \glab{a}{b} \fdo{c}{d}. \tag{L38} \label{eq:wdofdo1}
\end{align}

\hrule

\vspace{5mm}

\noindent The following relations are implied by the defining relations above, and will be useful later. They are proven in Appendix \ref{app:redundant}.

\vspace{4mm}
\noindent\begin{minipage}[t][][t]{0.45\textwidth}
\noindent \textit{For all $n$:}
\begin{align}
\wup{a}{b}e_{n-1}\wdo{c}{d} &= \fup{a}{c}\fdo{b}{d}, && n \geq 2. \label{eq:wupen1wdo}
\end{align}

\vspace{3mm}
\noindent\textit{For $n$ odd:}
\begin{align}
\fup{a}{b}E &= E \fup{a}{b}, \\
\fdo{a}{b}O &= O\fdo{a}{b}, \\
\wup{a}{b}\, O &= E\, \wup{a}{b}, \label{eq:wupO} \\
\wdo{a}{b}\, E &= O\, \wdo{a}{b}, \label{eq:wdoE} \\[5mm]
e_1 \wup{a}{b}E &= \wdo{a}{b}E, &&n>1,\\
e_{n-1} \wdo{a}{b}O &= \wup{a}{b}O , &&n>1,\\
E\wdo{a}{b}e_1 &= \wup{a}{b} O, &&n>1,\\
O\wup{a}{b}e_{n-1} &= \wdo{a}{b}E, &&n>1,\\[5mm]
\fup{c}{a}\, E \, \fdo{b}{d} \, O &= \glab{a}{b}\, \wup{c}{d}\, O, \label{eq:fupEfdoO}\\[1mm]
\fdo{d}{b}\, O\, \fup{a}{c} \, E &= \glab{a}{b} \,\wdo{c}{d}\, E, \label{eq:fdoOfupE}\\[1mm]
\wup{c}{b}\, O\,  \fup{a}{d}\, E &= \glab{a}{b}\, \fup{c}{d}\, E, \label{eq:wupOfupE}\\[1mm]
\wdo{a}{c}\, E \, \fdo{b}{d}\, O &= \glab{a}{b}\, \fdo{c}{d}\, O. \label{eq:wdoEfdoO}
\end{align}
\end{minipage}\hspace{0.1\textwidth}\begin{minipage}[t][][t]{0.45\textwidth}
\noindent\textit{For $n$ even:}

\noindent Let
\vspace{-1mm}
\begin{align}
\Omega \coloneqq  \prod_{k=1}^{\frac{n}{2}-1} e_{2k}.
\end{align}
Then 
\vspace{-1mm}
\begin{align}
\Theta \wup{a}{b} &= \Theta \wdo{a}{b}, \label{eq:Thetawup} \\[1mm]
\wup{a}{b} \Theta &= \wdo{a}{b} \Theta, \label{eq:wupTheta} \\[1mm]
\wup{a}{b}\Theta \wup{c}{d} &= \fup{a}{c}\fdo{b}{d} \Omega, \label{eq:wupThetawup} \\[1mm]
e_1\wup{a}{b} \Omega &= \Theta \wup{a}{b}, \label{eq:e1wupOmega} \\[1mm]
e_{n-1}\wdo{a}{b} \Omega &= \Theta \wdo{a}{b}. \label{eq:en1wdoOmega}
\end{align}

\vspace{7mm}
\noindent\textit{For $n=1$:}
\vspace{-1mm}
\begin{align}
\fup{c}{a}\wdo{d}{b} &= \glab{a}{b}\,\fup{c}{d}, \label{eq:fupwdo1}\\[1mm]
\fdo{c}{b}\,\wup{a}{d} &= \glab{a}{b}\,\fdo{c}{d}, \label{eq:fdowup1}\\[1mm]
\wup{c}{b}\wup{a}{d} &= \glab{a}{b}\, \wup{c}{d}, \label{eq:wupwup1}\\[1mm]
\wdo{a}{d}\wdo{c}{b} &= \glab{a}{b}\, \wdo{c}{d}. \label{eq:wdowdo1}
\end{align}
\end{minipage}

\vspace{6mm}

We now introduce a map $\phi$ between the algebraic and diagrammatic label algebras. The goal for the remainder of the paper is to show that this is an isomorphism. Define a map $\phi: \A_n(X) \to \lab_n(X)$ by
\begin{align}
e_j &\mapsto \de{j}, &&1\leq j \leq n-1, \label{eq:phidefstart} \\
\fup{a}{b} &\mapsto \dfup{a}{b}, &&\mathrm{a},\mathrm{b} \in X, \\
\fdo{a}{b} &\mapsto \dfdo{a}{b}, &&\mathrm{a},\mathrm{b} \in X, \\
\wup{a}{b} &\mapsto \dwup{a}{b}, &&\mathrm{a},\mathrm{b} \in X, \\
\wdo{a}{b} &\mapsto \dwdo{a}{b}, &&\mathrm{a},\mathrm{b} \in X. \label{eq:phidefend}
\end{align}

\begin{proposition}\label{prop:welldef}
Equations \eqref{eq:phidefstart}--\eqref{eq:phidefend} define a homomorphism.
\end{proposition}
\begin{proof}
It suffices to show that the diagrammatic images of the generators of $\A_n(X)$ satisfy the defining relations of $\A_n(X)$. This is easy to confirm using the diagrammatic relations of $\lab_n(X)$.
\end{proof}

\begin{proposition}\label{prop:surj}
The homomorphism $\phi:\A_n(X) \to \lab_n(X)$ is surjective.
\end{proposition}
\begin{proof}
Recall from Theorem \ref{thm:diaggen} that $\lab_n(X)$ is generated by $\did$, $\de{i}$, $\dfup{a}{b}$, $\dfdo{a}{b}$, $\dwup{a}{b}$ and $\dwdo{a}{b}$, where $i = 1, \dots, n-1$ and $\mathrm{a},\mathrm{b} \in X$, and observe that these diagrams are the images of the generators of $\A_n(X)$.
\end{proof}

Note that, for any word $\mathbf{u}$ in $\A_n(X)$, $\phi(\mathbf{u})$ is a scalar multiple of an $\lab_n(X)$-diagram. Indeed, if $\mathbf{u}$ is the word $s_1s_2 \dots s_k$, then $\phi(\mathbf{u}) = \phi(s_1)\phi(s_2)\dots \phi(s_k)$ is a product of diagrammatic generators, and any product of $\lab_n(X)$-diagrams is a scalar multiple of an $\lab_n(X)$-diagram. We use this result throughout the rest of the paper.

\section{Parity and reduced words}\label{s:ParityReduced}
In this section, we introduce the notion of parity for words in $\A_n(X)$, and define two different generalisations of the idea of a reduced word, namely, \textit{label-reduced} and \textit{even-reduced}. Each of these definitions is constructed to be equivalent to a diagrammatic property as follows. Each word in $\A_n(X)$ maps to a scalar multiple of a $\lab_n(X)$-diagram of the same parity. A word is label-reduced if and only if concatenating the corresponding diagrammatic generators produces no boundary arcs. Finally, even-reduced words correspond to even diagrams, expressed as minimal-length products of the even diagrammatic generators. Some of these results will not be established until after we prove that $\phi$ is an isomorphism, but this intuition informs our approach to that proof throughout the remainder of the paper.

We will refer to the generators $\fup{a}{b}$, $\fdo{a}{b}$, $\wup{a}{b}$ and $\wdo{a}{b}$ as \textit{label generators}. A word in $\A_n(X)$ is called \textit{label-reduced} if it is not equal to a scalar multiple of any word containing strictly fewer label generators. Note that a label-reduced word need not be reduced in the usual sense, that it is not equal to a scalar multiple of a strictly shorter word; indeed, the word $e_1e_2e_1$ is label-reduced, but not reduced. 

At this point, we can establish one direction of the claimed correspondence between label-reduced words and diagrammatic products without boundary arcs.

\begin{lemma}\label{lem:labelreducedbdylinks}
Let $\mathbf{u}$ be a word in $\A_n(X)$, so $\phi(\mathbf{u})$ is a scalar multiple of some $\lab_n(X)$-diagram $\undertilde{D}$. Then the number of boundary links in $\undertilde{D}$ is at most twice the number of label generators in $\mathbf{u}$, with equality occurring only if $\mathbf{u}$ is label-reduced.
\end{lemma}
\begin{proof}
Observe that the diagrams $\phi(\fup{a}{b})$, $\phi(\fdo{a}{b})$, $\phi(\wup{a}{b})$ and $\phi(\wdo{a}{b})$ each have two boundary links, for any $\mathrm{a}, \mathrm{b} \in X$, while the diagrams $\phi(e_i)$ have no boundary links, for any $i=1, \dots, n-1$. Thus the number of boundary links appearing in the concatenation of the $\phi$-images of each generator in $\mathbf{u}$ is precisely twice the number of label generators in $\mathbf{u}$. Since $\undertilde{D}$ is obtained by removing boundary arcs and loops from this concatenation, it follows that the number of boundary links in $\undertilde{D}$ is at most twice the number of label generators in $\mathbf{u}$, and that equality occurs if and only if no boundary arcs are produced in the concatenation. Using the constructions in Propositions \ref{prop:evengen} and \ref{prop:oddgen}, we can write $\undertilde{D}$ as a product of diagrammatic generators where no boundary arcs are produced, so this minimum can be achieved.

We now want to show that equality occurs only if $\mathbf{u}$ is label-reduced. Suppose $\mathbf{u}$ is not label-reduced. Then there is a word $\mathbf{v}$ in $\A_n(X)$ such that $\mathbf{u}=a\, \mathbf{v}$ for some (nonzero) scalar $a$, and $\mathbf{v}$ contains strictly fewer label generators than $\mathbf{u}$. The concatenation of $\phi$-images of the generators of $\mathbf{v}$ thus contains strictly fewer boundary links than the concatenation for $\mathbf{u}$. Let $b$ be the scalar such that $\phi(\mathbf{u})=b\, \undertilde{D}$. Then we have $\phi(\mathbf{v})= \frac{b}{a}\,  \undertilde{D}$, so the number of boundary links in $\undertilde{D}$ is at most twice the number of label generators in $\mathbf{v}$, and thus strictly less than twice the number of label generators in $\mathbf{u}$.
\end{proof}

A word is called \textit{even} if it is equal to a scalar multiple of a label-reduced word containing no generators of the form $\wup{a}{b}$ or $\wdo{a}{b}$, and \textit{odd} otherwise. We will refer to $\wup{a}{b}$ and $\wdo{a}{b}$ as \textit{odd generators}, and to all other generators as \textit{even generators}. It is clear that the even generators are even words, but not that the odd generators are odd. In fact, we will show that $\A_n(X)$ and $\lab_n(X)$ are isomorphic without explicitly using that odd generators are odd, but for completeness, this is proven in Corollary \ref{cor:oddgensodd}.

We say a word is \textit{even-reduced} if it is label-reduced, it consists only of even generators, and it is not equal to a scalar multiple of any strictly shorter word in the even generators. It follows that all even-reduced words are even, and that every even word is equal to a scalar multiple of an even-reduced word. Similar to label-reduced words, even-reduced words need not be reduced in the usual sense. For example, in $\A_4(\{\mathrm{a},\mathrm{b},\mathrm{c},\mathrm{d}\})$, the word $\fup{a}{b}e_1e_2e_3\fdo{c}{d}$ is even-reduced, but equal to the shorter word $\wup{a}{c}\wup{b}{d}$, by \eqref{eq:wupwup}.

The following corollary establishes part of the relationship between even-reduced words in $\A_n(X)$ and even diagrams in $\lab_n(X)$.

\begin{corollary}\label{cor:evenredsurj}
Let $\undertilde{T}$ be an even diagram in $\lab_n(X)$. Then there exists an even-reduced word $T \in \A_n(X)$ such that $\phi(T)=\undertilde{T}$, and the number of boundary links in $\undertilde{T}$ is double the number of label generators in $T$.
\end{corollary}
\begin{proof}
Since $\undertilde{T}$ is an even diagram, from Proposition \ref{prop:evengen}, we can write $\undertilde{T}$ as a product of the even diagrammatic generators $\de{i}$, $i=1, \dots, n-1$; $\dfup{a}{b}$, $\mathrm{a},\mathrm{b} \in X$; and $\dfdo{a}{b}$, $\mathrm{a}, \mathrm{b} \in X$. Observe that these are the images under $\phi$ of the even generators $e_i$, $i=1, \dots, n-1$; $\fup{a}{b}$, $\mathrm{a},\mathrm{b} \in X$; and $\fdo{a}{b}$, $\mathrm{a}, \mathrm{b} \in X$; respectively. We can thus let $T'$ be the word in $\A_n(X)$ whose generators correspond to the diagrammatic generators in that product. Then $T'$ is a word in the even generators, and $\phi(T')=\undertilde{T}$. Moreover, recall that the construction from Proposition \ref{prop:evengen} does not produce any boundary arcs or loops. This means that the number of boundary links in $\undertilde{T}$ is precisely twice the number of label generators in $T'$. Hence, by Lemma \ref{lem:labelreducedbdylinks}, $T'$ is label-reduced. Since $T'$ is a label-reduced word in the even generators, it is even, so it is equal to a scalar multiple of an even-reduced word. Thus we have $T' = \mu\,  T$, for some scalar $\mu$ and even-reduced word $T$. Then $\undertilde{T}=\phi(T')=\mu\, \phi(T)$. Since the $\phi$-image of each word in the $\A_n(X)$ generators is a scalar multiple of a diagram, we have that $\phi(T)=\nu\,  \undertilde{S}$ for some scalar $\nu$ and $\lab_n(X)$-diagram $\undertilde{S}$. Hence $\undertilde{T}=\mu\nu\, \undertilde{S}$. Since $\undertilde{T}$ and $\undertilde{S}$ are both $\lab_n(X)$-diagrams, we must have $\mu = \nu = 1$, and thus $\phi(T) = \undertilde{T}$. We have that $T'$ is label-reduced, because it is even-reduced, and that $T$ is label-reduced and $T'=\mu\, T$; it follows that $T$ and $T'$ have the same number of label generators, and therefore the number of boundary links in $\undertilde{T}$ is indeed double the number of label generators in $T$.
\end{proof}

Our next goal is to show that each word in $\A_n(X)$ is equal to a scalar multiple of a label-reduced word containing at most one odd generator. Lemma \ref{lem:sWreduction} establishes some technical details, and then Proposition \ref{prop:wreduce} shows how products of even generators with an odd generator on each side can be reduced to scalar multiples of words containing at most one odd generator. Applying this repeatedly, we find the desired result in Corollary \ref{cor:watmost1}.

\begin{lemma}\label{lem:sWreduction}
Let $0\leq m \leq n-1$, and let $s \neq e_{m+1}$ be an even generator of $\A_n(X)$. If $m=n-1$, let $s \neq \fdo{c}{d}$ for any $\mathrm{c},\mathrm{d} \in X$.
\begin{enumerate}[label=(\roman*)]
\item The word $\mathbf{u}\coloneqq se_me_{m-1}\dots e_2e_1\wup{a}{b}$ is a scalar multiple of a word of the form $\mathbf{v}T$, where
\begin{itemize}
    \item $\mathbf{v}=\id$, or $\mathbf{v}=e_ie_{i-1}\dots e_2e_1 \wup{e}{f}$ for some $\mathrm{e},\mathrm{f}\in X$ and $0 \leq i \leq m$;
    \item $T$ is a word in the even generators; and
    \item the number of label generators in $\mathbf{v}T$ is less than or equal to the number of label generators in $\mathbf{u}$.
\end{itemize}
\item The word $\mathbf{u}'\coloneqq\wdo{a}{b}e_1e_2\dots e_{m-1}e_ms$ is a scalar multiple of a word of the form $T'\mathbf{v}'$, where
\begin{itemize}
    \item $\mathbf{v}'=\id$, or $\mathbf{v}'=\wdo{e}{f}e_1e_2\dots e_{i-1}e_i$ for some $\mathrm{e},\mathrm{f}\in X$ and $0 \leq i \leq m$;
    \item $T'$ is a word in the even generators; and
    \item the number of label generators in $T'\mathbf{v}'$ is less than or equal to the number of label generators in $\mathbf{u}'$.
\end{itemize}
\end{enumerate}
\end{lemma}
\begin{proof}
We will prove (i) in detail; (ii) is analogous. We first consider the case $m=0$. 
Applying the relations \eqref{eq:fupwup}, \eqref{eq:fdowup}, \eqref{eq:fdowup1} and \eqref{eq:ejwup} to $\mathbf{u}=s\wup{a}{b}$ for each possible $s\neq e_1$, we have
\begin{align}
\fup{c}{d}\wup{a}{b} &= \aup{d}{a}\wup{c}{b}, \\[2mm]
\fdo{c}{d}\wup{a}{b} &= \begin{cases}
    \wup{a}{c}e_{n-1}\fdo{d}{b}, &n\geq 2, \\[2mm]
    \glab{a}{d}\fdo{c}{b}, &n=1,
\end{cases}\\[2mm]
e_j\wup{a}{b} &= \wup{a}{b}e_{j-1}, \qquad 2 \leq j \leq n-1.
\end{align}
Thus the result holds for each possible $s\neq e_{m+1}$ in the case $m=0$.

Now consider $1 \leq m \leq n-1$. This implies $n\geq 2$. If $s=\fup{c}{d}$ for some $\mathrm{c},\mathrm{d} \in X$, then since $m\geq 1$, we have
\begin{align}
\fup{c}{d}e_me_{m-1}\dots e_2e_1 \wup{a}{b} &= e_me_{m-1}\dots e_2\fup{c}{d}e_1\wup{a}{b} &&\text{by \eqref{eq:fupej}} \\
&= e_me_{m-1}\dots e_2 \wup{c}{b}\fup{d}{a} &&\text{by \eqref{eq:wupfup}} \\
&= \wup{c}{b} e_{m-1}e_{m-2}e_1\fup{d}{a} &&\text{by \eqref{eq:ejwup}}.
\end{align}
If $s=\fdo{c}{d}$ and $m<n-1$, then $n\geq 3$ and we have
\begin{align}
\fdo{c}{d}e_me_{m-1} \dots e_1\wup{a}{b} 
&= e_me_{m-1}\dots e_1 \fdo{c}{d}\wup{a}{b} &&\text{by \eqref{eq:fdoej}} \\
&= e_me_{m-1}\dots e_1 \wup{a}{c}e_{n-1}\fdo{d}{b} &&\text{by \eqref{eq:fdowup}}.
\end{align}
The case $s=\fdo{c}{d}$ and $m=n-1$ is excluded in the statement of the Lemma.
If $s=e_j$, for $j\leq m-1$, then
\begin{align}
e_j e_m e_{m-1}\dots e_1\wup{a}{b}
&= e_m e_{m-1}\dots e_{j+2}e_je_{j+1}e_j e_{j-1} \dots e_1 \wup{a}{b} &&\text{by \eqref{eq:eiej}} \\
&= e_m e_{m-1} \dots e_{j+2}e_j e_{j-1} \dots e_1 \wup{a}{b} &&\text{by \eqref{eq:TLreduce}} \\
&= e_je_{j-1} \dots e_1 e_me_{m-1} \dots e_{j+2} \wup{a}{b} &&\text{by \eqref{eq:eiej}} \\
&= e_je_{j-1} \dots e_1 \wup{a}{b} e_{m-1}e_{m-2} \dots e_{j+1} &&\text{by \eqref{eq:ejwup}}.
\end{align}
If $s=e_m$, then
\begin{align}
e_me_m e_{m-1} \dots e_1 \wup{a}{b}
&= \beta e_me_{m-1} \dots e_1 \wup{a}{b} &&\text{by \eqref{eq:beta}}.
\end{align}
If $s=e_j$ for $j \geq m+2$, then
\begin{align}
e_je_me_{m-1}\dots e_1 \wup{a}{b}
&= e_me_{m-1}\dots e_1 e_j\wup{a}{b} &&\text{by \eqref{eq:eiej}} \\
&= e_me_{m-1} \dots e_1 \wup{a}{b} e_{j-1} &&\text{by \eqref{eq:ejwup}}.
\end{align}
In each of these cases we have shown that $\mathbf{u}$ is equal to a word of the desired form, so the result holds for $1\leq m \leq n-1$ as well.
\end{proof}

\begin{proposition} \label{prop:wreduce}
Let $\mathbf{u}=s_ks_{k-1} \dots s_1$, where $s_i$ is an even generator for each $i=1, \dots, k$. Then each of the words
\begin{align}
\wup{a}{b} \mathbf{u}\, \wup{c}{d},
&&\wup{a}{b}\mathbf{u}\, \wdo{c}{d},
&&\wdo{a}{b} \mathbf{u}\, \wup{c}{d},
&&\wdo{a}{b}\mathbf{u}\, \wdo{c}{d}
\end{align}
is equal to a scalar multiple of a word containing at most one odd generator, and no more label generators than the original word.
\end{proposition}
\begin{proof}
The main idea of this proof is to induct on $k$, reducing the number of generators in the word $\mathbf{u}$. If $k=0$, then by relations \eqref{eq:wupwup}, \eqref{eq:wupwdo}, \eqref{eq:wdowup} and \eqref{eq:wdowdo}, we have
\begin{align}
\wup{a}{b}\wup{c}{d} &= \fup{a}{c}e_1e_2 \dots e_{n-1} \fdo{b}{d}, \\
\wup{a}{b}\wdo{c}{d} &= \ddo{b}{d} \fup{a}{c}, \\
\wdo{a}{b}\wup{c}{d} &= \aup{a}{c} \fdo{b}{d}, \\
\wdo{a}{b}\wdo{c}{d} &= \fdo{b}{d}e_{n-1}e_{n-2}\dots e_1 \fup{a}{c}.
\end{align}

For the inductive step of the proof, we will treat the cases ending with $\wup{c}{d}$ in detail; the $\wdo{c}{d}$ cases are analogous. Fix $k> 0$. Assume that for each word $\mathbf{v}$ consisting of fewer than $k$ even generators, each word $\mathbf{v}' \in \left\lbrace \wup{e}{f}\, \mathbf{v}\, \wup{g}{h},\;
\wup{e}{f}\,\mathbf{v}\, \wdo{g}{h},\; 
\wdo{e}{f}\, \mathbf{v}\, \wup{g}{h},\; 
\wdo{e}{f}\,\mathbf{v}\, \wdo{g}{h}\right\rbrace$
is equal to a scalar multiple of a word containing at most one odd generator, and no more label generators than $\mathbf{v}'$.

Let $m$ be the largest integer such that $s_i=e_i$ for all $i \leq m$. Then $0 \leq m\leq k$, and also $m \leq n-1$. If $m<k$, then 
\begin{align}
s_{m+1}s_ms_{m-1}\dots s_1\wup{c}{d}
&=s_{m+1}e_me_{m-1}\dots e_1\wup{c}{d},
\end{align}
where $s_{m+1} \neq e_{m+1}$ is an even generator. First consider the case where $m\neq n-1$ or $s_{m+1} \neq \fdo{e}{f}$ for any $\mathrm{e},\mathrm{f} \in X$. Then by Lemma \ref{lem:sWreduction}(i), this is equal to a scalar multiple of a word of the form $\mathbf{w}T$, where $\mathbf{w}=\id$ or $\mathbf{w}=e_ie_{i-1}\dots e_1\wup{e}{f}$ for some $\mathrm{e},\mathrm{f} \in X$ and $0 \leq i \leq m$; $T$ is a word in the even generators; and the number of label generators in $\mathbf{w}T$ is less than or equal to the number of label generators in $s_{m+1}e_me_{m-1}\dots e_1\wup{c}{d}$. 

If $\mathbf{w}=\id$, then 
\begin{align}
\mathbf{u}\, \wup{c}{d}
&= s_ks_{k-1} \dots s_{m+2}s_{m+1}s_ms_{m-1} \dots s_1\wup{c}{d} \\
&= \mu\, s_ks_{k-1} \dots s_{m+2} T
\end{align}
for some scalar $\mu$. Hence both $\wup{a}{b}\,\mathbf{u}\,\wup{c}{d}$ and $\wdo{a}{b}\,\mathbf{u}\,\wup{c}{d}$ are equal to scalar multiples of words containing one odd generator ($\wup{a}{b}$ and $\wdo{a}{b}$, respectively), and no more label generators than $\wup{a}{b}\,\mathbf{u}\,\wup{c}{d}$ and $\wdo{a}{b}\,\mathbf{u}\,\wup{c}{d}$, respectively.

If $\mathbf{w}=e_ie_{i-1}\dots e_1 \wup{e}{f}$ for some $\mathrm{e},\mathrm{f}\in X$ and $0 \leq i \leq m$, then 
\begin{align}
\mathbf{u}\wup{c}{d}
&= s_ks_{k-1} \dots s_{m+2}s_{m+1}s_m s_{m-1} \dots s_1 \wup{c}{d} \\
&= \mu\, s_ks_{k-1} \dots s_{m+2} e_ie_{i-1} \dots e_1 \wup{e}{f} T
\end{align}
for some scalar $\mu$. The word $\mathbf{v}\coloneqq s_ks_{k-1} \dots s_{m+2}e_ie_{i-1} \dots e_1$ then consists of at most $k-1$ even generators, so by applying our inductive hypothesis, both
\begin{align}
\wup{a}{b}\, \mathbf{u}\, \wup{c}{d} &= \mu\, \wup{a}{b}\, \mathbf{v} \,\wup{e}{f} T
\end{align}
and
\begin{align}
\wdo{a}{b}\,\mathbf{u}\, \wup{c}{d} &= \mu\, \wdo{a}{b}\, \mathbf{v}\, \wup{e}{f} T
\end{align}
are equal to scalar multiples of words that contain at most one odd generator, and no more label generators than $\wup{a}{b}\,\mathbf{u}\, \wup{c}{d}$ and $\wdo{a}{b}\,\mathbf{u}\, \wup{c}{d}$, respectively. 

Now consider the case where $m=n-1$ and $s_{m+1}=\fdo{e}{f}$ for some $\mathrm{e},\mathrm{f}\in X$. Then for $n \geq 2$ we have
\begin{align}
\mathbf{u}\, \wup{c}{d}
&= s_ks_{k-1} \dots s_{m+2}s_{m+1}s_ms_{m-1} \dots e_1 \wup{c}{d} \\
&= s_ks_{k-1} \dots s_{n+1}\fdo{e}{f}e_{n-1}e_{n-2} \dots e_1\wup{c}{d} \\
&= s_ks_{k-1}\dots s_{n+1}\fdo{e}{f}e_{n-1}\wdo{c}{d} &&\text{by \eqref{eq:en1wdo}} \\
&= s_ks_{k-1}\dots s_{n+1}\wdo{c}{e}\fdo{f}{d} &&\text{by \eqref{eq:wdofdo}}
\end{align}
The word $\mathbf{v}\coloneqq s_k s_{k-1} \dots s_{n+1}$ then consists of $k-n$ even generators, so by applying our inductive hypothesis, both
\begin{align}
\wup{a}{b}\,\mathbf{u}\, \wup{c}{d} &= \wup{a}{b}\, \mathbf{v} \,\wdo{c}{e}\fdo{f}{d}
\end{align}
and
\begin{align}
\wdo{a}{b}\,\mathbf{u}\, \wup{c}{d} &= \wdo{a}{b}\, \mathbf{v}\, \wdo{c}{e}\fdo{f}{d}
\end{align}
are equal to scalar multiples of words that contain at most one odd generator, and no more label generators than $\wup{a}{b}\,\mathbf{u}\, \wup{c}{d}$ and $\wdo{a}{b}\,\mathbf{u}\, \wup{c}{d}$, respectively.

For $n=1$, with $m=n-1 = 0$ and $s_{m+1}=s_1 = \fdo{e}{f}$, we have
\begin{align}
\mathbf{u}\, \wup{c}{d} 
&= s_ks_{k-1} \dots s_{m+2}s_{m+1}s_{m}s_{m-1} \dots s_1 \wup{c}{d} \\
&= s_ks_{k-1} \dots s_2 \fdo{e}{f} \wup{c}{d} \\
&= \glab{c}{f}s_ks_{k-1}\dots s_2 \fdo{e}{d} &&\text{by \eqref{eq:fdowup1}}.
\end{align}
This contains no odd generators, and contains one less label generator than $\mathbf{u}\, \wup{c}{d}$. Hence each of the words $\wup{a}{b}\,  \mathbf{u}\, \wup{c}{d}$ and $\wdo{a}{b}\,\mathbf{u}\, \wup{c}{d}$ is equal to a scalar multiple of a word containing one odd generator, and one less label generator than $\wup{a}{b}\,  \mathbf{u}\, \wup{c}{d}$ and $\wdo{a}{b}\,\mathbf{u}\, \wup{c}{d}$ respectively. 

This completes the induction step for $m<k$.

If $m=k$, then $s_i=e_i$ for $i=1, \dots , k$, and thus $k\leq n-1$. Since $k>0$, this implies $n\geq 2$. If $k<n-1$, we have
\begin{align}
\wup{a}{b}\, \mathbf{u}\, \wup{c}{d}
&= \wup{a}{b} e_ke_{k-1} \dots e_1 \wup{c}{d} \\
&= e_{k+1}e_k \dots e_2 \wup{a}{b}\wup{c}{d} &&\text{by \eqref{eq:ejwup}},
\end{align}
which reduces to a $k=0$ case, and
\begin{align}
\wdo{a}{b}\, \mathbf{u} \, \wup{c}{d}
&= \wdo{a}{b} e_k e_{k-1} \dots e_2 e_1 \wup{c}{d} \\
&= e_{k-1}e_{k-2} \dots e_1 \wdo{a}{b} e_1 \wup{c}{d} &&\text{by \eqref{eq:ejwdo}} \\
&= e_{k-1}e_{k-2}\dots e_1 \fup{a}{c}\fdo{b}{d} &&\text{by \eqref{eq:wdoe1wup},}
\end{align}
which contains only even generators, and has the same number of label generators as $\wdo{a}{b}\, \mathbf{u}\, \wup{c}{d}$. If $k=n-1$, the latter still holds, while the former becomes
\begin{align}
\wup{a}{b} e_{n-1}e_{n-2} \dots e_1 \wup{c}{d}
&= \wdo{a}{b}e_1 \wup{c}{d} &&\text{by \eqref{eq:wdoe1}} \\
&= \fup{a}{c}\fdo{d}{b},
\end{align}
and this also has no extra label generators.
\end{proof}

\begin{corollary}\label{cor:watmost1}
Any word $\mathbf{u}$ in $\A_n(X)$ is equal to a scalar multiple of a label-reduced word containing at most one odd generator.
\end{corollary}
\begin{proof}
It is clear that any word $\mathbf{u}$ in $\A_n(X)$ is equal to a scalar multiple of a label-reduced word $\mathbf{u}'$. Proposition \ref{prop:wreduce} can be applied repeatedly to pairs of consecutive odd generators in $\mathbf{u}'$ to yield a scalar multiple of a word $\mathbf{u}''$ with at most one odd generator, without increasing the number of label generators. This means that $\mathbf{u}''$ is also label-reduced, implying the result.
\end{proof}

In the next section, we extend this result to construct a useful spanning set of $\A_n(X)$.

\section{Spanning set}\label{s:WTform}
The goal of this section is to establish a convenient spanning set of words in the $\A_n(X)$ generators. Each even word in the set will correspond to an even $\lab_n(X)$-diagram expressed as a minimal-length product of even diagrammatic generators, while each odd word will be a product similar to the diagrammatic product $\dw{a}{b}(j)\undertilde{T}$ from Proposition \ref{prop:oddgen}. Lemma \ref{lem:wuponly} builds upon the results of the previous section to write each word as a scalar multiple of a word containing at most one odd generator, where the odd generator is now $\wup{a}{b}$ for some $\mathrm{a},\mathrm{b} \in X$, if possible. Proposition \ref{prop:WTintro} and Corollary \ref{cor:WTspanning} then establish the desired spanning set. 

\begin{lemma}\label{lem:wuponly}
Any word in $\A_n(X)$ is equal to a scalar multiple of $\wdo{a}{b}$, for some $\mathrm{a},\mathrm{b}\in X$, or equal to a scalar multiple of a label-reduced word containing at most one generator of the form $\wup{a}{b}$ and no generators of the form $\wdo{a}{b}$.
\end{lemma}
\begin{proof}
By Corollary \ref{cor:watmost1}, any word $\mathbf{u}$ in $\A_n(X)$ is equal to a scalar multiple of a label-reduced word $\mathbf{u}'$ containing at most one odd generator. If $\mathbf{u}'$ does not contain any generators of the form $\wdo{a}{b}$, or it contains only $\wdo{a}{b}$, then we are done.

Otherwise, $\mathbf{u}'$ contains precisely one generator of the form $\wdo{a}{b}$, and at least one other generator, which must be even. We now show that each product of $\wdo{a}{b}$ with an even generator is equal to (i) a word with the same number of label generators, where at most one is of the form $\wup{e}{f}$, and none are of the form $\wdo{e}{f}$, or (ii) a word with fewer label generators. In the former case, this means that making these substitutions into $\mathbf{u}'$ yields a label-reduced word containing at most one generator of the form $\wup{e}{f}$ and no generators of the form $\wdo{e}{f}$, as required, while the latter contradicts our assumption that $\mathbf{u}'$ is label-reduced, so cannot occur.

For $n\geq 2$, for each $j=1, \dots, n-1$, we have
\begin{align}
\wdo{a}{b}e_j &= \wdo{a}{b}e_je_{j-1} \dots e_2e_1e_2\dots e_{j-1}e_j &&\text{by \eqref{eq:TLreduce}} \\
&= e_{j-1}e_{j-2} \dots e_1 \wdo{a}{b} e_1e_2\dots e_{j-1}e_j &&\text{by \eqref{eq:ejwdo}} \\
&= e_{j-1}e_{j-2}\dots e_1 \wup{a}{b} e_{n-1}e_{n-2}\dots e_2e_1e_2\dots e_{j-1}e_j &&\text{by \eqref{eq:wdoe1}} \\
&= e_{j-1}e_{j-2}\dots e_1\wup{a}{b} e_{n-1}e_{n-2}\dots e_j &&\text{by \eqref{eq:TLreduce}.}
\end{align}
Note that this does hold for $j=1$, with $e_{j-1}e_{j-2}\dots e_1 = e_2\dots e_{j-1}e_j= \id$. 
For $n\geq 2$,
\begin{align}
\wdo{a}{b}\fdo{c}{d} &= \fdo{b}{c}e_{n-1}\wdo{a}{d} &&\text{by \eqref{eq:wdofdo}} \\
&= \fdo{b}{c}e_{n-1}e_{n-2}\dots e_2e_1\wup{a}{d} &&\text{by \eqref{eq:en1wdo},}
\end{align}
and for $n=1$,
\begin{align}
\wdo{a}{b}\fdo{c}{d} &= \glab{a}{c}\fdo{b}{d} &&\text{by \eqref{eq:wdofdo1}}.
\end{align}
We also have
\begin{align}
\wdo{a}{b}\fup{c}{d} = \aup{a}{c}\wdo{d}{b} &&\text{by \eqref{eq:wdofup}},
\end{align}
so if $\mathbf{u}'$ contains $\wdo{a}{b}\fup{c}{d}$, or $\wdo{a}{b}\fdo{c}{d}$ with $n=1$, then it is not label-reduced.

Analogously,
\begin{align}
e_j\wdo{a}{b} &= e_je_{j-1}\dots e_1 \wup{a}{b} e_{n-1}e_{n-2}\dots e_j e_{j+1}, \\[2mm]
\fup{c}{d}\wdo{a}{b} &= \begin{cases}
    \glab{d}{b}\fup{c}{a}, &n=1,\\[2mm]
    \wup{c}{b}e_{n-1}e_{n-2}\dots e_2e_1\fup{d}{a}, &n\geq 2,
\end{cases} \\[2mm]
\fdo{c}{d}\wdo{a}{b}&= \ddo{d}{b}\wdo{a}{c}.
\end{align}
\end{proof}

We define a \textit{subword} of a word $s_1s_2\dots s_k$ to be a word $s_is_{i+1} \dots s_{k-j-1}s_{k-j}$, where $0 \leq i,j \leq k$. That is, a subword can be obtained from a word by deleting its first $i$ generators, and its last $j$ generators. For example, if $\mathbf{u} \coloneqq s_1s_2s_3s_4s_5$, then $\id$, $s_2s_3$, $s_1$, $s_3s_4s_5$ and $\mathbf{u}$ itself are all subwords of $\mathbf{u}$, but $s_2s_4$ is not.
Note that this is stricter than the usual definition used for words in the symmetric group, for example.

\begin{proposition}\label{prop:WTintro}
Each word $\mathbf{u}$ in $\A_n(X)$ is equal to a scalar multiple of a label-reduced word $WT$, where $W$ is one of 
\begin{align}
\id, \\
\w{a}{b}(0) &\coloneqq  \wup{a}{b}, \\
\w{a}{b}(i) &\coloneqq  e_ie_{i-1} \dots e_2e_1 \wup{a}{b}, &&1 \leq i \leq  n-1, \\
\w{a}{b}(n) &\coloneqq  \wdo{a}{b},
\end{align}
for some $\mathrm{a},\mathrm{b}\in X$, and $T$ is an even-reduced word.
\end{proposition}
\begin{proof}
By Lemma \ref{lem:wuponly}, any word $\mathbf{u}$ in $\A_n(X)$ is equal to a scalar multiple of $\wdo{a}{b}$ for some $\mathrm{a},\mathrm{b} \in X$, or equal to a scalar multiple of a label-reduced word $\mathbf{u}'$ containing at most one generator of the form $\wup{a}{b}$ and no generators of the form $\wdo{a}{b}$. In the former case, we can take $W=\w{a}{b}(n)=\wdo{a}{b}$ and $T=\id$; indeed, $T$ is even-reduced, and by Lemma \ref{lem:labelreducedbdylinks}, the whole word $\wdo{a}{b}$ is label-reduced, because $\phi\left(\wdo{a}{b}\right) = \dwdo{a}{b}$ has two boundary links.

In the latter case, if $\mathbf{u}'$ does not contain a generator of the form $\wup{a}{b}$, then it is even, so it is equal to a scalar multiple of some even-reduced word $T$. We can then take $W=\id$, and so $WT=T$ is label-reduced as required, since even-reduced words are label-reduced.

Otherwise, $\mathbf{u}'$ contains one generator of the form $\wup{a}{b}$ and no generators of the form $\wdo{a}{b}$. It thus has the form $s_ks_{k-1}\dots s_2s_1\wup{a}{b}T'$, for some $\mathrm{a},\mathrm{b} \in X$, where each $s_i$ is an even generator and $T'$ is a word in the even generators. Note that $T'$ is label-reduced, because $\mathbf{u}'$ is label-reduced, and thus $T'$ is even.

We now induct on $k$, reducing the number of generators on the left of $\wup{a}{b}$.
If $k=0$, then $\mathbf{u}'=\wup{a}{b}T'$. Since $T'$ is even, it is equal to a scalar multiple of an even-reduced word $T$. As both $T$ and $T'$ are label-reduced, it follows that they contain the same number of label generators, so $\wup{a}{b} T$ is label-reduced. Taking $W=\w{a}{b}(0)=\wup{a}{b}$, we have that $\mathbf{u}$ is equal to a scalar multiple of $WT$, where $WT$ is label-reduced and $T$ is even-reduced.

For the inductive step of the proof, fix $k>0$. Assume that for any $i<k$, if $S$ and $S'$ are words in the even generators, where $S$ consists of $i$ generators and $\mathbf{v}\coloneqq S\wup{a}{b}S'$ is label-reduced, then $\mathbf{v}$ is equal to a scalar multiple of a label-reduced word of the form $W'S''$, where $W'$ is $\id$ or $\w{c}{d}(j)$ for some $j=0, \dots, n$ and $\mathrm{c}, \mathrm{d} \in X$, and $S''$ is an even-reduced word.

Let $m$ be the largest integer such that $s_i=e_i$ for all $i \leq m$. Then $0 \leq m \leq k$, and also $m \leq n-1$. If $m=k$, then
\begin{align}
\mathbf{u}' &= e_ke_{k-1}\dots e_2e_1\wup{a}{b} T'\\
&= \w{a}{b}(k)T'.
\end{align}
Since $T'$ is even, it is equal to a scalar multiple of some even-reduced word $T$. Recalling that $\mathbf{u}$ is a scalar multiple of $\mathbf{u}'$, if we take $W=\w{a}{b}(k)$, then $\mathbf{u}$ is a scalar multiple of $WT$. Since $T'$ is label-reduced and equal to a scalar multiple of $T$, which is even-reduced and therefore label-reduced, $T$ and $T'$ must have the same number of label generators. It follows that $\w{a}{b}(k)T$ is label-reduced, because $\w{a}{b}(k)T'$ was label-reduced. Hence the result holds for $m=k$.

If $m<k$, then 
\begin{align}
s_{m+1} s_m s_{m-1} \dots s_2s_1 \wup{a}{b} T' 
&=  s_{m+1} e_m e_{m-1} \dots e_2e_1 \wup{a}{b} T' ,
\end{align}
where $s_{m+1}\neq e_{m+1}$ is an even generator. First consider the case where $m\neq n-1$, or $s_{m+1} \neq \fdo{c}{d}$ for any $\mathrm{c},\mathrm{d}\in X$. Then by Lemma \ref{lem:sWreduction}(i), this is equal to a scalar multiple of a word of the form $\mathbf{w}T''$, where $\mathbf{w}=\id$ or $\mathbf{w}=e_ie_{i-1}\dots e_2e_1\wup{e}{f}$ for some $\mathrm{e},\mathrm{f} \in X$ and $0 \leq i \leq m$; $T''$ is a word in the even generators; and $\mathbf{w}T''$ contains no more label generators than $s_{m+1} s_m s_{m-1} \dots s_2s_1 \wup{a}{b}$. This means that
\begin{align}
\mathbf{u}' &= s_ks_{k-1} \dots s_{m+2}s_{m+1} s_m s_{m-1} \dots s_2s_1 \wup{a}{b} T'  \\
&= \mu\, s_ks_{k-1} \dots s_{m+2}\mathbf{w}T'' T',
\end{align}
for some scalar $\mu$, and the word $s_ks_{k-1} \dots s_{m+2}\mathbf{w}T'' T'$ is label-reduced.

If $\mathbf{w}=\id$, then 
\begin{align}
\mathbf{u}'
&= \mu\, s_ks_{k-1} \dots s_{m+2}T'' T',
\end{align}
is a label-reduced word in the even generators, so is even. This means it is equal to a scalar multiple of an even-reduced word $T$. Since $\mathbf{u}$ is equal to a scalar multiple of $\mathbf{u}'$, it follows that $\mathbf{u}$ is equal to a scalar multiple of $T$. Taking $W=\id$, we have that $\mathbf{u}$ is a scalar multiple of $WT = T$, where $T$ is even-reduced, and so $WT$ is label-reduced, as required.

If $\mathbf{w}=e_ie_{i-1}\dots e_2e_1\wup{e}{f}$ for some $\mathrm{e},\mathrm{f} \in X$ and $0 \leq i \leq m$, then
\begin{align}
\mathbf{u}'
&= \mu\, s_ks_{k-1} \dots s_{m+2}e_ie_{i-1}\dots e_2e_1\wup{e}{f}T'' T'.
\end{align}
Let $S=s_ks_{k-1} \dots s_{m+2}e_ie_{i-1}\dots e_2e_1$. Then $S$ is a word consisting of at most $k-1$ even generators, and $T''T'$ is also a word in the even generators. Recall that $\mathbf{w}T''$ contains no more label generators than $s_{m+1}s_m \dots s_2s_1 \wup{a}{b}$. This means that $S\wup{e}{f}T'' T'$ contains no more label generators than $\mathbf{u}'$. Since $\mathbf{u}'$ is equal to a scalar multiple of $S\wup{e}{f}T'' T'$, and $\mathbf{u}'$ is label-reduced, it follows that $S\wup{e}{f}T'' T'$ is label-reduced. Then by our inductive hypothesis, this is equal to a scalar multiple of a label-reduced word of the form $W'S''$, where $W'=\id$ or $W'=\w{c}{d}(j)$ for some $0\leq j \leq n$ and $\mathrm{c},\mathrm{d}\in X$, and $S''$ is even-reduced. Hence $\mathbf{u}$ is equal to a scalar multiple of $W'S''$, which has the desired form.

It remains to consider the case where $m=n-1$, and $s_{m+1}=\fdo{c}{d}$ for some $\mathrm{c},\mathrm{d}\in X$. Then $n \neq 1$, because for $n=1$, $\mathbf{u}'$ contains the subword $\fdo{c}{d}\wup{a}{b}$, which is not label-reduced when $n=1$. For $n\geq 2$, then, we have
\begin{align}
\mathbf{u}' &= s_ks_{k-1} \dots s_{m+2}s_{m+1}s_m \dots s_2s_1 \wup{a}{b} T' \\
&= s_ks_{k-1} \dots s_{m+2}\fdo{c}{d}e_{n-1} \dots e_2e_1 \wup{a}{b} T' \label{eq:uprime} \\
&= s_ks_{k-1} \dots s_{m+2}\fdo{c}{d}e_{n-1}\wdo{a}{b} T' &&\text{by \eqref{eq:en1wdo}} \\
&= s_ks_{k-1} \dots s_{m+2}\wdo{a}{c}\fdo{d}{b} T' &&\text{by \eqref{eq:wdofdo}}. \label{eq:uprimered}
\end{align}
We have that $\mathbf{u}'$ (given in \eqref{eq:uprime}) is label-reduced, and equal to \eqref{eq:uprimered}. Since \eqref{eq:uprime} and \eqref{eq:uprimered} have the same number of label generators, it follows that \eqref{eq:uprimered} is label-reduced.

Now, if $k=m+1$, then the expression \eqref{eq:uprimered} is precisely $\w{a}{c}(n)\fdo{d}{b}T'$.
As this word is label-reduced, its subword $\fdo{d}{b}T'$ is also label-reduced, and thus even, as it contains only even generators. This means $\fdo{d}{b}T'$ is equal to a scalar multiple of some even-reduced word $T$, and setting $W=\w{a}{c}(n)$ makes $\mathbf{u}$ a scalar multiple of a word $WT$ of the desired form. 

If $k>m+1$, consider the word $\mathbf{v}\coloneqq s_{m+2}\fdo{c}{d}e_{n-1} \dots e_2e_1 \wup{a}{b}$. This is a subword of $\mathbf{u}'$, so must be label-reduced. If $s_{m+2}=\fdo{e}{f}$ for some $\mathrm{e},\mathrm{f} \in X$, then
\begin{align}
\mathbf{v} &= \fdo{e}{f} \fdo{c}{d} e_{n-1} \dots e_2e_1 \wup{a}{b} \\
&= \ddo{f}{c} \fdo{e}{d} e_{n-1} \dots e_2e_1\wup{a}{b} &&\text{by \eqref{eq:fdofdo}},
\end{align}
contradicting the fact that $\mathbf{v}$ is label-reduced. If $s_{m+2}=e_{n-1}$, then 
\begin{align}
\mathbf{v}&= e_{n-1}\fdo{c}{d}e_{n-1} \dots e_2e_1 \wup{a}{b} \\
&= \ddo{c}{d} e_{n-1}\dots e_2e_1 \wup{a}{b} &&\text{by \eqref{eq:en1fdoen1}},
\end{align}
again contradicting the fact that $\mathbf{v}$ is label-reduced. If $s_{m+2}=e_j$, for $1 \leq j \leq n-3$, then
\begin{align}
\mathbf{v} &= e_j \fdo{c}{d} e_{n-1} \dots e_{j+2}e_{j+1}e_je_{j-1} \dots e_2e_1 \wup{a}{b} \\
&=  \fdo{c}{d} e_j e_{n-1} \dots e_{j+2}e_{j+1}e_je_{j-1} \dots e_2e_1 \wup{a}{b} &&\text{by \eqref{eq:fdoej}} \\
&= \fdo{c}{d}  e_{n-1} \dots e_{j+2}e_je_{j+1}e_je_{j-1} \dots e_2e_1 \wup{a}{b} &&\text{by \eqref{eq:eiej}} \\
&= \fdo{c}{d}  e_{n-1} \dots e_{j+2}e_je_{j-1} \dots e_2e_1 \wup{a}{b} &&\text{by \eqref{eq:TLreduce}}.
\end{align}
This has the same number of label generators as $\mathbf{v}$. Then
\begin{align}
\mathbf{u}' &= s_ks_{k-1} \dots s_{m+3}\fdo{c}{d}  e_{n-1} \dots e_{j+2}e_je_{j-1} \dots e_2e_1 \wup{a}{b}T',
\end{align}
where the right hand side is label-reduced, with $k-1$ even generators before $\wup{a}{b}$, and only even generators after $\wup{a}{b}$. We can thus apply our inductive hypothesis to the right hand side to find that $\mathbf{u}'$ and thus $\mathbf{u}$ is a scalar multiple of a word of the required form. If $s_{m+2}=e_{n-2}$, then 
\begin{align}
\mathbf{v} &= e_{n-2} \fdo{c}{d} e_{n-1} e_{n-2}e_{n-3} \dots e_2e_1 \wup{a}{b} \\
&= \fdo{c}{d}e_{n-2} e_{n-1} e_{n-2}e_{n-3} \dots e_2e_1 \wup{a}{b} &&\text{by \eqref{eq:fdoej}}\\
&= \fdo{c}{d} e_{n-2}e_{n-3} \dots e_2e_1 \wup{a}{b} &&\text{by \eqref{eq:TLreduce}}.
\end{align}
This word has the same number of label generators as $\mathbf{v}$, and has $n-1=m<m+2$ even generators before $\wup{a}{b}$. Similar to the $s_{m+2}=e_j$, $1\leq j\leq n-3$ case, we can substitute this back into $\mathbf{u}'$ and use our inductive hypothesis to find that $\mathbf{u}$ is equal to a scalar multiple of a word of the desired form. Finally, if $s_{m+2} = \fup{e}{f}$ for some $\mathrm{e}, \mathrm{f} \in X$, then
\begin{align}
\mathbf{v} &= \fup{e}{f}\fdo{c}{d}e_{n-1}\dots e_2e_1 \wup{a}{b} \\
&= \fdo{c}{d}e_{n-1}\dots e_2\fup{e}{f}e_1\wup{a}{b} &&\text{by \eqref{eq:fupfdo}, \eqref{eq:fupej}} \\
&= \fdo{c}{d}e_{n-1}\dots e_2\wup{e}{b}\fup{f}{a} &&\text{by \eqref{eq:wupfup}}.
\end{align}
The final expression has the same number of label generators as $\mathbf{v}$, and has $n-1$ even generators before $\wup{a}{b}$. As in the previous case, this means we can write $\mathbf{u}$ as a scalar multiple of a word of the desired form. This completes the induction with respect to $k$, thus completing the proof.
\end{proof}

We say a word is in \textit{$WT$ form} if it is a label-reduced product $WT$ where $T$ is even-reduced, and either $W=\id$ or $W=\w{a}{b}(j)$ for some $\mathrm{a},\mathrm{b} \in X$, with $j$ minimal. By minimality of $j$, we mean that, if $\w{a}{b}(j)T$ is in $WT$ form, and $\w{a}{b}(j)T=\mu \, \w{c}{d}(i)T'$ for some scalar $\mu$ and $\mathrm{c},\mathrm{d} \in X$, where $\w{c}{d}(i)T'$ is label-reduced and $T'$ is even-reduced, then $i\geq j$.

\begin{corollary}\label{cor:WTspanning}
Each word in $\A_n(X)$ is equal to a scalar multiple of a word in $WT$ form, and thus the set of all words in $WT$ form is a spanning set for $\A_n(X)$.
\end{corollary}
\begin{proof}
This follows from Proposition \ref{prop:WTintro}.
\end{proof}

Note, however, that the set of all words in $WT$ form is not linearly independent for $n\geq 2$; for example the words $\fup{a}{b}\fdo{c}{d}$ and $\fdo{c}{d}\fup{a}{b}$ are both in $WT$ form, and are equal by \eqref{eq:fupfdo}. 

Our goal will eventually be to show that each word in $WT$ form maps to a basis diagram in $\lab_n(X)$ under the homomorphism $\phi$, and then to show that if two words in $WT$ form map to the same $\lab_n(X)$-diagram, then they are equal in $\A_n(X)$. To do this, we will deduce some results by comparing our presentation of $\A_n(X)$ to the algebraic presentation of the symplectic blob algebra, which has been proven isomorphic to its diagrammatic presentation in \cite{SympBlobPres}. 

\section{Symplectic blob algebra}\label{s:sympblob}
In this section, we establish a relationship between the even generators of the algebraic label algebra $\A_n(X)$, and the generators of the symplectic blob algebra $\symp_n$. In \cite{SympBlobPres}, Green, Martin and Parker give an algebraic presentation of $\symp_n$, and show that this is isomorphic to the diagrammatic presentation from \cite{TowersI}. Loosely, the even generators of $\A_n(X)$ satisfy the the symplectic blob relations, up to labels and parameters, so certain results about the even generators of $\A_n(X)$ can be inferred from results for $\symp_n$; Proposition \ref{prop:sblabel} makes this precise. Ultimately, we use this to show in Proposition \ref{prop:evenwordinj} that if $\phi(\mathbf{u})=\phi(\mathbf{v})$ for even words $\mathbf{u}$ and $\mathbf{v}$, then $\mathbf{u}=\mathbf{v}$.

The \textit{symplectic blob algebra} $\symp_n(\beta;\alpha_1,\alpha_2,\delta_1,\delta_2,\kappa)$ is the associative unital algebra with identity $\id$, generated by $e_i$, $i=1, \dots , n-1$, $f_0$ and $f_n$, subject to the relations
\begin{align}
e_ie_j&=e_je_i, &&\abs{i-1}\geq 2, \tag{S1} \label{eq:sbeiej}\\
e_ie_je_i &= e_i, &&\abs{i-j}=1, \tag{S2} \label{eq:sbreduce}\\
e_j^2 &= \beta e_i, &&1\leq j \leq n-1,\tag{S3}\label{eq:sbbeta} \\
f_0e_j&=e_jf_0, &&2 \leq j \leq n-1, \tag{S4} \label{eq:sbf0ej} \\
e_1f_0e_1 &= \alpha_1e_1, &&n\geq 2, \tag{S5} \label{eq:sbe1f0e1}\\
f_0^2 &= \alpha_2 f_0, \tag{S6} \label{eq:sbf0f0}\\
f_ne_j &= e_jf_n, &&1\leq j\leq n-2, \tag{S7} \label{eq:sbfnej}\\
e_{n-1}f_ne_{n-1} &= \delta_1 f_n, &&n\geq 2,\tag{S8} \label{eq:sben1fnen1}\\
f_n^2 &= \delta_2 f_n, \tag{S9} \label{eq:sbfnfn}\\
f_0f_n &= f_nf_0, &&n\geq 2, \tag{S10} \label{eq:sbf0fn}\\
IJI &= \kappa I, \tag{S11} \label{eq:sbIJI}\\
JIJ &= \kappa J, \tag{S12} \label{eq:sbJIJ}
\end{align}
where
\begin{align}
I &= \begin{dcases}
    f_n \, \prod_{k=1}^{\frac{n-1}{2}} e_{2k-1}, &n \text{ odd,} \\[2mm]
    \, \prod_{k=1}^{\frac{n}{2}}e_{2k-1}, &n \text{ even},
\end{dcases} \\[4mm]
J &= \begin{dcases}
    f_0 \, \prod_{k=1}^{\frac{n-1}{2}} e_{2k}, &n \text{ odd}, \\[2mm]
    f_0f_n\, \prod_{k=1}^{\frac{n}{2}-1}e_{2k}, &n\text{ even}.
\end{dcases}
\end{align}

As the name suggests, the symplectic blob algebra can also be defined in terms of a diagrammatic presentation involving blobs. Basis diagrams of $\symp_n$, called $\symp_n$\textit{-diagrams}, contain strings that may be decorated with two kinds of blobs: top boundary blobs, drawn as filled circles, and bottom boundary blobs, drawn as unfilled circles. The decorated strings must be able to be deformed such that every blob touches its corresponding boundary without any strings intersecting, including self-intersection. For example,
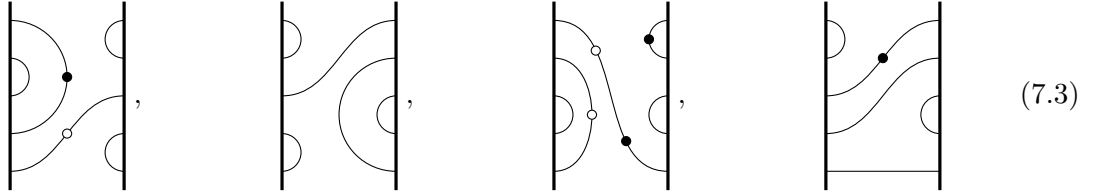
\begin{align}
\begin{tikzpicture}[baseline={([yshift=-1mm]current bounding box.center)},scale=0.5]
{
\draw [very thick] (0,-4.5)--(0,0.5);
\draw [very thick] (3,-4.5)--(3,0.5);
\draw (0,0) arc (90:-90:1.5);
\draw (0,-1) arc (90:-90:0.5);
\draw (0,-4) to[out=0,in=180] (3,-2);
\draw (3,0) arc(90:270:0.5);
\draw (3,-3) arc(90:270:0.5);
\filldraw (1.5,-1.5) circle (0.12);
\filldraw[fill=white] (1.5,-3) circle (0.12);
}
\end{tikzpicture}\; ,
&&
\begin{tikzpicture}[baseline={([yshift=-1mm]current bounding box.center)},scale=0.5]
{
\draw [very thick] (0,-4.5)--(0,0.5);
\draw [very thick] (3,-4.5)--(3,0.5);
\draw (0,0) arc (90:-90:0.5);
\draw (0,-2) to[out=0,in=180] (3,0);
\draw (0,-3) arc(90:-90:0.5);
\draw (3,-1) arc(90:270:1.5);
\draw (3,-2) arc(90:270:0.5);
}
\end{tikzpicture}\; ,
&&
\begin{tikzpicture}[baseline={([yshift=-1mm]current bounding box.center)},scale=0.5]
{
\draw [very thick] (0,-4.5)--(0,0.5);
\draw [very thick] (3,-4.5)--(3,0.5);
\draw (0,0) to[out=0,in=180] (3,-4);
\draw (0,-1) ..controls (0.7,-1) and (1,-1.8).. (1,-2.5) ..controls (1,-3.2) and (0.7,-4).. (0,-4);
\draw (0,-2) arc(90:-90:0.5);
\draw (3,0) arc(90:270:0.5);
\draw (3,-2) arc(90:270:0.5);
\filldraw (2.5,-0.5) circle (0.12);
\filldraw (1.9,-3.2) circle (0.12);
\filldraw[fill=white] (1,-2.5) circle (0.12);
\filldraw[fill=white] (1.1,-0.8) circle (0.12);
}
\end{tikzpicture}\; ,
&&
\begin{tikzpicture}[baseline={([yshift=-1mm]current bounding box.center)},scale=0.5]
{
\draw [very thick] (0,-4.5)--(0,0.5);
\draw [very thick] (3,-4.5)--(3,0.5);
\draw (0,0) arc (90:-90:0.5);
\draw (0,-2) to[out=0,in=180] (3,0);
\draw (0,-3) to[out=0,in=180] (3,-1);
\draw (3,-2) arc(90:270:0.5);
\draw (0,-4)--(3,-4);
\filldraw (1.5,-1) circle (0.12);
}
\end{tikzpicture}
\end{align}
are all $\symp_5$-diagrams, as their strings can be deformed as in
\begin{align}
\begin{tikzpicture}[baseline={([yshift=-1mm]current bounding box.center)},scale=0.5]
{
\draw [very thick] (0,-4.5)--(0,0.5);
\draw [very thick] (3,-4.5)--(3,0.5);
\draw (0,0) to[out=0,in=-90] (1,0.5) to[out=-90,in=0] (0,-3);
\draw (0,-1) arc (90:-90:0.5);
\draw (0,-4) to[out=0,in=90] (1.5,-4.5) to[out=90,in=180] (3,-2);
\draw (3,0) arc(90:270:0.5);
\draw (3,-3) arc(90:270:0.5);
\filldraw (1,0.5) circle (0.12);
\filldraw[fill=white] (1.5,-4.5) circle (0.12);
}
\end{tikzpicture}\; ,
&&
\begin{tikzpicture}[baseline={([yshift=-1mm]current bounding box.center)},scale=0.5]
{
\draw [very thick] (0,-4.5)--(0,0.5);
\draw [very thick] (3,-4.5)--(3,0.5);
\draw (0,0) arc (90:-90:0.5);
\draw (0,-2) to[out=0,in=180] (3,0);
\draw (0,-3) arc(90:-90:0.5);
\draw (3,-1) arc(90:270:1.5);
\draw (3,-2) arc(90:270:0.5);
}
\end{tikzpicture}\; ,
&&
\begin{tikzpicture}[baseline={([yshift=-1mm]current bounding box.center)},scale=0.5]
{
\draw [very thick] (0,-4.5)--(0,0.5);
\draw [very thick] (3,-4.5)--(3,0.5);
\draw (0,0) ..controls (1,0) and (1.3,-3).. (1.3,-4.5) to[out=90,in=-90] (1.7,0.5) ..controls (1.7,-1) and (2,-4).. (3,-4);
\draw (0,-1) ..controls (1,-1) and (0.7,-3).. (0.7,-4.5) to[out=90,in=0] (0,-4);
\draw (0,-2) arc(90:-90:0.5);
\draw (3,0) to[out=180,in=270] (2.3,0.5) to[out=270,in=180] (3,-1);
\draw (3,-2) arc(90:270:0.5);
\filldraw (2.3,0.5) circle (0.12);
\filldraw (1.7,0.5) circle (0.12);
\filldraw[fill=white] (0.7,-4.5) circle (0.12);
\filldraw[fill=white] (1.3,-4.5) circle (0.12);
}
\end{tikzpicture}\; ,
&&
\begin{tikzpicture}[baseline={([yshift=-1mm]current bounding box.center)},scale=0.5]
{
\draw [very thick] (0,-4.5)--(0,0.5);
\draw [very thick] (3,-4.5)--(3,0.5);
\draw (0,0) arc (90:-90:0.5);
\draw (0,-2) to[out=0,in=-90] (1.5,0.5) to[out=270,in=180] (3,0);
\draw (0,-3) to[out=0,in=180] (3,-1);
\draw (3,-2) arc(90:270:0.5);
\draw (0,-4)--(3,-4);
\filldraw (1.5,0.5) circle (0.12);
\filldraw[fill opacity=0,draw opacity=0] (1,-4.5) circle (0.12);
}
\end{tikzpicture}\; .
\end{align}
However, no such deformation exists for
\begin{align}
\begin{tikzpicture}[baseline={([yshift=-1mm]current bounding box.center)},scale=0.5]
{
\draw [very thick] (0,-4.5)--(0,0.5);
\draw [very thick] (3,-4.5)--(3,0.5);
\draw (0,0) to[out=0,in=180] (3,-2);
\draw (0,-1) arc (90:-90:0.5);
\draw (0,-3) arc (90:-90:0.5);
\draw (3,0) arc(90:270:0.5);
\draw (3,-3) arc(90:270:0.5);
\filldraw (1.5,-1) circle (0.12);
\filldraw (0.5,-3.5) circle (0.12);
}
\end{tikzpicture}\; ,
&&
\begin{tikzpicture}[baseline={([yshift=-1mm]current bounding box.center)},scale=0.5]
{
\draw [very thick] (0,-4.5)--(0,0.5);
\draw [very thick] (3,-4.5)--(3,0.5);
\draw (0,0)--(3,0);
\draw (0,-1)--(3,-1);
\draw (0,-2) to[out=0,in=180] (3,-4);
\draw (0,-3) arc(90:-90:0.5);
\draw (3,-2) arc(90:270:0.5);
\filldraw[fill=white] (1.5,0) circle (0.12);
}
\end{tikzpicture}\; ,
\end{align}
so these are not $\symp_5$-diagrams, and similarly
\begin{align}
\begin{tikzpicture}[baseline={([yshift=-1mm]current bounding box.center)},scale=0.5]
{
\draw [very thick] (0,-3.5)--(0,0.5);
\draw [very thick] (3,-3.5)--(3,0.5);
\draw (0,0) arc (90:-90:1.5);
\draw (0,-1) arc (90:-90:0.5);
\draw (3,0) arc(90:270:0.5);
\draw (3,-2) arc(90:270:0.5);
\filldraw (-45:1.5)+(0,-1.5) circle (0.12);
\filldraw[fill=white] (45:1.5)+(0,-1.5) circle (0.12);
}
\end{tikzpicture}
\end{align}
is not an $\symp_4$-diagram.

Multiplication of $\symp_n$-diagrams is defined by concatenation, and certain diagrammatic features may be removed or simplified at the cost of a parameter, as summarised in Table \ref{tab:sympblobparams} below. Diagrams that satisfy the requirements listed so far and do not contain any of these features form the diagram basis for $\symp_n$.

As an example of multiplication in $\symp_6$, we have
\begin{align}
\begin{tikzpicture}[baseline={([yshift=-1mm]current bounding box.center)},scale=0.5]
{
\draw [very thick] (0,-5.5)--(0,0.5);
\draw [very thick] (3,-5.5)--(3,0.5);
\draw [very thick] (6,-5.5)--(6,0.5);
\draw (0,0) to[out=0,in=180] (3,-2);
\draw (0,-1) arc(90:-90:0.5);
\draw (0,-3)--(3,-3);
\draw (0,-4) arc(90:-90:0.5);
\draw (3,0) arc(90:270:0.5);
\draw (3,-4) arc(90:270:0.5);
\filldraw (1.5,-1) circle (0.12);
\filldraw (2.5,-0.5) circle (0.12);
\filldraw (5.5,-0.5) circle (0.12);
\filldraw[fill=white] (1.5,-3) circle (0.12);
\filldraw[fill=white] (2.5,-4.5) circle (0.12);
\draw (3,0) arc(90:-90:0.5);
\draw (3,-2) arc(90:-90:0.5);
\draw (3,-4)--(6,-4);
\draw (3,-5)--(6,-5);
\draw (6,0) arc(90:270:0.5);
\draw (6,-2) arc(90:270:0.5);
\filldraw (3.5,-2.5) circle (0.12);
\filldraw (4.5,-4) circle (0.12);
\filldraw[fill=white] (4.5,-5) circle (0.12);
}
\end{tikzpicture}
\; = \; \alpha_1 \alpha_2 \delta_2\; \; 
\begin{tikzpicture}[baseline={([yshift=-1mm]current bounding box.center)},scale=0.5]
{
\draw [very thick] (0,-5.5)--(0,0.5);
\draw [very thick] (3,-5.5)--(3,0.5);
\draw (0,0) arc(90:-90:1.5);
\draw (0,-1) arc(90:-90:0.5);
\draw (0,-4) arc(90:-90:0.5);
\draw (3,-4) arc(90:270:0.5);
\filldraw (0,-1.5)+(45:1.5) circle (0.12);
\filldraw (3,-4.5)+(150:0.5) circle (0.12);
\filldraw[fill=white] (0,-1.5)+(-45:1.5) circle (0.12);
\filldraw[fill=white] (3,-4.5)+(210:0.5) circle (0.12);
\draw (3,0) arc(90:270:0.5);
\draw (3,-2) arc(90:270:0.5);
\filldraw (2.5,-0.5) circle (0.12);
}
\end{tikzpicture}
\; = \; \kappa \alpha_1\alpha_2\delta_2\;\;
\begin{tikzpicture}[baseline={([yshift=-1mm]current bounding box.center)},scale=0.5]
{
\draw [very thick] (0,-5.5)--(0,0.5);
\draw [very thick] (3,-5.5)--(3,0.5);
\draw (0,0) to[out=0,in=180] (3,-4);
\draw (0,-3) to[out=0,in=180] (3,-5);
\draw (0,-1) arc(90:-90:0.5);
\draw (0,-4) arc(90:-90:0.5);
\filldraw (1.5,-2) circle (0.12);
\filldraw[fill=white] (1.5,-4) circle (0.12);
\draw (3,0) arc(90:270:0.5);
\draw (3,-2) arc(90:270:0.5);
\filldraw (2.5,-0.5) circle (0.12);
}
\end{tikzpicture}\; .
\end{align}
Note that the simplification rules for $n$ odd or $n$ even must not be applied in the other parity. The features in the rules for $n$ even cannot occur with $n$ odd, but with $n$ even, we have for example
\begin{align}
\begin{tikzpicture}[baseline={([yshift=-1mm]current bounding box.center)},scale=0.5]
{
\draw [very thick] (0,-1.5)--(0,0.5);
\draw [very thick] (3,-1.5)--(3,0.5);
\draw [very thick] (6,-1.5)--(6,0.5);
\draw (0,0) arc(90:-90:0.5);
\draw (3,0) arc(90:270:0.5);
\filldraw (3,-0.5)+(150:0.5) circle (0.12);
\filldraw[fill=white] (3,-0.5)+(210:0.5) circle (0.12);
\draw (3,0) arc(90:-90:0.5);
\draw (6,0) arc(90:270:0.5);
\filldraw (3.5,-0.5) circle (0.12);
}
\end{tikzpicture}
\; &=\; \alpha_2\; \; 
\begin{tikzpicture}[baseline={([yshift=-1mm]current bounding box.center)},scale=0.5]
{
\draw [very thick] (0,-1.5)--(0,0.5);
\draw [very thick] (3,-1.5)--(3,0.5);
\draw [very thick] (6,-1.5)--(6,0.5);
\draw (0,0) arc(90:-90:0.5);
\draw (3,0) arc(90:270:0.5);
\filldraw[fill=white] (3,-0.5)+(180:0.5) circle (0.12);
\draw (3,0) arc(90:-90:0.5);
\draw (6,0) arc(90:270:0.5);
\filldraw (3.5,-0.5) circle (0.12);
}
\end{tikzpicture} \;
=\; \kappa \alpha_2 \; \; 
\begin{tikzpicture}[baseline={([yshift=-1mm]current bounding box.center)},scale=0.5]
{
\draw [very thick] (0,-1.5)--(0,0.5);
\draw [very thick] (3,-1.5)--(3,0.5);
\draw (0,0) arc(90:-90:0.5);
\draw (3,0) arc(90:270:0.5);
}
\end{tikzpicture}\; .
\end{align}
Incorrectly applying the relations intended for odd $n$ here would give
\begin{align}
\begin{tikzpicture}[baseline={([yshift=-1mm]current bounding box.center)},scale=0.5]
{
\draw [very thick] (0,-1.5)--(0,0.5);
\draw [very thick] (3,-1.5)--(3,0.5);
\draw [very thick] (6,-1.5)--(6,0.5);
\draw (0,0) arc(90:-90:0.5);
\draw (3,0) arc(90:270:0.5);
\filldraw (3,-0.5)+(150:0.5) circle (0.12);
\filldraw[fill=white] (3,-0.5)+(210:0.5) circle (0.12);
\draw (3,0) arc(90:-90:0.5);
\draw (6,0) arc(90:270:0.5);
\filldraw (3.5,-0.5) circle (0.12);
}
\end{tikzpicture}
\; &=\; \kappa\; \; 
\begin{tikzpicture}[baseline={([yshift=-1mm]current bounding box.center)},scale=0.5]
{
\draw [very thick] (0,-1.5)--(0,0.5);
\draw [very thick] (3,-1.5)--(3,0.5);
\draw [very thick] (6,-1.5)--(6,0.5);
\draw (0,0) arc(90:-90:0.5);
\draw (3,0) arc(90:270:0.5);
\draw (3,0) arc(90:-90:0.5);
\draw (6,0) arc(90:270:0.5);
\filldraw (3.5,-0.5) circle (0.12);
}
\end{tikzpicture} \;
=\; \kappa \alpha_1 \; \; 
\begin{tikzpicture}[baseline={([yshift=-1mm]current bounding box.center)},scale=0.5]
{
\draw [very thick] (0,-1.5)--(0,0.5);
\draw [very thick] (3,-1.5)--(3,0.5);
\draw (0,0) arc(90:-90:0.5);
\draw (3,0) arc(90:270:0.5);
}
\end{tikzpicture}\; ,
\end{align}
implying $\alpha_1=\alpha_2$.

\begin{table}[t]
    \caption{Diagrammatic simplification rules for the symplectic blob algebra. The shaded regions in the last relation may be filled with any number of strings, though the strings in corresponding regions in the two diagrams must match. Strings in a shaded region may be decorated with blobs only if the region meets a boundary, and only with the type of blob corresponding to that boundary.}
    \label{tab:sympblobparams}
    \centering
\begin{tabularx}{\textwidth}{|X|}
\hline

For all $n \in \N$:
\begin{align*}
\begin{tikzpicture}[baseline={([yshift=-1mm]current bounding box.center)},scale=0.5]
{
\draw (0,0) circle (1);
\filldraw (1,0) circle (0.12);
}
\end{tikzpicture}
\;\;  \mapsto \; \; 
\alpha_1 \hspace{30mm}
\begin{tikzpicture}[baseline={([yshift=-1mm]current bounding box.center)},scale=0.5]
{
\draw (0,0)--(2,0);
\filldraw (0.75,0) circle (0.12);
\filldraw (1.25,0) circle (0.12);
}
\end{tikzpicture}
\;\;  \mapsto \; \; 
\alpha_2\; \; 
\begin{tikzpicture}[baseline={([yshift=-1mm]current bounding box.center)},scale=0.5]
{
\draw (0,0)--(2,0);
\filldraw (1,0) circle (0.12);
}
\end{tikzpicture}
\end{align*}
\begin{align*}
\begin{tikzpicture}[baseline={([yshift=-1mm]current bounding box.center)},scale=0.5]
{
\draw (0,0) circle (1);
\filldraw[fill=white] (1,0) circle (0.12);
}
\end{tikzpicture}
\;\;  \mapsto \; \; 
\delta_1 \hspace{30mm}
\begin{tikzpicture}[baseline={([yshift=-1mm]current bounding box.center)},scale=0.5]
{
\draw (0,0)--(2,0);
\filldraw[fill=white] (0.75,0) circle (0.12);
\filldraw[fill=white] (1.25,0) circle (0.12);
}
\end{tikzpicture}
\;\;  \mapsto \; \; 
\delta_2\; \; 
\begin{tikzpicture}[baseline={([yshift=-1mm]current bounding box.center)},scale=0.5]
{
\draw (0,0)--(2,0);
\filldraw[fill=white] (1,0) circle (0.12);
}
\end{tikzpicture}
\end{align*}
\begin{align*}
\begin{tikzpicture}[baseline={([yshift=-1mm]current bounding box.center)},scale=0.5]
{
\draw[white] (1,0) circle (0.12);
\draw (0,0) circle (1);
}
\end{tikzpicture}
\;\;  \mapsto \; \; 
\beta
\end{align*}
\\
\hline

For $n$ odd only:
\begin{align*}
\begin{tikzpicture}[baseline={([yshift=-1mm]current bounding box.center)},scale=0.5]
{
\draw (0,0)--(2,0);
\filldraw (0.5,0) circle (0.12);
\filldraw[fill=white] (1,0) circle (0.12);
\filldraw (1.5,0) circle (0.12);
}
\end{tikzpicture}
\;\;  \mapsto \; \; 
\kappa\; \; 
\begin{tikzpicture}[baseline={([yshift=-1mm]current bounding box.center)},scale=0.5]
{
\draw (0,0)--(2,0);
\filldraw (1,0) circle (0.12);
}
\end{tikzpicture}
\hspace{30mm}
\begin{tikzpicture}[baseline={([yshift=-1mm]current bounding box.center)},scale=0.5]
{
\draw (0,0)--(2,0);
\filldraw[fill=white] (0.5,0) circle (0.12);
\filldraw (1,0) circle (0.12);
\filldraw[fill=white] (1.5,0) circle (0.12);
}
\end{tikzpicture}
\;\;  \mapsto \; \; 
\kappa\; \; 
\begin{tikzpicture}[baseline={([yshift=-1mm]current bounding box.center)},scale=0.5]
{
\draw (0,0)--(2,0);
\filldraw[fill=white] (1,0) circle (0.12);
}
\end{tikzpicture}
\end{align*}
\\[-2mm]
\hline

For $n$ even only:
\begin{align*}
\begin{tikzpicture}[baseline={([yshift=-1mm]current bounding box.center)},scale=0.5]
{
\draw (0,0) circle (1);
\filldraw (-1,0) circle (0.12);
\filldraw[fill=white] (1,0) circle (0.12);
}
\end{tikzpicture}
\;\;  \mapsto \; \; 
\kappa
\hspace{30mm}
\begin{tikzpicture}[baseline={([yshift=-1mm]current bounding box.center)},scale=0.5]
{
\fill[fill=gray!40] (0,0) arc(-90:0:0.5)--(1.2,0.5) arc(0:-90:1.2 and 1.5) --cycle;
\fill[fill=gray!40] (3,0) arc(270:180:0.5)--(1.8,0.5) arc(180:270:1.2 and 1.5)--cycle;
\fill[gray!40] (0,-2) arc(90:-90:0.5)--cycle;
\fill[gray!40] (3,-2) arc(90:270:0.5)--cycle;
\fill[gray!40] (0,-5) arc (90:0:0.5)--(1.2,-5.5) arc(0:90:1.2 and 1.5)--cycle;
\fill[gray!40] (3,-5) arc (90:180:0.5)--(1.8,-5.5) arc(180:90:1.2 and 1.5)--cycle;
\draw (0,-1.5) arc(90:-90:1);
\draw (3,-1.5) arc(90:270:1);
\filldraw (30:1)+(0,-2.5) circle (0.12);
\filldraw[fill=white] (-30:1)+(0,-2.5) circle (0.12);
\filldraw (150:1)+(3,-2.5) circle (0.12);
\filldraw[fill=white] (210:1)+(3,-2.5) circle (0.12);
\draw[very thick] (0,0.5)--(0,-5.5);
\draw[very thick] (3,0.5)--(3,-5.5);
}
\end{tikzpicture}\; \;
\mapsto \;\; \kappa \;\; 
\begin{tikzpicture}[baseline={([yshift=-1mm]current bounding box.center)},scale=0.5]
{
\fill[fill=gray!40] (0,0) arc(-90:0:0.5)--(1.2,0.5) arc(0:-90:1.2 and 1.5) --cycle;
\fill[fill=gray!40] (3,0) arc(270:180:0.5)--(1.8,0.5) arc(180:270:1.2 and 1.5)--cycle;
\fill[gray!40] (0,-2) arc(90:-90:0.5)--cycle;
\fill[gray!40] (3,-2) arc(90:270:0.5)--cycle;
\fill[gray!40] (0,-5) arc (90:0:0.5)--(1.2,-5.5) arc(0:90:1.2 and 1.5)--cycle;
\fill[gray!40] (3,-5) arc (90:180:0.5)--(1.8,-5.5) arc(180:90:1.2 and 1.5)--cycle;
\draw (0,-1.5)--(3,-1.5);
\draw (0,-3.5)--(3,-3.5);
\filldraw (1.5,-1.5) circle (0.12);
\filldraw[fill=white] (1.5,-3.5) circle (0.12);
\draw[very thick] (0,0.5)--(0,-5.5);
\draw[very thick] (3,0.5)--(3,-5.5);
}
\end{tikzpicture}
\end{align*}
\\
\hline
\end{tabularx}
\end{table}

The isomorphism between the non-diagrammatic and diagrammatic presentations of $\symp_n$ is given by
\begin{align}
e_i &\mapsto \; \begin{tikzpicture}[baseline={([yshift=-1mm]current bounding box.center)},scale=0.5]
{
\draw [very thick](0,-5.5) -- (0,0.5);
\draw (0,0) -- (3,0);
\draw (1.5,-0.3) node[anchor=center]{$\vdots$};
\draw (0,-1) -- (3,-1);
\draw (0,-1) node[font=\scriptsize,anchor=east]{$i-1$};
\draw (0,-2) arc (90:-90:0.5);
\draw (3,-2) arc (90:270:0.5);
\draw (0,-2) node[font=\scriptsize,anchor=east]{$i$};
\draw (0,-3) node[font=\scriptsize,anchor=east]{$i+1$};
\draw (0,-4) -- (3,-4);
\draw (0,-4) node[font=\scriptsize,anchor=east]{$i+2$};
\draw (1.5,-4.3) node[anchor=center]{$\vdots$};
\draw (0,-5) -- (3,-5);
\draw (0,-5) node[font=\scriptsize,anchor=east]{$n$};
\draw [very thick](3,-5.5) -- (3,0.5);
\draw (0,0) node[font=\scriptsize,anchor=east]{$1$};
}
\end{tikzpicture}\; ,
&f_0&\mapsto\; \begin{tikzpicture}[baseline={([yshift=-1mm]current bounding box.center)},scale=0.5]
{
\draw [very thick] (0,-5.5)--(0,0.5);
\draw [very thick] (3,-5.5)--(3,0.5);
\draw (0,0)--(3,0);
\filldraw (1.5,0) circle (0.12);
\draw (0,-1)--(3,-1);
\draw (0,-5)--(3,-5);
\node at (1.5,-2.8) [anchor=center] {$\vdots$};
}
\end{tikzpicture}\; , 
&f_n &\mapsto \; \begin{tikzpicture}[baseline={([yshift=-1mm]current bounding box.center)},scale=0.5]
{
\draw [very thick] (0,-5.5)--(0,0.5);
\draw [very thick] (3,-5.5)--(3,0.5);
\draw (0,0)--(3,0) ;
\draw (0,-4)--(3,-4);
\draw (0,-5)--(3,-5);
\filldraw[fill=white] (1.5,-5) circle (0.12);
\node at (1.5,-1.8) [anchor=center] {$\vdots$};
}
\end{tikzpicture}\; .
\end{align}
A rigorous treatment of the isomorphism between the two presentations is given in \cite{SympBlobPres}.

We note that omitting the relations \eqref{eq:sbIJI} and \eqref{eq:sbJIJ} gives a presentation for the two-boundary TL algebra $\BBTL_n$. The basis diagrams of $\BBTL_n$ are usually drawn with strings connected to the boundaries, as seen in the introduction, instead of blobs. The diagrammatic presentations of the two algebras may be related by replacing each top or bottom boundary blob in an $\symp_n$-diagram by a pair of top or bottom boundary links; the $\symp_n$ relations \eqref{eq:sbIJI} and \eqref{eq:sbJIJ} then amount to removing pairs of adjacent top-to-bottom boundary arcs and replacing each pair by a factor of $\kappa$.

\begin{proposition}\label{prop:sblabel}
In the left-hand side of each symplectic blob relation, suppose we replace each occurrence of $f_0$ and $f_n$ with $\fup{a}{b}$ and $\fdo{a}{b}$, respectively, with different labels on every occurrence. Then we can replace each $f_0$ and $f_n$ in the right-hand side with $\fup{a}{b}$ and $\fdo{a}{b}$ respectively, with appropriately chosen labels, and replace each symplectic blob parameter with some product of $\A_n(X)$ parameters, such that the resulting relations hold in $\A_n(X)$. In this sense, the $\A_n(X)$ generators satisfy the symplectic blob relations, up to labels and parameters.
\end{proposition}
\begin{proof}
The relations (\ref{eq:sbeiej}--\ref{eq:sbf0fn}) correspond directly to the relations \eqref{eq:eiej}, \eqref{eq:TLreduce}, \eqref{eq:beta}, \eqref{eq:fupej}, \eqref{eq:e1fupe1}, \eqref{eq:fupfup}, \eqref{eq:fdoej}, \eqref{eq:en1fdoen1}, \eqref{eq:fdofdo}, \eqref{eq:fupfdo}, respectively; namely,
\begin{align}
e_ie_j&=e_je_i, &&\abs{i-1}\geq 2, \\
e_ie_je_i &= e_i, &&\abs{i-j}=1, \\
e_j^2 &= \beta e_i, &&1\leq j \leq n-1, \\
\fup{a}{b}e_j&=e_j\fup{a}{b}, &&2 \leq j \leq n-1,  \\
e_1\fup{a}{b}e_1 &= \aup{a}{b}e_1, &&n\geq 2, \\
\fup{c}{a}\fup{b}{d} &= \aup{a}{b}\fup{c}{d}, \\
\fdo{a}{b}e_j &= e_j\fdo{a}{b}, &&1\leq j\leq n-2, \\
e_{n-1}\fdo{a}{b}e_{n-1} &= \ddo{a}{b} f_n, &&n\geq 2,\\
\fdo{c}{a}\fdo{b}{d} &= \ddo{a}{b} \fdo{c}{d},\\
\fup{a}{b}\fdo{c}{d} &= \fdo{c}{d}\fup{a}{b}, &&n\geq 2.
\end{align}
It remains to check \eqref{eq:sbIJI} and \eqref{eq:sbJIJ}.

For $n$ even, note that $I=\prod_{k=1}^{n/2}e_{2k-1}$ in $\symp_n$ corresponds to $\Theta=\prod_{k=1}^{n/2} e_{2k-1}$ in $\A_n(X)$, and recall from \eqref{eq:wupThetawup} that
\begin{align}
\wup{a}{b}\Theta \wup{c}{d} 
&= \fup{a}{c} \fdo{b}{d}\,  \Omega = \fup{a}{c}\fdo{b}{d} \prod_{k=1}^{\frac{n}{2}-1} e_{2k},
\end{align}
which corresponds to $J$. Hence the product corresponding to $IJI$ is
\begin{align}
\Theta \wup{a}{b}\Theta \wup{c}{d} \Theta 
&= \glab{a}{b}\glab{c}{d} \Theta &&\text{by \eqref{eq:ThetawupTheta}},
\end{align}
corresponding to $\kappa I$, up to parameters. Similarly, the product corresponding to $JIJ$ is
\begin{align}
\wup{a}{b} \Theta \wup{c}{d} \Theta \wup{e}{f}\Theta \wup{g}{h} = \glab{c}{d}\glab{e}{f}\wup{a}{b}\Theta \wup{g}{h},
\end{align}
corresponding to $\kappa J$, up to labels and parameters.

For $n$ odd, $I=f_n\prod_{k=1}^\frac{n-1}{2}e_{2k-1}$ corresponds to $\fdo{a}{b}O$, and $J=f_0\prod_{k=1}^\frac{n-1}{2} e_{2k}$ corresponds to $\fup{a}{b}E$, for any $\mathrm{a},\mathrm{b} \in X$. Then the product corresponding to $IJI$ is
\begin{align}
\fdo{a}{b}O \fup{c}{d}E \fdo{e}{f}O
&= \glab{c}{b}\,\wdo{d}{a}E \fdo{e}{f}O &&\text{by \eqref{eq:fdoOfupE}} \\
&= \glab{c}{b} \glab{d}{e} \fdo{a}{f} O&&\text{by \eqref{eq:wdoEfdoO}},
\end{align}
corresponding to $\kappa I$ up to labels and parameters. The product corresponding to $JIJ$ is
\begin{align}
\fup{a}{b}E \fdo{c}{d} O \fup{e}{f} E
&= \glab{b}{c} \wup{a}{d}\, O \fup{e}{f} E &&\text{by \eqref{eq:fupEfdoO}} \\
&= \glab{b}{c} \glab{e}{d} \fup{a}{f} E &&\text{by \eqref{eq:wupOfupE}},
\end{align}
corresponding to $\kappa J$ up to labels and parameters.
\end{proof}

A \textit{reduced monomial} in the symplectic blob algebra is a word that is not equal to a scalar multiple of a strictly shorter word in $\symp_n$, using the $\symp_n$ relations. From Proposition \ref{prop:sblabel}, it follows that, given any even-reduced word in $\A_n(X)$, if we replace each generator of the form $\fup{a}{b}$ and $\fdo{a}{b}$ with $f_0$ and $f_n$, respectively, the result is a reduced monomial in $\symp_n$.

It is known that reduced monomials in the algebraic presentation of $\symp_n$ map to basis diagrams in the diagrammatic presentation of $\symp_n$, with coefficient 1. That is, concatenating the diagrammatic generators corresponding to the generators in a reduced monomial does not produce any of the diagrammatic features that can be simplified using the rules from Table \ref{tab:sympblobparams}. It follows that the numbers of $f_0$ and $f_n$ generators in a reduced monomial are the same as the numbers of top and bottom blobs in the corresponding diagram, respectively.

Next, we give a technical result about reduced monomials and their corresponding $\symp_n$-diagrams. From here on, we will identify the algebraic generators of $\symp_n$ with their corresponding diagrammatic generators, as per the isomorphism from \cite{SympBlobPres}.

\begin{lemma}\label{lem:sbtopbot}
Let $D$ be a reduced monomial in $\symp_n$ whose corresponding $\symp_n$-diagram contains a string decorated with both a top blob and a bottom blob. Then $D$ is equal to a word $T$ in $\symp_n$ that contains $IJ$ or $JI$ as a subword, and contains the same number of copies of each of $f_0$ and $f_n$ as $D$. 
\end{lemma}
\begin{proof}
We first show that every diagram that contains a string decorated with both a top blob and a bottom blob can be expressed as a product containing $IJ$ or $JI$, where the numbers of top and bottom blobs in the product are the same as in the original diagram. We can then find a reduced monomial for each of the diagrams in the product, and take $T$ to be the product of these reduced monomials, giving the desired result.

For $n$ odd, it follows from parity and planarity considerations that the doubly-decorated string must be a throughline, and that it must be the only throughline in the diagram. We have
\begin{align}
\begin{tikzpicture}[baseline={([yshift=-3mm]current bounding box.center)},xscale=0.5,yscale=0.5]
{
\filldraw[gray!40] (0,0) arc(-90:0:0.5) --(1.3,0.5) to[out=-90,in=0] (0,-3.5) --cycle;
\filldraw[gray!40] (3,-1.5) ..controls (2.2,-1.5) and (2,-0.5).. (2,0.5) -- (2.5,0.5) arc (180:270:0.5) --cycle;
\draw (0,-4) to[out=0,in=180] (3,-2);
\filldraw[gray!40] (0,-4.5) to[out=0,in=90] (1,-6.5) -- (0.5,-6.5) arc(0:90:0.5) -- cycle;
\filldraw[gray!40] (3,-2.5) to[out=180,in=90] (1.7,-6.5) -- (2.5,-6.5) arc (180:90:0.5) -- cycle;
\draw [very thick](0,-6.5) -- (0,0.5);
\draw [very thick](3,-6.5)--(3,0.5);
\filldraw (1,-3.55) circle (0.12);
\filldraw[fill=white] (2,-2.45) circle (0.12);
}
\end{tikzpicture}
\; = \;
\begin{tikzpicture}[baseline={([yshift=-3mm]current bounding box.center)},xscale=0.5,yscale=0.5]
{
\filldraw[gray!40] (0,0) arc(-90:0:0.5) --(1.3,0.5) to[out=-90,in=0] (0,-3.5) --cycle;
\draw (0,-4) to[out=0,in=180] (3,-6);
\filldraw[gray!40] (0,-4.5) to[out=0,in=90] (1,-6.5) -- (0.5,-6.5) arc(0:90:0.5) -- cycle;
\draw (3,-5) arc(90:270:0.3 and 0.25);
\draw (3,-4) arc(90:270:0.3 and 0.25);
\draw (3,0) arc(90:270:0.3 and 0.25);
\draw (3,-1) arc(90:270:0.3 and 0.25);
\node at (2.7,-2.55) {$\vdots$};
\begin{scope}[shift={(3,0)}]
\draw (0,0)--(3,0);
\draw (0,-0.5) arc(90:-90:0.3 and 0.25);
\draw (3,-0.5) arc(90:270:0.3 and 0.25);
\draw (0,-1.5) arc(90:-90:0.3 and 0.25);
\draw (3,-1.5) arc(90:270:0.3 and 0.25);
\draw (0,-4.5) arc(90:-90:0.3 and 0.25);
\draw (3,-4.5) arc(90:270:0.3 and 0.25);
\draw (0,-5.5) arc(90:-90:0.3 and 0.25);
\draw (3,-5.5) arc(90:270:0.3 and 0.25);
\node at (0.3,-3.05) {$\vdots$};
\node at (2.7,-3.05) {$\vdots$};
\filldraw (1.5,0) circle (0.12);
\end{scope}
\begin{scope}[shift={(6,0)}]
\draw (3,-5) arc(90:270:0.3 and 0.25);
\draw (3,-4) arc(90:270:0.3 and 0.25);
\draw (3,0) arc(90:270:0.3 and 0.25);
\draw (3,-1) arc(90:270:0.3 and 0.25);
\draw (0,-5) arc(90:-90:0.3 and 0.25);
\draw (0,-4) arc(90:-90:0.3 and 0.25);
\draw (0,0) arc(90:-90:0.3 and 0.25);
\draw (0,-1) arc(90:-90:0.3 and 0.25);
\draw (0,-6)--(3,-6);
\node at (0.3,-2.55) {$\vdots$};
\node at (2.7,-2.55) {$\vdots$};
\filldraw[fill=white] (1.5,-6) circle (0.12);
\end{scope}
\begin{scope}[shift={(9,0)}]
\filldraw[gray!40] (3,-1.5) ..controls (2.2,-1.5) and (2,-0.5).. (2,0.5) -- (2.5,0.5) arc (180:270:0.5) --cycle;
\filldraw[gray!40] (3,-2.5) to[out=180,in=90] (1.7,-6.5) -- (2.5,-6.5) arc (180:90:0.5) -- cycle;
\draw (0,0) to[out=0,in=180] (3,-2);
\draw (0,-0.5) arc(90:-90:0.3 and 0.25);
\draw (0,-1.5) arc(90:-90:0.3 and 0.25);
\draw (0,-4.5) arc(90:-90:0.3 and 0.25);
\draw (0,-5.5) arc(90:-90:0.3 and 0.25);
\node at (0.3,-3.05) {$\vdots$};
\end{scope}
\draw [very thick](0,-6.5) -- (0,0.5);
\draw [very thick](3,-6.5)--(3,0.5);
\draw [very thick](6,-6.6)--(6,0.5);
\draw [very thick](9,-6.6)--(9,0.5);
\draw [very thick](12,-6.6)--(12,0.5);
}
\end{tikzpicture}\; ,
\end{align}
where the second and third diagrams are $J$ and $I$, respectively. The case where the top blob is on the right of the bottom blob is handled similarly; simply reflect the above across a vertical line.

For $n$ even, having a doubly-decorated string implies that there are no throughlines, and thus the doubly-decorated string has both ends attached to the same side of the diagram. The opposite side of the diagram must have some (possibly decorated) link \textit{exposed} to the side with the doubly-decorated string. That is, the link in question can be deformed to touch the side with the doubly-decorated string, without crossing any other strings. In the case with the doubly-decorated string on the left, and an undecorated exposed link on the right, we have 
\begin{align}
\begin{tikzpicture}[baseline={([yshift=-1mm]current bounding box.center)},scale=0.5]
{
\fill[fill=gray!40] (0,0) arc(-90:0:0.5)--(1.2,0.5) arc(0:-90:1.2 and 1.5) --cycle;
\fill[fill=gray!40] (3,0) arc(270:180:0.5)--(1.8,0.5) to[out=270,in=180] (3,-2)--cycle;
\fill[gray!40] (0,-2) arc(90:-90:0.5)--cycle;
\fill[gray!40] (3,-3) arc(90:270:0.5)--cycle;
\fill[gray!40] (0,-6) arc (90:0:0.5)--(1.2,-6.5) to[out=90,in=0] (0,-4)--cycle;
\fill[gray!40] (3,-6) arc (90:180:0.5)--(1.8,-6.5) arc(180:90:1.2 and 1.5)--cycle;
\draw (0,-1.5) arc(90:-90:1);
\draw (3,-2.5) arc(90:270:1);
\filldraw (0,-2.5)+(30:1) circle (0.12);
\filldraw[fill=white] (0,-2.5)+(-30:1) circle (0.12);
\draw[very thick] (0,0.5)--(0,-6.5);
\draw[very thick] (3,0.5)--(3,-6.5);
}
\end{tikzpicture}
\; =\;
\begin{tikzpicture}[baseline={([yshift=-1mm]current bounding box.center)},scale=0.5]
{
\fill[fill=gray!40] (0,0) arc(-90:0:0.5)--(1.2,0.5) arc(0:-90:1.2 and 1.5) --cycle;
\fill[gray!40] (0,-2) arc(90:-90:0.5)--cycle;
\fill[gray!40] (0,-6) arc (90:0:0.5)--(1.2,-6.5) to[out=90,in=0] (0,-4)--cycle;
\draw (0,-1.5) to[out=0,in=180] (3,0);
\draw (0,-3.5) to[out=0,in=180] (3,-0.5);
\draw (3,-1) arc(90:270:0.3 and 0.25);
\draw (3,-2) arc(90:270:0.3 and 0.25);
\draw (3,-4.5) arc(90:270:0.3 and 0.25);
\draw (3,-5.5) arc(90:270:0.3 and 0.25);
\node at (2.7,-3.3) {$\vdots$};
\begin{scope}[shift={(3,0)}]
\draw (0,0)--(3,0);
\draw (0,-6)--(3,-6);
\draw (0,-0.5) arc(90:-90:0.3 and 0.25);
\draw (0,-1.5) arc(90:-90:0.3 and 0.25);
\draw (0,-4) arc(90:-90:0.3 and 0.25);
\draw (0,-5) arc(90:-90:0.3 and 0.25);
\draw (3,-0.5) arc(90:270:0.3 and 0.25);
\draw (3,-1.5) arc(90:270:0.3 and 0.25);
\draw (3,-4) arc(90:270:0.3 and 0.25);
\draw (3,-5) arc(90:270:0.3 and 0.25);
\node at (0.3,-2.8) {$\vdots$};
\node at (2.7,-2.8) {$\vdots$};
\filldraw (1.5,0) circle (0.12);
\filldraw[fill=white] (1.5,-6) circle (0.12);
\end{scope}
\begin{scope}[shift={(6,0)}]
\draw (0,0) arc(90:-90:0.3 and 0.25);
\draw (0,-1) arc(90:-90:0.3 and 0.25);
\draw (0,-2) arc(90:-90:0.3 and 0.25);
\draw (0,-4.5) arc(90:-90:0.3 and 0.25);
\draw (0,-5.5) arc(90:-90:0.3 and 0.25);
\draw (3,0) arc(90:270:0.3 and 0.25);
\draw (3,-1) arc(90:270:0.3 and 0.25);
\draw (3,-2) arc(90:270:0.3 and 0.25);
\draw (3,-4.5) arc(90:270:0.3 and 0.25);
\draw (3,-5.5) arc(90:270:0.3 and 0.25);
\node at (0.3,-3.3) {$\vdots$};
\node at (2.7,-3.3) {$\vdots$};
\end{scope}
\begin{scope}[shift={(9,0)}]
\draw (0,0) to[out=0,in=180] (3,-2.5);
\draw (0,-6) to[out=0,in=180] (3,-4.5);
\draw (0,-0.5) arc(90:-90:0.3 and 0.25);
\draw (0,-1.5) arc(90:-90:0.3 and 0.25);
\draw (0,-4) arc(90:-90:0.3 and 0.25);
\draw (0,-5) arc(90:-90:0.3 and 0.25);
\node at (0.3,-2.8) {$\vdots$};
\fill[fill=gray!40] (3,0) arc(270:180:0.5)--(1.8,0.5) to[out=270,in=180] (3,-2)--cycle;
\fill[gray!40] (3,-3) arc(90:270:0.5)--cycle;
\fill[gray!40] (3,-6) arc (90:180:0.5)--(1.8,-6.5) arc(180:90:1.2 and 1.5)--cycle;
\end{scope}
\draw[very thick] (0,0.5)--(0,-6.5);
\draw[very thick] (3,0.5)--(3,-6.5);
\draw[very thick] (6,0.5)--(6,-6.5);
\draw[very thick] (9,0.5)--(9,-6.5);
\draw[very thick] (12,0.5)--(12,-6.5);
}
\end{tikzpicture}\; ,
\end{align}
where the second and third diagrams in the product are $J$ and $I$, respectively. If there is no undecorated exposed link on the right, one can similarly deform the lowest exposed link with a top blob, or the highest exposed link with a bottom blob, positioning the blob on the top or bottom throughline in the fourth diagram, respectively. The cases with the doubly-decorated string on the right are handled similarly by reflecting the above about a vertical line.
\end{proof}

The preceding results are used in Lemma \ref{lem:evendiag} to show that even-reduced words map to even diagrams, as indicated in Section \ref{s:ParityReduced}. Some further claims about parity from that section are deduced in Corollaries \ref{cor:evenwordtodiag} and \ref{cor:oddgensodd}.

\begin{lemma}\label{lem:evendiag}
Let $T$ be an even-reduced word in $\A_n(X)$. Then $\phi(T)$ is an even diagram in $\lab_n(X)$, with coefficient $1$, and the number of boundary links in $\phi(T)$ is double the number of label generators in $T$.
\end{lemma}
\begin{proof}
Let $T$ be the word $s_1s_2 \dots s_k$, where each $s_i$ is an even generator. It is clear from the definition of $\phi$ that $\phi(T)$ is a scalar multiple of a basis diagram of $\lab_n(X)$. 

Consider the symplectic blob generators $\sigma(s_i)$ corresponding to each generator $s_i$, where
\begin{align}
\sigma(e_i) &= e_i = \; \begin{tikzpicture}[baseline={([yshift=-1mm]current bounding box.center)},scale=0.5]
{
\draw [very thick](0,-5.5) -- (0,0.5);
\draw (0,0) -- (3,0);
\draw (1.5,-0.3) node[anchor=center]{$\vdots$};
\draw (0,-1) -- (3,-1);
\draw (0,-1) node[font=\scriptsize,anchor=east]{$i-1$};
\draw (0,-2) arc (90:-90:0.5);
\draw (3,-2) arc (90:270:0.5);
\draw (0,-2) node[font=\scriptsize,anchor=east]{$i$};
\draw (0,-3) node[font=\scriptsize,anchor=east]{$i+1$};
\draw (0,-4) -- (3,-4);
\draw (0,-4) node[font=\scriptsize,anchor=east]{$i+2$};
\draw (1.5,-4.3) node[anchor=center]{$\vdots$};
\draw (0,-5) -- (3,-5);
\draw (0,-5) node[font=\scriptsize,anchor=east]{$n$};
\draw [very thick](3,-5.5) -- (3,0.5);
\draw (0,0) node[font=\scriptsize,anchor=east]{$1$};
}
\end{tikzpicture}\; , &
\sigma(\fup{a}{b}) &= f_0 = \; \begin{tikzpicture}[baseline={([yshift=-1mm]current bounding box.center)},scale=0.5]
{
\draw [very thick] (0,-5.5)--(0,0.5);
\draw [very thick] (3,-5.5)--(3,0.5);
\draw (0,0)--(3,0);
\filldraw (1.5,0) circle (0.12);
\draw (0,-1)--(3,-1);
\draw (0,-5)--(3,-5);
\node at (1.5,-2.8) [anchor=center] {$\vdots$};
}
\end{tikzpicture}\; ,  &
\sigma\left(\fdo{a}{b}\right) &= f_n = \; \begin{tikzpicture}[baseline={([yshift=-1mm]current bounding box.center)},scale=0.5]
{
\draw [very thick] (0,-5.5)--(0,0.5);
\draw [very thick] (3,-5.5)--(3,0.5);
\draw (0,0)--(3,0) ;
\draw (0,-4)--(3,-4);
\draw (0,-5)--(3,-5);
\filldraw[fill=white] (1.5,-5) circle (0.12);
\node at (1.5,-1.8) [anchor=center] {$\vdots$};
}
\end{tikzpicture}
\end{align}
for each $\mathrm{a},\mathrm{b}\in X$. We can identify these algebraic generators with their diagrammatic counterparts because the algebraic and diagrammatic presentations of $\symp_n$ are known to be isomorphic. Comparing these to the corresponding images under $\phi$, namely
\begin{align}
\phi(e_i) &= \; \begin{tikzpicture}[baseline={([yshift=-1mm]current bounding box.center)},scale=0.5]
{
\draw [very thick](0,-5.5) -- (0,0.5);
\draw (0,0) -- (3,0);
\draw (1.5,-0.3) node[anchor=center]{$\vdots$};
\draw (0,-1) -- (3,-1);
\draw (0,-1) node[font=\scriptsize,anchor=east]{$i-1$};
\draw (0,-2) arc (90:-90:0.5);
\draw (3,-2) arc (90:270:0.5);
\draw (0,-2) node[font=\scriptsize,anchor=east]{$i$};
\draw (0,-3) node[font=\scriptsize,anchor=east]{$i+1$};
\draw (0,-4) -- (3,-4);
\draw (0,-4) node[font=\scriptsize,anchor=east]{$i+2$};
\draw (1.5,-4.3) node[anchor=center]{$\vdots$};
\draw (0,-5) -- (3,-5);
\draw (0,-5) node[font=\scriptsize,anchor=east]{$n$};
\draw [very thick](3,-5.5) -- (3,0.5);
\draw (0,0) node[font=\scriptsize,anchor=east]{$1$};
\node at (2.5,0.4) [anchor=south] {\scriptsize \vphantom{b}};
\node at (2.5,-6.3) [anchor=south] {\scriptsize \vphantom{b}};
}
\end{tikzpicture}\; , &
\phi(\fup{a}{b}) &=  \; \begin{tikzpicture}[baseline={([yshift=-1mm]current bounding box.center)},scale=0.5]
{
\draw [very thick] (0,-5.5)--(0,0.5);
\draw [very thick] (3,-5.5)--(3,0.5);
\draw (0,0) arc(-90:0:0.5);
\draw (3,0) arc(270:180:0.5);
\node at (0.5,0.4) [anchor=south] {\scriptsize a};
\node at (2.5,0.4) [anchor=south] {\scriptsize b};
\node at (2.5,-6.3) [anchor=south] {\scriptsize \vphantom{b}};
\draw (0,-1)--(3,-1);
\draw (0,-5)--(3,-5);
\node at (1.5,-2.8) [anchor=center] {$\vdots$};
}
\end{tikzpicture}\; ,  &
\phi\left(\fdo{a}{b}\right) &=  \; \begin{tikzpicture}[baseline={([yshift=-1mm]current bounding box.center)},scale=0.5]
{
\draw [very thick] (0,-5.5)--(0,0.5);
\draw [very thick] (3,-5.5)--(3,0.5);
\draw (0,0)--(3,0) ;
\draw (0,-4)--(3,-4);
\draw (0,-5) arc(90:0:0.5);
\draw (3,-5) arc(90:180:0.5);
\node at (0.5,-6.3) [anchor=south] {\scriptsize a};
\node at (2.5,-6.3) [anchor=south] {\scriptsize b};
\node at (2.5,0.4) [anchor=south] {\scriptsize \vphantom{b}};
\node at (1.5,-1.8) [anchor=center] {$\vdots$};
}
\end{tikzpicture}\; ,
\end{align}
we note that the only difference between the $\phi(s_i)$ and $\sigma(s_i)$ diagrams is that top (bottom) blobs from the $\sigma(s_i)$ diagrams are replaced by pairs of labelled top (bottom) boundary links in $\phi(s_i)$.

We now want to show that the concatenation of the $\phi(s_i)$ diagrams does not produce any top-to-bottom boundary arcs. Indeed, suppose for sake of contradiction that this does produce a top-to-bottom boundary arc in some position. Then the concatenation of the $\sigma(s_i)$ diagrams must contain a doubly-decorated string in the same position, since a blob on a string in a $\sigma(s_i)$ diagram corresponds to a pair of boundary links in the same position in a $\phi(s_i)$ diagram. By Lemma \ref{lem:sbtopbot}, this means the word $S\coloneqq\sigma(s_1)\sigma(s_2)\dots \sigma(s_k)$ is equal to another word $S'$ in $\symp_n$ that contains $IJ$ or $JI$ as a subword, and has the same number of copies of each of $f_0$ and $f_n$ as $S$. This means there is a finite sequence of symplectic blob relations that can be applied to $S$ to get $S'$. By Proposition \ref{prop:sblabel}, each of these relations has an equivalent in $\A_n(X)$ that can be applied to $T$ to get a scalar multiple of a word $T'$ in $\A_n(X)$, where $T'$ contains the same numbers of generators of the forms $\fup{a}{b}$ and $\fdo{a}{b}$, $\mathrm{a}, \mathrm{b} \in X$, as $T$, and has the equivalent of $IJ$ or $JI$ as a subword. That is, for $n$ odd, $T'$ contains one of
\begin{align}
\fdo{a}{b}\, O\, \fup{c}{d}\, E, && \fup{a}{b}\, E\, \fdo{c}{d}\, O
\end{align}
as a subword, for some $\mathrm{a}, \mathrm{b}, \mathrm{c}, \mathrm{d} \in X$; for $n$ even, $T'$ contains one of
\begin{align}
\Theta \fup{a}{b} \fdo{c}{d} \Omega , &&
\fup{a}{b} \fdo{c}{d} \Omega \, \Theta
\end{align}
as a subword, for some $\mathrm{a}, \mathrm{b}, \mathrm{c}, \mathrm{d} \in X$. Here we are using the words corresponding to $I$ and $J$, as discussed in Proposition \ref{prop:sblabel}. We then have
\begin{align}
\fdo{a}{b}\, O\, \fup{c}{d}\, E &= \glab{c}{b}\,  \wdo{d}{a} \, E &&\text{by \eqref{eq:fdoOfupE}}, \\
\fup{a}{b}\, E\, \fdo{c}{d}\, O &= \glab{b}{c}\, \wup{a}{d}\, O &&\text{by \eqref{eq:fupEfdoO}},
\end{align}
for $n$ odd, and 
\begin{align}
\Theta \fup{a}{b} \fdo{c}{d} \, \Omega &= \Theta\,  \wup{a}{c}\,  \Theta\,  \wup{b}{d} &&\text{by \eqref{eq:wupThetawup}} \\
&= \glab{a}{c} \, \Theta \, \wup{b}{d} &&\text{by \eqref{eq:ThetawupTheta}}, \\
\fup{a}{b} \fdo{c}{d} \, \Omega\,  \Theta &=
\wup{a}{c}\,  \Theta\,  \wup{b}{d} \, \Theta &&\text{by \eqref{eq:wupThetawup}} \\
&= \glab{b}{d}\, \wup{a}{c}\, \Theta &&\text{by \eqref{eq:ThetawupTheta}},
\end{align}
for $n$ even. Each of these contains only one label generator, when the original subword contained two. This contradicts our assumption that $T$ was label-reduced, since $T$ is a scalar multiple of $T'$, $T$ and $T'$ have the same number of label generators, and $T'$ is equal to a word containing strictly fewer label generators. It follows that the concatenation of the diagrams $\phi(s_i)$ cannot contain any top-to-bottom boundary arcs. 

Since each of the diagrams $\phi(s_i)$ is even, it follows that the product $\phi(s_1) \phi(s_2) \dots \phi(s_k)$ is a scalar multiple of an even diagram. This is because the only way for this product to be a scalar multiple of an odd diagram would be if the concatenation contained an odd (and thus nonzero) number of top-to-bottom boundary arcs.

Moreover, the concatenation of $\phi(s_i)$ diagrams cannot contain any other boundary arcs or loops. Indeed, each top (bottom) boundary arc in the concatenation of $\phi(s_i)$ diagrams corresponds to a pair of top (bottom) blobs on a string, or a loop decorated with a top (bottom) blob, in the concatenation of $\sigma(s_i)$ diagrams. Each loop similarly corresponds to an undecorated loop. These diagrammatic features can be simplified in $\symp_n$, meaning $S$ is not a reduced monomial, so there exists a sequence of $\symp_n$ relations that can be applied to the word $S$ to get a scalar multiple of a shorter word $S'$. Applying the equivalent $\A_n(X)$ relations to $T$ gives a scalar multiple of a strictly shorter word in the even generators, contradicting our assumption that $T$ is even-reduced.

We have now shown that the concatenation of $\phi(s_i)$ diagrams cannot contain any boundary arcs or loops. It follows that the product $\phi(s_1)\phi(s_2) \dots \phi(s_k)$ is this concatenated diagram, with no factors of any of the parameters, as desired. This also means that each boundary link from each $\phi(s_i)$ remains in the diagram $\phi(T)$; there are two of these for each $\phi(\fup{a}{b})$ or $\fdo{a}{b}$, and none for each $\phi(e_j)$. Hence the number of boundary links in $\phi(T)$ is two times the number of label generators in $T$.
\end{proof}

\begin{corollary}\label{cor:evenwordtodiag}
If $\mathbf{u}$ is an even word in $\A_n(X)$, then $\phi(\mathbf{u})$ is a scalar multiple of an even diagram.
\end{corollary}
\begin{proof}
If $\mathbf{u}$ is even, then it is equal to $\mu\, T$ for some scalar $\mu$ and even-reduced word $T$. Then $\phi(\mathbf{u})= \phi(\mu\, T) = \mu \, \phi(T)$, and $\phi(T)$ is an even diagram by Lemma \ref{lem:evendiag}.
\end{proof}

\begin{corollary}\label{cor:oddgensodd}
For any $\mathrm{a},\mathrm{b}\in X$ and $0 \leq j \leq n$, the word $\w{a}{b}(j) \in \A_n(X)$, as defined in Proposition \ref{prop:WTintro}, is odd. In particular, $\w{a}{b}(0)=\wup{a}{b}$ and $\w{a}{b}(n) = \wdo{a}{b}$ are odd.
\end{corollary}
\begin{proof}
If $\w{a}{b}(j)$ is even, then by Corollary \ref{cor:evenwordtodiag}, $\phi(\w{a}{b}(j))$ must be a scalar multiple of an even diagram. However, $\phi(\w{a}{b}(j))=\dw{a}{b}(j)$, as defined in Proposition \ref{prop:oddgen}, and $\dw{a}{b}(j)$ is an odd diagram for all $0 \leq j\leq n$. This is a contradiction, and thus $\w{a}{b}(j)$ is odd for all $0 \leq j \leq n$.
\end{proof}

We now show that $\phi$ is injective when restricted to even-reduced words, in Lemma \ref{lem:evenredinj}, and deduce the same result for even words in Proposition \ref{prop:evenwordinj}.
\begin{lemma}\label{lem:evenredinj}
Let $S$ and $T$ be even-reduced words in $\A_n(X)$. If $\phi(S) = \phi(T)$, then $S=T$ in $\A_n(X)$.
\end{lemma}
\begin{proof}
Recall that even-reduced words consist only of even generators; we can thus let the words $S$ and $T$ be $s_1s_2\dots s_k$ and $t_1t_2\dots t_l$ respectively, where each $s_i$ and $t_i$ is an even generator. From Lemma \ref{lem:evendiag}, $\phi(S)=\phi(T)$ is an even diagram in $\lab_n(X)$, given by the concatenation of the diagrams $\phi(s_i)$, and also given by the concatenation of the diagrams $\phi(t_i)$. 

Consider the symplectic blob diagrams $\sigma(s_i)$ and $\sigma(t_i)$ corresponding to each generator $s_i$ and $t_i$, with $\sigma$ defined as in Lemma \ref{lem:evendiag}. Since $\phi(s_1)\phi(s_2)\dots \phi(s_k) = \phi(t_1)\phi(t_2)\dots \phi(t_l)$, and each of the corresponding concatenations produces no loops or boundary arcs, it follows diagrammatically that $\sigma(s_1)\sigma(s_2)\dots \sigma(s_k) = \sigma(t_1)\sigma(t_2)\dots \sigma(t_l)$ in $\symp_n$. This means there exists a sequence of $\symp_n$ relations that can be applied to the word $\sigma(s_1)\sigma(s_2)\dots \sigma(s_k)$ to get the word $\sigma(t_1)\sigma(t_2) \dots \sigma(t_l)$.
By Proposition \ref{prop:sblabel}, applying the corresponding $\A_n(X)$ relations to $s_1s_2 \dots s_k$ then gives $t_1t_2 \dots t_l$, up to labels and parameters. That is, we have
\begin{align}
s_1s_2 \dots s_k &= \mu \, t_1't_2'\dots t_l', \label{eq:stprime}
\end{align}
where $\mu$ is a product of parameters (possibly including negative powers), and for each $m=1, 2, \dots, l$, 
\begin{itemize}
    \item If $t_m = e_i$, then $t_m' = e_i$, and
    \item If $t_m= \fup{a}{b}$ or $t_m=\fdo{a}{b}$ for some $\mathrm{a}, \mathrm{b} \in X$, then $t_m' = \fup{c}{d}$ or $t_m'=\fdo{c}{d}$ for some $\mathrm{c},\mathrm{d}\in X$, respectively.
\end{itemize}

Now consider the $\A_n(X)$ relations containing label generators, and note that all of these relations hold for any choice of labels, as long as each label matches the corresponding label on the other side of the relation. It follows that the property of being even-reduced cannot depend on the labels appearing in a word. Thus, since $t_1t_2 \dots t_l$ is even-reduced, it follows that $t_1't_2' \dots t_l'$ is also even-reduced. 
Then by Lemma \ref{lem:evendiag}, $\phi(t_1't_2'\dots t_l')$ is an even diagram with coefficient $1$. But $\phi(s_1s_2\dots s_k)$ is also an even diagram with coefficient $1$, and
\begin{align}
\phi(s_1s_2 \dots s_k) = \phi(\mu \, t_1't_2'\dots t_l') = \mu \, \phi(t_1't_2'\dots t_l'),
\end{align}
so we must have $\mu =1$, and thus
\begin{align}
    s_1s_2 \dots s_k = t_1't_2' \dots t_l'.
\end{align}

We know that $t_1't_2'\dots t_l'$ and $t_1t_2\dots t_l$ are the same sequence of generators, except possibly with different labels; we now want to show that they have the same labels. As $\phi(s_1s_2\dots s_k)$ is equal to the diagram obtained by simply concatenating the diagrams $\phi(s_i)$, without removing any boundary arcs, it follows that the labels on the top and bottom of the diagram are the superscript and subscript labels in the word $s_1s_2\dots s_k$, respectively, read left-to-right. Since $\phi(t_1t_2\dots t_l)$ and $\phi(t_1't_2'\dots t_l')$ are both equal to the diagram $\phi(s_1s_2\dots s_k)$, and $t_1t_2 \dots t_l$ and $t_1't_2'\dots t_l'$ are both even-reduced, they must also have the same superscript and subscript labels as $s_1s_2\dots s_k$. Hence $t_1t_2 \dots t_l$ and $t_1't_2'\dots t_l'$ are actually the same word; that is, for each $i$, $t_i$ and $t_i'$ are the same generator of $\A_n(X)$. It follows that $s_1s_2 \dots s_k = t_1t_2 \dots t_l$, as desired.
\end{proof}

\begin{proposition}\label{prop:evenwordinj}
Let $\mathbf{u}$ and $\mathbf{v}$ be even words in $\A_n(X)$. If $\phi(\mathbf{u})=\phi(\mathbf{v})$, then $\mathbf{u}=\mathbf{v}$ in $\A_n(X)$.
\end{proposition}
\begin{proof}
Since $\mathbf{u}$ and $\mathbf{v}$ are even words, they are equal to nonzero scalar multiples of even-reduced words, and thus $\mathbf{u} = \mu\, S$ and $\mathbf{v}=\nu \, T$ for some nonzero scalars $\mu$ and $\nu$, and even-reduced words $S$ and $T$. Then from $\phi(\mathbf{u})=\phi(\mathbf{v})$, it follows that $\mu\, \phi(S)=\nu\, \phi(T)$. From Lemma \ref{lem:evendiag}, $\phi(S)$ and $\phi(T)$ are even diagrams, so this implies both $\mu=\nu$ and $\phi(S)=\phi(T)$, as $\mu$ and $\nu$ are nonzero. Lemma \ref{lem:evenredinj} then implies $S=T$, and thus $\mathbf{u}=\mu\, S = \nu\, T = \mathbf{v}$.
\end{proof}

\section{Left descent sets}\label{s:leftdescentsets}
In this section, we provide some restrictions on the possible words $T$ for which $\w{a}{b}(j)T$ is in $WT$ form, for each $j$. 

For an even-reduced word $T$ in $\A_n(X)$, we define the \textit{left descent set} $\mathcal{L}(T)$ of $T$ to be the set of all $\A_n(X)$ generators $s$ such that $T=sT'$, where $sT'$ is an even-reduced word. That is, $\mathcal{L}(T)$ is the set containing the leftmost generator of each even-reduced word equal to $T$.

Note that this definition of left descent set differs from the one in \cite{SympBlobPres} for words in $\symp_n$, but serves the same purpose, at least for our even-reduced words in $\A_n(X)$. Their definition works for any word in $\symp_n$, but only considers those related by the commutation relations 
\begin{align}
e_ie_j&=e_je_i, &&\abs{i-j}\geq 2, \\
f_0e_j&=e_jf_0, &&2 \leq j \leq n-1, \\
f_ne_j&=e_jf_n, && 1 \leq j \leq n-2, \\
f_0f_n &=  f_nf_0, && n\geq 2,
\end{align}
instead of using the full set of $\symp_n$ relations. One might thus wish to define the left descent set of any word in $\A_n(X)$, but this would need to be much more cumbersome to accommodate the generators $\wup{a}{b}$ and $\wdo{a}{b}$, and we really only need the concept for even-reduced words. The analogous definition of a right descent set is also unnecessary for our purposes.

\begin{proposition}\label{prop:leftdescent}
Let $\w{a}{b}(j)T$ be a word in $WT$ form, for some $\mathrm{a},\mathrm{b} \in X$ and $0 \leq j \leq n$.
\begin{enumerate}[label=(\roman*)]
    \item If $0 \leq j \leq n-1$, then $\fdo{e}{f} \notin \mathcal{L}(T)$ for any $\mathrm{e},\mathrm{f} \in X$.
    \item If $1 \leq j \leq n$, then $\fup{e}{f} \notin \mathcal{L}(T)$ for any $\mathrm{e},\mathrm{f} \in X$.
    \item For all $0\leq j \leq n$, if $i<j$, then $e_i \notin \mathcal{L}(T)$.
    \end{enumerate}
\end{proposition}
\begin{proof}
We will prove each part of this lemma by contradiction. Recall from Proposition \ref{prop:WTintro} that
\begin{align}
\w{a}{b}(0) &\coloneqq  \wup{a}{b}, \\
\w{a}{b}(i) &\coloneqq  e_ie_{i-1} \dots e_2e_1 \wup{a}{b}, &&1 \leq i \leq n-1, \\
\w{a}{b}(n) &\coloneqq  \fdo{c}{d}e_{n-1}e_{n-2}\dots e_2e_1\wup{a}{b}.
\end{align}

\vspace{4mm}
\noindent\textbf{Part (i):} Since $0 \leq j \leq n-1$, the rightmost generator of $\w{a}{b}(j)$ is $\wup{a}{b}$. If $\fdo{e}{f} \in \mathcal{L}(T)$, then $T=\fdo{e}{f}T'$, where $\fdo{e}{f}T'$ is an even-reduced word. Recall that even-reduced words are label-reduced, by definition. Then since $T$ and $\fdo{e}{f}T'$ are equal, and both label-reduced, it follows that they contain the same number of label generators, so $T'$ contains one less label generator than $T$. We then have
\begin{align}
\w{a}{b}(j)T
&= e_je_{j-1} \dots e_1 \wup{a}{b} T &&\text{by def. of $\w{a}{b}(j)$} \\
&= e_je_{j-1} \dots e_1 \wup{a}{b} \fdo{e}{f} T' &&\text{since $T=\fdo{e}{f}T'$} \\
&= \ddo{b}{e} e_je_{j-1} \dots e_1 \wup{a}{f} T' &&\text{by \eqref{eq:wupfdo}}.
\end{align}
We note that the final expression is a scalar multiple of a word that contains one less label generator than $\w{a}{b}(j)T$, but is equal to $\w{a}{b}(j)T$. This is a contradiction as $\w{a}{b}(j)T$ is in $WT$ form, and thus label-reduced.

\vspace{4mm}
\noindent\textbf{Part (ii):} If $\fup{e}{f} \in \mathcal{L}(T)$, then $T=\fup{e}{f}T'$, where $\fup{e}{f}T'$ is an even-reduced word. As in (i), $T'$ contains one less label generator than $T$.

First consider $1 \leq j \leq n-1$, so $\w{a}{b}(j) = e_je_{j-1}\dots e_1\wup{a}{b}$. We then have $n\geq 2$, since $e_1$ appears in each of these words. Then we have
\begin{align}
\w{a}{b}(j)T
&= e_je_{j-1} \dots e_1 \wup{a}{b} T&&\text{by def. of $\w{a}{b}(j)$} \\
&= e_je_{j-1} \dots e_1 \wup{a}{b} \fup{e}{f} T' &&\text{since $T=\fup{e}{f}T'$} \\
&= e_je_{j-1} \dots e_1 \fup{a}{e} e_1 \wup{f}{b} T' &&\text{by \eqref{eq:wupfup}} \\
&= \aup{a}{e} e_je_{j-1} \dots e_1 \wup{f}{b} T' &&\text{by \eqref{eq:e1fupe1}},
\end{align}
and the last expression is a scalar multiple of a word containing one less label generator than $\w{a}{b}(j)T$. This is a contradiction as $\w{a}{b}(j)T$ is label-reduced. Hence $\fup{e}{f}\notin \mathcal{L}(T)$ for $1\leq j \leq n-1$.

For $j=n$, we have $\w{a}{b}(j)=\w{a}{b}(n)=\wdo{a}{b}$. Then
\begin{align}
\w{a}{b}(j)T
&= \wdo{a}{b}T &&\text{by def. of $\w{a}{b}(n)$} \\
&= \wdo{a}{b}\fup{e}{f} T' &&\text{since $T=\fup{e}{f}T'$} \\
&= \aup{a}{e} \wdo{f}{b} T' &&\text{by \eqref{eq:wdofup}},
\end{align}
where the last expression is again a scalar multiple of a word containing one less label generator than $\w{a}{b}(j)T$, contradicting the fact that $\w{a}{b}(j)T$ is label-reduced. Hence $\fup{e}{f}\notin \mathcal{L}(T)$ for $j=n$.

\vspace{4mm}
\noindent\textbf{Part (iii):} If $e_i \in \mathcal{L}(T)$, then $T=e_i T'$, where $e_i T'$ is even-reduced. Since $T$ and $e_iT'$ are both even-reduced, and thus label-reduced, they must contain the same number of label generators. This means $T$ and $T'$ have the same number of label generators. 

For $j=0$ and $j=1$, $e_i$ with $i<j$ does not exist. For $2 \leq j \leq n-1$, and $i<j$, we have
\begin{align}
\w{a}{b}(j)T
&= e_je_{j-1} \dots e_{i+1}e_ie_{i-1}e_{i-2} \dots e_1 \wup{a}{b} T&&\text{by def. of $\w{a}{b}(j)$} \\
&= e_je_{j-1} \dots e_{i+1}e_ie_{i-1}e_{i-2} \dots e_1 \wup{a}{b} e_i T'&&\text{since $T=e_iT'$} \\
&= e_je_{j-1} \dots e_{i+1}e_ie_{i-1}e_{i-2} \dots e_1 e_{i+1}\wup{a}{b} T' &&\text{by \eqref{eq:ejwup}, since $i<n-1$} \\
&= e_je_{j-1} \dots e_{i+1}e_ie_{i+1}e_{i-1}e_{i-2} \dots e_1 \wup{a}{b} T' &&\text{by \eqref{eq:eiej}} \\
&= e_je_{j-1} \dots e_{i+1}e_{i-1}e_{i-2} \dots e_1 \wup{a}{b} T' &&\text{by \eqref{eq:TLreduce}}  \\
&= (e_{i-1}e_{i-2} \dots e_1) (e_je_{j-1} \dots e_{i+1})\wup{a}{b} T' &&\text{by \eqref{eq:eiej}} \\
&= (e_{i-1}e_{i-2} \dots e_1) \wup{a}{b} (e_{j-1}e_{j-2} \dots e_{i}) T' &&\text{by \eqref{eq:ejwup}}.
\end{align}
Note that the final expression is equal to $\w{a}{b}(j)T$, and contains the same number of label generators, so it is label-reduced. This means the subword $(e_{j-1}e_{j-2} \dots e_{i+1}e_{i}) T'$ must also be label-reduced. Since this subword is label-reduced and contains only even generators, it is even, and thus equal to a scalar multiple (say $\mu$) of some even-reduced word $T''$ containing the same number of label generators. But then 
\begin{align}
\w{a}{b}(j)T = \mu (e_{i-1} \dots e_1) \wup{a}{b} T'' = \mu \w{a}{b}(i-1)T'',
\end{align}
where $i-1<j$, and $T''$ is even-reduced. Since $T''$ contains the same number of label generators as $(e_{j-1}e_{j-2} \dots e_{i+1}e_{i}) T'$, and $(e_{i-1} \dots e_1) \wup{a}{b} (e_{j-1}e_{j-2} \dots e_{i+1}e_{i}) T'$ is label-reduced, it follows that $\w{a}{b}(i-1)T''$ is label-reduced. This is a contradiction as words $\w{a}{b}(j)T$ in $WT$ form must have $j$ minimal; that is, they cannot be equal to a scalar multiple of some other label-reduced word $\w{c}{d}(k)T''$ with $k<j$ and $T''$ even-reduced. Hence $e_i \notin \mathcal{L}(T)$ for any $i<j$. This proves (iii) for $0 \leq j \leq n-1$.

For $j=n$, $\w{a}{b}(n)=\wdo{a}{b}$. If $n=1$ then there are no $e_i$ with $i<j=n$, so the result holds. Hence take $n\geq 2$. Then, as seen in the proof of Lemma \ref{lem:wuponly}, 
\begin{align}
\wdo{a}{b}e_i &= e_{i-1}e_{i-2} \dots e_1 \wup{a}{b} e_{n-1} e_{n-2} \dots e_i.
\end{align}
It follows that
\begin{align}
\w{a}{b}(n)T
&= \wdo{a}{b} e_iT' \\
&= e_{i-1}e_{i-2} \dots e_1 \wup{a}{b} e_{n-1} e_{n-2} \dots e_iT'.
\end{align}
Since $T$ and $T'$ have the same number of label generators, the final expression has the same number of label generators as $\w{a}{b}(n)T$, and is thus label-reduced, since they are equal. Then $e_{n-1}e_{n-2} \dots e_i T'$ is a label-reduced word in the even generators, so it is even, and thus equal to a scalar multiple of some even-reduced word $T''$. Thus,
\begin{align}
\w{a}{b}(n)T &= \mu \, \w{a}{b}(i-1)T''
\end{align}
for some scalar $\mu$. Since $\w{a}{b}(n)T$ and $\w{a}{b}(i-1)T''$ each have the same number of label generators, it follows that $\w{a}{b}(i-1)T''$ is label-reduced. This contradicts the minimality of $j=n$ in $\w{a}{b}(n)T$, as $i-1<j$. Hence $e_i \notin \mathcal{L}(T)$ for any $i <j=n$. This completes the proof of (iii).
\end{proof}

\begin{lemma}\label{lem:simplelinklds}
Let $T$ be an even-reduced word such that $\phi(T)$ has a simple link at $i$, for some $i=1, \dots, n-1$. Then the left descent set of $T$ contains $e_i$.
\end{lemma}
\begin{proof}
Let $T$ be the word $s_1s_2 \dots s_k$, where each $s_j$ is an even generator. Since $T$ is even-reduced, by Lemma \ref{lem:evendiag}, $\phi(T)$ is an even diagram with coefficient $1$. In particular, this diagram is equal to the concatenation of the diagrams $\phi(s_j)$, and this concatenation does not produce any loops or boundary arcs. The corresponding concatenation of symplectic blob diagrams $\sigma(s_j)$, defined in Lemma \ref{lem:evendiag}, must then contain none of the features from Table \ref{tab:sympblobparams} that can be simplified, meaning $\sigma(s_1)\sigma(s_2)\dots \sigma(s_k)$ is a reduced monomial in $\symp_n$. Moreover, the diagram $\sigma(s_1)\sigma(s_2)\dots \sigma(s_k)$ has an undecorated simple link at $i$ on the left. By \cite[Lemma 6.2]{SympBlobPres}, noting the slightly different definition of left descent set in that paper, there exists a sequence of the commutation relations \eqref{eq:sbeiej}, \eqref{eq:sbf0ej}, \eqref{eq:sbfnej} and \eqref{eq:sbfnfn}, that can be applied to $\sigma(s_1)\sigma(s_2)\dots \sigma(s_k)$, to get a word of the same length beginning with $e_i$. Applying the analogous $\A_n(X)$ relations to $s_1s_2\dots s_k$ yields a word $S$ of the same length, beginning with $e_i$, consisting only of even generators. Then $S$ is even-reduced, because $T$ is even-reduced, and $S$ is equal to $T$ and has the same length. Since $S$ and $T$ are even-reduced, and $S$ begins with $e_i$, we have $e_i \in \mathcal{L}(T)$, as desired.
\end{proof}

\section{Injectivity}\label{s:labeliso}
In this section, we establish the injectivity of the map $\phi: \A_n(X)\to \lab_n(X)$, and thus show that the algebraic and diagrammatic label algebras are isomorphic. 

We first want to show that words in $WT$ form map to basis diagrams of $\lab_n(X)$ under $\phi$. This is the content of Proposition \ref{prop:WTdiag}. For $W=\id$, this follows from Lemma \ref{lem:evendiag}, so it remains to show that the products $\phi(\w{a}{b}(j))\phi(T)$ do not produce any loops or boundary arcs. It is impossible for such products to produce any loops, as the diagrams $\phi(\w{a}{b}(j))=\dw{a}{b}(j)$ have no links on their right sides. The proof for boundary arcs is much more involved, and relies on Lemmas \ref{lem:topbdylinkat1}--\ref{lem:evenlonglink} to cover the cases listed in Figure \ref{fig:WTbdyarccases}. Each of these lemmas assumes that the diagram $\phi(T)$ contains a particular boundary link or link, and finds another word that is equal to $T$ and contains the same number of label generators. This word enables us to show in Proposition \ref{prop:WTdiag} that for such $T$, if $\phi(\w{a}{b}(j))\phi(T)$ produces a boundary arc, then $\w{a}{b}(j)T$ is not label-reduced, and thus not in $WT$ form.

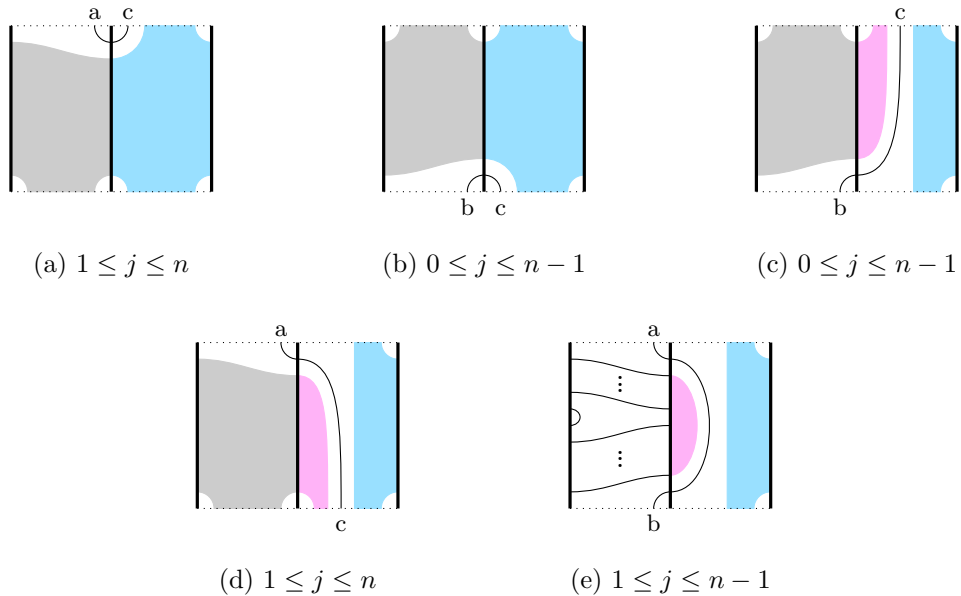
\begin{figure}[H]
\centering
\begin{subfigure}[b]{0.3\textwidth}
\begin{align*}
\begin{tikzpicture}[baseline={([yshift=-1mm]current bounding box.center)},scale=0.44]
{
\filldraw[gray!40] (0,0) to[out=0,in=180] (3,-0.5)--(3,-4) arc(90:180:0.5)--(0.5,-4.5) arc(0:90:0.5) --cycle;
\filldraw[blue!30!cyan!40] (3,-0.5) arc(-90:0:1) --(5.5,0.5) arc(180:270:0.5) --(6,-4) arc(90:180:0.5) --(3.5,-4.5) arc(0:90:0.5)--cycle;
\draw [very thick](0,-4.5) -- (0,0.5);
\draw [very thick](3,-4.5)--(3,0.5);
\draw [very thick](6,-4.5)--(6,0.5);
\draw (2.5,0.5) arc(-180:0:0.5);
\draw [dotted] (0,0.5)--(6,0.5);
\draw[dotted] (0,-4.5)--(6,-4.5);
\node at (2.5,0.4) [anchor=south] {\footnotesize $\mathrm{a}$};
\node at (3.5,0.4) [anchor=south] {\footnotesize $\mathrm{c}$};
\node at (2.5,-5.5) [anchor=south] {\footnotesize $\vphantom{\mathrm{b}}$};
}
\end{tikzpicture}
\end{align*}
\vspace{-3mm}
\caption{$1\leq j \leq n$}
\end{subfigure}
\begin{subfigure}[b]{0.3\textwidth}
\begin{align*}
\begin{tikzpicture}[baseline={([yshift=-1mm]current bounding box.center)},scale=0.44]
{
\filldraw[gray!40] (0,0) arc(-90:0:0.5) --(2.5,0.5) arc(180:270:0.5)--(3,-3.5) to[out=180,in=0] (0,-4)--cycle;
\filldraw[blue!30!cyan!40] (3,0) arc(-90:0:0.5) --(5.5,0.5) arc(180:270:0.5) --(6,-4) arc(90:180:0.5) --(4,-4.5) arc(0:90:1)--cycle;
\draw [very thick](0,-4.5) -- (0,0.5);
\draw [very thick](3,-4.5)--(3,0.5);
\draw [very thick](6,-4.5)--(6,0.5);
\draw (2.5,-4.5) arc(180:0:0.5);
\draw [dotted] (0,0.5)--(6,0.5);
\draw[dotted] (0,-4.5)--(6,-4.5);
\node at (2.5,0.4) [anchor=south] {\footnotesize $\vphantom{\mathrm{ac}}$};
\node at (2.5,-5.5) [anchor=south] {\footnotesize $\vphantom{\mathrm{b}}\mathrm{b}$};
\node at (3.5,-5.5) [anchor=south] {\footnotesize $\vphantom{\mathrm{b}}\mathrm{c}$};
}
\end{tikzpicture}
\end{align*}
\vspace{-3mm}
\caption{$0 \leq j \leq n-1$}
\end{subfigure}
\begin{subfigure}[b]{0.3\textwidth}
\begin{align*}
\begin{tikzpicture}[baseline={([yshift=-1mm]current bounding box.center)},scale=0.44]
{
\filldraw[gray!40] (0,0) arc(-90:0:0.5) --(2.5,0.5) arc(180:270:0.5)--(3,-3.5) to[out=180,in=0] (0,-4)--cycle;
\filldraw[blue!30!cyan!40] (6,-4) arc(90:180:0.5) --(4.7,-4.5)--(4.7,0.5)--(5.5,0.5) arc(180:270:0.5)--cycle;
\filldraw[red!10!magenta!30] (3,0) arc(-90:0:0.5) --(3.9,0.5) ..controls (3.9,-2) and (3.9,-3.5).. (3,-3.5)--cycle;
\draw [very thick](0,-4.5) -- (0,0.5);
\draw [very thick](3,-4.5)--(3,0.5);
\draw [very thick](6,-4.5)--(6,0.5);
\draw (3,-4) arc(90:180:0.5);
\draw (3,-4) ..controls (4.3,-4) and (4.3,-2).. (4.3,0.5);
\draw [dotted] (0,0.5)--(6,0.5);
\draw[dotted] (0,-4.5)--(6,-4.5);
\node at (4.3,0.4) [anchor=south] {\footnotesize $\vphantom{\mathrm{a}}\mathrm{c}$};
\node at (2.5,-5.5) [anchor=south] {\footnotesize $\vphantom{\mathrm{b}}\mathrm{b}$};
}
\end{tikzpicture}
\end{align*}
\vspace{-3mm}
\caption{$0\leq j \leq n-1$}
\end{subfigure}
\begin{subfigure}[b]{0.3\textwidth}
\begin{align*}
\begin{tikzpicture}[baseline={([yshift=-1mm]current bounding box.center)},scale=0.44]
{
\filldraw[gray!40] (0,0) to[out=0,in=180] (3,-0.5)--(3,-4) arc(90:180:0.5)--(0.5,-4.5) arc(0:90:0.5) --cycle;
\filldraw[blue!30!cyan!40] (6,-4) arc(90:180:0.5) --(4.7,-4.5)--(4.7,0.5)--(5.5,0.5) arc(180:270:0.5)--cycle;
\filldraw[red!10!magenta!30] (3,-4) arc(90:0:0.5) --(3.9,-4.5) ..controls (3.9,-2) and (3.9,-0.5).. (3,-0.5)--cycle;
\draw [very thick](0,-4.5) -- (0,0.5);
\draw [very thick](3,-4.5)--(3,0.5);
\draw [very thick](6,-4.5)--(6,0.5);
\draw (3,0) arc(270:180:0.5);
\draw (3,0) ..controls (4.3,0) and (4.3,-2).. (4.3,-4.5);
\draw [dotted] (0,0.5)--(6,0.5);
\draw[dotted] (0,-4.5)--(6,-4.5);
\node at (2.5,0.4) [anchor=south] {\footnotesize $\mathrm{a}\vphantom{\mathrm{c}}$};
\node at (4.3,-5.5) [anchor=south] {\footnotesize $\vphantom{\mathrm{b}}\mathrm{c}$};
}
\end{tikzpicture}
\end{align*}
\vspace{-3mm}
\caption{$1 \leq j \leq n$}
\end{subfigure}
\begin{subfigure}[b]{0.3\textwidth}
\begin{align*}
\begin{tikzpicture}[baseline={([yshift=-1mm]current bounding box.center)},scale=0.44]
{
\filldraw[blue!30!cyan!40] (6,-4) arc(90:180:0.5) --(4.7,-4.5)--(4.7,0.5)--(5.5,0.5) arc(180:270:0.5)--cycle;
\filldraw[red!10!magenta!30] (3,-0.5) arc(90:-90:0.8 and 1.5) --cycle;
\draw (0,0) to[out=0,in=180] (3,-0.5);
\draw (0,-4) to[out=0,in=180] (3,-3.5);
\draw (0,-1.5) arc(90:-90:0.3 and 0.25);
\draw (0,-1) to[out=0,in=180] (3,-1.5);
\draw (0,-2.5) to[out=0,in=180] (3,-2);
\draw [very thick](0,-4.5) -- (0,0.5);
\draw [very thick](3,-4.5)--(3,0.5);
\draw [very thick](6,-4.5)--(6,0.5);
\draw (3,0) arc(270:180:0.5);
\draw (3,-4) arc(90:180:0.5);
\draw (3,0) to[out=0,in=0] (3,-4);
\draw [dotted] (0,0.5)--(6,0.5);
\draw[dotted] (0,-4.5)--(6,-4.5);
\node at (1.5,-0.55) [anchor=center] {.};
\node at (1.5,-0.75) [anchor=center] {.};
\node at (1.5,-0.95) [anchor=center] {.};
\node at (1.5,-2.8) [anchor=center] {.};
\node at (1.5,-3) [anchor=center] {.};
\node at (1.5,-3.2) [anchor=center] {.};
\node at (2.5,0.4) [anchor=south] {\footnotesize $\mathrm{a}\vphantom{\mathrm{c}}$};
\node at (2.5,-5.5) [anchor=south] {\footnotesize $\mathrm{b}$};
}
\end{tikzpicture}
\end{align*}
\vspace{-3mm}
\caption{$1 \leq j \leq n-1$}
\end{subfigure}
\caption{The five cases considered in Proposition \ref{prop:WTdiag}, where concatenating $\phi(\w{a}{b}(j))$ and $\phi(T)$ produces a boundary arc.}
    \label{fig:WTbdyarccases}
\end{figure}

\pagebreak
\begin{lemma}[Cases (a), (b)]\label{lem:topbdylinkat1}
Let $T$ be an even-reduced word in $\A_n(X)$.
\begin{enumerate}[label=(\roman*)]
    \item If $\phi(T)$ has a top boundary link at $1$ on the left, then $T=\fup{a}{b}T'$ for some $\mathrm{a},\mathrm{b}\in X$, where $T'$ is even-reduced and contains one less label generator than $T$.
    \item If $\phi(T)$ has a bottom boundary link at $n$ on the left, then $T=\fdo{a}{b}T'$ for some $\mathrm{a},\mathrm{b}\in X$, where $T'$ is even-reduced and contains one less label generator than $T$.
\end{enumerate}
\end{lemma}
\begin{proof}
We prove (i); (ii) is proven analogously by reflecting the diagrams about a horizontal line. We thus let $T$ be an even-reduced word in $\A_n(X)$ such that $\phi(T)$ has a top boundary link at $1$ on the left. As $T$ is even-reduced, we have from Lemma \ref{lem:evendiag} that $\phi(T)$ is an even diagram. This means $\phi(T)$ has at least one more top boundary link, in addition to the one at $1$ on the left. 

Consider the second-leftmost top boundary link. If this comes from the left side of the diagram, we have
\begin{align}
\phi(T) &= \; \begin{tikzpicture}[baseline={([yshift=-3mm]current bounding box.center)},xscale=0.5,yscale=0.5]
{
\draw (0,0) arc(-90:0:0.5);
\filldraw[blue!30!cyan!40] (0,-0.5) arc(90:-90:1 and 1.5) --cycle;
\draw (0,-4) to[out=0,in=-90] (1.5,0.5);
\filldraw[red!10!magenta!30] (0,-4.5) to[out=0,in=-90] (2,0.5) -- (2.5,0.5) arc (180:270:0.5) -- (3,-6) arc(90:180:0.5)--(0.5,-6.5) arc(0:90:0.5) -- cycle;
\node at (0.5,0.4) [anchor=south] {\footnotesize $\mathrm{a}$};
\node at (1.5,0.4) [anchor=south] {\footnotesize $\mathrm{b}$};
\draw[dotted] (0,0.5)--(3,0.5);
\draw[dotted] (0,-6.5)--(3,-6.5);
\draw [very thick](0,-6.5) -- (0,0.5);
\draw [very thick](3,-6.5)--(3,0.5);
}
\end{tikzpicture} 
\; = \; 
\begin{tikzpicture}[baseline={([yshift=-3mm]current bounding box.center)},xscale=0.5,yscale=0.5]
{
\draw (0,0) arc(-90:0:0.5);
\draw (3,0) arc(270:180:0.5);
\draw (0,-0.5)--(3,-0.5);
\draw (0,-3.5)--(3,-3.5);
\draw (0,-4)--(3,-4);
\draw (0,-4.5)--(3,-4.5);
\draw (0,-6)--(3,-6);
\node at (0.5,0.4) [anchor=south] {\footnotesize $\mathrm{a}$};
\node at (2.5,0.4) [anchor=south] {\footnotesize $\mathrm{b}$};
\node at (1.5,-1.8) {$\vdots$};
\node at (1.5,-5.05) {$\vdots$};
\begin{scope}[shift={(3,0)}]
\filldraw[blue!30!cyan!40] (0,-0.5) arc(90:-90:1 and 1.5) --cycle;
\draw (0,0) arc(90:-90:1.5 and 2);
\filldraw[red!10!magenta!30] (0,-4.5) to[out=0,in=-90] (2,0.5) -- (2.5,0.5) arc (180:270:0.5) -- (3,-6) arc(90:180:0.5)--(0.5,-6.5) arc(0:90:0.5) -- cycle;
\end{scope}
\draw[dotted] (0,0.5)--(6,0.5);
\draw[dotted] (0,-6.5)--(6,-6.5);
\draw [very thick](0,-6.5) -- (0,0.5);
\draw [very thick](3,-6.5)--(3,0.5);
\draw[very thick] (6,-6.5)--(6,0.5);
}
\end{tikzpicture}\; .
\end{align}
If the second-leftmost top boundary link comes from the right side of the diagram, then we have
\begin{align}
\phi(T) &= \; \begin{tikzpicture}[baseline={([yshift=-3mm]current bounding box.center)},xscale=0.5,yscale=0.5]
{
\draw (0,0) arc(-90:0:0.5);
\filldraw[blue!30!cyan!40] (3,0) arc(270:180:0.5) --(1.5,0.5) to[out=270,in=180] (3,-3) --cycle;
\draw (3,-3.5) to[out=180,in=270] (1,0.5);
\filldraw[red!10!magenta!30] (0,-0.5) ..controls (0.8,0) and (0.7,-4).. (3,-4) -- (3,-6) arc (90:180:0.5) -- (0.5,-6.5) arc(0:90:0.5) -- cycle;
\node at (0.5,0.4) [anchor=south] {\footnotesize $\mathrm{a}$};
\node at (1,0.4) [anchor=south] {\footnotesize $\mathrm{b}$};
\draw[dotted] (0,0.5)--(3,0.5);
\draw[dotted] (0,-6.5)--(3,-6.5);
\draw [very thick](0,-6.5) -- (0,0.5);
\draw [very thick](3,-6.5)--(3,0.5);
}
\end{tikzpicture} 
\; = \; 
\begin{tikzpicture}[baseline={([yshift=-3mm]current bounding box.center)},xscale=0.5,yscale=0.5]
{
\draw (0,0) arc(-90:0:0.5);
\draw (3,0) arc(270:180:0.5);
\draw (0,-0.5)--(3,-0.5);
\draw (0,-6)--(3,-6);
\node at (0.5,0.4) [anchor=south] {\footnotesize $\mathrm{a}$};
\node at (2.5,0.4) [anchor=south] {\footnotesize $\mathrm{b}$};
\node at (1.5,-3.05) {$\vdots$};
\begin{scope}[shift={(3,0)}]
\filldraw[blue!30!cyan!40] (3,0) arc(270:180:0.5) --(1.5,0.5) to[out=270,in=180] (3,-3) --cycle;
\draw (3,-3.5) ..controls (0.9,-3.5) and (1.4,0).. (0,0);
\filldraw[red!10!magenta!30] (0,-0.5) ..controls (0.8,0) and (0.7,-4).. (3,-4) -- (3,-6) arc (90:180:0.5) -- (0.5,-6.5) arc(0:90:0.5) -- cycle;
\end{scope}
\draw[dotted] (0,0.5)--(6,0.5);
\draw[dotted] (0,-6.5)--(6,-6.5);
\draw [very thick](0,-6.5) -- (0,0.5);
\draw [very thick](3,-6.5)--(3,0.5);
\draw [very thick](6,-6.5)--(6,0.5);
}
\end{tikzpicture} \; .
\end{align}
Let the second diagram in the appropriate product above be $\undertilde{T}'$. Note that this diagram is even. Hence from Corollary \ref{cor:evenredsurj}, there is an even-reduced word $T'$ such that $\phi(T')=\undertilde{T}'$. Since concatenating $\phi(\fup{a}{b})$ and $\undertilde{T}'$ produces no boundary arcs, the number of boundary links in $\phi(\fup{a}{b}T')= \phi(T)$ is the sum of the number of boundary links in each of $\phi(\fup{a}{b})$ and $\undertilde{T}'$. From Lemma \ref{lem:evendiag}, the number of boundary links in $\undertilde{T}'$ is twice the number of label generators in $T'$, so the number of boundary links in $\phi(T)$ is this, plus the two from $\phi(\fup{a}{b})$.
As $T$ is even-reduced, we also have that the number of boundary links in $\phi(T)$ is twice the number of label generators in $T$, from Lemma \ref{lem:evendiag}. Hence $T$ contains one more label generator than $T'$.
Finally, by Proposition \ref{prop:evenwordinj}, $\phi(\fup{a}{b}T')=\phi(T)$ implies $T=\fup{a}{b}T'$.
\end{proof}

Recall that, for $n$ odd, 
\begin{align}
O&\coloneqq \prod_{k=1}^\frac{n-1}{2} e_{2k-1}, &E&\coloneqq \prod_{k=1}^\frac{n-1}{2} e_{2k}.
\end{align}

\pagebreak
\begin{lemma}[Cases (c), (d); $n$ odd]\label{lem:oddlongbdylink}
Let $n$ be odd, and let $T$ be an even-reduced word in $\A_n(X)$.
\begin{enumerate}[label=(\roman*)]
    \item If $\phi(T)$ has a top boundary link at $n$ on the left, then $T=LO\fup{a}{b}ET'$, for some $\mathrm{a},\mathrm{b} \in X$, where
    \begin{itemize}
        \item $L$ is a word in the generators $e_i$, $1\leq i\leq n-3$, and $\fup{c}{d}$, $\mathrm{c}, \mathrm{d} \in X$;
        \item $T'$ is a word in the even generators; and
        \item $L$ and $T'$ together contain one less label generator than $T$.
    \end{itemize}
    \item If $\phi(T)$ has a bottom boundary link at $1$ on the left, then $T=LE\fdo{a}{b}OT'$, for some $\mathrm{a},\mathrm{b} \in X$, where
    \begin{itemize}
        \item $L$ is a word in the generators $e_i$, $3\leq i\leq n-1$, and $\fdo{c}{d}$, $\mathrm{c}, \mathrm{d} \in X$;
        \item $T'$ is a word in the even generators; and
        \item $L$ and $T'$ together contain one less label generator than $T$.
    \end{itemize}
\end{enumerate}
\end{lemma}
\begin{proof}
We prove (i); (ii) is proven analogously by reflecting the diagrams about a horizontal line. For $n=1$, note that $T=\fup{a}{b}$, and $O=E=\id$, so the result holds with $L=T'=\id$. Hereafter assume $n\geq 3$. 

Since $T$ is even-reduced, $\phi(T)$ is an even diagram, by Lemma \ref{lem:evendiag}. Due to parity and planarity considerations, having a top boundary link at $n$ on the left implies that the number of top boundary links on the left side of the diagram is odd, and thus the number of top boundary links on the right side of the diagram is also odd. It follows that there is at least one top boundary link coming from the right side of $\phi(T)$; we will draw the leftmost of these explicitly in the following diagrams. We now consider two cases, according to whether the string at $n-1$ on the left is a top boundary link, or a link.

In the first case, we have
\begin{align}
\phi(T) &= \; \begin{tikzpicture}[baseline={([yshift=-3mm]current bounding box.center)},xscale=0.5,yscale=0.5]
{
\draw (0,-6) ..controls (2.2,-6) and (1.5,-2.5).. (1.5,0.5);
\draw (0,-5.5) ..controls (1.8,-5.5) and (1.1,-2).. (1.1,0.5);
\draw (3,-2) to[out=180,in=270] (1.9,0.5);
\filldraw[red!10!magenta!30] (0,0) arc(-90:0:0.4 and 0.5)--(0.7,0.5) ..controls (0.7,-2) and (1.3,-5).. (0,-5) --cycle;
\filldraw[blue!30!cyan!40] (3,0) arc(270:180:0.4 and 0.5)--(2.3,0.5) to[out=270,in=180] (3,-1.5) --cycle;
\filldraw[orange!40] (3,-2.5) ..controls (1.7,-2.5) and (1.8,-5).. (1.8,-6.5)--(2.5,-6.5) arc(180:90:0.5) --cycle;
\node at (1.1,0.4) [anchor=south] {\footnotesize $\mathrm{c}$};
\node at (1.5,0.4) [anchor=south] {\footnotesize $\mathrm{a}$};
\node at (1.9,0.4) [anchor=south] {\footnotesize $\mathrm{b}$};
\draw[dotted] (0,0.5)--(3,0.5);
\draw[dotted] (0,-6.5)--(3,-6.5);
\draw [very thick](0,-6.5) -- (0,0.5);
\draw [very thick](3,-6.5)--(3,0.5);
}
\end{tikzpicture} 
\; = \;
\begin{tikzpicture}[baseline={([yshift=-3mm]current bounding box.center)},xscale=0.5,yscale=0.5]
{
\filldraw[red!10!magenta!30] (0,0) arc(-90:0:0.4 and 0.5)--(0.7,0.5) ..controls (0.7,-2) and (1.3,-5).. (0,-5) --cycle;
\draw (0,-6)--(3,-6);
\draw (0,-5.5)--(3,-5.5);
\draw (3,0) arc(270:180:0.5);
\node at (2.5,0.4) [anchor=south] {\footnotesize $\mathrm{c}$};
\draw (3,-0.5) arc(90:270:0.3 and 0.25);
\draw (3,-1.5) arc(90:270:0.3 and 0.25);
\draw (3,-3.5) arc(90:270:0.3 and 0.25);
\draw (3,-4.5) arc(90:270:0.3 and 0.25);
\node at (2.7,-2.55) {$\vdots$};
\begin{scope}[shift={(3,0)}]
\draw (0,0) arc(90:-90:0.3 and 0.25);
\draw (0,-1) arc(90:-90:0.3 and 0.25);
\draw (0,-4) arc(90:-90:0.3 and 0.25);
\draw (0,-5) arc(90:-90:0.3 and 0.25);
\draw (3,0) arc(90:270:0.3 and 0.25);
\draw (3,-1) arc(90:270:0.3 and 0.25);
\draw (3,-4) arc(90:270:0.3 and 0.25);
\draw (3,-5) arc(90:270:0.3 and 0.25);
\draw (0,-6)--(3,-6);
\node at (0.3,-2.55) {$\vdots$};
\node at (2.7,-2.55) {$\vdots$};
\end{scope}
\begin{scope}[shift={(6,0)}]
\draw (0,0) arc(-90:0:0.5);
\draw (3,0) arc(270:180:0.5);
\draw (0,-0.5) arc(90:-90:0.3 and 0.25);
\draw (3,-0.5) arc(90:270:0.3 and 0.25);
\draw (0,-1.5) arc(90:-90:0.3 and 0.25);
\draw (3,-1.5) arc(90:270:0.3 and 0.25);
\draw (0,-5.5) arc(90:-90:0.3 and 0.25);
\draw (3,-5.5) arc(90:270:0.3 and 0.25);
\draw (0,-4.5) arc(90:-90:0.3 and 0.25);
\draw (3,-4.5) arc(90:270:0.3 and 0.25);
\node at (0.5,0.4) [anchor=south] {\footnotesize $\mathrm{a}$};
\node at (2.5,0.4) [anchor=south] {\footnotesize $\mathrm{b}$};
\node at (0.3,-3.05) {$\vdots$};
\node at (2.7,-3.05) {$\vdots$};
\end{scope}
\begin{scope}[shift={(9,0)}]
\filldraw[blue!30!cyan!40] (3,0) arc(270:180:0.4 and 0.5)--(2.3,0.5) to[out=270,in=180] (3,-1.5) --cycle;
\filldraw[orange!40] (3,-2.5) ..controls (1.7,-2.5) and (1.8,-5).. (1.8,-6.5)--(2.5,-6.5) arc(180:90:0.5) --cycle;
\node at (0.3,-2.55) {$\vdots$};
\draw (0,-6) to[out=0,in=180] (3,-2);
\draw (0,-5) arc(90:-90:0.3 and 0.25);
\draw (0,-4) arc(90:-90:0.3 and 0.25);
\draw (0,0) arc(90:-90:0.3 and 0.25);
\draw (0,-1) arc(90:-90:0.3 and 0.25);
\end{scope}
\draw[dotted] (0,0.5)--(12,0.5);
\draw[dotted] (0,-6.5)--(12,-6.5);
\draw [very thick](0,-6.5) -- (0,0.5);
\draw [very thick](3,-6.5)--(3,0.5);
\draw [very thick](6,-6.5)--(6,0.5);
\draw [very thick](9,-6.5)--(9,0.5);
\draw [very thick](12,-6.5)--(12,0.5);
}
\end{tikzpicture} \; .
\end{align}
In the second case,
\begin{align}
\phi(T) &= \; \begin{tikzpicture}[baseline={([yshift=-3mm]current bounding box.center)},xscale=0.5,yscale=0.5]
{
\draw (0,-6) ..controls (2,-6) and (1.3,-2.5).. (1.3,0.5);
\draw (0,-3) arc(90:-90:1 and 1.25);
\draw (3,-2) to[out=180,in=270] (1.75,0.5);
\filldraw[red!10!magenta!30] (0,0) arc(-90:0:0.4 and 0.5)--(0.9,0.5) ..controls (0.9,-1) and (1.3,-2.5).. (0,-2.5) --cycle;
\filldraw[blue!30!cyan!40] (3,0) arc(270:180:0.4 and 0.5)--(2.2,0.5) to[out=270,in=180] (3,-1.5) --cycle;
\filldraw[orange!40] (3,-2.5) ..controls (1.7,-2.5) and (1.8,-5).. (1.8,-6.5)--(2.5,-6.5) arc(180:90:0.5) --cycle;
\filldraw[green!90!blue!30] (0,-3.5) arc(90:-90:0.7 and 0.75)--cycle;
\node at (1.3,0.4) [anchor=south] {\footnotesize $\mathrm{a}$};
\node at (1.75,0.4) [anchor=south] {\footnotesize $\mathrm{b}$};
\draw[dotted] (0,0.5)--(3,0.5);
\draw[dotted] (0,-6.5)--(3,-6.5);
\draw [very thick](0,-6.5) -- (0,0.5);
\draw [very thick](3,-6.5)--(3,0.5);
}
\end{tikzpicture} 
\; = \;
\begin{tikzpicture}[baseline={([yshift=-3mm]current bounding box.center)},xscale=0.5,yscale=0.5]
{
\filldraw[red!10!magenta!30] (0,0) arc(-90:0:0.4 and 0.5)--(0.9,0.5) ..controls (0.9,-1) and (1.3,-2.5).. (0,-2.5) --cycle;
\draw (0,-6)--(3,-6);
\draw (0,-3) to[out=0,in=180] (3,0);
\draw (0,-5.5)--(3,-5.5);
\filldraw[green!90!blue!30] (0,-3.5) arc(90:-90:0.7 and 0.75)--cycle;
\draw (3,-0.5) arc(90:270:0.3 and 0.25);
\draw (3,-1.5) arc(90:270:0.3 and 0.25);
\draw (3,-3.5) arc(90:270:0.3 and 0.25);
\draw (3,-4.5) arc(90:270:0.3 and 0.25);
\node at (2.7,-2.55) {$\vdots$};
\begin{scope}[shift={(3,0)}]
\draw (0,0) arc(90:-90:0.3 and 0.25);
\draw (0,-1) arc(90:-90:0.3 and 0.25);
\draw (0,-4) arc(90:-90:0.3 and 0.25);
\draw (0,-5) arc(90:-90:0.3 and 0.25);
\draw (3,0) arc(90:270:0.3 and 0.25);
\draw (3,-1) arc(90:270:0.3 and 0.25);
\draw (3,-4) arc(90:270:0.3 and 0.25);
\draw (3,-5) arc(90:270:0.3 and 0.25);
\draw (0,-6)--(3,-6);
\node at (0.3,-2.55) {$\vdots$};
\node at (2.7,-2.55) {$\vdots$};
\end{scope}
\begin{scope}[shift={(6,0)}]
\draw (0,0) arc(-90:0:0.5);
\draw (3,0) arc(270:180:0.5);
\draw (0,-0.5) arc(90:-90:0.3 and 0.25);
\draw (3,-0.5) arc(90:270:0.3 and 0.25);
\draw (0,-1.5) arc(90:-90:0.3 and 0.25);
\draw (3,-1.5) arc(90:270:0.3 and 0.25);
\draw (0,-5.5) arc(90:-90:0.3 and 0.25);
\draw (3,-5.5) arc(90:270:0.3 and 0.25);
\draw (0,-4.5) arc(90:-90:0.3 and 0.25);
\draw (3,-4.5) arc(90:270:0.3 and 0.25);
\node at (0.5,0.4) [anchor=south] {\footnotesize $\mathrm{a}$};
\node at (2.5,0.4) [anchor=south] {\footnotesize $\mathrm{b}$};
\node at (0.3,-3.05) {$\vdots$};
\node at (2.7,-3.05) {$\vdots$};
\end{scope}
\begin{scope}[shift={(9,0)}]
\draw (3,-2) to[out=180,in=0] (0,-6);
\filldraw[blue!30!cyan!40] (3,0) arc(270:180:0.4 and 0.5)--(2.2,0.5) to[out=270,in=180] (3,-1.5) --cycle;
\filldraw[orange!40] (3,-2.5) ..controls (1.7,-2.5) and (1.8,-5).. (1.8,-6.5)--(2.5,-6.5) arc(180:90:0.5) --cycle;
\draw (0,-5) arc(90:-90:0.3 and 0.25);
\draw (0,-4) arc(90:-90:0.3 and 0.25);
\draw (0,0) arc(90:-90:0.3 and 0.25);
\draw (0,-1) arc(90:-90:0.3 and 0.25);
\node at (0.3,-2.55) {$\vdots$};
\end{scope}
\draw[dotted] (0,0.5)--(12,0.5);
\draw[dotted] (0,-6.5)--(12,-6.5);
\draw [very thick](0,-6.5) -- (0,0.5);
\draw [very thick](3,-6.5)--(3,0.5);
\draw [very thick](6,-6.5)--(6,0.5);
\draw [very thick](9,-6.5)--(9,0.5);
\draw [very thick](12,-6.5)--(12,0.5);
}
\end{tikzpicture}\; .
\end{align}
In each of these products, the first diagram has throughlines connecting $n-1$ and $n$ on the left, to $n-1$ and $n$ on the right, respectively, and thus it is generated by $\de{i}$, $1\leq i \leq n-3$ and $\dfup{e}{f}$, $\mathrm{e}, \mathrm{f} \in X$. One can deduce this from the construction in Proposition \ref{prop:evengen}. Hence the first diagram is $\phi(L)$ for some word $L$ in the generators $e_i$, $1\leq i \leq n-3$ and $\fup{e}{f}$, $\mathrm{e}, \mathrm{f} \in X$. The construction from Proposition \ref{prop:evengen} does not produce any arcs, so the number of boundary links in $\phi(L)$ is double the number of label generators in $L$. The second diagram is $\phi(O)$, and the third diagram is $\phi(\fup{a}{b}E)$; these have zero and two boundary links, respectively, which is double the number of label generators in $O$ and $\fup{a}{b}E$, respectively. The fourth diagram is even, and is therefore equal to $\phi(T')$ for some even-reduced word $T'$, by Corollary \ref{cor:evenredsurj}. From Corollary \ref{cor:evenredsurj}, we also have that the number of boundary links in $\phi(T')$ is equal to double the number of label generators in $T'$. 

Thus, $\phi(T)=\phi(LO\fup{a}{b}ET')$. Since the above products of diagrams produce no boundary arcs, we have that the number of boundary arcs in $\phi(T)=\phi(LO\fup{a}{b}ET')$ is the total number of boundary links in $\phi(L)$, $\phi(O)$, $\phi(\fup{a}{b}E)$ and $\phi(T')$. In each case this is exactly double the number of label generators in the argument of $\phi$, so the number of boundary links in $\phi(T)=\phi(LO\fup{a}{b}ET')$ is double the number of label generators in $LO\fup{a}{b}ET'$. Hence $LO\fup{a}{b}ET'$ is label-reduced, by Lemma \ref{lem:labelreducedbdylinks}. Since this is a label-reduced word in the even generators, it follows that it is even. As $T$ is also even, and $\phi(T)=\phi(LO\fup{a}{b}ET')$, we have $T=LO\fup{a}{b}ET'$ by Proposition \ref{prop:evenwordinj}. Then since $T=LO\fup{a}{b}ET'$, and both $T$ and $LO\fup{a}{b}ET'$ are label-reduced, $T$ and $LO\fup{a}{b}ET'$ have the same number of label generators. As $O$ and $E$ each contain no label generators, and $\fup{a}{b}$ is a label generator, it follows that $L$ and $T'$ together contain one less label generator than $T$. This proves (i).
\end{proof}

Recall that, for $n$ even,
\begin{align}
\Theta &\coloneqq  \prod_{k=1}^{\frac{n}{2}} e_{2k-1}, &\Omega &\coloneqq  \prod_{k=1}^{\frac{n}{2}-1} e_{2k}.
\end{align}

\begin{lemma}[Cases (c), (d); $n$ even]\label{lem:evenlongbdylink}
Let $n$ be even, and let $T$ be an even-reduced word in $\A_n(X)$.
\begin{enumerate}[label=(\roman*)]
    \item If $\phi(T)$ has a top boundary link at $n$ on the left, then $T=L\fup{a}{b} \Omega \Theta T'$, for some $\mathrm{a},\mathrm{b} \in X$, where
    \begin{itemize}
        \item $L$ is a word in the generators $e_i$, $1\leq i \leq n-2$, and $\fup{c}{d}$, $\mathrm{c},\mathrm{d} \in X$;
        \item $T'$ is a word in the even generators; and
        \item $L$ and $T'$ together contain one less label generator than $T$.
    \end{itemize}
    \item If $\phi(T)$ has a bottom boundary link at $1$ on the left, then $T=L\fdo{a}{b} \Omega \Theta T'$, for some $\mathrm{a},\mathrm{b} \in X$, where
    \begin{itemize}
        \item $L$ is a word in the generators $e_i$, $2\leq i \leq n-1$, and $\fdo{c}{d}$, $\mathrm{c},\mathrm{d} \in X$;
        \item $T'$ is a word in the even generators; and
        \item $L$ and $T'$ together contain one less label generator than $T$.
    \end{itemize}
\end{enumerate}
\end{lemma}
\begin{proof}
We prove (i); (ii) is proven analogously by reflecting the diagrams about a horizontal line. We thus let $T$ be an even-reduced word in $\A_n(X)$ such that $\phi(T)$ has a top boundary link at $n$ on the left. As $T$ is even-reduced, we have from Lemma \ref{lem:evendiag} that $\phi(T)$ is an even diagram. From parity and planarity considerations, given the top boundary link at $n$ on the left, $\phi(T)$ must have an even number of top boundary links coming from its left side. Hence there is at least one more top boundary link coming from the left, in addition to the one at $n$; we will draw the rightmost of these explicitly in the following diagrams.

The right side of the diagram must have at least one string exposed to the left; we consider three cases, using an exposed link, and the leftmost top or bottom boundary link.

In the first case, with an exposed link on the right, we have
\begin{align}
\phi(T) &= \; \begin{tikzpicture}[baseline={([yshift=-3mm]current bounding box.center)},xscale=0.5,yscale=0.5]
{
\draw (0,-6) ..controls (2,-6) and (1.75,-2).. (1.75,0.5);
\draw (0,-3) to[out=0,in=-90] (1.3,0.5);
\draw (3,-2) arc(90:270:1 and 1.25);
\filldraw[red!10!magenta!30] (0,0) arc(-90:0:0.4 and 0.5)--(0.9,0.5) ..controls (0.9,-1) and (0.9,-2.5).. (0,-2.5) --cycle;
\filldraw[blue!30!cyan!40] (3,0) arc(270:180:0.4 and 0.5)--(2.2,0.5) to[out=270,in=180] (3,-1.5) --cycle;
\filldraw[orange!40] (3,-5) ..controls (2,-5) and (2,-5.9).. (2,-6.5)--(2.6,-6.5) arc(180:90:0.4 and 0.5) --cycle;
\filldraw[green!90!blue!30] (0,-3.5) arc(90:-90:0.8 and 1)--cycle;
\filldraw[gray!40] (3,-2.5) arc(90:270:0.65 and 0.75)--cycle;
\node at (1.3,0.4) [anchor=south] {\footnotesize $\mathrm{a}$};
\node at (1.75,0.4) [anchor=south] {\footnotesize $\mathrm{b}$};
\draw[dotted] (0,0.5)--(3,0.5);
\draw[dotted] (0,-6.5)--(3,-6.5);
\draw [very thick](0,-6.5) -- (0,0.5);
\draw [very thick](3,-6.5)--(3,0.5);
}
\end{tikzpicture} 
\; = \; 
\begin{tikzpicture}[baseline={([yshift=-3mm]current bounding box.center)},xscale=0.5,yscale=0.5]
{
\draw (0,-6)--(3,-6);
\filldraw[red!10!magenta!30] (0,0) arc(-90:0:0.4 and 0.5)--(0.9,0.5) ..controls (0.9,-1) and (0.9,-2.5).. (0,-2.5) --cycle;
\filldraw[green!90!blue!30] (0,-3.5) arc(90:-90:0.8 and 1)--cycle;
\draw (0,-3) to[out=0,in=180] (3,-5.5);
\draw (3,0) arc(90:270:0.3 and 0.25);
\draw (3,-1) arc(90:270:0.3 and 0.25);
\draw (3,-3.5) arc(90:270:0.3 and 0.25);
\draw (3,-4.5) arc(90:270:0.3 and 0.25);
\node at (2.7,-2.3) {$\vdots$};
\begin{scope}[shift={(3,0)}]
\draw (0,-6)--(3,-6);
\draw (0,0) arc(-90:0:0.5);
\draw (3,0) arc(270:180:0.5);
\draw (0,-0.5) arc(90:-90:0.3 and 0.25);
\draw (0,-1.5) arc(90:-90:0.3 and 0.25);
\draw (0,-4) arc(90:-90:0.3 and 0.25);
\draw (0,-5) arc(90:-90:0.3 and 0.25);
\draw (3,-0.5) arc(90:270:0.3 and 0.25);
\draw (3,-1.5) arc(90:270:0.3 and 0.25);
\draw (3,-4) arc(90:270:0.3 and 0.25);
\draw (3,-5) arc(90:270:0.3 and 0.25);
\node at (0.5,0.4) [anchor=south] {\footnotesize $\mathrm{a}$};
\node at (2.5,0.4) [anchor=south] {\footnotesize $\mathrm{b}$};
\node at (0.3,-2.8) {$\vdots$};
\node at (2.7,-2.8) {$\vdots$};
\end{scope}
\begin{scope}[shift={(6,0)}]
\draw (0,0) arc(90:-90:0.3 and 0.25);
\draw (0,-1) arc(90:-90:0.3 and 0.25);
\draw (0,-4.5) arc(90:-90:0.3 and 0.25);
\draw (0,-5.5) arc(90:-90:0.3 and 0.25);
\draw (3,0) arc(90:270:0.3 and 0.25);
\draw (3,-1) arc(90:270:0.3 and 0.25);
\draw (3,-4.5) arc(90:270:0.3 and 0.25);
\draw (3,-5.5) arc(90:270:0.3 and 0.25);
\node at (0.3,-2.8) {$\vdots$};
\node at (2.7,-2.8) {$\vdots$};
\end{scope}
\begin{scope}[shift={(9,0)}]
\filldraw[blue!30!cyan!40] (3,0) arc(270:180:0.4 and 0.5)--(2.2,0.5) to[out=270,in=180] (3,-1.5) --cycle;
\filldraw[orange!40] (3,-5) ..controls (2,-5) and (2,-5.9).. (2,-6.5)--(2.6,-6.5) arc(180:90:0.4 and 0.5) --cycle;
\filldraw[gray!40] (3,-2.5) arc(90:270:0.65 and 0.75)--cycle;
\draw (3,-2) to[out=180,in=0] (0,0);
\draw (3,-4.5) to[out=180,in=0] (0,-6);
\draw (0,-0.5) arc(90:-90:0.3 and 0.25);
\draw (0,-1.5) arc(90:-90:0.3 and 0.25);
\draw (0,-4) arc(90:-90:0.3 and 0.25);
\draw (0,-5) arc(90:-90:0.3 and 0.25);
\node at (0.3,-2.8) {$\vdots$};
\end{scope}
\draw[dotted] (0,0.5)--(12,0.5);
\draw[dotted] (0,-6.5)--(12,-6.5);
\draw [very thick](0,-6.5) -- (0,0.5);
\draw [very thick](3,-6.5)--(3,0.5);
\draw [very thick](6,-6.5)--(6,0.5);
\draw [very thick](9,-6.5)--(9,0.5);
\draw [very thick](12,-6.5)--(12,0.5);
}
\end{tikzpicture} \; .
\end{align}
In the second case, with a top boundary link on the right, we have
\begin{align}
\phi(T) &= \; \begin{tikzpicture}[baseline={([yshift=-3mm]current bounding box.center)},xscale=0.5,yscale=0.5]
{
\draw (0,-6) ..controls (1.7,-6) and (1.5,-2).. (1.5,0.5);
\draw (0,-3) to[out=0,in=-90] (1.1,0.5);
\draw (3,-2) to[out=180,in=270] (1.9,0.5);
\filldraw[red!10!magenta!30] (0,0) arc(-90:0:0.4 and 0.5)--(0.7,0.5) ..controls (0.7,-1) and (0.9,-2.5).. (0,-2.5) --cycle;
\filldraw[green!90!blue!30] (0,-3.5) arc(90:-90:0.8 and 1)--cycle;
\filldraw[blue!30!cyan!40] (3,0) arc(270:180:0.4 and 0.5)--(2.3,0.5) to[out=270,in=180] (3,-1.5) --cycle;
\filldraw[orange!40] (3,-2.5) ..controls (1.7,-2.5) and (1.8,-5).. (1.8,-6.5)--(2.5,-6.5) arc(180:90:0.5) --cycle;
\node at (1.1,0.4) [anchor=south] {\footnotesize $\mathrm{a}$};
\node at (1.5,0.4) [anchor=south] {\footnotesize $\mathrm{b}$};
\node at (1.9,0.4) [anchor=south] {\footnotesize $\mathrm{c}$};
\draw[dotted] (0,0.5)--(3,0.5);
\draw[dotted] (0,-6.5)--(3,-6.5);
\draw [very thick](0,-6.5) -- (0,0.5);
\draw [very thick](3,-6.5)--(3,0.5);
}
\end{tikzpicture} 
\; = \; 
\begin{tikzpicture}[baseline={([yshift=-3mm]current bounding box.center)},xscale=0.5,yscale=0.5]
{
\draw (0,-6)--(3,-6);
\filldraw[red!10!magenta!30] (0,0) arc(-90:0:0.4 and 0.5)--(0.7,0.5) ..controls (0.7,-1) and (0.9,-2.5).. (0,-2.5) --cycle;
\filldraw[green!90!blue!30] (0,-3.5) arc(90:-90:0.8 and 1)--cycle;
\draw (0,-3) to[out=0,in=180] (3,-5.5);
\draw (3,0) arc(90:270:0.3 and 0.25);
\draw (3,-1) arc(90:270:0.3 and 0.25);
\draw (3,-3.5) arc(90:270:0.3 and 0.25);
\draw (3,-4.5) arc(90:270:0.3 and 0.25);
\node at (2.7,-2.3) {$\vdots$};
\begin{scope}[shift={(3,0)}]
\draw (0,-6)--(3,-6);
\draw (0,0) arc(-90:0:0.5);
\draw (3,0) arc(270:180:0.5);
\draw (0,-0.5) arc(90:-90:0.3 and 0.25);
\draw (0,-1.5) arc(90:-90:0.3 and 0.25);
\draw (0,-4) arc(90:-90:0.3 and 0.25);
\draw (0,-5) arc(90:-90:0.3 and 0.25);
\draw (3,-0.5) arc(90:270:0.3 and 0.25);
\draw (3,-1.5) arc(90:270:0.3 and 0.25);
\draw (3,-4) arc(90:270:0.3 and 0.25);
\draw (3,-5) arc(90:270:0.3 and 0.25);
\node at (0.5,0.4) [anchor=south] {\footnotesize $\mathrm{a}$};
\node at (2.5,0.4) [anchor=south] {\footnotesize $\mathrm{b}$};
\node at (0.3,-2.8) {$\vdots$};
\node at (2.7,-2.8) {$\vdots$};
\end{scope}
\begin{scope}[shift={(6,0)}]
\draw (0,0) arc(90:-90:0.3 and 0.25);
\draw (0,-1) arc(90:-90:0.3 and 0.25);
\draw (0,-4.5) arc(90:-90:0.3 and 0.25);
\draw (0,-5.5) arc(90:-90:0.3 and 0.25);
\draw (3,0) arc(90:270:0.3 and 0.25);
\draw (3,-1) arc(90:270:0.3 and 0.25);
\draw (3,-4.5) arc(90:270:0.3 and 0.25);
\draw (3,-5.5) arc(90:270:0.3 and 0.25);
\node at (0.3,-2.8) {$\vdots$};
\node at (2.7,-2.8) {$\vdots$};
\end{scope}
\begin{scope}[shift={(9,0)}]
\filldraw[blue!30!cyan!40] (3,0) arc(270:180:0.4 and 0.5)--(2.3,0.5) to[out=270,in=180] (3,-1.5) --cycle;
\filldraw[orange!40] (3,-2.5) ..controls (1.7,-2.5) and (1.8,-5).. (1.8,-6.5)--(2.5,-6.5) arc(180:90:0.5) --cycle;
\draw (3,-2) to[out=180,in=0] (0,-6);
\draw (0,0) arc(-90:0:0.5);
\draw (0,-0.5) arc(90:-90:0.3 and 0.25);
\draw (0,-1.5) arc(90:-90:0.3 and 0.25);
\draw (0,-4) arc(90:-90:0.3 and 0.25);
\draw (0,-5) arc(90:-90:0.3 and 0.25);
\node at (0.3,-2.8) {$\vdots$};
\node at (0.5,0.4) [anchor=south] {\footnotesize $\mathrm{c}$};
\end{scope}
\draw[dotted] (0,0.5)--(12,0.5);
\draw[dotted] (0,-6.5)--(12,-6.5);
\draw [very thick](0,-6.5) -- (0,0.5);
\draw [very thick](3,-6.5)--(3,0.5);
\draw [very thick](6,-6.5)--(6,0.5);
\draw [very thick](9,-6.5)--(9,0.5);
\draw [very thick](12,-6.5)--(12,0.5);
}
\end{tikzpicture} \; .
\end{align}
In the third case, with a bottom boundary link on the right, we have
\begin{align}
\phi(T) &= \; \begin{tikzpicture}[baseline={([yshift=-2mm]current bounding box.center)},xscale=0.5,yscale=0.5]
{
\draw (0,-6) ..controls (1.7,-6) and (1.5,-1).. (1.5,0.5);
\draw (0,-3) to[out=0,in=-90] (1.1,0.5);
\draw (3,-2) ..controls (1.5,-2) and (1.5,-5).. (1.5,-6.5);
\filldraw[red!10!magenta!30] (0,0) arc(-90:0:0.4 and 0.5)--(0.7,0.5) ..controls (0.7,-1) and (0.9,-2.5).. (0,-2.5) --cycle;
\filldraw[green!90!blue!30] (0,-3.5) arc(90:-90:0.8 and 1)--cycle;
\filldraw[blue!30!cyan!40] (3,0) arc(270:180:0.4 and 0.5)--(2,0.5) to[out=270,in=180] (3,-1.5) --cycle;
\filldraw[orange!40] (3,-2.5) ..controls (1.9,-2.5) and (1.9,-5).. (1.9,-6.5)--(2.5,-6.5) arc(180:90:0.5) --cycle;
\node at (1.1,0.4) [anchor=south] {\footnotesize $\mathrm{a}$};
\node at (1.5,0.4) [anchor=south] {\footnotesize $\mathrm{b}$};
\node at (1.5,-7.3) [anchor=south] {\footnotesize $\mathrm{c}$};
\draw[dotted] (0,0.5)--(3,0.5);
\draw[dotted] (0,-6.5)--(3,-6.5);
\draw [very thick](0,-6.5) -- (0,0.5);
\draw [very thick](3,-6.5)--(3,0.5);
}
\end{tikzpicture} 
\; = \; 
\begin{tikzpicture}[baseline={([yshift=-2mm]current bounding box.center)},xscale=0.5,yscale=0.5]
{
\draw (0,-6)--(3,-6);
\filldraw[red!10!magenta!30] (0,0) arc(-90:0:0.4 and 0.5)--(0.7,0.5) ..controls (0.7,-1) and (0.9,-2.5).. (0,-2.5) --cycle;
\filldraw[green!90!blue!30] (0,-3.5) arc(90:-90:0.8 and 1)--cycle;
\draw (0,-3) to[out=0,in=180] (3,-5.5);
\draw (3,0) arc(90:270:0.3 and 0.25);
\draw (3,-1) arc(90:270:0.3 and 0.25);
\draw (3,-3.5) arc(90:270:0.3 and 0.25);
\draw (3,-4.5) arc(90:270:0.3 and 0.25);
\node at (2.7,-2.3) {$\vdots$};
\begin{scope}[shift={(3,0)}]
\draw (0,-6)--(3,-6);
\draw (0,0) arc(-90:0:0.5);
\draw (3,0) arc(270:180:0.5);
\draw (0,-0.5) arc(90:-90:0.3 and 0.25);
\draw (0,-1.5) arc(90:-90:0.3 and 0.25);
\draw (0,-4) arc(90:-90:0.3 and 0.25);
\draw (0,-5) arc(90:-90:0.3 and 0.25);
\draw (3,-0.5) arc(90:270:0.3 and 0.25);
\draw (3,-1.5) arc(90:270:0.3 and 0.25);
\draw (3,-4) arc(90:270:0.3 and 0.25);
\draw (3,-5) arc(90:270:0.3 and 0.25);
\node at (0.5,0.4) [anchor=south] {\footnotesize $\mathrm{a}$};
\node at (2.5,0.4) [anchor=south] {\footnotesize $\mathrm{b}$};
\node at (0.3,-2.8) {$\vdots$};
\node at (2.7,-2.8) {$\vdots$};
\end{scope}
\begin{scope}[shift={(6,0)}]
\draw (0,0) arc(90:-90:0.3 and 0.25);
\draw (0,-1) arc(90:-90:0.3 and 0.25);
\draw (0,-4.5) arc(90:-90:0.3 and 0.25);
\draw (0,-5.5) arc(90:-90:0.3 and 0.25);
\draw (3,0) arc(90:270:0.3 and 0.25);
\draw (3,-1) arc(90:270:0.3 and 0.25);
\draw (3,-4.5) arc(90:270:0.3 and 0.25);
\draw (3,-5.5) arc(90:270:0.3 and 0.25);
\node at (0.3,-2.8) {$\vdots$};
\node at (2.7,-2.8) {$\vdots$};
\end{scope}
\begin{scope}[shift={(9,0)}]
\filldraw[blue!30!cyan!40] (3,0) arc(270:180:0.4 and 0.5)--(2,0.5) to[out=270,in=180] (3,-1.5) --cycle;
\filldraw[orange!40] (3,-2.5) ..controls (1.9,-2.5) and (1.9,-5).. (1.9,-6.5)--(2.5,-6.5) arc(180:90:0.5) --cycle;
\draw (3,-2) to[out=180,in=0] (0,0);
\draw (0,-6) arc(90:0:0.5);
\draw (0,-0.5) arc(90:-90:0.3 and 0.25);
\draw (0,-1.5) arc(90:-90:0.3 and 0.25);
\draw (0,-4) arc(90:-90:0.3 and 0.25);
\draw (0,-5) arc(90:-90:0.3 and 0.25);
\node at (0.3,-2.8) {$\vdots$};
\node at (0.5,-7.3) [anchor=south] {\footnotesize $\mathrm{c}$};
\end{scope}
\draw[dotted] (0,0.5)--(12,0.5);
\draw[dotted] (0,-6.5)--(12,-6.5);
\draw [very thick](0,-6.5) -- (0,0.5);
\draw [very thick](3,-6.5)--(3,0.5);
\draw [very thick](6,-6.5)--(6,0.5);
\draw [very thick](9,-6.5)--(9,0.5);
\draw [very thick](12,-6.5)--(12,0.5);
}
\end{tikzpicture} \; .
\end{align}
In each of these products, the first diagram is even, and has a string connecting $n$ on the left to $n$ on the right. Using the construction from Proposition \ref{prop:evengen}, then, this is equal to a product of the diagrammatic generators $\de{i}$, $1\leq i \leq n-2$, and $\dfup{e}{f}$, $\mathrm{e}, \mathrm{f} \in X$, and this product does not produce any boundary arcs. Since these diagrammatic generators are the $\phi$-images of $e_i$, $1\leq i \leq n-2$, and $\fup{e}{f}$, $\mathrm{e}, \mathrm{f} \in X$, respectively, it follows that the first diagram is $\phi(L)$, where $L$ is the corresponding word in these $\A_n(X)$ generators. Moreover, because the diagrammatic product produces no boundary arcs, the number of boundary links in $\phi(L)$ is precisely double the number of label generators in $L$. The second diagram is $\phi(\fup{a}{b} \Omega)$, and the third diagram is $\phi(\Theta)$; these diagrams contain two and zero boundary links, respectively, which is double the number of label generators in each of these words. The fourth diagram is also even, and thus by Corollary \ref{cor:evenredsurj}, it is equal $\phi(T')$ for some even-reduced word $T'$. Also from Corollary \ref{cor:evenredsurj}, the number of boundary links in $\phi(T')$ is double the number of label generators in $T'$.

Thus, $\phi(T)=\phi(L\fup{a}{b}\Omega\Theta T')$. Since the above products of diagrams produce no boundary arcs, we have that the number of boundary links in $\phi(T)=\phi(L\fup{a}{b}\Omega\Theta T')$ is the total number of boundary links in $\phi(L)$, $\phi(\fup{a}{b}\Omega)$, $\phi(\Theta)$ and $\phi(T')$. In each case this is exactly double the number of label generators in the argument of $\phi$, so the number of boundary links in $\phi(T)=\phi(L\fup{a}{b}\Omega\Theta T')$ is double the number of label generators in $L\fup{a}{b}\Omega\Theta T'$. Hence $L\fup{a}{b} \Omega\Theta T'$ is label-reduced, by Lemma \ref{lem:labelreducedbdylinks}. Since this is a label-reduced word in the even generators, it follows that it is even. As $T$ is also even, and $\phi(T)=\phi(L\fup{a}{b}\Omega\Theta T')$, we have $T=L\fup{a}{b}\Omega \Theta T'$ by Proposition \ref{prop:evenwordinj}. Then since $T=L\fup{a}{b}\Omega \Theta T'$, and both $T$ and $L\fup{a}{b}\Omega\Theta T'$ are label-reduced, $T$ and $L\fup{a}{b}\Omega\Theta T'$ have the same number of label generators. As $\Omega$ and $\Theta$ each contain no label generators, and $\fup{a}{b}$ is a label generator, it follows that $L$ and $T'$ together contain one less label generator than $T$. This proves (i).
\end{proof}

\begin{lemma}[Case (e)]\label{lem:evenlonglink}
Let $T$ be an even-reduced word in $\A_n(X)$ such that $\phi(T)$ has a link connecting node $1$ to node $n$ on the left side of the diagram. Then $T=L\Omega\Theta T'$, where 
\begin{itemize}
    \item $L$ is a word in the generators $e_i$, $2\leq i \leq n-2$;
    \item $T'$ is a word in the even generators; and 
    \item $T'$ contains the same number of label generators as $T$.
\end{itemize}
\end{lemma}
\begin{proof}
Since $\phi(T)$ has a link connecting node $1$ to node $n$ on the same side of the diagram, $n$ must be even. As $T$ is even-reduced, we have from Lemma \ref{lem:evendiag} that $\phi(T)$ is an even diagram. If $n=2$, then $\phi(T)$ is one of $\de{1}$, $\de{1}\dfup{a}{b}$, or $\de{1}\dfdo{a}{b}$, for some $\mathrm{a},\mathrm{b} \in X$. These diagrams are $\phi(e_1)$, $\phi(e_1 \fup{a}{b})$ and $\phi(e_1\fdo{a}{b})$, respectively. Taking $L=\id$, and $T'=\id$, $\fup{a}{b}$ or $\fdo{a}{b}$ in these cases respectively, the three bullet points are satisfied. Since the words $e_1$, $e_1 \fup{a}{b}$ and $e_1\fdo{a}{b}$ are all even, and $\phi(T)=\phi(L\Omega\Theta T')$, we have $T=L\Omega\Theta T'$ by Lemma \ref{prop:evenwordinj}. Hereafter, take $n \geq 4$.

As $\phi(T)$ has a link from $1$ to $n$ on the left, it is not $\did$, so must have at least one link or boundary link on its right side, and at least one of these must be exposed to the left side of the diagram. We now consider three cases, according to whether $\phi(T)$ has a link, top boundary link or bottom boundary link on its right side, that is exposed to the left side.

In the first case, with an exposed link on the right, we have
\begin{align}
\phi(T) &= \; \begin{tikzpicture}[baseline={([yshift=-1mm]current bounding box.center)},xscale=0.5,yscale=0.5]
{
\draw (0,0) ..controls (2,0) and (2,-6).. (0,-6);
\draw (3,-2) arc(90:270:1 and 1.25);
\filldraw[red!10!magenta!30] (0,-0.5) arc(90:-90:0.5)--cycle;
\filldraw[blue!30!cyan!40] (3,0) arc(270:180:0.4 and 0.5)--(2.2,0.5) to[out=270,in=180] (3,-1.5) --cycle;
\filldraw[yellow!90!orange!80] (0,-2.5) arc(90:-90:0.5)--cycle;
\draw (0,-2) arc(90:-90:1);
\filldraw[orange!40] (3,-5) ..controls (2,-5) and (2,-5.9).. (2,-6.5)--(2.6,-6.5) arc(180:90:0.4 and 0.5) --cycle;
\filldraw[green!90!blue!30] (0,-4.5) arc(90:-90:0.5)--cycle;
\filldraw[gray!40] (3,-2.5) arc(90:270:0.65 and 0.75)--cycle;
\draw[dotted] (0,0.5)--(3,0.5);
\draw[dotted] (0,-6.5)--(3,-6.5);
\draw [very thick](0,-6.5) -- (0,0.5);
\draw [very thick](3,-6.5)--(3,0.5);
}
\end{tikzpicture} 
\; = \; 
\begin{tikzpicture}[baseline={([yshift=-1mm]current bounding box.center)},xscale=0.5,yscale=0.5]
{
\draw (0,0)--(3,0);
\draw (0,-6)--(3,-6);
\filldraw[red!10!magenta!30] (0,-0.5) arc(90:-90:0.5)--cycle;
\filldraw[yellow!90!orange!80] (0,-2.5) arc(90:-90:0.5)--cycle;
\draw (0,-2) to[out=0,in=180] (3,-0.5);
\draw (0,-4) to[out=0,in=180] (3,-5.5);
\filldraw[green!90!blue!30] (0,-4.5) arc(90:-90:0.5)--cycle;
\draw (3,-1) arc(90:270:0.3 and 0.25);
\draw (3,-3.5) arc(90:270:0.3 and 0.25);
\draw (3,-4.5) arc(90:270:0.3 and 0.25);
\node at (2.7,-2.3) {$\vdots$};
\begin{scope}[shift={(3,0)}]
\draw (0,-6)--(3,-6);
\draw (0,0)--(3,0);
\draw (0,-0.5) arc(90:-90:0.3 and 0.25);
\draw (0,-1.5) arc(90:-90:0.3 and 0.25);
\draw (0,-4) arc(90:-90:0.3 and 0.25);
\draw (0,-5) arc(90:-90:0.3 and 0.25);
\draw (3,-0.5) arc(90:270:0.3 and 0.25);
\draw (3,-1.5) arc(90:270:0.3 and 0.25);
\draw (3,-4) arc(90:270:0.3 and 0.25);
\draw (3,-5) arc(90:270:0.3 and 0.25);
\node at (0.3,-2.8) {$\vdots$};
\node at (2.7,-2.8) {$\vdots$};
\end{scope}
\begin{scope}[shift={(6,0)}]
\draw (0,0) arc(90:-90:0.3 and 0.25);
\draw (0,-1) arc(90:-90:0.3 and 0.25);
\draw (0,-4.5) arc(90:-90:0.3 and 0.25);
\draw (0,-5.5) arc(90:-90:0.3 and 0.25);
\draw (3,0) arc(90:270:0.3 and 0.25);
\draw (3,-1) arc(90:270:0.3 and 0.25);
\draw (3,-4.5) arc(90:270:0.3 and 0.25);
\draw (3,-5.5) arc(90:270:0.3 and 0.25);
\node at (0.3,-2.8) {$\vdots$};
\node at (2.7,-2.8) {$\vdots$};
\end{scope}
\begin{scope}[shift={(9,0)}]
\filldraw[blue!30!cyan!40] (3,0) arc(270:180:0.4 and 0.5)--(2.2,0.5) to[out=270,in=180] (3,-1.5) --cycle;
\filldraw[orange!40] (3,-5) ..controls (2,-5) and (2,-5.9).. (2,-6.5)--(2.6,-6.5) arc(180:90:0.4 and 0.5) --cycle;
\filldraw[gray!40] (3,-2.5) arc(90:270:0.65 and 0.75)--cycle;
\draw (3,-2) to[out=180,in=0] (0,0);
\draw (3,-4.5) to[out=180,in=0] (0,-6);
\draw (0,-0.5) arc(90:-90:0.3 and 0.25);
\draw (0,-1.5) arc(90:-90:0.3 and 0.25);
\draw (0,-4) arc(90:-90:0.3 and 0.25);
\draw (0,-5) arc(90:-90:0.3 and 0.25);
\node at (0.3,-2.8) {$\vdots$};
\end{scope}
\draw[dotted] (0,0.5)--(12,0.5);
\draw[dotted] (0,-6.5)--(12,-6.5);
\draw [very thick](0,-6.5) -- (0,0.5);
\draw [very thick](3,-6.5)--(3,0.5);
\draw [very thick](6,-6.5)--(6,0.5);
\draw [very thick](9,-6.5)--(9,0.5);
\draw [very thick](12,-6.5)--(12,0.5);
}
\end{tikzpicture} \; .
\end{align}
In the second case, with an exposed top boundary link on the right,
\begin{align}
\phi(T) &= \; \begin{tikzpicture}[baseline={([yshift=-3mm]current bounding box.center)},xscale=0.5,yscale=0.5]
{
\draw (0,0) ..controls (2,0) and (2,-6).. (0,-6);
\filldraw[red!10!magenta!30] (0,-0.5) arc(90:-90:0.5)--cycle;
\filldraw[yellow!90!orange!80] (0,-2.5) arc(90:-90:0.5)--cycle;
\draw (0,-2) arc(90:-90:1);
\filldraw[green!90!blue!30] (0,-4.5) arc(90:-90:0.5)--cycle;
\filldraw[gray!40] (3,-2.5) arc(90:270:0.65 and 0.75)--cycle;
\draw (3,-2) to[out=180,in=270] (1.9,0.5);
\filldraw[blue!30!cyan!40] (3,0) arc(270:180:0.4 and 0.5)--(2.3,0.5) to[out=270,in=180] (3,-1.5) --cycle;
\filldraw[orange!40] (3,-2.5) ..controls (1.7,-2.5) and (1.8,-5).. (1.8,-6.5)--(2.5,-6.5) arc(180:90:0.5) --cycle;
\node at (1.9,0.4) [anchor=south] {\footnotesize $\mathrm{a}$};
\draw[dotted] (0,0.5)--(3,0.5);
\draw[dotted] (0,-6.5)--(3,-6.5);
\draw [very thick](0,-6.5) -- (0,0.5);
\draw [very thick](3,-6.5)--(3,0.5);
}
\end{tikzpicture} 
\; = \; 
\begin{tikzpicture}[baseline={([yshift=-3mm]current bounding box.center)},xscale=0.5,yscale=0.5]
{
\draw (0,0)--(3,0);
\draw (0,-6)--(3,-6);
\filldraw[red!10!magenta!30] (0,-0.5) arc(90:-90:0.5)--cycle;
\filldraw[yellow!90!orange!80] (0,-2.5) arc(90:-90:0.5)--cycle;
\draw (0,-2) to[out=0,in=180] (3,-0.5);
\draw (0,-4) to[out=0,in=180] (3,-5.5);
\filldraw[green!90!blue!30] (0,-4.5) arc(90:-90:0.5)--cycle;
\draw (3,-1) arc(90:270:0.3 and 0.25);
\draw (3,-3.5) arc(90:270:0.3 and 0.25);
\draw (3,-4.5) arc(90:270:0.3 and 0.25);
\node at (2.7,-2.3) {$\vdots$};
\begin{scope}[shift={(3,0)}]
\draw (0,-6)--(3,-6);
\draw (0,0)--(3,0);
\draw (0,-0.5) arc(90:-90:0.3 and 0.25);
\draw (0,-1.5) arc(90:-90:0.3 and 0.25);
\draw (0,-4) arc(90:-90:0.3 and 0.25);
\draw (0,-5) arc(90:-90:0.3 and 0.25);
\draw (3,-0.5) arc(90:270:0.3 and 0.25);
\draw (3,-1.5) arc(90:270:0.3 and 0.25);
\draw (3,-4) arc(90:270:0.3 and 0.25);
\draw (3,-5) arc(90:270:0.3 and 0.25);
\node at (0.3,-2.8) {$\vdots$};
\node at (2.7,-2.8) {$\vdots$};
\end{scope}
\begin{scope}[shift={(6,0)}]
\draw (0,0) arc(90:-90:0.3 and 0.25);
\draw (0,-1) arc(90:-90:0.3 and 0.25);
\draw (0,-4.5) arc(90:-90:0.3 and 0.25);
\draw (0,-5.5) arc(90:-90:0.3 and 0.25);
\draw (3,0) arc(90:270:0.3 and 0.25);
\draw (3,-1) arc(90:270:0.3 and 0.25);
\draw (3,-4.5) arc(90:270:0.3 and 0.25);
\draw (3,-5.5) arc(90:270:0.3 and 0.25);
\node at (0.3,-2.8) {$\vdots$};
\node at (2.7,-2.8) {$\vdots$};
\end{scope}
\begin{scope}[shift={(9,0)}]
\draw (0,0) arc(-90:0:0.5);
\filldraw[blue!30!cyan!40] (3,0) arc(270:180:0.4 and 0.5)--(2.3,0.5) to[out=270,in=180] (3,-1.5) --cycle;
\filldraw[orange!40] (3,-2.5) ..controls (1.7,-2.5) and (1.8,-5).. (1.8,-6.5)--(2.5,-6.5) arc(180:90:0.5) --cycle;
\draw (3,-2) to[out=180,in=0] (0,-6);
\draw (0,-0.5) arc(90:-90:0.3 and 0.25);
\draw (0,-1.5) arc(90:-90:0.3 and 0.25);
\draw (0,-4) arc(90:-90:0.3 and 0.25);
\draw (0,-5) arc(90:-90:0.3 and 0.25);
\node at (0.3,-2.8) {$\vdots$};
\node at (0.5,0.4) [anchor=south] {\footnotesize $\mathrm{a}$};
\end{scope}
\draw[dotted] (0,0.5)--(12,0.5);
\draw[dotted] (0,-6.5)--(12,-6.5);
\draw [very thick](0,-6.5) -- (0,0.5);
\draw [very thick](3,-6.5)--(3,0.5);
\draw [very thick](6,-6.5)--(6,0.5);
\draw [very thick](9,-6.5)--(9,0.5);
\draw [very thick](12,-6.5)--(12,0.5);
}
\end{tikzpicture} \; .
\end{align}
In the third case, with an exposed bottom boundary link on the right,
\begin{align}
\phi(T) &= \; \begin{tikzpicture}[baseline={([yshift=1mm]current bounding box.center)},xscale=0.5,yscale=0.5]
{
\draw (0,0) ..controls (2,0) and (2,-6).. (0,-6);
\filldraw[red!10!magenta!30] (0,-0.5) arc(90:-90:0.5)--cycle;
\filldraw[yellow!90!orange!80] (0,-2.5) arc(90:-90:0.5)--cycle;
\draw (0,-2) arc(90:-90:1);
\filldraw[green!90!blue!30] (0,-4.5) arc(90:-90:0.5)--cycle;
\filldraw[blue!30!cyan!40] (3,0) arc(270:180:0.5)--(2,0.5) ..controls (2,-1) and (2,-2.5).. (3,-2.5) --cycle;
\filldraw[orange!40] (3,-3.5) ..controls (2,-3.5) and (2,-5).. (2,-6.5)--(2.5,-6.5) arc(180:90:0.5) --cycle;
\draw (3,-3) ..controls (1.6,-3) and (1.6,-5).. (1.6,-6.5);
\node at (1.6,-7.3) [anchor=south] {\footnotesize $\mathrm{a}$};
\draw[dotted] (0,0.5)--(3,0.5);
\draw[dotted] (0,-6.5)--(3,-6.5);
\draw [very thick](0,-6.5) -- (0,0.5);
\draw [very thick](3,-6.5)--(3,0.5);
}
\end{tikzpicture} 
\; = \; 
\begin{tikzpicture}[baseline={([yshift=1mm]current bounding box.center)},xscale=0.5,yscale=0.5]
{
\draw (0,0)--(3,0);
\draw (0,-6)--(3,-6);
\filldraw[red!10!magenta!30] (0,-0.5) arc(90:-90:0.5)--cycle;
\filldraw[yellow!90!orange!80] (0,-2.5) arc(90:-90:0.5)--cycle;
\draw (0,-2) to[out=0,in=180] (3,-0.5);
\draw (0,-4) to[out=0,in=180] (3,-5.5);
\filldraw[green!90!blue!30] (0,-4.5) arc(90:-90:0.5)--cycle;
\draw (3,-1) arc(90:270:0.3 and 0.25);
\draw (3,-3.5) arc(90:270:0.3 and 0.25);
\draw (3,-4.5) arc(90:270:0.3 and 0.25);
\node at (2.7,-2.3) {$\vdots$};
\begin{scope}[shift={(3,0)}]
\draw (0,-6)--(3,-6);
\draw (0,0)--(3,0);
\draw (0,-0.5) arc(90:-90:0.3 and 0.25);
\draw (0,-1.5) arc(90:-90:0.3 and 0.25);
\draw (0,-4) arc(90:-90:0.3 and 0.25);
\draw (0,-5) arc(90:-90:0.3 and 0.25);
\draw (3,-0.5) arc(90:270:0.3 and 0.25);
\draw (3,-1.5) arc(90:270:0.3 and 0.25);
\draw (3,-4) arc(90:270:0.3 and 0.25);
\draw (3,-5) arc(90:270:0.3 and 0.25);
\node at (0.3,-2.8) {$\vdots$};
\node at (2.7,-2.8) {$\vdots$};
\end{scope}
\begin{scope}[shift={(6,0)}]
\draw (0,0) arc(90:-90:0.3 and 0.25);
\draw (0,-1) arc(90:-90:0.3 and 0.25);
\draw (0,-4.5) arc(90:-90:0.3 and 0.25);
\draw (0,-5.5) arc(90:-90:0.3 and 0.25);
\draw (3,0) arc(90:270:0.3 and 0.25);
\draw (3,-1) arc(90:270:0.3 and 0.25);
\draw (3,-4.5) arc(90:270:0.3 and 0.25);
\draw (3,-5.5) arc(90:270:0.3 and 0.25);
\node at (0.3,-2.8) {$\vdots$};
\node at (2.7,-2.8) {$\vdots$};
\end{scope}
\begin{scope}[shift={(9,0)}]
\draw (0,-6) arc(90:0:0.5);
\filldraw[blue!30!cyan!40] (3,0) arc(270:180:0.5)--(2,0.5) ..controls (2,-1) and (2,-2.5).. (3,-2.5) --cycle;
\filldraw[orange!40] (3,-3.5) ..controls (2,-3.5) and (2,-5).. (2,-6.5)--(2.5,-6.5) arc(180:90:0.5) --cycle;
\draw (3,-3) to[out=180,in=0] (0,0);
\draw (0,-0.5) arc(90:-90:0.3 and 0.25);
\draw (0,-1.5) arc(90:-90:0.3 and 0.25);
\draw (0,-4) arc(90:-90:0.3 and 0.25);
\draw (0,-5) arc(90:-90:0.3 and 0.25);
\node at (0.3,-2.8) {$\vdots$};
\node at (0.5,-7.3) [anchor=south] {\footnotesize $\mathrm{a}$};
\end{scope}
\draw[dotted] (0,0.5)--(12,0.5);
\draw[dotted] (0,-6.5)--(12,-6.5);
\draw [very thick](0,-6.5) -- (0,0.5);
\draw [very thick](3,-6.5)--(3,0.5);
\draw [very thick](6,-6.5)--(6,0.5);
\draw [very thick](9,-6.5)--(9,0.5);
\draw [very thick](12,-6.5)--(12,0.5);
}
\end{tikzpicture} \; .
\end{align}
In each of these products, the first diagram is even, and has strings connecting $1$ and $n$ on the left to $1$ and $n$ on the right, respectively. It follows that this is equal to a product of the diagrammatic generators $\de{i}$, $2 \leq i \leq n-2$, and this product does not produce any boundary arcs. Hence the first diagram is equal to $\phi(L)$, where $L$ is the corresponding word in the generators $e_i$, $2\leq i \leq n-2$. Note that $L$ contains no label generators, and $\phi(L)$ has no boundary links. The second and third diagrams are $\phi(\Omega)$ and $\phi(\Theta)$, respectively; observe that these diagrams have no boundary links, and recall that $\Omega$ and $\Theta$ contain no label generators. The fourth diagram is even, so by Corollary \ref{cor:evenredsurj}, it is equal to $\phi(T')$ for some even-reduced word $T'$, where the number of boundary links in $\phi(T')$ is twice the number of label generators in $T'$. Note that each of these diagram products produces no boundary arcs, so the number of boundary links in $\phi(T)$ is the sum of the number of boundary links in each of the diagrams $\phi(L)$, $\phi(\Omega)$, $\phi(\Theta)$ and $\phi(T')$, which is twice the number of label generators in each argument of $\phi$. Hence the number of boundary links in $\phi(L\Omega\Theta T')$ is twice the number of label generators in $L\Omega \Theta T'$, so $L\Omega\Theta T'$ is label-reduced. Since $L\Omega\Theta T'$ is a label-reduced word in the even generators, it is even. As $T$ is also even, and $\phi(L\Omega\Theta T') = \phi(T)$, we have from Proposition \ref{prop:evenwordinj} that $L\Omega\Theta T' = T$.

Moreover, the number of boundary links in $\phi(T)$ is the same as the number of boundary links in $\phi(T')$, which is twice the number of label generators in $T'$. However, since $T$ is even-reduced, we have from Lemma \ref{lem:evendiag} that the number of boundary links in $\phi(T)$ is twice the number of label generators in $T$. Therefore, $T$ and $T'$ have the same number of label generators, as required.
\end{proof}

\begin{proposition}\label{prop:WTdiag}
Let $\mathbf{u}=WT$ be in $WT$ form. Then $\phi(\mathbf{u})$ is an $\lab_n(X)$-diagram, with coefficient $1$. Moreover, if $W=\id$, then $\phi(\mathbf{u})$ is an even diagram; otherwise, $\phi(\mathbf{u})$ is an odd diagram.
\end{proposition}
\begin{proof}
If $W=\id$, then $\mathbf{u}=T$ is an even-reduced word, and thus $\phi(\mathbf{u})$ is an even $\lab_n(X)$-diagram with coefficient $1$, by Lemma \ref{lem:evendiag}. Hence let $W=\w{a}{b}(j)$.

From our definitions of $\w{a}{b}(j)$ and $\dw{a}{b}(j)$ in Propositions \ref{prop:WTintro} and \ref{prop:oddgen}, it follows that $\phi(\w{a}{b}(j)) = \dw{a}{b}(j)$, and this is an odd diagram with coefficient $1$. Since $T$ is even-reduced, from Lemma \ref{lem:evendiag}, $\phi(T)$ is an even diagram with coefficient $1$. Hence it suffices to show that concatenating the diagrams $\dw{a}{b}(j)$ and $\phi(T)$ produces no loops or boundary arcs, because then their product is found by simply concatenating the diagrams, and the concatenation gives an odd $\lab_n(X)$-diagram.

Note that none of the diagrams $\dw{a}{b}(j)$ have any links on the right side, so it is impossible for the concatenation of $\dw{a}{b}(j)$ and $\phi(T)$ to produce any loops. 

We now show that the concatenation of $\phi(\w{a}{b}(j))$ and $\phi(T)$ does not produce any boundary arcs. Since $\w{a}{b}(j) T$ is in $WT$ form, it must be label-reduced. We now consider the ways in which concatenating $\phi(\w{a}{b}(j))$ and $\phi(T)$ could produce a boundary arc, given that all of the strings connected to the right side of $\phi(\w{a}{b}(j))$ are throughlines, except for a top boundary link at $1$, if $1\leq j \leq n$, and a bottom boundary link at $n$, if $0 \leq j \leq n-1$. There are five cases to consider; these were listed earlier in Figure \ref{fig:WTbdyarccases}. In each case, we find that $\w{a}{b}(j) T$ is equal to some expression containing fewer label generators, thus contradicting the fact that it is label-reduced.

\vspace{4mm}
\noindent\textbf{Case (a):} Take $1 \leq j \leq n$, so that $\phi(\w{a}{b}(j))$ has a top boundary link at $1$ on the right. Observe that, for $1\leq j \leq n-1$,
\begin{align}
\w{a}{b}(j) &= e_je_{j-1}\dots e_2e_1\wup{a}{b} \\
&= e_je_{j-1} \dots e_2e_1e_2\dots e_{j-1}e_je_{j+1} \dots e_{n-2}e_{n-1} \wdo{a}{b} &&\text{by \eqref{eq:e1wup}}\\
&= e_je_{j+1} \dots e_{n-2}e_{n-1} \wdo{a}{b} &&\text{by \eqref{eq:TLreduce}}, \label{eq:wflip}
\end{align}
so in fact $\w{a}{b}(j) = e_je_{j+1} \dots e_{n-2}e_{n-1} \wdo{a}{b}$ for all $1 \leq j \leq n$, and the expression on the right contains one label generator, as does $\w{a}{b}(j)$.

As $\phi(T)$ has a top boundary link at $1$ on the left, we have from Lemma \ref{lem:topbdylinkat1}(i) that $T=\fup{c}{d}T'$ for some $\mathrm{c},\mathrm{d}\in X$, and some word $T'$ in the even generators that contains one less label generator than $T$. Then we have
\begin{align}
\w{a}{b}(j)T
&= e_je_{j+1} \dots e_{n-2}e_{n-1} \wdo{a}{b} \fup{c}{d}T' \\
&= \aup{a}{c} e_je_{j+1} \dots e_{n-1}e_{n-1} \wdo{d}{b}T' &&\text{by \eqref{eq:wdofup}},
\end{align}
where the final expression contains one less label generator than $\w{a}{b}(j)T$, contradicting the fact that $\w{a}{b}(j)T$ is label-reduced.

\vspace{4mm}
\noindent\textbf{Case (b):} Take $0 \leq j \leq n-1$, so that $\phi(\w{a}{b}(j))$ has a bottom boundary link at $n$ on the right. Since $\phi(T)$ has a bottom boundary link at $n$ on the left, we have from Lemma \ref{lem:topbdylinkat1}(ii) that $T=\fdo{c}{d}T'$ for some $\mathrm{c},\mathrm{d}\in X$, and some word $T'$ in the even generators that contains one less label generator than $T$. Then we have
\begin{align}
\w{a}{b}(j)T
&= e_je_{j-1} \dots e_2e_1 \wup{a}{b}\fdo{c}{d}T' \\
&= \ddo{b}{c} e_je_{j-1}\dots e_2e_1 \wup{a}{d} T' &&\text{by \eqref{eq:wupfdo}},
\end{align}
where the final expression contains one less label generator than $\w{a}{b}(j)T$, contradicting the fact that $\w{a}{b}(j)T$ is label-reduced.

\vspace{4mm}
\noindent\textbf{Case (c):} Take $0 \leq j \leq n-1$, so that $\phi(\w{a}{b}(j))$ has a bottom boundary link at $n$ on the right. First consider $n$ odd. Since $\phi(T)$ has a top boundary link at $n$ on the left, we have from Lemma \ref{lem:oddlongbdylink}(i) that $T=LO\fup{c}{d}ET'$, where $L$ is a word in the generators $e_i$, $1\leq i \leq n-3$, and $\fup{e}{f}$, $\mathrm{e},\mathrm{f} \in X$; $T'$ is a word in the even generators; and $L$ and $T'$ together contain one less label generator than $T$. If $n=1$, then $\w{a}{b}(j) = \wup{a}{b}$ and $L=O=E=\id$, so $T=\fup{c}{d}T'$, and thus
\begin{align}
    \w{a}{b}(j)T &= \wup{a}{b}\fup{c}{d}T' = \glab{c}{b}\fup{a}{d}T' &&\text{by \eqref{eq:wupfup1}}.
\end{align}
The final expression contains one less label generator than $\w{a}{b}(j)T$, contradicting the fact that $\w{a}{b}(j)T$ is label-reduced.

If $n>1$, then from \eqref{eq:ejwup}, we have
\begin{align}
\wup{a}{b}e_i &= e_{i+1}\wup{a}{b}, && 1 \leq i \leq n-2,
\end{align}
and from \eqref{eq:wupfup},
\begin{align}
\wup{a}{b}\fup{e}{f} &= 
\fup{a}{e} e_1 \wup{f}{b},
\end{align}
for each $\mathrm{e},\mathrm{f} \in X$. Using these relations, we have that $\wup{a}{b}L=L'\wup{g}{h}$ for some $\mathrm{g},\mathrm{h} \in X$, where $L'$ is a word in the generators $e_i$, $1 \leq i \leq n-1$, and $\fup{e}{f}$, $\mathrm{e},\mathrm{f} \in X$, and $L'$ has the same number of label generators as $L$. It follows that $L'$ and $T'$ together contain one less label generator than $T$. We then have
\begin{align}
\w{a}{b}(j)T
&= e_je_{j-1}\dots e_2e_1 \wup{a}{b}LO\fup{c}{d}ET' \\
&= e_je_{j-1}\dots e_2e_1 L' \wup{g}{h} O\fup{c}{d}ET' \\
&= \glab{c}{h} e_je_{j-1}\dots e_2e_1 L' \fup{g}{d}ET' &&\text{by \eqref{eq:wupOfupE}}.
\end{align}
The final expression contains one less label generator than $\w{a}{b}(j)T$, giving the desired contradiction.

Now consider $n$ even. Since $\phi(T)$ has a top boundary link at $n$ on the left, we have from Lemma \ref{lem:evenlongbdylink}(i) that $T=L\fup{c}{d}\Omega\Theta T'$, where $L$ is a word in the generators $e_i$, $1\leq i \leq n-2$, and $\fup{e}{f}$, $\mathrm{e},\mathrm{f} \in X$; $T'$ is a word in the even generators; and $L$ and $T'$ together contain one less label generator than $T$. As for $n$ odd, we have $\wup{a}{b}L=L'\wup{g}{h}$ for some $\mathrm{g},\mathrm{h} \in X$, where $L'$ is a word in the generators $e_i$, $1 \leq i \leq n-1$, and $\fup{e}{f}$, $\mathrm{e},\mathrm{f} \in X$, and $L'$ has the same number of label generators as $L$. Again, $L'$ and $T'$ together contain one less label generator than $T$. We then have
\begin{align}
\w{a}{b}(j)T
&= e_je_{j-1}\dots e_2e_1 \wup{a}{b}L\fup{c}{d}\Omega\Theta T' \\
&= e_je_{j-1}\dots e_2e_1 L' \wup{g}{h}\fup{c}{d}\Omega\Theta T' \\
&= e_je_{j-1}\dots e_2e_1 L' \fup{g}{c} e_1 \wup{d}{h} \Omega\Theta T' &&\text{by \eqref{eq:wupfup}} \\
&= e_je_{j-1}\dots e_2e_1 L' \fup{g}{c} \Theta \wup{d}{h} \Theta T' &&\text{by \eqref{eq:e1wupOmega}} \\
&= \glab{d}{h} e_je_{j-1}\dots e_2e_1 L' \fup{g}{c} \Theta T' &&\text{by \eqref{eq:ThetawupTheta}}.
\end{align}
Since $\Theta$ contains no label generators, the final expression contains one less label generator than $\w{a}{b}(j)T$, giving the desired contradiction.

\vspace{4mm}
\noindent\textbf{Case (d):} Take $1 \leq j \leq n$, so that $\phi(\w{a}{b}(j))$ has a top boundary link at $1$ on the right. First consider $n$ odd. Since $\phi(T)$ has a bottom boundary link at $1$ on the left, we have from Lemma \ref{lem:oddlongbdylink}(ii) that $T=LE\fdo{c}{d}OT'$, where $L$ is a word in the generators $e_i$, $3\leq i \leq n-1$, and $\fdo{e}{f}$, $\mathrm{e},\mathrm{f} \in X$; $T'$ is a word in the even generators; and $L$ and $T'$ together contain one less label generator than $T$. If $n=1$, then $\w{a}{b}(j)=\wdo{a}{b}$ and $L=E=O=\id$, so $T=\fdo{c}{d}T'$, and thus
\begin{align}
\w{a}{b}(j)T &= \wdo{a}{b} \fdo{c}{d}T' = \glab{a}{c}\fdo{b}{d} T' &&\text{by \eqref{eq:wdofdo1}}.
\end{align}
The final expression contains one less label generator than $\w{a}{b}(j)T$, contradicting the fact that $\w{a}{b}(j)T$ is label-reduced.

If $n>1$, then from \eqref{eq:ejwdo}, we have
\begin{align}
\wdo{a}{b}e_i &= e_{i-1}\wdo{a}{b}, &&2 \leq i \leq n-1,
\end{align}
and from \eqref{eq:wdofdo},
\begin{align}
\wdo{a}{b}\fdo{e}{f} &= 
\fdo{b}{e} e_{n-1} \wdo{a}{f},
\end{align}
for each $\mathrm{e},\mathrm{f} \in X$. Using these relations, we have that $\wdo{a}{b}L=L'\wdo{g}{h}$ for some $\mathrm{g},\mathrm{h} \in X$, where $L'$ is a word in the generators $e_i$, $1 \leq i \leq n-1$, and $\fdo{e}{f}$, $\mathrm{e},\mathrm{f} \in X$, and $L'$ has the same number of label generators as $L$. It follows that $L'$ and $T'$ together contain one less label generator than $T$. We then have
\begin{align}
\w{a}{b}(j)T
&= e_je_{j+1} \dots e_{n-2}e_{n-1} \wdo{a}{b}LE\fdo{c}{d}OT' &&\text{using \eqref{eq:wflip}}\\
&= e_je_{j+1} \dots e_{n-2}e_{n-1} L' \wdo{g}{h}E\fdo{c}{d}OT' \\
&= \glab{g}{c}e_je_{j+1} \dots e_{n-2}e_{n-1} L' \fdo{h}{d}OT' &&\text{by \eqref{eq:wdoEfdoO}}.
\end{align}
Recalling that $O$ contains no label generators, the final expression contains one less label generator than $\w{a}{b}(j)T$, giving the desired contradiction.

Now consider $n$ even. Since $\phi(T)$ has a bottom boundary link at $1$ on the left, we have from Lemma \ref{lem:evenlongbdylink}(ii) that $T=L\fdo{c}{d}\Omega\Theta T'$, where $L$ is a word in the generators $e_i$, $2\leq i \leq n-1$, and $\fdo{e}{f}$, $\mathrm{e},\mathrm{f} \in X$; $T'$ is a word in the even generators; and $L$ and $T'$ together contain one less label generator than $T$. As for $n$ odd, we have $\wdo{a}{b}L=L'\wdo{g}{h}$ for some $\mathrm{g},\mathrm{h} \in X$, where $L'$ is a word in the generators $e_i$, $1 \leq i \leq n-1$, and $\fdo{e}{f}$, $\mathrm{e},\mathrm{f} \in X$, and the number of label generators in $L'$ is equal to the number of label generators in $L$. Again, $L'$ and $T'$ together contain one less label generator than $T$. We then have
\begin{align}
\w{a}{b}(j)T
&= e_je_{j+1} \dots e_{n-2}e_{n-1} \wdo{a}{b}L\fdo{c}{d}\Omega\Theta T' &&\text{using \eqref{eq:wflip}}\\
&= e_je_{j+1} \dots e_{n-2}e_{n-1} L' \wdo{g}{h} \fdo{c}{d}\Omega\Theta T' \\
&= e_je_{j+1} \dots e_{n-2}e_{n-1} L' \fdo{h}{c}e_{n-1}\wdo{g}{d}\Omega\Theta T' &&\text{by \eqref{eq:wdofdo}} \\
&= e_je_{j+1} \dots e_{n-2}e_{n-1} L' \fdo{h}{c}\Theta \wdo{g}{d}\Theta T' &&\text{by \eqref{eq:en1wdoOmega}} \\
&= \glab{g}{d} e_je_{j+1} \dots e_{n-2}e_{n-1} L' \fdo{h}{c}\Theta T' &&\text{by \eqref{eq:Thetawup} and \eqref{eq:ThetawupTheta}}.
\end{align}
Recalling that $\Theta$ contains no label generators, the final expression contains one less label generator than $\w{a}{b}(j)T$, giving the desired contradiction.

\vspace{4mm}
\noindent\textbf{Case (e):} Take $1 \leq j \leq n-1$, so that $\w{a}{b}(j)$ has a top boundary link at $1$ and a bottom boundary link at $n$ on the right. Observe that $\phi(T)$ has a link from $1$ to $n$ on the left, and that this can only occur for $n$ even. It thus follows from Lemma \ref{lem:evenlonglink} that $T=L\Omega\Theta T'$, where $L$ is a word in the generators $e_i$, $2\leq i \leq n-2$, and $T'$ is a word in the even generators that contains the same number of label generators as $T$. From \eqref{eq:ejwup}, we have that $\wup{a}{b}e_i = e_{i+1}\wup{a}{b}$ for all $1 \leq i \leq n-2$. It follows that $\wup{a}{b}L = L'\wup{a}{b}$, where $L'$ is a word in the generators $e_i$, $3 \leq i \leq n-1$. Then we have
\begin{align}
\w{a}{b}(j)T
&= e_je_{j-1} \dots e_2e_1 \wup{a}{b} L \Omega \Theta T' \\
&= e_je_{j-1} \dots e_2e_1 L' \wup{a}{b} \Omega \Theta T' \\
&= e_je_{j-1} \dots e_2 L' e_1 \wup{a}{b} \Omega \Theta T' &&\text{by \eqref{eq:eiej}} \\
&= e_je_{j-1} \dots e_2 L' \Theta \wup{a}{b} \Theta T' &&\text{by \eqref{eq:e1wupOmega}} \\
&= \glab{a}{b} e_je_{j-1} \dots e_2 L' \Theta T' &&\text{by \eqref{eq:ThetawupTheta}}.
\end{align}
Since $L'$ and $\Theta$ contain no label generators, and $T'$ contains the same number of label generators as $T$, the final expression contains one less label generator than $\w{a}{b}(j)T$, giving the desired contradiction.
\end{proof}

We can now show that, if two words in $WT$ form map to the same $\lab_n(X)$-diagram, then they must be equal in $\A_n(X)$.

\begin{proposition}\label{prop:WTinj}
Let $W_1T_1$ and $W_2T_2$ be words in $WT$ form. If $\phi(W_1T_1) = \phi(W_2T_2)$, then $W_1T_1=W_2T_2$ in $\A_n(X)$.
\end{proposition}
\begin{proof}
From Proposition \ref{prop:WTdiag}, we have $\phi(W_1T_1)=\phi(W_2T_2)=\undertilde{D}$ for some $\lab_n(X)$-diagram $\undertilde{D}$. If $\undertilde{D}$ is even, then Proposition \ref{prop:WTdiag} also tells us that $W_1=W_2=\id$. Then $\phi(W_1T_1)=\phi(W_2T_2)$ implies $\phi(T_1)=\phi(T_2)$, and $T_1$ and $T_2$ are both even-reduced, so by Lemma \ref{lem:evenredinj}, $T_1=T_2$. Thus, if $\undertilde{D}$ is even, we have $W_1T_1=W_2T_2$, as required.

Hence suppose $\undertilde{D}$ is odd. Then from Proposition \ref{prop:WTdiag}, $W_1=\w{a}{b}(j_1)$ and $W_2=\w{c}{d}(j_2)$ for some $\mathrm{a},\mathrm{b},\mathrm{c},\mathrm{d} \in X$ and $0 \leq j_1,j_2 \leq n$. From Proposition \ref{prop:WTdiag} we also have that the concatenation of $\phi(W_1)$ with $\phi(T_1)$, and of $\phi(W_2)$ with $\phi(T_2)$, produces no boundary arcs. Hence the leftmost top and bottom boundary links in $\phi(W_1T_1)=\undertilde{D}$ are labelled $\mathrm{a}$ and $\mathrm{b}$, respectively, and those of $\phi(W_2T_2)=\undertilde{D}$ are labelled $\mathrm{c}$ and $\mathrm{d}$, respectively. Thus $\mathrm{a}=\mathrm{c}$ and $\mathrm{b}=\mathrm{d}$, so $W_2 = \w{a}{b}(j_2)$.

Now suppose $j_1 \neq j_2$. Without loss of generality, take $j_1<j_2$. We will consider the possible values of $j_1$, and arrive at a contradiction in each case, implying $j_1=j_2$. 

First consider $j_1=0$. Then $\phi(\w{a}{b}(j_1))$ has a top boundary link at $1$ on the left, so $\undertilde{D}=\phi(\w{a}{b}(j_1)T_1)=\phi(\w{a}{b}(j_2)T_2)$ must also. If $j_2=1$, then $\phi(\w{a}{b}(j_2))$ has a simple link at $1$ on the left, if $n\geq 2$; or a bottom boundary link at $1$ on the left, if $n=1$; so $\undertilde{D}$ must also. Either way, this contradicts the fact that $\undertilde{D}$ has a top boundary link at $1$ on the left. Hence $j_2>1$.
This means that $\phi(\w{a}{b}(j_2))$ must have a throughline connecting node $1$ on the left to node $2$ on the right, and a top boundary link at $1$ on the right. For $\undertilde{D}$ to have a top boundary link at $1$ on the left, the string at node $1$ on the left of $\phi(T_2)$ must eventually connect to node $2$ on the left of $\phi(T_2)$. This means the string at node $1$ on the left of $\phi(T_2)$ must connect to some alternating sequence of links on the right of $\phi(\w{a}{b}(j_2))$ and the left of $\phi(T_2)$, ending at node $2$ on the left of $\phi(T_2)$. As $\phi(\w{a}{b}(j_2))$ does not have any links on its right side, it follows that $\phi(T_2)$ must have a simple link at $1$ on the left. That is, for $2 \leq j_2 \leq n-1$,
\begin{align}
\phi(\w{a}{b}(j_2))\phi(T_2) = \; \begin{tikzpicture}[baseline={([yshift=-1mm]current bounding box.center)},xscale=0.5,yscale=0.5]
{
\draw (0,0) to[out=0,in=180] (3,-0.5);
\draw (0,-0.5) to[out=0,in=180] (3,-1);
\draw (0,-2) to[out=0,in=180] (3,-2.5);
\draw (0,-2.5) arc (90:-90:0.3 and 0.25);
\draw (0,-3.5) to[out=0,in=180] (3,-3);
\draw (0,-6) to[out=0,in=180] (3,-5.5);
\draw (3,0) arc (270:180:0.5);
\draw (3,-6) arc (90:180:0.5);
\node at (2.5,0.4) [anchor=south] {\footnotesize $\mathrm{a}$};
\node at (2.5,-6.4) [anchor=north] {\footnotesize $\mathrm{b}$};
\node at (0,-2.5) [anchor=east] {\scriptsize $j_2$};
\node at (1.5,-1.3) [anchor=center] {$\vdots$};
\node at (1.5,-4.3) [anchor=center] {$\vdots$};
\begin{scope}[shift={(3,0)}]
\draw (0,0) arc(90:-90:0.3 and 0.25);
\filldraw[gray!40] (0,-1) to[out=0,in=-90] (1,0.5)--(2.5,0.5) arc(180:270:0.5) --(3,-6) arc(90:180:0.5) --(0.5,-6.5) arc(0:90:0.5)--cycle;
\end{scope}
\draw[dotted] (0,0.5)--(6,0.5);
\draw[dotted] (0,-6.5)--(6,-6.5);
\draw [very thick](0,-6.5) -- (0,0.5);
\draw [very thick](3,-6.5)--(3,0.5);
\draw [very thick](6,-6.5)--(6,0.5);
}
\end{tikzpicture}\; ,
\end{align}
and for $j_2=n$,
\begin{align}
\phi(\w{a}{b}(j_2))\phi(T_2) = \; \begin{tikzpicture}[baseline={([yshift=-1mm]current bounding box.center)},xscale=0.5,yscale=0.5]
{
\draw (0,0) to[out=0,in=180] (3,-0.5);
\draw (0,-0.5) to[out=0,in=180] (3,-1);
\draw (0,-5.5) to[out=0,in=180] (3,-6);
\draw (0,-6) arc (90:0:0.5);
\draw (3,0) arc (270:180:0.5);
\node at (2.5,0.4) [anchor=south] {\footnotesize $\mathrm{a}$};
\node at (0.5,-6.4) [anchor=north] {\footnotesize $\mathrm{b}$};
\node at (1.5,-3.05) [anchor=center] {$\vdots$};
\begin{scope}[shift={(3,0)}]
\draw (0,0) arc(90:-90:0.3 and 0.25);
\filldraw[gray!40] (0,-1) to[out=0,in=-90] (1,0.5)--(2.5,0.5) arc(180:270:0.5) --(3,-6) arc(90:180:0.5) --(0.5,-6.5) arc(0:90:0.5)--cycle;
\end{scope}
\draw[dotted] (0,0.5)--(6,0.5);
\draw[dotted] (0,-6.5)--(6,-6.5);
\draw [very thick](0,-6.5) -- (0,0.5);
\draw [very thick](3,-6.5)--(3,0.5);
\draw [very thick](6,-6.5)--(6,0.5);
}
\end{tikzpicture}\; .
\end{align}

By Lemma \ref{lem:simplelinklds}, then, the left descent set of $T_2$ must contain $e_1$. From Lemma \ref{prop:leftdescent}, however, $e_i \notin \mathcal{L}(T_2)$ for any $i<j_2$, since $\w{a}{b}(j_2)T_2$ is in $WT$ form. This is a contradiction, as $1<j_2$. It follows that $j_1 \neq 0$.

Now consider $1 \leq j_1 \leq n-1$. Then $\phi(\w{a}{b}(j_1))$ has a simple link at $j_1$ on the left, so $\undertilde{D}$ has a simple link at $j_1$ on the left; this link connects $j_1$ to $j_1+1$. If $j_2\leq n-1$, then $\phi(\w{a}{b}(j_2))$ has a simple link at $j_2$ on the left, so $\undertilde{D}$ also has a simple link at $j_2$ on the left. If $j_2=n$, then $\phi(\w{a}{b}(j_2))$ has a bottom boundary link at $j_2=n$ on the left, so $\undertilde{D}$ has a bottom boundary link at $n$ on the left. If $j_2=j_1+1$, each of these gives a contradiction, so $j_2>j_1+1$. Note that $\phi(\w{a}{b}(j_2))$ has throughlines connecting $i$ on the left to $i+1$ on the right, for each $1 \leq i \leq j_2-1$. Since $j_2>j_1+1$, this means there are throughlines in $\phi(\w{a}{b}(j_2))$ connecting $j_1$ and $j_1+1$ on the left, to $j_1+1$ and $j_1+2$ on the right, respectively. As there are no links on the right side of $\phi(\w{a}{b}(j_2))$, the throughlines at $j_1+1$ and $j_1+2$ on the right of $\phi(\w{a}{b}(j_2))$ must be connected to each other directly by a simple link at $j_1+1$ on the left side of $\phi(T_2)$. That is, for $j_2<n$,
\begin{align}
\phi(\w{a}{b}(j_2))\phi(T_2) =\;  \begin{tikzpicture}[baseline={([yshift=-1mm]current bounding box.center)},xscale=0.5,yscale=0.5]
{
\draw (0,0) to[out=0,in=180] (3,-0.5);
\draw (3,-6) arc(90:180:0.5);
\draw (0,-1) to[out=0,in=180] (3,-1.5);
\draw (0,-1.5) to[out=0,in=180] (3,-2);
\draw (0,-2) to[out=0,in=180] (3,-2.5);
\draw (0,-2.5) to[out=0,in=180] (3,-3);
\draw (0,-3.5) to[out=0,in=180] (3,-4);
\draw (0,-4) arc (90:-90:0.3 and 0.25);
\draw (0,-5) to[out=0,in=180] (3,-4.5);
\draw (0,-6) to[out=0,in=180] (3,-5.5);
\draw (3,0) arc (270:180:0.5);
\node at (2.5,0.4) [anchor=south] {\footnotesize $\mathrm{a}$};
\node at (2.5,-6.4) [anchor=north] {\footnotesize $\mathrm{b}$};
\node at (1.5,-0.55) [anchor=center] {.};
\node at (1.5,-0.75) [anchor=center] {.};
\node at (1.5,-0.95) [anchor=center] {.};
\node at (1.5,-3.05) [anchor=center] {.};
\node at (1.5,-3.25) [anchor=center] {.};
\node at (1.5,-3.45) [anchor=center] {.};
\node at (1.5,-5.05) [anchor=center] {.};
\node at (1.5,-5.25) [anchor=center] {.};
\node at (1.5,-5.45) [anchor=center] {.};
\node at (0,-1.5) [anchor=east] {\scriptsize $j_1$};
\node at (0,-4) [anchor=east] {\scriptsize $j_2$};
\begin{scope}[shift={(3,0)}]
\draw (0,-2) arc (90:-90:0.3 and 0.25);
\filldraw[gray!40] (0,0) arc(-90:0:0.5)--(2.5,0.5) arc(180:270:0.5) --(3,-6) arc(90:180:0.5) --(0.5,-6.5) arc(0:90:0.5) -- (0,-3) arc(-90:90:0.75)--cycle;
\end{scope}
\draw[dotted] (0,0.5)--(6,0.5);
\draw[dotted] (0,-6.5)--(6,-6.5);
\draw [very thick](0,-6.5) -- (0,0.5);
\draw [very thick](3,-6.5)--(3,0.5);
\draw [very thick](6,-6.5)--(6,0.5);
}
\end{tikzpicture}\; ,
\end{align}
and for $j_2=n$,
\begin{align}
\phi(\w{a}{b}(j_2))\phi(T_2) =\;  \begin{tikzpicture}[baseline={([yshift=-1mm]current bounding box.center)},xscale=0.5,yscale=0.5]
{
\draw (0,0) to[out=0,in=180] (3,-0.5);
\draw (0,-1) to[out=0,in=180] (3,-1.5);
\draw (0,-1.5) to[out=0,in=180] (3,-2);
\draw (0,-2) to[out=0,in=180] (3,-2.5);
\draw (0,-2.5) to[out=0,in=180] (3,-3);
\draw (0,-5.5) to[out=0,in=180] (3,-6);
\node at (1.5,-0.55) [anchor=center] {.};
\node at (1.5,-0.75) [anchor=center] {.};
\node at (1.5,-0.95) [anchor=center] {.};
\node at (1.5,-4.05) [anchor=center] {.};
\node at (1.5,-4.25) [anchor=center] {.};
\node at (1.5,-4.45) [anchor=center] {.};
\node at (0,-1.5) [anchor=east] {\scriptsize $j_1$};
\draw (0,-6) arc(90:0:0.5);
\draw (3,0) arc (270:180:0.5);
\node at (2.5,0.4) [anchor=south] {\footnotesize $\mathrm{a}$};
\node at (0.5,-6.4) [anchor=north] {\footnotesize $\mathrm{b}$};
\begin{scope}[shift={(3,0)}]
\draw (0,-2) arc (90:-90:0.3 and 0.25);
\filldraw[gray!40] (0,0) arc(-90:0:0.5)--(2.5,0.5) arc(180:270:0.5) --(3,-6) arc(90:180:0.5) --(0.5,-6.5) arc(0:90:0.5) -- (0,-3) arc(-90:90:0.75)--cycle;
\end{scope}
\draw[dotted] (0,0.5)--(6,0.5);
\draw[dotted] (0,-6.5)--(6,-6.5);
\draw [very thick](0,-6.5) -- (0,0.5);
\draw [very thick](3,-6.5)--(3,0.5);
\draw [very thick](6,-6.5)--(6,0.5);
}
\end{tikzpicture}\; .
\end{align}
Then $\mathcal{L}(T_2)$ contains $e_{j_1+1}$, by Lemma \ref{lem:simplelinklds}. However, $j_1+1<j_2$, and by Proposition \ref{prop:leftdescent}, $e_i \notin \mathcal{L}(T_2)$ for any $i < j_2$, as $\w{a}{b}(j_2)T_2$ is in $WT$ form. This is a contradiction, so we cannot have $1 \leq j_1 \leq n-1$.

Finally, if $j_1=n$, then $j_2>j_1$ contradicts $0 \leq j_2 \leq n$. 

Therefore, if $\w{a}{b}(j_1)T_1$ and $\w{a}{b}(j_2)T_2$ are in $WT$ form, and $\phi(\w{a}{b}(j_1)T_1) = \phi(\w{a}{b}(j_2)T_2)$, then $j_1=j_2$. 

Let $j\coloneqq j_1=j_2$. Then $W_1=W_2= \w{a}{b}(j)$, so $\phi(W_1)=\phi(W_2)=\dw{a}{b}(j)$. Hence $\undertilde{D} = \phi(W_1)\phi(T_1) = \dw{a}{b}(j)\phi(T_1)$, and $\undertilde{D} = \phi(W_2)\phi(T_2) = \dw{a}{b}(j) \phi(T_2)$. From Proposition \ref{prop:oddgen}, then, $\phi(T_1)=\phi(T_2)$. As $T_1$ and $T_2$ are even-reduced, this implies $T_1=T_2$, by Lemma \ref{lem:evenredinj}. Therefore, $W_1T_1=W_2T_2$, as required.
\end{proof}

We can now deduce the main result of this paper, that $\A_n(X)$ and $\lab_n(X)$ are isomorphic.

\begin{theorem}\label{thm:labeliso}
The algebraic and diagrammatic label algebras $\A_n(X)$ and $\lab_n(X)$ are isomorphic.
\end{theorem}
\begin{proof}
From Propositions \ref{prop:welldef} and \ref{prop:surj}, the map $\phi:\A_n(X) \to \lab_n(X)$ is a surjective homomorphism. From Corollary \ref{cor:WTspanning}, the set of all words in $WT$ form is a spanning set for $\A_n(X)$, and from Proposition \ref{prop:WTinj}, if two words in $WT$ form map to the same element of $\lab_n(X)$, then they are equal; it follows that $\phi$ is injective. Therefore $\phi$ is an isomorphism, and $\A_n(X) \cong \lab_n(X)$.
\end{proof}

We conclude this section with some corollaries that establish the outstanding claims from the beginning of Section \ref{s:ParityReduced}.

\begin{corollary}
A word $\mathbf{u}$ in $\A_n(X)$ is label-reduced if and only if concatenating the corresponding diagrammatic generators produces no boundary arcs.
\end{corollary}
\begin{proof}
We have that $\phi(\mathbf{u}) = \mu\, \undertilde{D}$ for some scalar $\mu$ and $\lab_n(X)$-diagram $\undertilde{D}$. The backwards implication follows from Lemma \ref{lem:labelreducedbdylinks}, since the number of boundary links in $\undertilde{D}$ must be exactly twice the number of label generators in $\mathbf{u}$, if no boundary arcs are formed in the concatenation. 

For the forwards implication, assume $\mathbf{u}$ is label-reduced. Using Proposition \ref{prop:evengen} or \ref{prop:oddgen}, we can write $\undertilde{D}$ as a product of the diagrammatic generators of $\lab_n(X)$ that produces no boundary arcs, so the number of boundary links in $\undertilde{D}$ is twice the total number of copies of $\dfup{a}{b}$, $\dfdo{a}{b}$, $\dwup{a}{b}$ and $\dwdo{a}{b}$ in the product, over all $\mathrm{a},\mathrm{b} \in X$. Let $\mathbf{v}$ be the corresponding word in the $\A_n(X)$ generators, so $\phi(\mathbf{v})=\undertilde{D}$, and the number of boundary links in $\undertilde{D}$ is twice the number of label generators in $\mathbf{v}$. Then since $\phi$ is an isomorphism, $\mathbf{u}= \mu \, \mathbf{v}$. Since $\mathbf{u}$ is label-reduced, $\mathbf{v}$ has at least as many label generators as $\mathbf{u}$. Thus the number of boundary links in $\undertilde{D}$ is at least twice the number of label generators in $\mathbf{u}$. Since the number of boundary links in $\undertilde{D}$ cannot be more than twice the number of label generators in $\mathbf{u}$, these numbers are equal. Hence concatenating the $\phi$-images of the generators of $\mathbf{u}$ produces no boundary arcs, as required.
\end{proof}

\begin{corollary}
Each word in $\A_n(X)$ maps to a scalar multiple of an $\lab_n(X)$-diagram of the same parity.
\end{corollary}
\begin{proof}
We have already shown in Corollary \ref{cor:evenwordtodiag} that each even word in $\A_n(X)$ maps to a scalar multiple of an even $\lab_n(X)$-diagram. Thus let $\mathbf{u}$ be an odd word in $\A_n(X)$. We know that $\phi(\mathbf{u})= \mu \, \undertilde{D}$ for some scalar $\mu$ and some $\lab_n(X)$-diagram $\undertilde{D}$. Suppose for sake of contradiction that $\undertilde{D}$ is even. Then by Corollary \ref{cor:evenredsurj}, there exists some even-reduced word $T$ in $\A_n(X)$ such that $\phi(T)=\undertilde{D}$. Then $\phi(\mathbf{u}) = \phi(\mu\, T)$, so $\mathbf{u} = \mu \, T$, since $\phi$ is an isomorphism. Since $\mathbf{u}$ is a scalar multiple of the even-reduced word $T$, $\mathbf{u}$ is even, contradicting our assumption that $\mathbf{u}$ was odd. Hence each odd word in $\A_n(X)$ maps to a scalar multiple of an odd diagram in $\lab_n(X)$.
\end{proof}

\appendix

\section{Additional relations}\label{app:redundant}
These relations follow from the defining relations of $\A_n(X)$ in Section \ref{s:algrels}.
\begin{lemma}
For $n\geq 2$,
\begin{align}
\wup{a}{b}e_{n-1}\wdo{c}{d} &= \fup{a}{c}\fdo{b}{d}. \label{eq:wupen1wdoapp}
\end{align}
\end{lemma}
\begin{proof}
We have
\begin{align}
\wup{a}{b}e_{n-1}\wdo{c}{d} &= \wdo{a}{b}e_1e_2\dots e_{n-2}e_{n-1}\wdo{c}{d} &&\text{by \eqref{eq:wupen1}} \\
&= \wdo{a}{b} e_1e_2 \dots e_{n-2}e_{n-1} e_{n-2} \dots e_2e_1 \wup{c}{d} &&\text{by \eqref{eq:en1wdo}} \\
&= \wdo{a}{b} e_1 \wup{c}{d} &&\text{by \eqref{eq:TLreduce}, repeatedly} \\
&= \fup{a}{c}\fdo{b}{d} &&\text{by \eqref{eq:wdoe1wup}.}
\end{align}
\end{proof}

\begin{lemma}
For $n$ odd, we have
\begin{align}
\fup{a}{b}E &= E \fup{a}{b}, \label{eq:fupEapp}\\
\fdo{a}{b}O &= O\fdo{a}{b}. \label{eq:fdoOapp}
\end{align}
\end{lemma}
\begin{proof}
This follows from the definitions of $E$ and $O$, and \eqref{eq:fupej} and \eqref{eq:fdoej}.
\end{proof}

\begin{lemma}
For $n$ odd, we have
\begin{align}
\wup{a}{b}\, O &= E\, \wup{a}{b}, \label{eq:wupOapp} \\
\wdo{a}{b}\, E &= O\, \wdo{a}{b}. \label{eq:wdoEapp}
\end{align}
\end{lemma}
\begin{proof}
This follows from the definitions of $E$ and $O$, and \eqref{eq:ejwup} and \eqref{eq:ejwdo}.
\end{proof}

\begin{lemma}
For odd $n>1$, we have
\begin{align}
e_1 \wup{a}{b}E = O\wup{a}{b}e_{n-1} &= \wdo{a}{b}E, \label{eq:e1wupEapp} \\
E\wdo{a}{b}e_1 = e_{n-1} \wdo{a}{b}O &= \wup{a}{b}O . \label{eq:Ewdowe1app}
\end{align}
\end{lemma}
\begin{proof}
We have
\begin{align}
e_1 \wup{a}{b}E &= (e_1 e_2 \dots e_{n-2}e_{n-1}) \wdo{a}{b} E &&\text{by \eqref{eq:e1wup}} \\
&= (e_1 e_2 \dots e_{n-2}e_{n-1}) O \wdo{a}{b} &&\text{by \eqref{eq:wdoEapp}} \\
&= (e_1 e_2 \dots e_{n-2}e_{n-1}) (e_1e_3 \dots e_{n-4}e_{n-2})\wdo{a}{b} &&\text{by definition of $O$} \\
&= (e_1e_2e_1)(e_3e_4e_3) \dots (e_{n-2}e_{n-1}e_{n-2}) \wdo{a}{b} &&\text{by \eqref{eq:eiej}, repeatedly} \\
&= e_1e_3 \dots e_{n-2} \wdo{a}{b} &&\text{by \eqref{eq:TLreduce}} \\
&= O \wdo{a}{b} &&\text{by definition of $O$} \\
&= \wdo{a}{b} E &&\text{by \eqref{eq:wdoEapp}}.
\end{align}
We also have
\begin{align}
O\wup{a}{b}e_{n-1} &= e_1e_3 \dots e_{n-4}e_{n-2}\wup{a}{b}e_{n-1} &&\text{by definition of O} \\
&= e_1 \wup{a}{b}e_2 e_4 \dots e_{n-5}e_{n-3}e_{n-1} &&\text{by \eqref{eq:ejwup}} \\
&= e_1 \wup{a}{b} E &&\text{by definition of }E.
\end{align}
This proves \eqref{eq:e1wupEapp}; \eqref{eq:Ewdowe1app} is proven analogously.
\end{proof}

\begin{lemma}
For $n$ odd, we have
\begin{align}
\fup{c}{a}\, E \, \fdo{b}{d} \, O &= \glab{a}{b}\, \wup{c}{d}\, O, \label{eq:fupEfdoOapp}\\[1mm]
\fdo{d}{b}\, O\, \fup{a}{c} \, E &= \glab{a}{b} \,\wdo{c}{d}\, E, \label{eq:fdoOfupEapp}\\[1mm]
\wup{c}{b}\, O\,  \fup{a}{d}\, E &= \glab{a}{b}\, \fup{c}{d}\, E, \label{eq:wupOfupEapp}\\[1mm]
\wdo{a}{c}\, E \, \fdo{b}{d}\, O &= \glab{a}{b}\, \fdo{c}{d}\, O. \label{eq:wdoEfdoOapp}
\end{align}
\end{lemma}
\begin{proof}
Assume $n$ is odd throughout this proof. For \eqref{eq:fupEfdoOapp}, we have, for $n>1$,
\begin{align}
\fup{c}{a}\, E \, \fdo{b}{d} \, O
&= E \, \fup{c}{a}\, \fdo{b}{d} \, O &&\text{by \eqref{eq:fupEapp}} \\
&= E\, \wdo{c}{b}e_1\wup{a}{d}\, O &&\text{by \eqref{eq:wdoe1wup}} \\
&= \wup{c}{b}\, O\,  \wup{a}{d} \, O &&\text{by \eqref{eq:Ewdowe1app}} \\
&= \glab{a}{b}\, \wup{c}{d}\, O &&\text{by \eqref{eq:wupOwupO}}.
\end{align}
For $n=1$, we have
\begin{align}
\fup{c}{a}\, E \, \fdo{b}{d} \, O
&= \fup{c}{a}\fdo{b}{d} &&\text{by definitions of $O$ and $E$} \\
&= \glab{a}{b}\, \wup{c}{d} &&\text{by \eqref{eq:fupfdo1}} \\
&= \glab{a}{b}\,\wup{c}{d}\, O &&\text{by definition of $O$}.
\end{align}
For \eqref{eq:fdoOfupEapp}, we have, for $n>1$,
\begin{align}
\fdo{d}{b}\, O\, \fup{a}{c}\, E 
&= O\, \fdo{d}{b}\, \fup{a}{c}\, E &&\text{by \eqref{eq:fdoOapp}}\\
&= O\, \fup{a}{c}\, \fdo{d}{b}\, E &&\text{by \eqref{eq:fupfdo}} \\
&= O\, \wup{a}{d}e_{n-1}\wdo{c}{b}\, E &&\text{by \eqref{eq:wupen1wdoapp}} \\
&= \wdo{a}{d}\, E \, \wdo{c}{b}\, E &&\text{by \eqref{eq:e1wupEapp}} \\
&= \glab{a}{b}\, \wdo{c}{d} \, E &&\text{by \eqref{eq:wdoEwdoE}}.
\end{align}
For $n=1$,
\begin{align}
\fdo{d}{b}\, O\, \fup{a}{c}\, E &= \fdo{d}{b}\fup{a}{c} &&\text{by definitions of }O\text{ and }E \\
&= \glab{a}{b}\, \wdo{c}{d} &&\text{by \eqref{eq:fdofup1}} \\
&= \glab{a}{b}\, \wdo{c}{d}\, E &&\text{by definition of }E.
\end{align}
For \eqref{eq:wupOfupEapp}, we have, for $n>1$,
\begin{align}
\wup{c}{b}\, O\, \fup{a}{d}\, E &= E\, \wup{c}{b}\fup{a}{d}\, E &&\text{by \eqref{eq:wupOapp}} \\
&= E\, \fup{c}{a}e_1 \wup{d}{b}\, E &&\text{by \eqref{eq:wupfup}} \\
&=\fup{c}{a}\, E\, e_1\wup{d}{b}E &&\text{by \eqref{eq:fupEapp}} \\
&= \fup{c}{a}\, E\, \wdo{d}{b}\, E &&\text{by \eqref{eq:e1wupEapp}} \\
&= \glab{a}{b}\, \fup{c}{d}\, E &&\text{by \eqref{eq:fupEwdoE}}.
\end{align}
For $n=1$,
\begin{align}
\wup{c}{b}\, O\, \fup{a}{d}\, E &= \wup{c}{b}\, \fup{a}{d} &&\text{by definitions of }O\text{ and }E \\
&= \glab{a}{b}\, \fup{c}{d} &&\text{by \eqref{eq:wupfup1}} \\
&= \glab{a}{b}\,\fup{c}{d}\, E &&\text{by definition of }E.
\end{align}
For \eqref{eq:wdoEfdoOapp}, we have, for $n>1$,
\begin{align}
\wdo{a}{c}\, E\, \fdo{b}{d}\, O &= O\, \wdo{a}{c}\, \fdo{b}{d}\, O &&\text{by \eqref{eq:wdoEapp}} \\
&= O\, \fdo{c}{b}e_{n-1}\wdo{a}{d}\, O &&\text{by \eqref{eq:wdofdo}} \\
&= \fdo{c}{b}\, O\, e_{n-1}\wdo{a}{d}\, O &&\text{by \eqref{eq:fdoOapp}} \\
&= \fdo{c}{b}\, O\, \wup{a}{d} \, O &&\text{by \eqref{eq:Ewdowe1app}} \\
&= \glab{a}{b}\, \fdo{c}{d}\, O &&\text{by \eqref{eq:fdoOwupO}}.
\end{align}
For $n=1$,
\begin{align}
\wdo{a}{c}\, E\, \fdo{b}{d}\, O &= \wdo{a}{c}\, \fdo{b}{d} &&\text{by definitions of }E\text{ and }O \\
&= \glab{a}{b}\, \fdo{c}{d} &&\text{by \eqref{eq:wdofdo1}} \\
&= \glab{a}{b}\, \fdo{c}{d}\, E &&\text{by definition of }E.
\end{align}
\end{proof}

\begin{lemma}
For $n$ even,
\begin{align}
\Theta \wup{a}{b} &= \Theta \wdo{a}{b}, \label{eq:Thetawupapp} \\[1mm]
\wup{a}{b} \Theta &= \wdo{a}{b} \Theta. \label{eq:wupThetaapp}
\end{align}
\end{lemma}
\begin{proof}
We have
\begin{align}
\Theta \wup{a}{b} &= e_1e_3e_5 \dots e_{n-1} \wup{a}{b} &&\text{by definition of }\Theta \\
&= e_3e_5 \dots e_{n-1} e_1\wup{a}{b} &&\text{by \eqref{eq:eiej}, repeatedly} \\
&= (e_3e_5 \dots e_{n-1})(e_1e_2e_3e_4e_5\dots e_{n-2}e_{n-1}) \wdo{a}{b} &&\text{by \eqref{eq:e1wup}} \\
&= e_1 (e_3e_2e_3)(e_5e_4e_5) \dots (e_{n-1}e_{n-2}e_{n-1}) \wdo{a}{b} &&\text{by \eqref{eq:eiej}, repeatedly} \\
&= e_1e_3e_5 \dots e_{n-1} \wdo{a}{b} &&\text{by \eqref{eq:TLreduce}} \\
&= \Theta \wdo{a}{b} &&\text{by definition of }\Theta.
\end{align}
This proves \eqref{eq:Thetawupapp}; \eqref{eq:wupThetaapp} is proven analogously.
\end{proof}

\begin{lemma}
For $n$ even,
\begin{align}
\wup{a}{b}\Theta \wup{c}{d} = \fup{a}{c} \fdo{b}{d} \Omega. 
\end{align}
\end{lemma}
\begin{proof}
We have
\begin{align}
\wup{a}{b}\Theta \wup{c}{d} &= \wup{a}{b}e_1e_3e_5\dots e_{n-1} \wup{c}{d} &&\text{by def. of $\Theta$}\\
&= \wup{a}{b} e_1 \wup{c}{d} e_2e_4\dots e_{n-2} &&\text{by \eqref{eq:ejwup}} \\
&= e_2 \wup{a}{b}\wup{c}{d} e_2e_4\dots e_{n-2} &&\text{by \eqref{eq:ejwup}} \\
&= e_2 \fup{a}{c}e_1e_2e_3 \dots e_{n-2}e_{n-1}\fdo{b}{d} e_2e_4\dots e_{n-2} &&\text{by \eqref{eq:wupwup}} \\
&= e_2 \fup{a}{c}e_1(e_2e_3 \dots e_{n-2}e_{n-1}) (e_2e_4 \dots e_{n-4}e_{n-2}) \fdo{b}{d} &&\text{by \eqref{eq:fdoej}}\\
&= e_2 \fup{a}{c}e_1(e_2e_3e_2)(e_4e_5e_4) \dots (e_{n-4}e_{n-3}e_{n-4})(e_{n-2}e_{n-1}e_{n-2})  \fdo{b}{d} &&\text{by \eqref{eq:eiej}}\\
&= e_2 \fup{a}{c}e_1(e_2e_4 \dots e_{n-4}e_{n-2})  \fdo{b}{d} &&\text{by \eqref{eq:TLreduce}} \\
&= \fup{a}{c} e_2e_1e_2 e_4 \dots e_{n-4}e_{n-2} \fdo{b}{d} &&\text{by \eqref{eq:fupej}} \\
&= \fup{a}{c} e_2e_4 \dots e_{n-4}e_{n-2} \fdo{b}{d} &&\text{by \eqref{eq:TLreduce}} \\
&= \fup{a}{c}\fdo{b}{d} e_2e_4 \dots e_{n-4}e_{n-2} &&\text{by \eqref{eq:fdoej}} \\
&= \fup{a}{c} \fdo{b}{d} \Omega &&\text{by def. of $\Omega$}.
\end{align}
\end{proof}

\begin{lemma}
For $n$ even,
\begin{align}
    e_1 \wup{a}{b} \Omega &= \Theta \wup{a}{b}, \\
    e_{n-1} \wdo{a}{b} \Omega &= \Theta \wdo{a}{b}.
\end{align}
\end{lemma}
\begin{proof}
We have 
\begin{align}
e_1 \wup{a}{b} \Omega
&= e_1 \wup{a}{b} e_2e_4\dots e_{n-4}e_{n-2} &&\text{by definition of $\Omega$}\\  
&= e_1 e_3e_5\dots e_{n-3}e_{n-1} \wup{a}{b} &&\text{by \eqref{eq:ejwup}} \\ 
&= \Theta \wup{a}{b} &&\text{by definition of $\Theta$}.
\end{align}
We also have
\begin{align}
e_{n-1} \wdo{a}{b} \Omega
&= e_{n-1} \wdo{a}{b} e_2e_4\dots e_{n-4}e_{n-2} &&\text{by definition of $\Omega$} \\
&= e_{n-1} e_1e_3\dots e_{n-5}e_{n-3}\wdo{a}{b} &&\text{by \eqref{eq:ejwdo}} \\
&=   e_1e_3\dots e_{n-5}e_{n-3}e_{n-1}\wdo{a}{b} &&\text{by \eqref{eq:eiej}} \\
&=  \Theta \wdo{a}{b} &&\text{by definition of $\Theta$}. 
\end{align}
\end{proof}

\begin{lemma}
For $n=1$,
\begin{align}
\fup{c}{a}\wdo{d}{b} &= \glab{a}{b}\,\fup{c}{d}, \label{eq:fupwdo1app}\\
\fdo{c}{b}\wup{a}{d} &= \glab{a}{b}\,\fdo{c}{d}, \label{eq:fdowup1app}\\
\wup{c}{b}\wup{a}{d} &= \glab{a}{b}\,\wup{c}{d}, \label{eq:wupwup1app}\\
\wdo{a}{d}\wdo{c}{b} &= \glab{a}{b}\,\wdo{c}{d}.\label{eq:wdowdo1app}
\end{align}
\end{lemma}
\begin{proof}
For \eqref{eq:fupwdo1app}, using \eqref{eq:fupEwdoE} and that $E=\id$ for $n=1$,
\begin{align}
\fup{c}{a}\wdo{d}{b} = \fup{c}{a}\, E\, \wdo{d}{b}\, E = \glab{a}{b}\, \fup{c}{d}\, E = \glab{a}{b}\, \fup{c}{d}.
\end{align}
Similarly, \eqref{eq:fdowup1app} is proven using \eqref{eq:fdoOwupO} and $O=\id$ for $n=1$.

For \eqref{eq:wupwup1app}, we have from \eqref{eq:wupwup} and \eqref{eq:fupfdo1} that
\begin{align}
\wup{c}{b}\wup{a}{d} &= \fup{c}{a}\fdo{b}{d} = \glab{a}{b} \wup{c}{d}.
\end{align}
Similarly, \eqref{eq:wdowdo1app} is proven using \eqref{eq:wdowdo} and \eqref{eq:fdofup1}.
\end{proof}

\vspace{10mm}
\noindent \textbf{Madeline Nurcombe} (\texttt{madeline.nurcombe@sydney.edu.au})

\noindent \textit{School of Mathematics and Statistics, University of Sydney}

\noindent \textit{Camperdown, Sydney, New South Wales 2006, Australia}


\begin{thebibliography}{99}
\bibitem{Ghost}
M. Nurcombe. 
\textit{The ghost algebra and the dilute ghost algebra},
J. Stat. Mech. 
(2024)
023102, 
\href{https://arxiv.org/abs/2308.11966}{\textsf{arXiv:2308.11966}}. 

\bibitem{Thesis}
M. Nurcombe.
\textit{Algebraic aspects of lattice models},
PhD thesis,
University of Queensland
(2024).

\bibitem{TL}
H. Temperley and E. Lieb. 
\textit{Relations between the ‘percolation’ and ‘colouring’ problem and other graph-theoretic problems associated with regular plane lattices: some exact results for the ‘percolation’ problem}, 
Proc. Roy. Soc. London Ser. A
\textbf{322}
(1971)
251--280. 

\bibitem{JonesTL}
V. Jones. 
\textit{Index for subfactors},
Invent. Math.
\textbf{72}
(1983)
1--25.

\bibitem{MNdGB}
S. Mitra, B. Nienhuis, J. de Gier and M. Batchelor.
\textit{Exact expressions for correlations in the ground state of the dense $O(1)$
loop model},
J. Stat. Mech.
(2004)
P09010,
\href{https://arxiv.org/pdf/cond-mat/0401245.pdf}{\textsf{arxiv:cond-mat/0401245}}.

\bibitem{dGN}
J. de Gier and A. Nichols.
\textit{The two-boundary Temperley-Lieb algebra},
J. Algebra
\textbf{321}
(2009)
1132--1167,
\href{https://arxiv.org/abs/math/0703338}{\textsf{arXiV:math/0703338}}.




\bibitem{MartinSaleur}
P. Martin and H. Saleur.
\textit{The blob algebra and the periodic Temperley-Lieb algebra},
Lett. Math. Phys.
\textbf{30}
(1994)
189--206,
\href{https://arxiv.org/abs/hep-th/9302094}{\textsf{arXiv:hep-th/9302094}}. 

\bibitem{TowersI}
P. Martin, R. Green and A. Parker. 
\textit{Towers of recollement and bases for diagram algebras: planar diagrams and a little beyond},
J. Algebra
\textbf{316}
(2007)
392--452,
\href{https://arxiv.org/abs/math/0610971}{\textsf{arXiv:math/0610971}}.

\bibitem{SympBlobPres}
R. Green, P. Martin and A. Parker. 
\textit{A presentation for the symplectic blob algebra}, 
J. Algebra Appl. 
\textbf{11} 
(2012) 
1250060, \href{https://arxiv.org/abs/1808.04206}{\textsf{arXiv:1808.04206 [math.RT]}}.


\bibitem{Tipunin}
P. Pearce, J. Rasmussen and I. Tipunin.
\textit{Critical dense polymers with Robin boundary conditions, half-integer Kac labels and $\Z_4$ fermions},
Nucl. Phys. B
\textbf{889}
(2014)
580--636, 
\href{https://arxiv.org/abs/1405.0550}{\textsf{arXiv:1405.0550 [hep-th]}}. 

\bibitem{GL}
J. Graham and G. Lehrer.
\textit{Cellular algebras},
Invent. Math.
\textbf{123}
(1996)
1--34.

\bibitem{BloteNienhuis}
H. Bl\"{o}te and B. Nienhuis.
\textit{Critical behaviour and conformal anomaly of the $O(n)$ model on the square lattice},
J. Phys. A: Math. Gen.
\textbf{22}
(1989)
1415--1438.


\bibitem{Grimm}
U. Grimm. 
\textit{Dilute algebras and solvable lattice models}, in \textit{Statistical Models, Yang-Baxter Equation and Related Topics}, 
Proceedings of the Satellite Meeting of STATPHYS-19, World Scientific (1996).


\bibitem{RSA}
D. Ridout and Y. Saint-Aubin. 
\textit{Standard modules, induction and the Temperley-Lieb algebra}, 
Adv. Theor. Math. Phys.
\textbf{18}
(2014)
957--1041,
\href{https://arxiv.org/abs/1204.4505}{\textsf{arXiv:hep-th/1204.4505 [hep-th]}}.



\end{thebibliography}
\end{document}